\documentclass{article}
\usepackage{xcolor}
\usepackage{amsmath}
\usepackage{amssymb}
\usepackage{amsthm}
\usepackage{hyperref}
\hypersetup{
	colorlinks=true,
	linkcolor=cyan,
	filecolor=blue,
	urlcolor=red,
	citecolor=green,
}

\newtheorem{theorem}{Theorem}
\newtheorem{lemma}{Lemma}

\newtheorem{observation}{Observation}

\title{Four fault-free $B_{n-2}$'s in $B_{n}$ under the random node fault model}
\author{Kaiyue Meng and Yuxing Yang
\footnote{They are in School of Mathematics and Information Science,
Henan Normal University,
Xinxiang, Henan 453007, China.
E-mails:\ mengkaiyue@stu.htu.edu.cn (Kaiyue Meng),\ yyx@htu.edu.cn (Yuxing Yang)}}
\date{}

\begin{document}
\maketitle

\begin{abstract}
Let $n\geq 4$. Each $B_{n-2}$ in $B_n$ has one of the forms $a_1a_2X^{n-2}$, $a_1X^{n-2}a_2$ and $X^{n-2}a_1a_2$. Let $1-p$ be the fault probiability of each node in the $n$-dimensional bubble-sort network $B_{n}$ under the random node fault model. In this paper, we determine the probability that there are four distinct fault-free $B_{n-2}$'s in $B_{n}$ by considering all possible combinatorial cases of the four fault-free $B_{n-2}$'s.
\end{abstract}

\section{Preliminaries}
Let $N_n$ be the set of positive integers no more than $n$ for any $n\geq 1$. In this paper, refer to \cite{Akers} for definition and properties of the $n$-dimensionl bubble-sort graph $B_n$, and refer to \cite{Bondy} for graph-theoretical concepts and notations. Please see \cite{Chang} for basic knowledge on the random fault model and its applications.

\begin{lemma}[See \cite{Yang2015}]\label{Yang}
For $n\geq 4$, there are $(2+1)!\binom{n}{2}=3n(n-1)$ distinct $B_{n-2}$'s in $B_n$, which can be divided into three disjoint sets $H_{1,2}$, $H_{1,n}$ and $H_{n-1,n-2}$, where
$$H_{1,2}=\{X^{n-2}a_1a_2:a_1,a_2\in N_n\mbox{ and }a_1\neq a_2\},$$
$$H_{1,n}=\{a_1X^{n-2}a_2:a_1,a_2\in N_n\mbox{ and }a_1\neq a_2\},$$
$$H_{n-1,n-2}=\{a_1a_2X^{n-2}:a_1,a_2\in N_n\mbox{ and }a_1\neq a_2\}.$$
\end{lemma}

\begin{observation}\label{ob:1}
The $n(n-1)$ $B_{n-2}$-subnetworks in any one of $\{H_{1,2},H_{1,n},H_{n-1,n}\}$ are pairwise disjoint.
\end{observation}

\begin{observation}\label{ob:2}
A $B_{n-2}$-subnetwork $ab X^{n-2}$ in $H_{1,2}$ and a $B_{n-2}$-subnetwork $cX^{n-2}d$ in $H_{1,n}$ are not disjoint if and only if $a=c$ and $b\neq d$.
\end{observation}

\begin{observation}\label{ob:3}
A $B_{n-2}$-subnetwork $X^{n-2}ab$ in $H_{n-1,n}$ and a $B_{n-2}$-subnetwork $cX^{n-2}d$ in $H_{1,n}$ are not disjoint if and only if $a\neq c$ and $b=d$.
\end{observation}

\begin{observation}\label{ob:4}
A $B_{n-2}$-subnetwork $ab X^{n-2}$ in $H_{1,2}$ and a $B_{n-2}$-subnetwork $X^{n-2}cd$ in $H_{n-1,n}$ are not disjoint if and only if $\{a,b\}\cap\{c,d\}=\emptyset$.
\end{observation}

\section{Main results}

In the rest of the paper, we will derive the probiability that there are four distinct fault-free $B_{n-2}$'s in $B_{n}$. Denote by $P(i,j,k)$
the probability that there are four fault-free $B_{n-2}$'s, $i$ of which lie in $H_{1,2}$, $j$ of which lie in $H_{1,n-1}$ and $k$ of which lie in $H_{n-1,n-2}$.

\begin{theorem}\label{th1}
$P(4,0,0)=P(0,4,0)=P(0,0,4)=\frac{1}{8}(n^8-4n^7+14n^5-6n^4-16n^3+5n^2+6n)p^{4(n-2)!}$.
\end{theorem}
\begin{proof}
Clearly, $P(4,0,0)=P(0,4,0)=P(0,0,4)$. It suffices to derive $P(4,0,0)$. Since $|H_{1,2}|=n(n-1)$, there are $\binom{n(n-1)}{4}$ ways to choose four $B_{n-2}$'s $a_1 a_2 X^{n-2}$, $b_1 b_2 X^{n-2}$, $c_1 c_2 X^{n-2}$ and $d_1 d_2 X^{n-2}$ in $H_{1,2}$. Observation \ref{ob:1} implies that any four $B_{n-2}$'s in $H_{1,2}$ are disjoint. Thus, $|V(a_1 a_2 X^{n-2})\cup V(b_1 b_2 X^{n-2})\cup V(c_1 c_2 X^{n-2})\cup V(d_1 d_2 X^{n-2})|=|V(a_1 a_2 X^{n-2})|+|V(b_1 b_2 X^{n-2})|+|V(c_1 c_2 X^{n-2})|+|V(d_1 d_2 X^{n-2})|=4(n-2)!$, and so $P(0,4,0)=P(0,0,4)=P(4,0,0)=3\binom{n(n-1)}{4}p^{4(n-2)!}=\frac{1}{8}(n^8-4n^7+14n^5-6n^4-16n^3+5n^2+6n)p^{4(n-2)!}$.
\end{proof}

\begin{theorem}\label{th2}
$P(3,1,0)=\frac{1}{6}(n^8-7n^7+21n^6-35n^5+34n^4-18n^3+4n^2)p^{4(n-2)!}+\frac{1}{2}(n^7-7n^6+21n^5-35n^4+34n^3-18n^2+4n)p^{4(n-2)!-(n-3)!}+\frac{1}{2}(n^6-8n^5+25n^4-40n^3+34n^2-12n)p^{4(n-2)!-2(n-3)!}+\frac{1}{6}(n^5-10n^4+35n^3-50n^2+24n)p^{4(n-2)!-3(n-3)!}$.
\end{theorem}
\begin{proof}
Denote by $a_1 a_2 X^{n-2}\in H_{1,2}$, $b_1 b_2 X^{n-2}\in H_{1,2}$, $c_1 c_2 X^{n-2}\in H_{1,2}$ and $d_1 X^{n-2} d_2\in H_{1,n}$ these four $B_{n-2}$'s. Observation \ref{ob:1} implies that $a_1 a_2 X^{n-2}$, $b_1 b_2 X^{n-2}$ and $c_1 c_2 X^{n-2}$ are pairwise disjoint.

Since we have implicitly sorted the values of $a_1 a_2$, $b_1 b_2$ and $c_1 c_2$, the number of ways to choose four $B_{n-2}$'s as required needs to be divided by 6. It will be applied directly in following analyses.

We will analyze this problem in 5 cases as follows.

{\it Case 1.} $d_1\notin\{a_1, b_1, c_1\}$.

{\it Case 1.1.} $a_1$, $b_1$ and $c_1$ are pairwise distinct.

There are $n$ ways to choose $d_1$ from $N_n$, $n-1$ ways to choose $a_1$ from $N_n\setminus\{d_1\}$, $n-2$ ways to choose $b_1$ from $N_n\setminus\{a_1, d_1\}$, $n-3$ ways to choose $c_1$ from $N_n\setminus\{a_1, b_1, d_1\}$, $n-1$ ways to choose $a_2$ from $N_n\setminus\{a_1\}$, $n-1$ ways to choose $b_2$ from $N_n\setminus\{b_1\}$, $n-1$ ways to choose $c_2$ from $N_n\setminus\{c_1\}$, $n-1$ ways to choose $d_2$ from $N_n\setminus\{d_1\}$.
Therefore, there are $n(n-1)^5(n-2)(n-3)=n^8-10n^7+41n^6-90n^5+115n^4-86n^3+35n^2-6n$ ways to choose four $B_{n-2}$'s as required.

{\it Case 1.2.} $a_1=b_1\neq c_1$ or $a_1=c_1\neq b_1$ or $b_1=c_1\neq a_1$.

W.l.o.g, assume that $a_1=b_1\neq c_1$. There are $n$ ways to choose $d_1$ from $N_n$, $n-1$ ways to choose $a_1(=b_1)$ from $N_n\setminus\{d_1\}$, $n-2$ ways to choose $c_1$ from $N_n\setminus\{a_1, d_1\}$, $n-1$ ways to choose $a_2$ from $N_n\setminus\{a_1\}$, $n-2$ ways to choose $b_2$ from $N_n\setminus\{b_1, a_2\}$, $n-1$ ways to choose $c_2$ from $N_n\setminus\{c_1\}$, $n-1$ ways to choose $d_2$ from $N_n\setminus\{d_1\}$.
Therefore, there are $3n(n-1)^4(n-2)^2=3(n^7-8n^6+26n^5-44n^4+41n^3-10n^2+4n)$ ways to choose four $B_{n-2}$'s as required.

{\it Case 1.3.} $a_1=b_1=c_1$.

There are $n$ ways to choose $d_1$ from $N_n$, $n-1$ ways to choose $a_1(=b_1=c_1)$ from $N_n\setminus\{d_1\}$, $n-1$ ways to choose $a_2$ from $N_n\setminus\{a_1\}$, $n-2$ ways to choose $b_2$ from $N_n\setminus\{b_1, a_2\}$, $n-3$ ways to choose $c_2$ from $N_n\setminus\{c_1, a_2, b_2\}$, $n-1$ ways to choose $d_2$ from $N_n\setminus\{d_1\}$.
Therefore, there are $n(n-1)^3(n-2)(n-3)=n^6-8n^5+24n^4-34n^3+23n^2-6n$ ways to choose four $B_{n-2}$'s as required.

In summary, there are $n^8-7n^7+18n^6-20n^5+7n^4+3n^3-2n^2$ ways to choose four $B_{n-2}$'s as required in Case 1.
Observation \ref{ob:2} implies that for any $i\in\{a, b, c\}$, $i_1 i_2 X^{n-2}$ and $d_1 X^{n-2} d_2$ are disjoint. Thus, $|V(a_1 a_2 X^{n-2})\cup V(b_1 b_2 X^{n-2})\cup V(c_1 c_2 X^{n-2})\cup V(d_1 X^{n-2} d_2)|=|V(a_1 a_2 X^{n-2})|+|V(b_1 b_2 X^{n-2})|+|V(c_1 c_2 X^{n-2})|+|V(d_1 X^{n-2} d_2)|=4(n-2)!$. 
The probability $P_1$ that there are four fault-free $B_{n-2}$'s chosen as in Case 1 is $\frac{1}{6}(n^8-7n^7+18n^6-20n^5+7n^4+3n^3-2n^2)p^{4(n-2)!}$

{\it Case 2.} $a_1=d_1\notin\{b_1, c_1\}$, $a_2\neq d_2$ or $b_1=d_1\notin\{a_1, c_1\}$, $b_2\neq d_2$ or $c_1=d_1\notin\{a_1, b_1\}$, $c_2\neq d_2$.

W.l.o.g, assume that $a_1=d_1\notin\{b_1, c_1\}$ and $a_2\neq d_2$.

{\it Case 2.1.} $b_1\neq c_1$.

There are $n$ ways to choose $a_1(=d_1)$ from $N_n$, $n-1$ ways to choose $b_1$ from $N_n\setminus\{d_1\}$, $n-2$ ways to choose $c_1$ from $N_n\setminus\{b_1, d_1\}$, $n-1$ ways to choose $a_2$ from $N_n\setminus\{a_1\}$, $n-1$ ways to choose $b_2$ from $N_n\setminus\{b_1\}$, $n-1$ ways to choose $c_2$ from $N_n\setminus\{c_1\}$, $n-2$ ways to choose $d_2$ from $N_n\setminus\{d_1, a_2\}$.
Therefore, there are $n(n-1)^4(n-2)^2=n^7-8n^6+26n^5-44n^4+41n^3-20n^2+4n$ ways to choose four $B_{n-2}$'s as required.

{\it Case 2.2.} $b_1=c_1$.

There are $n$ ways to choose $a_1(=d_1)$ from $N_n$, $n-1$ ways to choose $b_1(=c_1)$ from $N_n\setminus\{d_1\}$, $n-1$ ways to choose $a_2$ from $N_n\setminus\{a_1\}$, $n-1$ ways to choose $b_2$ from $N_n\setminus\{b_1\}$, $n-2$ ways to choose $c_2$ from $N_n\setminus\{c_1, b_2\}$, $n-2$ ways to choose $d_2$ from $N_n\setminus\{d_1, a_2\}$.
Therefore, there are $n(n-1)^3(n-2)^2=n^6-7n^5+19n^4-25n^3+16n^2-4n$ ways to choose four $B_{n-2}$'s as required.

In summary, there are $3(n^7-7n^6+19n^5-25n^4+16n^3-4n^2)$ ways to choose four $B_{n-2}$'s as required in Case 2.
Observation \ref{ob:2} implies that $a_1 a_2 X^{n-2}$ and $d_1 X^{n-2} d_2$ are not disjoint and $V(a_1 a_2 X^{n-2})\cap V(d_1 X^{n-2} d_2)=V(a_1 a_2 X^{n-3} d_2)$, and for any $i\in\{b, c\}$, $i_1 i_2 X^{n-2}$ and $d_1 X^{n-2} d_2$ are disjoint. Thus, $|V(a_1 a_2 X^{n-2})\cup V(b_1 b_2 X^{n-2})\cup V(c_1 c_2 X^{n-2})\cup V(d_1 X^{n-2} d_2)|=|V(a_1 a_2 X^{n-2})|+|V(b_1 b_2 X^{n-2})|+|V(c_1 c_2 X^{n-2})|+|V(d_1 X^{n-2}$ $d_2)|-|V(a_1 a_2 X^{n-2})\cap V(d_1 X^{n-2} d_2)|=4(n-2)!-(n-3)!$. 
The probability $P_2$ that there are four fault-free $B_{n-2}$'s chosen as in Case 2 is $\frac{1}{6}\times 3(n^7-7n^6+19n^5-25n^4+16n^3-4n^2)p^{4(n-2)!-(n-3)!}$

{\it Case 3.} $a_1=d_1\notin\{b_1, c_1\}$, $a_2=d_2$ or $b_1=d_1\notin\{a_1, c_1\}$, $b_2=d_2$ or $c_1=d_1\notin\{a_1, b_1\}$, $c_2=d_2$.

W.l.o.g, assume that $a_1=d_1\notin\{b_1, c_1\}$ and $a_2=d_2$.

{\it Case 3.1.} $b_1\neq c_1$.

There are $n$ ways to choose $a_1(=d_1)$ from $N_n$, $n-1$ ways to choose $b_1$ from $N_n\setminus\{d_1\}$, $n-2$ ways to choose $c_1$ from $N_n\setminus\{b_1, d_1\}$, $n-1$ ways to choose $a_2(=d_2)$ from $N_n\setminus\{a_1\}$, $n-1$ ways to choose $b_2$ from $N_n\setminus\{b_1\}$, $n-1$ ways to choose $c_2$ from $N_n\setminus\{c_1\}$.
Therefore, there are $n(n-1)^4(n-2)=n^6-6n^5+14n^4-16n^3+9n^2-2n$ ways to choose four $B_{n-2}$'s as required.

{\it Case 3.2.} $b_1=c_1$.

There are $n$ ways to choose $a_1(=d_1)$ from $N_n$, $n-1$ ways to choose $b_1(=c_1)$ from $N_n\setminus\{d_1\}$, $n-1$ ways to choose $a_2(=d_2)$ from $N_n\setminus\{a_1\}$, $n-1$ ways to choose $b_2$ from $N_n\setminus\{b_1\}$, $n-2$ ways to choose $c_2$ from $N_n\setminus\{c_1, b_2\}$.
Therefore, there are $n(n-1)^3(n-2)=n^5-5n^4+9n^3-7n^2+2n$ ways to choose four $B_{n-2}$'s as required.

In summary, there are $3(n^6-5n^5+9n^4-7n^3+2n^2)$ ways to choose four $B_{n-2}$'s as required in Case 3.
Observation \ref{ob:2} implies that for any $i\in\{a, b, c\}$, $i_1 i_2 X^{n-2}$ and $d_1 X^{n-2} d_2$ are disjoint. Thus, $|V(a_1 a_2 X^{n-2})\cup V(b_1 b_2 X^{n-2})\cup V(c_1 c_2 X^{n-2})\cup V(d_1 X^{n-2} d_2)|=|V(a_1 a_2 X^{n-2})|+|V(b_1 b_2 X^{n-2})|+|V(c_1 c_2 X^{n-2})|+|V(d_1 X^{n-2}$ $d_2)|=4(n-2)!$.  
The probability $P_3$ that there are four fault-free $B_{n-2}$'s chosen as in Case 3 is $\frac{1}{6}\times 3(n^6-5n^5+9n^4-7n^3+2n^2)p^{4(n-2)!}$

{\it Case 4.} $a_1=b_1=d_1\neq c_1$ or $a_1=c_1=d_1\neq b_1$ or $b_1=c_1=d_1\neq a_1$.

W.l.o.g, assume that $a_1=b_1=d_1\neq c_1$.

{\it Case 4.1.} $a_2\neq d_2$ and $b_2\neq d_2$.

There are $n$ ways to choose $a_1(=b_1=d_1)$ from $N_n$, $n-1$ ways to choose $d_2$ from $N_n\setminus\{d_1\}$, $n-2$ ways to choose $a_2$ from $N_n\setminus\{a_1, d_2\}$, $n-3$ ways to choose $b_2$ from $N_n\setminus\{b_1, a_2, d_2\}$, $n-1$ ways to choose $c_1$ from $N_n\setminus\{d_1\}$, $n-1$ ways to choose $c_2$ from $N_n\setminus\{c_1\}$.
Therefore, there are $n(n-1)^3(n-2)(n-3)=n^6-8n^5+24n^4-34n^3+23n^2-6n$ ways to choose four $B_{n-2}$'s as required. 
Observation \ref{ob:2} implies that for any $i\in\{a, b\}$, $i_1 i_2 X^{n-2}$ and $d_1 X^{n-2} d_2$ are not disjoint and $V(i_1 i_2 X^{n-2})\cap V(d_1 X^{n-2} d_2)=V(i_1 i_2 X^{n-3} d_2)$, and $c_1 c_2 X^{n-2}$ and $d_1 X^{n-2} d_2$ are disjoint. Thus, $|V(a_1 a_2 X^{n-2})\cup V(b_1 b_2 X^{n-2})\cup V(c_1 c_2 X^{n-2})\cup V(d_1 X^{n-2} d_2)|=|V(a_1 a_2 X^{n-2})|+|V(b_1 b_2 X^{n-2})|+|V(c_1 c_2 X^{n-2})|+|V(d_1 X^{n-2} d_2)|-|V(a_1 a_2 X^{n-2})\cap V(d_1 X^{n-2} d_2)|-|V(b_1 b_2$ $X^{n-2})\cap V(d_1 X^{n-2} d_2)|=4(n-2)!-2(n-3)!$.
The probability $P_{4.1}$ that there are four fault-free $B_{n-2}$'s chosen as in Case 4.1 is $\frac{1}{6}(n^6-8n^5+24n^4-34n^3+23n^2-6n)p^{4(n-2)!-2(n-3)!}$.

{\it Case 4.2.} $a_2=d_2\neq b_2$ or $b_2=d_2\neq a_2$.

W.l.o.g, assume that $a_2=d_2\neq b_2$.
There are $n$ ways to choose $a_1=b_1=d_1$ from $N_n$, $n-1$ ways to choose $a_2=d_2$ from $N_n\setminus\{d_1\}$, $n-2$ ways to choose $b_2$ from $N_n\setminus\{b_1, d_2\}$, $n-1$ ways to choose $c_1$ from $N_n\setminus\{d_1\}$, $n-1$ ways to choose $c_2$ from $N_n\setminus\{c_1\}$.
Therefore, there are $2n(n-1)^3(n-2)=2(n^5-5n^4+9n^3-7n^2+2n)$ ways to choose four $B_{n-2}$'s as required. 
Observation \ref{ob:2} implies that for any $i\in\{a, c\}$, $i_1 i_2 X^{n-2}$ and $d_1 X^{n-2} d_2$ are disjoint, and $b_1 b_2 X^{n-2}$ and $d_1 X^{n-2} d_2$ are not disjoint and $V(b_1 b_2 X^{n-2})\cap V(d_1 X^{n-2} d_2)=V(b_1 b_2 X^{n-3} d_2)$. Thus, $|V(a_1 a_2 X^{n-2})\cup V(b_1 b_2 X^{n-2})\cup V(c_1 c_2 X^{n-2})\cup V(d_1 X^{n-2} d_2)|=|V(a_1 a_2 X^{n-2})|+|V(b_1 b_2 X^{n-2})|+|V(c_1 c_2 X^{n-2})|+|V(d_1 X^{n-2} d_2)|-|V(b_1 b_2 X^{n-2})\cap V(d_1 X^{n-2} d_2)|=4(n-2)!-(n-3)!$.
The probability $P_{4.2}$ that there are four fault-free $B_{n-2}$'s chosen as in Case 4.2 is $\frac{1}{6}\times 2(n^5-5n^4+9n^3-7n^2+2n)p^{4(n-2)!-(n-3)!}$.

In summary, the probability $P_4$ that there are four fault-free $B_{n-2}$'s chosen as in Case 4 is $3(P_{4.1}+P_{4.2})=(n^5-5n^4+9n^3-7n^2+2n)p^{4(n-2)!-(n-3)!}+\frac{1}{2}(n^6-8n^5+24n^4-34n^3+23n^2-6n)p^{4(n-2)!-2(n-3)!}$.

{\it Case 5.} $a_1=b_1=c_1=d_1$.

Clearly, $a_2$, $b_2$ and $c_2$ are pairwise distinct in this case.

{\it Case 5.1.} $d_2\notin\{a_2, b_2, c_2\}$.

There are $n$ ways to choose $a_1=b_1=c_1=d_1$ from $N_n$, $n-1$ ways to choose $d_2$ from $N_n\setminus\{d_1\}$, $n-2$ ways to choose $a_2$ from $N_n\setminus\{a_1, d_2\}$, $n-3$ ways to choose $b_2$ from $N_n\setminus\{b_1, d_2, a_2\}$, $n-4$ ways to choose $c_2$ from $N_n\setminus\{c_1, d_2, a_2, b_2\}$.
Therefore, there are $n(n-1)(n-2)(n-3)(n-4)=n^5-10n^4+35n^3-50n^2+24n$ ways to choose four $B_{n-2}$'s as required.
Observation \ref{ob:2} implies that for any $i\in\{a, b, c\}$, $i_1 i_2 X^{n-2}$ and $d_1 X^{n-2} d_2$ are not disjoint and $V(i_1 i_2 X^{n-2})\cap V(d_1 X^{n-2} d_2)=V(i_1 i_2 X^{n-3} d_2)$. Thus, $|V(a_1 a_2 X^{n-2})\cup V(b_1 b_2 X^{n-2})\cup V(c_1 c_2 X^{n-2})\cup V(d_1 X^{n-2} d_2)|=|V(a_1 a_2 X^{n-2})|+|V(b_1 b_2 X^{n-2})|+|V(c_1 c_2 X^{n-2})|+|V(d_1 X^{n-2} d_2)|-|V(a_1 a_2 X^{n-2})\cap V(d_1 X^{n-2}$ $d_2)|-|V(b_1 b_2 X^{n-2})\cap V(d_1 X^{n-2} d_2)|-|V(c_1 c_2 X^{n-2})\cap V(d_1 X^{n-2} d_2)|=4(n-2)!-3(n-3)!$.
The probability $P_{5.1}$ that there are four fault-free $B_{n-2}$'s chosen as in Case 5.1 is $\frac{1}{6}(n^5-10n^4+35n^3-50n^2+24n)p^{4(n-2)!-3(n-3)!}$.

{\it Case 5.2.} $d_2\in\{a_2, b_2, c_2\}$.

There are 3 possible cases, w.l.o.g, assume that $d_2=a_2$.
There are $n$ ways to choose $a_1=b_1=c_1=d_1$ from $N_n$, $n-1$ ways to choose $d_2=a_2$ from $N_n\setminus\{d_1\}$, $n-2$ ways to choose $b_2$ from $N_n\setminus\{b_1, d_2\}$, $n-3$ ways to choose $c_2$ from $N_n\setminus\{c_1, d_2, b_2\}$.
Therefore, there are $3n(n-1)(n-2)(n-3)=3(n^4-6n^3+11n^2-6n)$ ways to choose four $B_{n-2}$'s as required.
Observation \ref{ob:2} implies that $a_1 a_2 X^{n-2}$ and $d_1 X^{n-2} d_2$ are disjoint, and for any $i\in\{b, c\}$, $i_1 i_2 X^{n-2}$ and $d_1 X^{n-2} d_2$ are not disjoint and $V(i_1 i_2 X^{n-2})\cap V(d_1 X^{n-2} d_2)=V(i_1 i_2 X^{n-3} d_2)$. Thus, $|V(a_1 a_2 X^{n-2})\cup V(b_1 b_2 X^{n-2})\cup V(c_1 c_2 X^{n-2})\cup V(d_1 X^{n-2} d_2)|=|V(a_1 a_2 X^{n-2})|+|V(b_1 b_2 X^{n-2})|+|V(c_1 c_2 X^{n-2})|+|V(d_1 X^{n-2} d_2)|-|V(b_1 b_2 X^{n-2})\cap V(d_1 X^{n-2} d_2)|-|V(c_1 c_2 X^{n-2})\cap V(d_1 X^{n-2} d_2)|=4(n-2)!-2(n-3)!$.
The probability $P_{5.2}$ that there are four fault-free $B_{n-2}$'s chosen as in Case 5.2 is $\frac{1}{6}\times 3(n^4-6n^3+11n^2-6n)p^{4(n-2)!-2(n-3)!}$.

In summary, the probability $P_5$ that there are four fault-free $B_{n-2}$'s chosen as in Case 5 is $P_{5.1}+P_{5.2}=\frac{1}{2}(n^4-6n^3+11n^2-6n)p^{4(n-2)!-2(n-3)!}+\frac{1}{6}(n^5-10n^4+35n^3-50n^2+24n)p^{4(n-2)!-3(n-3)!}$.

By the above computations, $P(3,1,0)=\sum_{i=1}^{5}P_i=\frac{1}{6}(n^8-7n^7+21n^6-35n^5+34n^4-18n^3+4n^2)p^{4(n-2)!}+\frac{1}{2}(n^7-7n^6+21n^5-35n^4+34n^3-18n^2+4n)p^{4(n-2)!-(n-3)!}+\frac{1}{2}(n^6-8n^5+25n^4-40n^3+34n^2-12n)p^{4(n-2)!-2(n-3)!}$ $+\frac{1}{6}(n^5-10n^4+35n^3-50n^2+24n)p^{4(n-2)!-3(n-3)!}$.
\end{proof}

\begin{theorem}\label{th3}
$P(1,3,0)=P(0,3,1)=P(0,1,3)=\frac{1}{6}(n^8-7n^7+21n^6-35n^5+34n^4-18n^3+4n^2)p^{4(n-2)!}+\frac{1}{2}(n^7-7n^6+21n^5-35n^4+34n^3-18n^2+4n)p^{4(n-2)!-(n-3)!}+\frac{1}{2}(n^6-8n^5+25n^4-40n^3+34n^2-12n)p^{4(n-2)!-2(n-3)!}+\frac{1}{6}(n^5-10n^4+35n^3-50n^2+24n)p^{4(n-2)!-3(n-3)!}$.
\end{theorem}
\begin{proof}
All of these three scenarios are similar to Theorem \ref{th2}, it is easy to get that $P(1,3,0)=P(0,3,1)=P(0,1,3)=P(3,1,0)=\frac{1}{6}(n^8-7n^7+21n^6-35n^5+34n^4-18n^3+4n^2)p^{4(n-2)!}+\frac{1}{2}(n^7-7n^6+21n^5-35n^4+34n^3-18n^2+4n)p^{4(n-2)!-(n-3)!}+\frac{1}{2}(n^6-8n^5+25n^4-40n^3+34n^2-12n)p^{4(n-2)!-2(n-3)!}+\frac{1}{6}(n^5-10n^4+35n^3-50n^2+24n)p^{4(n-2)!-3(n-3)!}$.
\end{proof}

\begin{theorem}\label{th4}
$P(3,0,1)=\frac{4}{3}(8n^5-50n^4+115n^3-115n^2+42n)p^{4(n-2)!}+(8n^6-74n^5+265n^4-460n^3+387n^2-126n)p^{4(n-2)!-(n-4)!}+(2n^7-25n^6+125n^5-320n^4+443n^3-315n^2+90n)p^{4(n-2)!-2(n-4)!}+\frac{1}{6}(n^8-16n^7+105n^6-365n^5+724n^4-819n^3+490n^2-120n)p^{4(n-2)!-3(n-4)!}$.
\end{theorem}
\begin{proof}
Denote by $a_1 a_2 X^{n-2}\in H_{1,2}$, $b_1 b_2 X^{n-2}\in H_{1,2}$, $c_1 c_2 X^{n-2}\in H_{1,2}$ and $X^{n-2} d_1 d_2\in H_{n-1,n}$ these four $B_{n-2}$'s. Note that $a_1 a_2 X^{n-2}$, $b_1 b_2 X^{n-2}$ and $c_1 c_2 X^{n-2}$ are pairwise disjoint.

Since we have implicitly sorted the values of $a_1 a_2$, $b_1 b_2$ and $c_1 c_2$, the number of ways to choose four $B_{n-2}$'s as required needs to be divided by 6. It will be applied directly in following analyses.

In the following analysis, we will use a 3-digit string to represent different cases. These three numbers represent $|\{a_1, a_2\}\cap\{d_1, d_2\}|$, $|\{b_1, b_2\}\cap\{d_1, d_2\}|$ and $|\{c_1, c_2\}\cap\{d_1, d_2\}|$ respectively. For example, 012 represents the case $\{a_1, a_2\}\cap\{d_1, d_2\}=\emptyset$, $|\{b_1, b_2\}\cap\{d_1, d_2\}|=1$ and $\{c_1, c_2\}=\{d_1, d_2\}$.

We will analyze this problem in 9 cases as follows.

{\it Case 1.} 221 or 212 or 122.

W.l.o.g, assume that the case of 221, that is, $\{a_1, a_2\}=\{b_1, b_2\}=\{d_1, d_2\}$ and $|\{c_1, c_2\}\cap\{d_1, d_2\}|=1$.
There are 2 possible cases $a_1=b_2=d_1$, $a_2=b_1=d_2$ and $a_1=b_2=d_2$, $a_2=b_1=d_1$ such that $\{a_1, a_2\}=\{b_1, b_2\}=\{d_1, d_2\}$. There are 4 possible cases ($\romannumeral1$) $c_1=d_1$, $c_2\notin\{d_1, d_2\}$, and ($\romannumeral2$) $c_1=d_2$, $c_2\notin\{d_1, d_2\}$, and ($\romannumeral3$) $c_2=d_1$, $c_1\notin\{d_1, d_2\}$, and ($\romannumeral4$) $c_2=d_2$, $c_1\notin\{d_1, d_2\}$ such that $|\{c_1, c_2\}\cap\{d_1, d_2\}|=1$. Therefore, there are 8 possible scenarios for this situation. W.l.o.g, assume that $a_1=b_2=d_1$, $a_2=b_1=d_2$, $c_1=d_1$, $c_2\notin\{d_1, d_2\}$.

There are $n$ ways to choose $a_1(=b_2=d_1=c_1)$ from $N_n$, $n-1$ ways to choose $a_2(=b_1=d_2)$ from $N_n\setminus\{a_1\}$, $n-2$ ways to choose $c_2$ from $N_n\setminus\{d_1, d_2\}$.
Therefore, there are $3\times 8n(n-1)(n-2)=24(n^3-3n^2+2n)$ ways to choose four $B_{n-2}$'s as required. Observation \ref{ob:4} implies that for any $i\in\{a, b, c\}$, $i_1 i_2 X^{n-2}$ and $X^{n-2} d_1 d_2$ are disjoint. Thus, $|V(a_1 a_2 X^{n-2})\cup V(b_1 b_2 X^{n-2})\cup V(c_1 c_2 X^{n-2})\cup V(X^{n-2} d_1 d_2)|=|V(a_1 a_2 X^{n-2})|+|V(b_1 b_2 X^{n-2})|+|V(c_1 c_2 X^{n-2})|+|V(X^{n-2} d_1 d_2)|=4(n-2)!$. 
The probability $P_1$ that there are four fault-free $B_{n-2}$'s chosen as in Case 1 is $\frac{1}{6}\times 24(n^3-3n^2+2n)p^{4(n-2)!}$

{\it Case 2.} 220 or 202 or 022.

W.l.o.g, assume that the case of 220, that is, $\{a_1, a_2\}=\{b_1, b_2\}=\{d_1, d_2\}$ and $\{c_1, c_2\}\cap\{d_1, d_2\}=\emptyset$.
There are 2 possible cases $a_1=b_2=d_1$, $a_2=b_1=d_2$ and $a_1=b_2=d_2$, $a_2=b_1=d_1$ such that $\{a_1, a_2\}=\{b_1, b_2\}=\{d_1, d_2\}$. W.l.o.g, assume that $a_1=b_2=d_1$, $a_2=b_1=d_2$.

There are $n$ ways to choose $a_1(=b_2=d_1)$ from $N_n$, $n-1$ ways to choose $a_2(=b_1=d_2)$ from $N_n\setminus\{a_1\}$, $n-2$ ways to choose $c_1$ from $N_n\setminus\{d_1, d_2\}$, $n-3$ ways to choose $c_2$ from $N_n\setminus\{d_1, d_2, c_1\}$.
Therefore, there are $3\times 2n(n-1)(n-2)(n-3)=6(n^4-6n^3+11n^2-6n)$ ways to choose four $B_{n-2}$'s as required. Observation \ref{ob:4} implies that for any $i\in\{a, b\}$, $i_1 i_2 X^{n-2}$ and $X^{n-2} d_1 d_2$ are disjoint, and $c_1 c_2 X^{n-2}$ and $X^{n-2} d_1 d_2$ are not disjoint and $V(c_1 c_2 X^{n-2})\cap V(X^{n-2} d_1 d_2)=V(c_1 c_2 X^{n-4} d_1 d_2)$. Thus, $|V(a_1 a_2 X^{n-2})\cup V(b_1 b_2 X^{n-2})\cup V(c_1 c_2 X^{n-2})\cup V(X^{n-2} d_1 d_2)|=|V(a_1 a_2 X^{n-2})|+|V(b_1 b_2 X^{n-2})|+|V(c_1 c_2 X^{n-2})|+|V(X^{n-2} d_1 d_2)|-|V(c_1 c_2 X^{n-2})\cap V(X^{n-2}$ $d_1 d_2)|=4(n-2)!-(n-4)!$. 
The probability $P_2$ that there are four fault-free $B_{n-2}$'s chosen as in Case 2 is $\frac{1}{6}\times 6(n^4-6n^3+11n^2-6n)p^{4(n-2)!-(n-4)!}$

{\it Case 3.} 211 or 121 or 112.

W.l.o.g, assume that the case of 211, that is, $\{a_1, a_2\}=\{d_1, d_2\}$, $|\{b_1, b_2\}\cap\{d_1, d_2\}|=1$ and $|\{c_1, c_2\}\cap\{d_1, d_2\}|=1$.
There are 2 possible cases $a_1=d_1$, $a_2=d_2$ and $a_1=d_2$, $a_2=d_1$ such that $\{a_1, a_2\}=\{d_1, d_2\}$. W.l.o.g, assume that $a_1=d_1$, $a_2=d_2$.
Note that for any $i\in\{b, c\}$, there are 4 possible cases ($\romannumeral1$) $i_1=d_1$, $i_2\notin\{d_1, d_2\}$, and ($\romannumeral2$) $i_1=d_2$, $i_2\notin\{d_1, d_2\}$, and ($\romannumeral3$) $i_2=d_1$, $i_1\notin\{d_1, d_2\}$, and ($\romannumeral4$) $i_2=d_2$, $i_1\notin\{d_1, d_2\}$ such that $|\{i_1, i_2\}\cap\{d_1, d_2\}|=1$.

{\it Case 3.1.} $b_1=c_1\in\{d_1, d_2\}$ or $b_2=c_2\in\{d_1, d_2\}$.

W.l.o.g, assume that $b_1(=c_1)\in\{d_1, d_2\}$. There are 2 possible cases $b_1(=c_1=d_1)$, $b_2\notin\{d_1, d_2\}$, $c_2\notin\{d_1, d_2\}$ and $b_1(=c_1=d_2)$, $b_2\notin\{d_1, d_2\}$, $c_2\notin\{d_1, d_2\}$. W.l.o.g, assume that the former applies.
There are $n$ ways to choose $b_1=c_1=d_1=a_1$ from $N_n$, $n-1$ ways to choose $a_2=d_2$ from $N_n\setminus\{a_1\}$, $n-2$ ways to choose $b_2$ from $N_n\setminus\{d_1, d_2\}$, $n-3$ ways to choose $c_2$ from $N_n\setminus\{d_1, d_2, b_2\}$.
Therefore, there are $2\times 2n(n-1)(n-2)(n-3)=4(n^4-6n^3+11n^2-6n)$ ways to choose four $B_{n-2}$'s as required.

{\it Case 3.2.} $b_1=c_2\in\{d_1, d_2\}$ or $b_2=c_1\in\{d_1, d_2\}$.

W.l.o.g, assume that $b_1=c_2\in\{d_1, d_2\}$. There are 2 possible cases $b_1=c_2=d_1$, $b_2\notin\{d_1, d_2\}$, $c_1\notin\{d_1, d_2\}$ and $b_1=c_2=d_2$, $b_2\notin\{d_1, d_2\}$, $c_1\notin\{d_1, d_2\}$. W.l.o.g, assume that the former applies.
There are $n$ ways to choose $b_1(=c_2=d_1=a_1)$ from $N_n$, $n-1$ ways to choose $a_2(=d_2)$ from $N_n\setminus\{a_1\}$, $n-2$ ways to choose $b_2$ from $N_n\setminus\{d_1, d_2\}$, $n-2$ ways to choose $c_2$ from $N_n\setminus\{d_1, d_2\}$.
Therefore, there are $2\times 2n(n-1)(n-2)^2=4(n^4-5n^3+8n^2-4n)$ ways to choose four $B_{n-2}$'s as required.

{\it Case 3.3.} For any $i, j\in\{1, 2\}$, $b_i\in\{d_1, d_2\}$ and $c_j\in\{d_1, d_2\}\setminus\{b_i\}$.

There are 8 possible cases. W.l.o.g, assume that $b_1=d_1$, $c_1=d_2$, $b_2\notin\{d_1, d_2\}$, $c_2\notin\{d_1, d_2\}$.
There are $n$ ways to choose $a_1(=d_1=b_1)$ from $N_n$, $n-1$ ways to choose $a_2(=d_2=c_1)$ from $N_n\setminus\{a_1\}$, $n-2$ ways to choose $b_2$ from $N_n\setminus\{d_1, d_2\}$, $n-2$ ways to choose $c_2$ from $N_n\setminus\{d_1, d_2\}$.
Therefore, there are $8n(n-1)(n-2)^2=8(n^4-5n^3+8n^2-4n)$ ways to choose four $B_{n-2}$'s as required.

In summary, there are $4(4n^4-21n^3+35n^2-18n)$ ways to choose four $B_{n-2}$'s as required in Case 3.
Observation \ref{ob:4} implies that for any $i\in\{a, b, c\}$, $i_1 i_2 X^{n-2}$ and $X^{n-2} d_1 d_2$ are disjoint. Thus, $|V(a_1 a_2 X^{n-2})\cup V(b_1 b_2 X^{n-2})\cup V(c_1 c_2 X^{n-2})\cup V(X^{n-2} d_1 d_2)|=|V(a_1 a_2 X^{n-2})|+|V(b_1 b_2 X^{n-2})|+|V(c_1 c_2 X^{n-2})|+|V(X^{n-2} d_1$ $d_2)|=4(n-2)!$. 
The probability $P_3$ that there are four fault-free $B_{n-2}$'s chosen as in Case 3 is $\frac{1}{6}\times 3\times 2\times 4(4n^4-21n^3+35n^2-18n)p^{4(n-2)!}$.

{\it Case 4.} 210 or 201 or 120 or 102 or 012 or 021.

W.l.o.g, assume that the case of 210, that is, $\{a_1, a_2\}=\{d_1, d_2\}$, $|\{b_1, b_2\}\cap\{d_1, d_2\}|=1$ and $\{c_1, c_2\}\cap\{d_1, d_2\}=\emptyset$.
There are 8 possible cases, w.l.o.g, assume that $a_1=d_1$, $a_2=d_2$, $b_1=d_1$, $b_2\notin\{d_1, d_2\}$.
There are $n$ ways to choose $a_1(=d_1=b_1)$ from $N_n$, $n-1$ ways to choose $a_2(=d_2)$ from $N_n\setminus\{a_1\}$, $n-2$ ways to choose $b_2$ from $N_n\setminus\{d_1, d_2\}$, $n-2$ ways to choose $c_1$ from $N_n\setminus\{d_1, d_2\}$, $n-3$ ways to choose $c_2$ from $N_n\setminus\{d_1, d_2, c_1\}$.
Therefore, there are $n(n-1)(n-2)^2(n-3)=n^5-8n^4+23n^3-28n^2+12n$ ways to choose four $B_{n-2}$'s as required. Observation \ref{ob:4} implies that for any $i\in\{a, b\}$, $i_1 i_2 X^{n-2}$ and $X^{n-2} d_1 d_2$ are disjoint, and $c_1 c_2 X^{n-2}$ and $X^{n-2} d_1 d_2$ are not disjoint and $V(c_1 c_2 X^{n-2})\cap V(X^{n-2} d_1 d_2)=V(c_1 c_2 X^{n-4} d_1 d_2)$. Thus, $|V(a_1 a_2 X^{n-2})\cup V(b_1 b_2 X^{n-2})\cup V(c_1 c_2 X^{n-2})\cup V(X^{n-2} d_1 d_2)|=|V(a_1 a_2 X^{n-2})|+|V(b_1 b_2 X^{n-2})|+|V(c_1 c_2 X^{n-2})|+|V(X^{n-2} d_1 d_2)|-|V(c_1 c_2 X^{n-2})\cap V(X^{n-2} d_1 d_2)|=4(n-2)!-(n-4)!$. 
The probability $P_4$ that there are four fault-free $B_{n-2}$'s chosen as in Case 4 is $\frac{1}{6}\times 6\times 8(n^5-8n^4+23n^3-28n^2+12n)p^{4(n-2)!-(n-4)!}$

{\it Case 5.} 200 or 020 or 002.

W.l.o.g, assume that the case of 200, that is, $\{a_1, a_2\}=\{d_1, d_2\}$, $\{b_1, b_2\}\cap\{d_1, d_2\}=\emptyset$ and $\{c_1, c_2\}\cap\{d_1, d_2\}=\emptyset$.
There are 2 possible cases, w.l.o.g, assume that $a_1=d_1$, $a_2=d_2$.

{\it Case 5.1.} $b_1\neq c_1$.

There are $n$ ways to choose $a_1(=d_1)$ from $N_n$, $n-1$ ways to choose $a_2(=d_2)$ from $N_n\setminus\{a_1\}$, $n-2$ ways to choose $b_1$ from $N_n\setminus\{d_1, d_2\}$, $n-3$ ways to choose $b_2$ from $N_n\setminus\{d_1, d_2, b_1\}$, $n-3$ ways to choose $c_1$ from $N_n\setminus\{d_1, d_2, b_1\}$, $n-3$ ways to choose $c_2$ from $N_n\setminus\{d_1, d_2, c_1\}$.
Therefore, there are $n(n-1)(n-2)(n-3)^3=n^6-12n^5+56n^4-126n^3+135n^2-54n$ ways to choose four $B_{n-2}$'s as required.

{\it Case 5.2.} $b_1=c_1$.

There are $n$ ways to choose $a_1(=d_1)$ from $N_n$, $n-1$ ways to choose $a_2(=d_2)$ from $N_n\setminus\{a_1\}$, $n-2$ ways to choose $b_1(=c_1)$ from $N_n\setminus\{d_1, d_2\}$, $n-3$ ways to choose $b_2$ from $N_n\setminus\{d_1, d_2, b_1\}$, $n-4$ ways to choose $c_2$ from $N_n\setminus\{d_1, d_2, c_1, b_2\}$.
Therefore, there are $n(n-1)(n-2)(n-3)(n-4)=n^5-10n^4+35n^3-50n^2+24n$ ways to choose four $B_{n-2}$'s as required.

In summary, there are $n^6-11n^5+46n^4-91n^3+85n^2-30n$ ways to choose four $B_{n-2}$'s as required in Case 5.
Observation \ref{ob:4} implies that $a_1 a_2 X^{n-2}$ and $X^{n-2} d_1 d_2$ are disjoint, and for any $i\in\{b, c\}$, $i_1 i_2 X^{n-2}$ and $X^{n-2} d_1 d_2$ are not disjoint and $V(i_1 i_2 X^{n-2})\cap V(X^{n-2} d_1 d_2)=V(i_1 i_2 X^{n-4} d_1 d_2)$. Thus, $|V(a_1 a_2 X^{n-2})\cup V(b_1 b_2 X^{n-2})\cup V(c_1 c_2 X^{n-2})\cup V(X^{n-2} d_1 d_2)|=|V(a_1 a_2 X^{n-2})|+|V(b_1 b_2 X^{n-2})|+|V(c_1 c_2 X^{n-2})|+|V(X^{n-2} d_1$ $d_2)|-|V(b_1 b_2 X^{n-2})\cap V(X^{n-2} d_1 d_2)|-|V(c_1 c_2 X^{n-2})\cap V(X^{n-2} d_1 d_2)|=4(n-2)!-2(n-4)!$. 
The probability $P_5$ that there are four fault-free $B_{n-2}$'s chosen as in Case 5 is $\frac{1}{6}\times 3\times 2(n^6-11n^5+46n^4-91n^3+85n^2-30n)p^{4(n-2)!-2(n-4)!}$.

{\it Case 6.} 111.

That is $|\{a_1, a_2\}\cap\{d_1, d_2\}|=1$, $|\{b_1, b_2\}\cap\{d_1, d_2\}|=1$ and $|\{c_1, c_2\}\cap\{d_1, d_2\}|=1$.
Note that for any $i\in\{a, b, c\}$, there are 4 possible cases ($\romannumeral1$) $i_1=d_1$, $i_2\notin\{d_1, d_2\}$, and ($\romannumeral2$) $i_1=d_2$, $i_2\notin\{d_1, d_2\}$, and ($\romannumeral3$) $i_2=d_1$, $i_1\notin\{d_1, d_2\}$, and ($\romannumeral4$) $i_2=d_2$, $i_1\notin\{d_1, d_2\}$ such that $|\{i_1, i_2\}\cap\{d_1, d_2\}|=1$.

{\it Case 6.1.} $a_1=b_1=c_1\in\{d_1, d_2\}$ or $a_2=b_2=c_2\in\{d_1, d_2\}$.

There are 4 possible cases, w.l.o.g, assume that $a_1=b_1=c_1=d_1$, $\{a_2, b_2, c_2\}\cap\{d_1, d_2\}=\emptyset$.
There are $n$ ways to choose $a_1(=b_1=c_1=d_1)$ from $N_n$, $n-1$ ways to choose $d_2$ from $N_n\setminus\{d_1\}$, $n-2$ ways to choose $a_2$ from $N_n\setminus\{d_1, d_2\}$, $n-3$ ways to choose $b_2$ from $N_n\setminus\{d_1, d_2, a_2\}$, $n-4$ ways to choose $c_2$ from $N_n\setminus\{d_1, d_2, a_2, b_2\}$.
Therefore, there are $4n(n-1)(n-2)(n-3)(n-4)=4(n^5-10n^4+35n^3-50n^2+24n)$ ways to choose four $B_{n-2}$'s as required.

{\it Case 6.2.} For any $i, j, k, l\in\{1, 2\}$, $a_i=b_j=c_k=d_l$, where $i=j\neq k$ or $i=k\neq j$ or $j=k\neq i$.

There are 12 possible cases, w.l.o.g, assume that $a_1=b_1=c_2=d_1$, $\{a_2, b_2\}\cap\{d_1, d_2\}=\emptyset$ and $c_1\notin\{d_1, d_2\}$.
There are $n$ ways to choose $a_1(=b_1=c_2=d_1)$ from $N_n$, $n-1$ ways to choose $d_2$ from $N_n\setminus\{d_1\}$, $n-2$ ways to choose $a_2$ from $N_n\setminus\{d_1, d_2\}$, $n-3$ ways to choose $b_2$ from $N_n\setminus\{d_1, d_2, a_2\}$, $n-2$ ways to choose $c_2$ from $N_n\setminus\{d_1, d_2\}$.
Therefore, there are $12n(n-1)(n-2)^2(n-3)=12(n^5-8n^4+23n^3-28n^2+12n)$ ways to choose four $B_{n-2}$'s as required.

{\it Case 6.3.} For any $i, j\in\{1, 2\}$, $x_i=y_i=d_j$, $\{d_1, d_2\}\setminus\{d_j\}\subset\{z_1, z_2\}$, where $\{x, y, z\}=\{a, b, c\}$.

There are 3 ways to choose two from $\{a, b, c\}$ as $x$ and $y$, and the remaining one as $z$. There are 2 ways to determine the values of $i$, $j$ and $\{d_1, d_2\}\setminus\{d_j\}\subset\{z_1, z_2\}$. Therefore, there are 24 possible cases, w.l.o.g, assume that $a_1=b_1=d_1$, $c_1=d_2$, $\{a_2, b_2\}\cap\{d_1, d_2\}=\emptyset$ and $c_2\notin\{d_1, d_2\}$.
There are $n$ ways to choose $a_1(=b_1=d_1)$ from $N_n$, $n-1$ ways to choose $c_1(=d_2)$ from $N_n\setminus\{d_1\}$, $n-2$ ways to choose $a_2$ from $N_n\setminus\{d_1, d_2\}$, $n-3$ ways to choose $b_2$ from $N_n\setminus\{d_1, d_2, a_2\}$, $n-2$ ways to choose $c_2$ from $N_n\setminus\{d_1, d_2\}$.
Therefore, there are $24n(n-1)(n-2)^2(n-3)=24(n^5-8n^4+23n^3-28n^2+12n)$ ways to choose four $B_{n-2}$'s as required.

{\it Case 6.4.} $x_1=y_2\in\{d_1, d_2\}$, $\{d_1, d_2\}\setminus\{x_1\}\subset\{z_1, z_2\}$, where $\{x, y, z\}=\{a, b, c\}$.

There are 3 ways to choose one from $\{a, b, c\}$ as $x$, and there are 2 ways to choose one from $\{a, b, c\}\setminus\{x\}$ as $y$, and the remaining one as $z$. $x_1=y_2$ (resp. $\{d_1, d_2\}\setminus\{x_1\}$) has two possible values in $\{d_1, d_2\}$ (resp. $\{z_1, z_2\}$). Therefore, there are 24 possible cases, w.l.o.g, assume that $a_1=b_2=d_1$, $c_1=d_2$, $a_2\notin\{d_1, d_2\}$, $b_1\notin\{d_1, d_2\}$ and $c_2\notin\{d_1, d_2\}$.
There are $n$ ways to choose $a_1(=b_2=d_1)$ from $N_n$, $n-1$ ways to choose $c_1(=d_2)$ from $N_n\setminus\{d_1\}$, $n-2$ ways to choose $a_2$ from $N_n\setminus\{d_1, d_2\}$, $n-2$ ways to choose $b_2$ from $N_n\setminus\{d_1, d_2\}$, $n-2$ ways to choose $c_2$ from $N_n\setminus\{d_1, d_2\}$.
Therefore, there are $24n(n-1)(n-2)^3=24(n^5-7n^4+18n^3-20n^2+8n)$ ways to choose four $B_{n-2}$'s as required.

In summary, there are $2(8n^5-62n^4+175n^3-211n^2+90n)$ ways to choose four $B_{n-2}$'s as required in Case 6. 
Observation \ref{ob:4} implies that for any $i\in\{a, b, c\}$, $i_1 i_2 X^{n-2}$ and $X^{n-2} d_1 d_2$ are disjoint. Thus, $|V(a_1 a_2 X^{n-2})\cup V(b_1 b_2 X^{n-2})\cup V(c_1 c_2 X^{n-2})\cup V(X^{n-2} d_1 d_2)|=|V(a_1 a_2 X^{n-2})|+|V(b_1 b_2 X^{n-2})|+|V(c_1 c_2 X^{n-2})|+|V(X^{n-2} d_1$ $d_2)|=4(n-2)!$. 
The probability $P_6$ that there are four fault-free $B_{n-2}$'s chosen as in Case 6 is $\frac{1}{6}\times 4\times 2(8n^5-62n^4+175n^3-211n^2+90n)p^{4(n-2)!}$.

{\it Case 7.} 110 or 101 or 011.

W.l.o.g, assume that the case of 110, that is, $|\{a_1, a_2\}\cap\{d_1, d_2\}|=1$, $|\{b_1, b_2\}\cap\{d_1, d_2\}|=1$ and $\{c_1, c_2\}\cap\{d_1, d_2\}=\emptyset$.

{\it Case 7.1.} $a_1=b_1\in\{d_1, d_2\}$ or $a_2=b_2\in\{d_1, d_2\}$.

There are 4 possible cases, w.l.o.g, assume that $a_1=b_1=d_1$, $\{a_2, b_2\}\cap\{d_1, d_2\}=\emptyset$, $\{c_1, c_2\}\cap\{d_1, d_2\}=\emptyset$.
There are $n$ ways to choose $a_1(=b_1=d_1)$ from $N_n$, $n-1$ ways to choose $d_2$ from $N_n\setminus\{d_1\}$, $n-2$ ways to choose $a_2$ from $N_n\setminus\{d_1, d_2\}$, $n-3$ ways to choose $b_2$ from $N_n\setminus\{d_1, d_2, a_2\}$, $n-2$ ways to choose $c_1$ from $N_n\setminus\{d_1, d_2\}$, $n-3$ ways to choose $c_2$ from $N_n\setminus\{d_1, d_2, c_1\}$.
Therefore, there are $4n(n-1)(n-2)^2(n-3)^2=4(n^6-11n^5+47n^4-97n^3+96n^2-36n)$ ways to choose four $B_{n-2}$'s as required.

{\it Case 7.2.} $a_1=b_2\in\{d_1, d_2\}$ or $a_2=b_1\in\{d_1, d_2\}$.

There are 4 possible cases, w.l.o.g, assume that $a_1=b_2=d_1$, $a_2\notin\{d_1, d_2\}$, $b_1\notin\{d_1, d_2\}$ and $\{c_1, c_2\}\cap\{d_1, d_2\}=\emptyset$.
There are $n$ ways to choose $a_1(=b_2=d_1)$ from $N_n$, $n-1$ ways to choose $d_2$ from $N_n\setminus\{d_1\}$, $n-2$ ways to choose $a_2$ from $N_n\setminus\{d_1, d_2\}$, $n-2$ ways to choose $b_1$ from $N_n\setminus\{d_1, d_2\}$, $n-2$ ways to choose $c_1$ from $N_n\setminus\{d_1, d_2\}$, $n-3$ ways to choose $c_2$ from $N_n\setminus\{d_1, d_2, c_1\}$.
Therefore, there are $4n(n-1)(n-2)^3(n-3)=4(n^6-10n^5+39n^4-74n^3+68n^2-24n)$ ways to choose four $B_{n-2}$'s as required.

{\it Case 7.3.} For any $i, j\in\{1, 2\}$, $a_i\in\{d_1, d_2\}$ and $b_j\in\{d_1, d_2\}\setminus\{a_i\}$.

There are 8 possible cases, w.l.o.g, assume that $a_1=d_1$, $b_1=d_2$, $a_2\notin\{d_1, d_2\}$, $b_2\notin\{d_1, d_2\}$ and $\{c_1, c_2\}\cap\{d_1, d_2\}=\emptyset$.
There are $n$ ways to choose $a_1(=d_1)$ from $N_n$, $n-1$ ways to choose $b_1(=d_2)$ from $N_n\setminus\{d_1\}$, $n-2$ ways to choose $a_2$ from $N_n\setminus\{d_1, d_2\}$, $n-2$ ways to choose $b_2$ from $N_n\setminus\{d_1, d_2\}$, $n-2$ ways to choose $c_1$ from $N_n\setminus\{d_1, d_2\}$, $n-3$ ways to choose $c_2$ from $N_n\setminus\{d_1, d_2, c_1\}$.
Therefore, there are $8n(n-1)(n-2)^3(n-3)=8(n^6-10n^5+39n^4-74n^3+68n^2-24n)$ ways to choose four $B_{n-2}$'s as required.

In summary, there are $4(4n^6-41n^5+164n^4-319n^3+300n^2-108n)$ ways to choose four $B_{n-2}$'s as required in Case 7.
Observation \ref{ob:4} implies that for any $i\in\{a, b\}$, $i_1 i_2 X^{n-2}$ and $X^{n-2} d_1 d_2$ are disjoint, and $c_1 c_2 X^{n-2}$ and $X^{n-2} d_1 d_2$ are not disjoint and $V(c_1 c_2 X^{n-2})\cap V(X^{n-2} d_1 d_2)=V(c_1 c_2 X^{n-4} d_1 d_2)$. Thus, $|V(a_1 a_2 X^{n-2})\cup V(b_1 b_2 X^{n-2})\cup V(c_1 c_2 X^{n-2})\cup V(X^{n-2} d_1 d_2)|=|V(a_1 a_2 X^{n-2})|+|V(b_1 b_2 X^{n-2})|+|V(c_1 c_2 X^{n-2})|+|V(X^{n-2} d_1 d_2)|-|V(c_1 c_2 X^{n-2})\cap V(X^{n-2} d_1 d_2)|=4(n-2)!-(n-4)!$. 
The probability $P_7$ that there are four fault-free $B_{n-2}$'s chosen as in Case 7 is $\frac{1}{6}\times 3\times 4(4n^6-41n^5+164n^4-319n^3+300n^2-108n)p^{4(n-2)!-(n-4)!}$.

{\it Case 8.} 100 or 010 or 001.

W.l.o.g, assume that the case of 100, that is, $|\{a_1, a_2\}\cap\{d_1, d_2\}|=1$, $\{b_1, b_2\}\cap\{d_1, d_2\}=\emptyset$ and $\{c_1, c_2\}\cap\{d_1, d_2\}=\emptyset$.
There are 4 possible cases, w.l.o.g, assume that $a_1=d_1$, $a_2\notin\{d_1, d_2\}$, $\{b_1, b_2\}\cap\{d_1, d_2\}=\emptyset$ and $\{c_1, c_2\}\cap\{d_1, d_2\}=\emptyset$.

{\it Case 8.1.} $b_1\neq c_1$.

There are $n$ ways to choose $a_1(=d_1)$ from $N_n$, $n-1$ ways to choose $d_2$ from $N_n\setminus\{d_1\}$, $n-2$ ways to choose $a_2$ from $N_n\setminus\{d_1, d_2\}$, $n-2$ ways to choose $b_1$ from $N_n\setminus\{d_1, d_2\}$, $n-3$ ways to choose $b_2$ from $N_n\setminus\{d_1, d_2, b_1\}$, $n-3$ ways to choose $c_1$ from $N_n\setminus\{d_1, d_2, b_1\}$, $n-3$ ways to choose $c_2$ from $N_n\setminus\{d_1, d_2, c_1\}$.
Therefore, there are $n(n-1)(n-2)^2(n-3)^3=n^7-14n^6+80n^5-238n^4+387n^3-324n^2+108n$ ways to choose four $B_{n-2}$'s as required.

{\it Case 8.2.} $b_1=c_1$.

There are $n$ ways to choose $a_1(=d_1)$ from $N_n$, $n-1$ ways to choose $d_2$ from $N_n\setminus\{d_1\}$, $n-2$ ways to choose $a_2$ from $N_n\setminus\{d_1, d_2\}$, $n-2$ ways to choose $b_1(=c_1)$ from $N_n\setminus\{d_1, d_2\}$, $n-3$ ways to choose $b_2$ from $N_n\setminus\{d_1, d_2, b_1\}$, $n-4$ ways to choose $c_2$ from $N_n\setminus\{d_1, d_2, c_1, b_2\}$.
Therefore, there are $n(n-1)(n-2)^2(n-3)(n-4)=n^6-12n^5+55n^4-120n^3+124n^2-48n$ ways to choose four $B_{n-2}$'s as required.

In summary, there are $n^7-13n^6+68n^5-183n^4+267n^3-200n^2+60n$ distinct $B_{n-2}$'s in Case 8. 
Observation \ref{ob:4} implies that $a_1 a_2 X^{n-2}$ and $X^{n-2} d_1 d_2$ are disjoint, and for any $i\in\{b, c\}$, $i_1 i_2 X^{n-2}$ and $X^{n-2} d_1 d_2$ are not disjoint and $V(i_1 i_2 X^{n-2})\cap V(X^{n-2} d_1 d_2)=V(i_1 i_2 X^{n-4} d_1 d_2)$. Thus, $|V(a_1 a_2 X^{n-2})\cup V(b_1 b_2 X^{n-2})\cup V(c_1 c_2 X^{n-2})\cup V(X^{n-2} d_1 d_2)|=|V(a_1 a_2 X^{n-2})|+|V(b_1 b_2 X^{n-2})|+|V(c_1 c_2 X^{n-2})|+|V(X^{n-2} d_1 d_2)|-|V(b_1 b_2$ $X^{n-2})\cap V(X^{n-2} d_1 d_2)|-|V(c_1 c_2 X^{n-2})\cap V(X^{n-2} d_1 d_2)|=4(n-2)!-2(n-4)!$. 
The probability $P_8$ that there are four fault-free $B_{n-2}$'s chosen as in Case 8 is $\frac{1}{6}\times 3\times 4(n^7-13n^6+68n^5-183n^4+267n^3-200n^2+60n)p^{4(n-2)!-2(n-4)!}$.

{\it Case 9.} 000.

That is $\{a_1, a_2\}\cap\{d_1, d_2\}=\emptyset$, $\{b_1, b_2\}\cap\{d_1, d_2\}=\emptyset$ and $\{c_1, c_2\}\cap\{d_1, d_2\}=\emptyset$.

{\it Case 9.1.} $a_1=b_1=c_1$.

There are $n$ ways to choose $d_1$ from $N_n$, $n-1$ ways to choose $d_2$ from $N_n\setminus\{d_1\}$, $n-2$ ways to choose $a_1(=b_1=c_1)$ from $N_n\setminus\{d_1, d_2\}$, $n-3$ ways to choose $a_2$ from $N_n\setminus\{d_1, d_2, a_1\}$, $n-4$ ways to choose $b_2$ from $N_n\setminus\{d_1, d_2, b_1, a_2\}$, $n-5$ ways to choose $c_2$ from $N_n\setminus\{d_1, d_2, c_1, a_2, b_2\}$.
Therefore, there are $n(n-1)(n-2)(n-3)(n-4)(n-5)=n^6-15n^5+85n^4-225n^3+274n^2-120n$ ways to choose four $B_{n-2}$'s as required.

{\it Case 9.2.} $a_1=b_1\neq c_1$ or $a_1=c_1\neq b_1$ or $b_1=c_1\neq a_1$.

There are $n$ ways to choose $d_1$ from $N_n$, $n-1$ ways to choose $d_2$ from $N_n\setminus\{d_1\}$, $n-2$ ways to choose $a_1(=b_1)$ from $N_n\setminus\{d_1, d_2\}$, $n-3$ ways to choose $c_1$ from $N_n\setminus\{d_1, d_2, a_1\}$, $n-3$ ways to choose $a_2$ from $N_n\setminus\{d_1, d_2, a_1\}$, $n-4$ ways to choose $b_2$ from $N_n\setminus\{d_1, d_2, b_1, a_2\}$, $n-3$ ways to choose $c_2$ from $N_n\setminus\{d_1, d_2, c_1\}$.
Therefore, there are $3n(n-1)(n-2)(n-3)^3(n-4)=3(n^7-16n^6+104n^5-350n^4+639n^3-594n^2+216n)$ ways to choose four $B_{n-2}$'s as required.

{\it Case 9.3.} $a_1$, $b_1$ and $c_1$ are pairwise distinct.

There are $n$ ways to choose $d_1$ from $N_n$, $n-1$ ways to choose $d_2$ from $N_n\setminus\{d_1\}$, $n-2$ ways to choose $a_1$ from $N_n\setminus\{d_1, d_2\}$, $n-3$ ways to choose $b_1$ from $N_n\setminus\{d_1, d_2, a_1\}$, $n-4$ ways to choose $c_1$ from $N_n\setminus\{d_1, d_2, a_1, b_1\}$, $n-3$ ways to choose $a_2$ from $N_n\setminus\{d_1, d_2, a_1\}$, $n-3$ ways to choose $b_2$ from $N_n\setminus\{d_1, d_2, b_1\}$, $n-3$ ways to choose $c_2$ from $N_n\setminus\{d_1, d_2, c_1\}$.
Therefore, there are $n(n-1)(n-2)(n-3)^4(n-4)=n^8-19n^7+152n^6-662n^5+1689n^4-2511n^3+1998n^2-648n$ ways to choose four $B_{n-2}$'s as required.

In summary, there are $n^8-16n^7+105n^6-365n^5+724n^4-819n^3+490n^2-120n$ ways to choose four $B_{n-2}$'s as required in Case 9. 
Observation \ref{ob:4} implies that for any $i\in\{a, b, c\}$, $i_1 i_2 X^{n-2}$ and $X^{n-2} d_1 d_2$ are not disjoint and $V(i_1 i_2 X^{n-2})\cap V(X^{n-2} d_1 d_2)=V(i_1 i_2 X^{n-4} d_1 d_2)$. Thus, $|V(a_1 a_2 X^{n-2})\cup V(b_1 b_2 X^{n-2})\cup V(c_1 c_2 X^{n-2})\cup V(X^{n-2} d_1 d_2)|=|V(a_1 a_2 X^{n-2})|+|V(b_1 b_2 X^{n-2})|+|V(c_1 c_2 X^{n-2})|+|V(X^{n-2} d_1 d_2)|-|V(a_1 a_2 X^{n-2})\cap V(X^{n-2}$ $d_1 d_2)|-|V(b_1 b_2 X^{n-2})\cap V(X^{n-2} d_1 d_2)|-|V(c_1 c_2 X^{n-2})\cap V(X^{n-2} d_1 d_2)|=4(n-2)!-3(n-4)!$. 
The probability $P_9$ that there are four fault-free $B_{n-2}$'s chosen as in Case 9 is $\frac{1}{6}(n^8-16n^7+105n^6-365n^5+724n^4-819n^3+490n^2-120n)p^{4(n-2)!-3(n-4)!}$.

By the above computations, $P(3,0,1)=\sum_{i=1}^{9}P_i=\frac{4}{3}(8n^5-50n^4+115n^3-115n^2+42n)p^{4(n-2)!}+(8n^6-74n^5+265n^4-460n^3+387n^2-126n)p^{4(n-2)!-(n-4)!}+(2n^7-25n^6+125n^5-320n^4+443n^3-315n^2+90n)p^{4(n-2)!-2(n-4)!}+\frac{1}{6}(n^8-16n^7+105n^6-365n^5+724n^4-819n^3+490n^2-120n)p^{4(n-2)!-3(n-4)!}$.
\end{proof}

\begin{theorem}\label{th5}
$P(1,0,3)=\frac{4}{3}(8n^5-50n^4+115n^3-115n^2+42n)p^{4(n-2)!}+(8n^6-74n^5+265n^4-460n^3+387n^2-126n)p^{4(n-2)!-(n-4)!}+(2n^7-25n^6+125n^5-320n^4+443n^3-315n^2+90n)p^{4(n-2)!-2(n-4)!}+\frac{1}{6}(n^8-16n^7+105n^6-365n^5+724n^4-819n^3+490n^2-120n)p^{4(n-2)!-3(n-4)!}$.
\end{theorem}
\begin{proof}
This scenario is similar to Theorem \ref{th4}, it is easy to get that $P(1,0,3)=P(3,0,1)=\frac{4}{3}(8n^5-50n^4+115n^3-115n^2+42n)p^{4(n-2)!}+(8n^6-74n^5+265n^4-460n^3+387n^2-126n)p^{4(n-2)!-(n-4)!}+(2n^7-25n^6+125n^5-320n^4+443n^3-315n^2+90n)p^{4(n-2)!-2(n-4)!}+\frac{1}{6}(n^8-16n^7+105n^6-365n^5+724n^4-819n^3+490n^2-120n)p^{4(n-2)!-3(n-4)!}$.
\end{proof}

\begin{theorem}\label{th6}
$P(2,2,0)=\frac{1}{4}(n^8-8n^7+30n^6-69n^5+104n^4-99n^3+53n^2-12n)p^{4(n-2)!}+(n^7-8n^6+29n^5-59n^4+68n^3-41n^2+10n)p^{4(n-2)!-(n-3)!}+\frac{1}{2}(3n^6-23n^5+67n^4-92n^3+59n^2-14n)p^{4(n-2)!-2(n-3)!}+(n^4-6n^3+11n^2-6n)p^{4(n-2)!-3(n-3)!}+\frac{1}{4}(n^5-10n^4+35n^3-50n^2+24n)p^{4(n-2)!-4(n-3)!}$.
\end{theorem}
\begin{proof}
Denote by $a_1 a_2 X^{n-2}\in H_{1, 2}$, $b_1 b_2 X^{n-2}\in H_{1, 2}$, $c_1 X^{n-2} c_2\in H_{1, n}$ and $d_1 X^{n-2} d_2\in H_{1, n}$ these four $B_{n-2}$'s. Note that $a_1 a_2 X^{n-2}$ and $b_1 b_2 X^{n-2}$ are disjoint, and $c_1 X^{n-2} c_2$ and $d_1 X^{n-2} d_2$ are disjoint.

Since we have implicitly sorted the values of $a_1 a_2$ and $b_1 b_2$, and the values of $c_1 c_2$ and $d_1 d_2$, the number of ways to choose four $B_{n-2}$'s as required needs to be divided by 4. It will be applied directly in following analyses.

We will analyze this problem in 7 cases as follows.

{\it Case 1.} $a_1=b_1=c_1=d_1$.

Clearly, in this case $a_2\neq b_2$ and $c_2\neq d_2$.

{\it Case 1.1.} $a_2\notin\{c_2, d_2\}$ and $b_2\notin\{c_2, d_2\}$.

There are $n$ ways to choose $a_1(=b_1=c_1=d_1)$ from $N_n$, $n-1$ ways to choose $c_2$ from $N_n\setminus\{c_1\}$, $n-2$ ways to choose $d_2$ from $N_n\setminus\{d_1, c_2\}$, $n-3$ ways to choose $a_2$ from $N_n\setminus\{a_1, c_2, d_2\}$, $n-4$ ways to choose $b_2$ from $N_n\setminus\{b_1, c_2, d_2, a_2\}$.
Therefore, there are $n(n-1)(n-2)(n-3)(n-4)=n^5-10n^4+35n^3-50n^2+24n$ ways to choose four $B_{n-2}$'s as required in Case 1.1. Observation \ref{ob:2} implies that for any $i\in\{a, b\}$ and $j\in\{c, d\}$, $i_1 i_2 X^{n-2}$ and $j_1 X^{n-2} j_2$ are not disjoint and $V(i_1 i_2 X^{n-2})\cap V(j_1 X^{n-2} j_2)=V(i_1 i_2 X^{n-3} j_2)$. Thus, $|V(a_1 a_2 X^{n-2})\cup V(b_1 b_2 X^{n-2})\cup V(X^{n-2} c_1 c_2)\cup V(X^{n-2} d_1 d_2)|=|V(a_1 a_2 X^{n-2})|+|V(b_1 b_2 X^{n-2})|+|V(X^{n-2} c_1 c_2)|+|V(X^{n-2} d_1 d_2)|-|V(a_1 a_2 X^{n-2})\cap V(c_1 X^{n-2} c_2)|-|V(a_1 a_2 X^{n-2})\cap V(d_1 X^{n-2} d_2)|-|V(b_1 b_2$ $X^{n-2})\cap V(c_1 X^{n-2} c_2)|-|V(b_1 b_2 X^{n-2})\cap V(d_1 X^{n-2} d_2)|=4(n-2)!-4(n-3)!$. 
The probability $P_{1.1}$ that there are four fault-free $B_{n-2}$'s chosen as in Case 1.1 is $\frac{1}{4}(n^5-10n^4+35n^3-50n^2+24n)p^{4(n-2)!-4(n-3)!}$.

{\it Case 1.2.} $a_2\notin\{c_2, d_2\}$, $b_2\in\{c_2, d_2\}$ or $a_2\in\{c_2, d_2\}$, $b_2\notin\{c_2, d_2\}$.

There are 4 possible cases, w.l.o.g, assume that $a_2\notin\{c_2, d_2\}$, $b_2=c_2\neq d_2$.

There are $n$ ways to choose $a_1(=b_1=c_1=d_1)$ from $N_n$, $n-1$ ways to choose $c_2(=b_2)$ from $N_n\setminus\{c_1\}$, $n-2$ ways to choose $d_2$ from $N_n\setminus\{d_1, c_2\}$, $n-3$ ways to choose $a_2$ from $N_n\setminus\{a_1, c_2, d_2\}$.
Therefore, there are $4n(n-1)(n-2)(n-3)=4(n^4-6n^3+11n^2-6n)$ ways to choose four $B_{n-2}$'s as required in Case 1.2.
Observation \ref{ob:2} implies that for any $i\in\{c, d\}$, $a_1 a_2 X^{n-2}$ and $i_1 X^{n-2} i_2$ are not disjoint and $V(a_1 a_2 X^{n-2})\cap V(i_1 X^{n-2} i_2)=V(a_1 a_2 X^{n-3} i_2)$, and $b_1 b_2 X^{n-2}$ and $c_1 X^{n-2} c_2$ are disjoint, and $b_1 b_2 X^{n-2}$ and $d_1 X^{n-2} d_2$ are not disjoint and $V(b_1 b_2 X^{n-2})\cap V(d_1 X^{n-2} d_2)=V(b_1 b_2 X^{n-3} d_2)$. Thus, $|V(a_1 a_2 X^{n-2})\cup V(b_1 b_2 X^{n-2})\cup V(X^{n-2} c_1 c_2)\cup V(X^{n-2} d_1 d_2)|=|V(a_1 a_2 X^{n-2})|+|V(b_1 b_2 X^{n-2})|+|V(X^{n-2} c_1 c_2)|+|V(X^{n-2} d_1 d_2)|-|V(a_1 a_2 X^{n-2})\cap V(c_1$ $X^{n-2} c_2)|-|V(a_1 a_2 X^{n-2})\cap V(d_1 X^{n-2} d_2)|-|V(b_1 b_2 X^{n-2})\cap V(d_1 X^{n-2} d_2)|=4(n-2)!-3(n-3)!$. 
The probability $P_{1.2}$ that there are four fault-free $B_{n-2}$'s chosen as in Case 1.2 is $\frac{1}{4}\times 4(n^4-6n^3+11n^2-6n)p^{4(n-2)!-3(n-3)!}$.

{\it Case 1.3.} $a_2=c_2\neq b_2=d_2$ or $a_2=d_2\neq b_2=c_2$.

W.l.o.g, assume that the former applies.

There are $n$ ways to choose $a_1(=b_1=c_1=d_1)$ from $N_n$, $n-1$ ways to choose $a_2(=c_2)$ from $N_n\setminus\{c_1\}$, $n-2$ ways to choose $b_2(=d_2)$ from $N_n\setminus\{d_1, c_2\}$.
Therefore, there are $2n(n-1)(n-2)=2(n^3-3n^2+2n)$ ways to choose four $B_{n-2}$'s as required in Case 1.3.
Observation \ref{ob:2} implies that $a_1 a_2 X^{n-2}$ and $c_1 X^{n-2} c_2$ are disjoint, $a_1 a_2 X^{n-2}$ and $d_1 X^{n-2} d_2$ are not disjoint and $V(a_1 a_2 X^{n-2})\cap V(d_1 X^{n-2} d_2)=V(a_1 a_2 X^{n-3} d_2)$, and $b_1 b_2 X^{n-2}$ and $c_1 X^{n-2} c_2$ are not disjoint and $V(b_1 b_2 X^{n-2})\cap V(c_1 X^{n-2} c_2)=V(b_1 b_2 X^{n-3} c_2)$, and $b_1 b_2 X^{n-2}$ and $d_1 X^{n-2} d_2$ are disjoint. Thus, $|V(a_1 a_2 X^{n-2})\cup V(b_1 b_2 X^{n-2})\cup V(X^{n-2} c_1 c_2)\cup V(X^{n-2} d_1 d_2)|=|V(a_1 a_2 X^{n-2})|+|V(b_1 b_2 X^{n-2})|+|V(X^{n-2} c_1 c_2)|+|V(X^{n-2} d_1 d_2)|-|V(a_1 a_2 X^{n-2})\cap V(d_1 X^{n-2} d_2)|-|V(b_1 b_2$ $X^{n-2})\cap V(c_1 X^{n-2} c_2)|=4(n-2)!-2(n-3)!$. 
The probability $P_{1.3}$ that there are four fault-free $B_{n-2}$'s chosen as in Case 1.3 is $\frac{1}{4}\times 2(n^3-3n^2+2n)p^{4(n-2)!-2(n-3)!}$.

In summary, the probability $P_1$ that there are four fault-free $B_{n-2}$'s chosen as in Case 1 is $P_{1.1}+P_{1.2}+P_{1.3}=\frac{1}{2}(n^3-3n^2+2n)p^{4(n-2)!-2(n-3)!}+(n^4-6n^3+11n^2-6n)p^{4(n-2)!-3(n-3)!}+\frac{1}{4}(n^5-10n^4+35n^3-50n^2+24n)p^{4(n-2)!-4(n-3)!}$.

{\it Case 2.} $a_1=c_1=d_1\neq b_1$ or $b_1=c_1=d_1\neq a_1$ or $a_1=b_1=c_1\neq d_1$ or $a_1=b_1=d_1\neq c_1$.

W.l.o.g, assume that $a_1=c_1=d_1\neq b_1$.

{\it Case 2.1.} $a_2\neq c_2$ and $a_2\neq d_2$.

There are $n$ ways to choose $a_1(=c_1=d_1)$ from $N_n$, $n-1$ ways to choose $b_1$ from $N_n\setminus\{a_1\}$, $n-1$ ways to choose $a_2$ from $N_n\setminus\{a_1\}$, $n-1$ ways to choose $b_2$ from $N_n\setminus\{b_1\}$, $n-2$ ways to choose $c_2$ from $N_n\setminus\{c_1, a_2\}$, $n-3$ ways to choose $d_2$ from $N_n\setminus\{d_1, c_2, a_2\}$.
Therefore, there are $n(n-1)^3(n-2)(n-3)=n^6-8n^5+24n^4-34n^3+23n^2-6n$ ways to choose four $B_{n-2}$'s as required. Observation \ref{ob:2} implies that for any $i\in\{c, d\}$, $a_1 a_2 X^{n-2}$ and $i_1 X^{n-2} i_2$ are not disjoint and $V(a_1 a_2 X^{n-2})\cap V(i_1 X^{n-2} i_2)=V(a_1 a_2 X^{n-3} i_2)$, and $b_1 b_2 X^{n-2}$ and $i_1 X^{n-2} i_2$ are disjoint. Thus, $|V(a_1 a_2 X^{n-2})\cup V(b_1 b_2 X^{n-2})\cup V(X^{n-2} c_1 c_2)\cup V(X^{n-2} d_1 d_2)|=|V(a_1 a_2 X^{n-2})|+|V(b_1 b_2 X^{n-2})|+|V(X^{n-2} c_1 c_2)|+|V(X^{n-2} d_1 d_2)|-|V(a_1 a_2 X^{n-2})\cap V(c_1 X^{n-2} c_2)|-|V(a_1 a_2 X^{n-2})\cap V(d_1 X^{n-2}$ $d_2)|=4(n-2)!-2(n-3)!$. 
The probability $P_{2.1}$ that there are four fault-free $B_{n-2}$'s chosen as in Case 2.1 is $\frac{1}{4}(n^6-8n^5+24n^4-34n^3+23n^2-6n)p^{4(n-2)!-2(n-3)!}$.

{\it Case 2.2.} $a_2=c_2\neq d_2$ or $a_2=d_2\neq c_2$.

W.l.o.g, assume that $a_2=c_2\neq d_2$.
There are $n$ ways to choose $a_1(=c_1=d_1)$ from $N_n$, $n-1$ ways to choose $b_1$ from $N_n\setminus\{a_1\}$, $n-1$ ways to choose $a_2(=c_2)$ from $N_n\setminus\{a_1\}$, $n-1$ ways to choose $b_2$ from $N_n\setminus\{b_1\}$, $n-2$ ways to choose $d_2$ from $N_n\setminus\{d_1, c_2\}$.
Therefore, there are $2n(n-1)^3(n-2)=2(n^5-5n^4+9n^3-7n^2+2n)$ ways to choose four $B_{n-2}$'s as required. Observation \ref{ob:2} implies that $a_1 a_2 X^{n-2}$ and $c_1 X^{n-2} c_2$ are disjoint, $a_1 a_2 X^{n-2}$ and $d_1 X^{n-2} d_2$ are not disjoint and $V(a_1 a_2 X^{n-2})\cap V(d_1 X^{n-2} d_2)=V(a_1 a_2 X^{n-3} d_2)$, and for any $i\in\{c, d\}$, and $b_1 b_2 X^{n-2}$ and $i_1 X^{n-2} i_2$ are disjoint. Thus, $|V(a_1 a_2 X^{n-2})\cup V(b_1 b_2 X^{n-2})\cup V(X^{n-2} c_1 c_2)\cup V(X^{n-2} d_1 d_2)|=|V(a_1 a_2 X^{n-2})|+|V(b_1 b_2 X^{n-2})|+|V(X^{n-2} c_1 c_2)|+|V(X^{n-2} d_1 d_2)|-|V(a_1 a_2 X^{n-2})\cap V(d_1 X^{n-2} d_2)|=4(n-2)!-(n-3)!$. 
The probability $P_{2.2}$ that there are four fault-free $B_{n-2}$'s chosen as in Case 2.2 is $\frac{1}{4}\times 2(n^5-5n^4+9n^3-7n^2+2n)p^{4(n-2)!-(n-3)!}$.

In summary, the probability $P_2$ that there are four fault-free $B_{n-2}$'s chosen as in Case 2 is $4(P_{2.1}+P_{2.2})=2(n^5-5n^4+9n^3-7n^2+2n)p^{4(n-2)!-(n-3)!}+(n^6-8n^5+24n^4-34n^3+23n^2-6n)p^{4(n-2)!-2(n-3)!}$.

{\it Case 3.} $a_1=c_1\neq b_1=d_1$ or $a_1=d_1\neq b_1=c_1$.

W.l.o.g, assume that $a_1=c_1\neq b_1=d_1$.

{\it Case 3.1.} $a_2\neq c_2$ and $b_2\neq d_2$.

There are $n$ ways to choose $a_1(=c_1)$ from $N_n$, $n-1$ ways to choose $b_1(=d_1)$ from $N_n\setminus\{a_1\}$, $n-1$ ways to choose $a_2$ from $N_n\setminus\{a_1\}$, $n-1$ ways to choose $b_2$ from $N_n\setminus\{b_1\}$, $n-2$ ways to choose $c_2$ from $N_n\setminus\{c_1, a_2\}$, $n-2$ ways to choose $d_2$ from $N_n\setminus\{d_1, b_2\}$.
Therefore, there are $n(n-1)^3(n-2)^2=n^6-7n^5+19n^4-25n^3+16n^2-4n$ ways to choose four $B_{n-2}$'s as required. Observation \ref{ob:2} implies that $a_1 a_2 X^{n-2}$ and $c_1 X^{n-2} c_2$ are not disjoint and $V(a_1 a_2 X^{n-2})\cap V(c_1 X^{n-2} c_2)=V(a_1 a_2 X^{n-3} c_2)$, and $a_1 a_2 X^{n-2}$ and $d_1 X^{n-2} d_2$ are disjoint, and $b_1 b_2 X^{n-2}$ and $c_1 X^{n-2} c_2$ are disjoint, and $b_1 b_2 X^{n-2}$ and $d_1 X^{n-2} d_2$ are not disjoint and $V(b_1 b_2 X^{n-2})\cap V(d_1 X^{n-2} d_2)=V(b_1 b_2 X^{n-3} d_2)$. Thus, $|V(a_1 a_2 X^{n-2})\cup V(b_1 b_2 X^{n-2})\cup V(X^{n-2} c_1 c_2)\cup V(X^{n-2} d_1 d_2)|=|V(a_1 a_2 X^{n-2})|+|V(b_1 b_2 X^{n-2})|+|V(X^{n-2} c_1 c_2)|+|V(X^{n-2} d_1 d_2)|-|V(a_1 a_2 X^{n-2})\cap V(c_1 X^{n-2} c_2)|-|V(b_1 b_2 X^{n-2})\cap V(d_1 X^{n-2}$ $d_2)|=4(n-2)!-2(n-3)!$. 
The probability $P_{3.1}$ that there are four fault-free $B_{n-2}$'s chosen as in Case 3.1 is $\frac{1}{4}(n^6-7n^5+19n^4-25n^3+16n^2-4n)p^{4(n-2)!-2(n-3)!}$.

{\it Case 3.2.} $a_2=c_2$, $b_2\neq d_2$ or $a_2\neq c_2$, $b_2=d_2$.

W.l.o.g, assume that the former applies.
There are $n$ ways to choose $a_1(=c_1)$ from $N_n$, $n-1$ ways to choose $b_1(=d_1)$ from $N_n\setminus\{a_1\}$, $n-1$ ways to choose $a_2(=c_2)$ from $N_n\setminus\{a_1\}$, $n-1$ ways to choose $b_2$ from $N_n\setminus\{b_1\}$, $n-2$ ways to choose $d_2$ from $N_n\setminus\{d_1, b_2\}$.
Therefore, there are $2n(n-1)^3(n-2)=2(n^5-5n^4+9n^3-7n^2+2n)$ ways to choose four $B_{n-2}$'s as required. Observation \ref{ob:2} implies that for any $i\in\{c, d\}$, $a_1 a_2 X^{n-2}$ and $i_1 X^{n-2} i_2$ are disjoint, and $b_1 b_2 X^{n-2}$ and $c_1 X^{n-2} c_2$ are disjoint, and $b_1 b_2 X^{n-2}$ and $d_1 X^{n-2} d_2$ are not disjoint and $V(b_1 b_2 X^{n-2})\cap V(d_1 X^{n-2} d_2)=V(b_1 b_2 X^{n-3} d_2)$. Thus, $|V(a_1 a_2 X^{n-2})\cup V(b_1 b_2 X^{n-2})\cup V(X^{n-2} c_1 c_2)\cup V(X^{n-2} d_1 d_2)|=|V(a_1 a_2 X^{n-2})|+|V(b_1 b_2 X^{n-2})|+|V(X^{n-2} c_1 c_2)|+|V(X^{n-2} d_1 d_2)|-|V(b_1 b_2 X^{n-2})\cap V(d_1$ $X^{n-2} d_2)|=4(n-2)!-(n-3)!$. 
The probability $P_{3.2}$ that there are four fault-free $B_{n-2}$'s chosen as in Case 3.2 is $\frac{1}{4}\times 2(n^5-5n^4+9n^3-7n^2+2n)p^{4(n-2)!-(n-3)!}$.

{\it Case 3.3.} $a_2=c_2$ and $b_2=d_2$.

There are $n$ ways to choose $a_1(=c_1)$ from $N_n$, $n-1$ ways to choose $b_1(=d_1)$ from $N_n\setminus\{a_1\}$, $n-1$ ways to choose $a_2(=c_2)$ from $N_n\setminus\{a_1\}$, $n-1$ ways to choose $b_2(=d_2)$ from $N_n\setminus\{b_1\}$.
Therefore, there are $n(n-1)^3=n^4-3n^3+3n^2-n$ ways to choose four $B_{n-2}$'s as required. Observation \ref{ob:2} implies that for any $i\in\{a, b\}$ and $j\in\{c, d\}$, $i_1 i_2 X^{n-2}$ and $j_1 X^{n-2} j_2$ are disjoint. Thus, $|V(a_1 a_2 X^{n-2})\cup V(b_1 b_2 X^{n-2})\cup V(X^{n-2} c_1 c_2)\cup V(X^{n-2} d_1 d_2)|=|V(a_1 a_2 X^{n-2})|+|V(b_1 b_2 X^{n-2})|+|V(X^{n-2} c_1 c_2)|+|V(X^{n-2} d_1 d_2)|=4(n-2)!$. 
The probability $P_{3.3}$ that there are four fault-free $B_{n-2}$'s chosen as in Case 3.3 is $\frac{1}{4}(n^4-3n^3+3n^2-n)p^{4(n-2)!}$.

In summary, the probability $P_3$ that there are four fault-free $B_{n-2}$'s chosen as in Case 3 is $2(P_{3.1}+P_{3.2}+P_{3.3})=\frac{1}{2}(n^4-3n^3+3n^2-n)p^{4(n-2)!}+(n^5-5n^4+9n^3-7n^2+2n)p^{4(n-2)!-(n-3)!}+\frac{1}{2}(n^6-7n^5+19n^4-25n^3+16n^2-4n)p^{4(n-2)!-2(n-3)!}$.

{\it Case 4.} $a_1=b_1\neq c_1=d_1$.

There are $n$ ways to choose $a_1(=b_1)$ from $N_n$, $n-1$ ways to choose $c_1(=d_1)$ from $N_n\setminus\{a_1\}$, $n-1$ ways to choose $a_2$ from $N_n\setminus\{a_1\}$, $n-2$ ways to choose $b_2$ from $N_n\setminus\{b_1, a_2\}$, $n-1$ ways to choose $c_2$ from $N_n\setminus\{c_1\}$, $n-2$ ways to choose $d_2$ from $N_n\setminus\{d_1, c_2\}$.
Therefore, there are $n(n-1)^3(n-2)^2=n^6-7n^5+19n^4-25n^3+16n^2-4n$ ways to choose four $B_{n-2}$'s as required. Observation \ref{ob:2} implies that for any $i\in\{a, b\}$ and $j\in\{c, d\}$, $i_1 i_2 X^{n-2}$ and $j_1 X^{n-2} j_2$ are disjoint. Thus, $|V(a_1 a_2 X^{n-2})\cup V(b_1 b_2 X^{n-2})\cup V(X^{n-2} c_1 c_2)\cup V(X^{n-2} d_1 d_2)|=|V(a_1 a_2 X^{n-2})|+|V(b_1 b_2 X^{n-2})|+|V(X^{n-2} c_1 c_2)|+|V(X^{n-2} d_1 d_2)|=4(n-2)!$. 
The probability $P_4$ that there are four fault-free $B_{n-2}$'s chosen as in Case 4 is $\frac{1}{4}(n^6-7n^5+19n^4-25n^3+16n^2-4n)p^{4(n-2)!}$.

{\it Case 5.} $a_1=b_1\notin\{c_1, d_1\}$ or $c_1=d_1\notin\{a_1, b_1\}$.

W.l.o.g, assume that $a_1=b_1\notin\{c_1, d_1\}$.
There are $n$ ways to choose $c_1$ from $N_n$, $n-1$ ways to choose $d_1$ from $N_n\setminus\{c_1\}$, $n-2$ ways to choose $a_1(=b_1)$ from $N_n\setminus\{c_1, d_1\}$, $n-1$ ways to choose $a_2$ from $N_n\setminus\{a_1\}$, $n-2$ ways to choose $b_2$ from $N_n\setminus\{b_1, a_2\}$, $n-1$ ways to choose $c_2$ from $N_n\setminus\{c_1\}$, $n-1$ ways to choose $d_2$ from $N_n\setminus\{d_1\}$.
Therefore, there are $2n(n-1)^4(n-2)^2=2(n^7-8n^6+26n^5-44n^4+41n^3-20n^2+4n)$ ways to choose four $B_{n-2}$'s as required. Observation \ref{ob:2} implies that for any $i\in\{a, b\}$ and $j\in\{c, d\}$, $i_1 i_2 X^{n-2}$ and $j_1 X^{n-2} j_2$ are disjoint. Thus, $|V(a_1 a_2 X^{n-2})\cup V(b_1 b_2 X^{n-2})\cup V(X^{n-2} c_1 c_2)\cup V(X^{n-2} d_1 d_2)|=|V(a_1 a_2 X^{n-2})|+|V(b_1 b_2 X^{n-2})|+|V(X^{n-2} c_1 c_2)|+|V(X^{n-2} d_1 d_2)|=4(n-2)!$. 
The probability $P_5$ that there are four fault-free $B_{n-2}$'s chosen as in Case 5 is $\frac{1}{4}\times 2(n^7-8n^6+26n^5-44n^4+41n^3-20n^2+4n)p^{4(n-2)!}$.

{\it Case 6.} For any $\{x, y\}=\{a, b\}$ and $\{z, w\}=\{c, d\}$, $x_1=z_1\neq w_1$, $y_1\notin\{c_1, d_1\}$.

There are 4 possible cases, w.l.o.g, assume that $a_1=c_1\neq d_1$ and $b_1\notin\{c_1, d_1\}$.

{\it Case 6.1.} $a_2\neq c_2$.

There are $n$ ways to choose $a_1{=c_1}$ from $N_n$, $n-1$ ways to choose $d_1$ from $N_n\setminus\{c_1\}$, $n-2$ ways to choose $b_1$ from $N_n\setminus\{c_1, d_1\}$, $n-1$ ways to choose $a_2$ from $N_n\setminus\{a_1\}$, $n-1$ ways to choose $b_2$ from $N_n\setminus\{b_1\}$, $n-2$ ways to choose $c_2$ from $N_n\setminus\{c_1, a_2\}$, $n-1$ ways to choose $d_2$ from $N_n\setminus\{d_1\}$.
Therefore, there are $n(n-1)^4(n-2)^2=n^7-8n^6+26n^5-44n^4+41n^3-20n^2+4n$ ways to choose four $B_{n-2}$'s as required. Observation \ref{ob:2} implies that $a_1 a_2 X^{n-2}$ and $c_1 X^{n-2} c_2$ are not disjoint and $V(a_1 a_2 X^{n-2})\cap V(c_1 X^{n-2} c_2)=V(a_1 a_2 X^{n-3} c_2)$, and $a_1 a_2 X^{n-2}$ and $d_1 X^{n-2} d_2$ are disjoint, and for any $i\in\{c, d\}$, $b_1 b_2 X^{n-2}$ and $i_1 X^{n-2} i_2$ are disjoint. Thus, $|V(a_1 a_2 X^{n-2})\cup V(b_1 b_2 X^{n-2})\cup V(X^{n-2} c_1 c_2)\cup V(X^{n-2} d_1 d_2)|=|V(a_1 a_2 X^{n-2})|+|V(b_1 b_2 X^{n-2})|+|V(X^{n-2} c_1 c_2)|+|V(X^{n-2} d_1 d_2)|-|V(a_1 a_2 X^{n-2})\cap V(c_1 X^{n-2} c_2)|=4(n-2)!-(n-3)!$. 
The probability $P_{6.1}$ that there are four fault-free $B_{n-2}$'s chosen as in Case 6.1 is $\frac{1}{4}(n^7-8n^6+26n^5-44n^4+41n^3-20n^2+4n)p^{4(n-2)!-(n-3)!}$.

{\it Case 6.2.} $a_2=c_2$.

There are $n$ ways to choose $a_1(=c_1)$ from $N_n$, $n-1$ ways to choose $d_1$ from $N_n\setminus\{c_1\}$, $n-2$ ways to choose $b_1$ from $N_n\setminus\{c_1, d_1\}$, $n-1$ ways to choose $a_2(=c_2)$ from $N_n\setminus\{a_1\}$, $n-1$ ways to choose $b_2$ from $N_n\setminus\{b_1\}$, $n-1$ ways to choose $d_2$ from $N_n\setminus\{d_1\}$.
Therefore, there are $n(n-1)^4(n-2)=n^6-6n^5+14n^4-16n^3+9n^2-2n$ ways to choose four $B_{n-2}$'s as required. Observation \ref{ob:2} implies that $a_1 a_2 X^{n-2}$ and $c_1 X^{n-2} c_2$ are not disjoint and $V(a_1 a_2 X^{n-2})\cap V(c_1 X^{n-2} c_2)=V(a_1 a_2 X^{n-3} c_2)$, and $a_1 a_2 X^{n-2}$ and $d_1 X^{n-2} d_2$ are disjoint, and for any $i\in\{c, d\}$, $b_1 b_2 X^{n-2}$ and $i_1 X^{n-2} i_2$ are disjoint. Thus, $|V(a_1 a_2 X^{n-2})\cup V(b_1 b_2 X^{n-2})\cup V(X^{n-2} c_1 c_2)\cup V(X^{n-2} d_1 d_2)|=|V(a_1 a_2 X^{n-2})|+|V(b_1 b_2 X^{n-2})|+|V(X^{n-2} c_1 c_2)|+|V(X^{n-2} d_1 d_2)|-|V(a_1 a_2 X^{n-2})\cap V(c_1$ $X^{n-2} c_2)|=4(n-2)!-(n-3)!$. 
The probability $P_{6.2}$ that there are four fault-free $B_{n-2}$'s chosen as in Case 6.2 is $\frac{1}{4}(n^6-6n^5+14n^4-16n^3+9n^2-2n)p^{4(n-2)!}$.

In summary, the probability $P_6$ that there are four fault-free $B_{n-2}$'s chosen as in Case 6 is $4(P_{6.1}+P_{6.2})=(n^6-6n^5+14n^4-16n^3+9n^2-2n)p^{4(n-2)!}+(n^7-8n^6+26n^5-44n^4+41n^3-20n^2+4n)p^{4(n-2)!-(n-3)!}$.

{\it Case 7.} $\{a_1, b_1\}\cap\{c_1, d_1\}=\emptyset$.

There are $n$ ways to choose $a_1$ from $N_n$, $n-1$ ways to choose $b_1$ from $N_n\setminus\{a_1\}$, $n-2$ ways to choose $c_1$ from $N_n\setminus\{a_1, b_1\}$, $n-3$ ways to choose $d_1$ from $N_n\setminus\{a_1, b_1, c_1\}$, $n-1$ ways to choose $a_2$ from $N_n\setminus\{a_1\}$, $n-1$ ways to choose $b_2$ from $N_n\setminus\{b_1\}$, $n-1$ ways to choose $c_2$ from $N_n\setminus\{c_1\}$, $n-1$ ways to choose $d_2$ from $N_n\setminus\{d_1\}$.
Therefore, there are $n(n-1)^5(n-2)(n-3)=n^8-10n^7+41n^6-90n^5+115n^4-86n^3+35n^2-6n$ ways to choose four $B_{n-2}$'s as required. Observation \ref{ob:2} implies that for any $i\in\{a, b\}$ and $j\in\{c, d\}$, $i_1 i_2 X^{n-2}$ and $j_1 X^{n-2} j_2$ are disjoint. Thus, $|V(a_1 a_2 X^{n-2})\cup V(b_1 b_2 X^{n-2})\cup V(X^{n-2} c_1 c_2)\cup V(X^{n-2} d_1 d_2)|=|V(a_1 a_2 X^{n-2})|+|V(b_1 b_2 X^{n-2})|+|V(X^{n-2} c_1 c_2)|+|V(X^{n-2} d_1 d_2)|=4(n-2)!$. 
The probability $P_7$ that there are four fault-free $B_{n-2}$'s chosen as in Case 7 is $\frac{1}{4}(n^8-10n^7+41n^6-90n^5+115n^4-86n^3+35n^2-6n)p^{4(n-2)!}$.

By the above computations, $P(2,2,0)=\sum_{i=1}^{7}P_i=\frac{1}{4}(n^8-8n^7+30n^6-69n^5+104n^4-99n^3+53n^2-12n)p^{4(n-2)!}+(n^7-8n^6+29n^5-59n^4+68n^3-41n^2+10n)p^{4(n-2)!-(n-3)!}+\frac{1}{2}(3n^6-23n^5+67n^4-92n^3+59n^2-14n)p^{4(n-2)!-2(n-3)!}+(n^4-6n^3+11n^2-6n)p^{4(n-2)!-3(n-3)!}+\frac{1}{4}(n^5-10n^4+35n^3-50n^2+24n)p^{4(n-2)!-4(n-3)!}$.
\end{proof}

\begin{theorem}\label{th7}
$P(0,2,2)=\frac{1}{4}(n^8-8n^7+30n^6-69n^5+104n^4-99n^3+53n^2-12n)p^{4(n-2)!}+(n^7-8n^6+29n^5-59n^4+68n^3-41n^2+10n)p^{4(n-2)!-(n-3)!}+\frac{1}{2}(3n^6-23n^5+67n^4-92n^3+59n^2-14n)p^{4(n-2)!-2(n-3)!}+(n^4-6n^3+11n^2-6n)p^{4(n-2)!-3(n-3)!}+\frac{1}{4}(n^5-10n^4+35n^3-50n^2+24n)p^{4(n-2)!-4(n-3)!}$.
\end{theorem}
\begin{proof}
This scenario is similar to Theorem \ref{th6}, it is easy to get that $P(0,2,2)=P(2,2,0)=\frac{1}{4}(n^8-8n^7+30n^6-69n^5+104n^4-99n^3+53n^2-12n)p^{4(n-2)!}+(n^7-8n^6+29n^5-59n^4+68n^3-41n^2+10n)p^{4(n-2)!-(n-3)!}+\frac{1}{2}(3n^6-23n^5+67n^4-92n^3+59n^2-14n)p^{4(n-2)!-2(n-3)!}+(n^4-6n^3+11n^2-6n)p^{4(n-2)!-3(n-3)!}+\frac{1}{4}(n^5-10n^4+35n^3-50n^2+24n)p^{4(n-2)!-4(n-3)!}$.
\end{proof}

\begin{theorem}\label{th8}
$P(2,0,2)=\frac{1}{2}(8n^5+8n^4-174n^3+347n^2-189n)p^{4(n-2)!}+16(3n^5-25n^4+75n^3-95n^2+42n)p^{4(n-2)!-(n-4)!}+4(6n^6-73n^5+345n^4-785n^3+849n^2-342n)p^{4(n-2)!-2(n-4)!}+2(2n^7-33n^6+223n^5-785n^4+1503n^3-1462n^2+552n)p^{4(n-2)!-3(n-4)!}+\frac{1}{4}(n^8-20n^7+172n^6-822n^5+2340n^4-3916n^3+3499n^2-1254n)p^{4(n-2)!-4(n-4)!}$.
\end{theorem}
\begin{proof}
Denote by $a_1 a_2 X^{n-2}\in H_{1,2}$, $b_1 b_2 X^{n-2}\in H_{1,2}$, $X^{n-2} c_1 c_2\in H_{n-1, n}$ and $X^{n-2} d_1 d_2\in H_{n-1,n}$ these four $B_{n-2}$'s. Note that $a_1 a_2 X^{n-2}$ and $b_1 b_2 X^{n-2}$ are disjoint, and $X^{n-2} c_1 c_2$ and $X^{n-2} d_1 d_2$ are disjoint.

Since we have implicitly sorted the values of $a_1 a_2$ and $b_1 b_2$, and the values of $c_1 c_2$ and $d_1 d_2$, the number of ways 
to choose four $B_{n-2}$'s as required needs to be divided by 4. It will be applied directly in following analyses.

In the following analysis, we will use a 4-digit string to represent different cases. These four numbers represent $|\{a_1, a_2\}\cap\{c_1, c_2\}|$, $|\{a_1, a_2\}\cap\{d_1, d_2\}|$, $|\{b_1, b_2\}\cap\{c_1, c_2\}|$, $|\{b_1, b_2\}\cap\{d_1, d_2\}|$ respectively. For example, 0121 represents the case $\{a_1, a_2\}\cap\{c_1, c_2\}=\emptyset$, $|\{a_1, a_2\}\cap\{d_1, d_2\}|=1$, $\{b_1, b_2\}=\{c_1, c_2\}$ and $|\{b_1, b_2\}\cap\{d_1, d_2\}|=1$.

For any $i\in\{a, b\}$ and $j\in\{c, d\}$, there are 2 possible cases ($\romannumeral1$) $i_1=j_1$, $i_2=j_2$, and ($\romannumeral2$) $i_1=j_2$, $i_2=j_1$ such that $|\{i_1, i_2\}\cap\{j_1, j_2\}|=2$, and 4 possible cases ($\romannumeral1$) $i_1=j_1$, $i_2\notin\{j_1, j_2\}$, and ($\romannumeral2$) $i_1=j_2$, $i_2\notin\{j_1, j_2\}$, and ($\romannumeral3$) $i_2=j_1$, $i_1\notin\{j_1, j_2\}$, and ($\romannumeral4$) $i_2=j_2$, $i_1\notin\{j_1, j_2\}$ such that $|\{i_1, i_2\}\cap\{j_1, j_2\}|=1$.

We will analyze this problem in 17 cases as follows.

{\it Case 1.} 2222.

There are 2 possible cases $a_1=b_2=c_1=d_2$, $a_2=b_1=c_2=d_1$ and $a_1=b_2=c_2=d_1$, $a_2=b_1=c_1=d_2$. W.l.o.g, assume that the former applies.
There are $n$ ways to choose $a_1(=b_2=c_1=d_2)$ from $N_n$, $n-1$ ways to choose $a_2(=b_1=c_2=d_1)$ from $N_n\setminus\{a_1\}$.
Therefore, there are $2n(n-1)=2(n^2-n)$ ways to choose four $B_{n-2}$'s as required. Observation \ref{ob:4} implies that for any $i\in\{a, b\}$ and $j\in\{c, d\}$, $i_1 i_2 X^{n-2}$ and $X^{n-2} j_1 j_2$ are disjoint. Thus, $|V(a_1 a_2 X^{n-2})\cup V(b_1 b_2 X^{n-2})\cup V(X^{n-2} c_1 c_2)\cup V(X^{n-2} d_1 d_2)|=|V(a_1 a_2 X^{n-2})|+|V(b_1 b_2 X^{n-2})|+|V(X^{n-2} c_1 c_2)|+|V(X^{n-2} d_1 d_2)|=4(n-2)!$. 
The probability $P_1$ that there are four fault-free $B_{n-2}$'s chosen as in Case 1 is $\frac{1}{4}\times 2(n^2-n)p^{4(n-2)!}$.

{\it Case 2.} 2211 or 1122 or 2121 or 1212.

W.l.o.g, assume that the case of 2211.
There are 2 possible cases $a_1=c_1=d_2$, $a_2=c_2=d_1$ and $a_1=c_2=d_1$, $a_2=c_1=d_2$ such that $\{a_1, a_2\}=\{c_1, c_2\}=\{d_1, d_2\}$. There are 4 possible cases ($\romannumeral1$) $b_1=c_1=d_2$, $b_2\notin\{c_1, c_2\}$, $b_2\notin\{d_1, d_2\}$, and ($\romannumeral2$) $b_1=c_2=d_1$, $b_2\notin\{c_1, c_2\}$, $b_2\notin\{d_1, d_2\}$, and ($\romannumeral3$) $b_2=c_1=d_2$, $b_1\notin\{c_1, c_2\}$, $b_1\notin\{d_1, d_2\}$, and ($\romannumeral4$) $b_2=c_2=d_1$, $b_1\notin\{c_1, c_2\}$, $b_1\notin\{d_1, d_2\}$ such that $|\{b_1, b_2\}\cap\{c_1, c_2\}|=1$ and $|\{b_1, b_2\}\cap\{d_1, d_2\}|=1$. Therefore, there are 8 possible scenarios for this situation. W.l.o.g, assume that $a_1=c_1=d_2$, $a_2=c_2=d_1$, $b_1=c_1=d_2$, $b_2\notin\{c_1, c_2\}$, $b_2\notin\{d_1, d_2\}$.

There are $n$ ways to choose $a_1(=c_1=d_2=b_1)$ from $N_n$, $n-1$ ways to choose $a_2(=c_2=d_1)$ from $N_n\setminus\{a_1\}$, $n-2$ ways to choose $b_2$ from $N_n\setminus\{c_1, c_2\}$.
Therefore, there are $4\times 8n(n-1)(n-2)=32(n^3-3n^2+2n)$ ways to choose four $B_{n-2}$'s as required. Observation \ref{ob:4} implies that for any $i\in\{a, b\}$ and $j\in\{c, d\}$, $i_1 i_2 X^{n-2}$ and $X^{n-2} j_1 j_2$ are disjoint. Thus, $|V(a_1 a_2 X^{n-2})\cup V(b_1 b_2 X^{n-2})\cup V(X^{n-2} c_1 c_2)\cup V(X^{n-2} d_1 d_2)|=|V(a_1 a_2 X^{n-2})|+|V(b_1 b_2 X^{n-2})|+|V(X^{n-2} c_1 c_2)|+|V(X^{n-2} d_1 d_2)|=4(n-2)!$. 
The probability $P_2$ that there are four fault-free $B_{n-2}$'s chosen as in Case 2 is $\frac{1}{4}\times 32(n^3-3n^2+2n)p^{4(n-2)!}$.

{\it Case 3.} 2112 or 1221.

W.l.o.g, assume that the case of 2112.
There are 4 possible cases such that $\{a_1, a_2\}=\{c_1, c_2\}$ and $\{b_1, b_2\}=\{d_1, d_2\}$. W.l.o.g, assume that $a_1=c_1$, $a_2=c_2$, $b_1=d_1$, $b_2=d_2$.
Note that if there is some $a_s=d_t$, then there must be some $b_p=c_q$, where $s, t, p, q\in\{1, 2\}$.
Thus, there are 4 possible cases ($\romannumeral1$) $a_1=c_1=b_1=d_1$, $a_2=c_2$, $b_2=d_2$, $a_2\notin\{d_1, d_2\}$, $b_2\notin\{c_1, c_2\}$, and ($\romannumeral2$) $a_2=c_2=b_1=d_1$, $a_1=c_1$, $b_2=d_2$, $a_1\notin\{d_1, d_2\}$, $b_2\notin\{c_1, c_2\}$, and ($\romannumeral3$) $a_1=c_1=b_2=d_2$, $a_2=c_2$, $b_1=d_1$, $a_2\notin\{d_1, d_2\}$, $b_1\notin\{c_1, c_2\}$, and ($\romannumeral4$) $a_2=c_2=b_2=d_2$, $a_1=c_1$, $b_1=d_1$, $a_1\notin\{d_1, d_2\}$, $b_1\notin\{c_1, c_2\}$ such that $|\{a_1, a_2\}\cap\{d_1, d_2\}|=1$ and $|\{b_1, b_2\}\cap\{c_1, c_2\}|=1$.
W.l.o.g, assume that the first case applies.

There are $n$ ways to choose $a_1(=c_1=d_2=b_1)$ from $N_n$, $n-1$ ways to choose $a_2(=c_2)$ from $N_n\setminus\{a_1\}$, $n-2$ ways to choose $b_2(=d_2)$ from $N_n\setminus\{c_1, c_2\}$.
Therefore, there are $2\times 4\times 4n(n-1)(n-2)=32(n^3-3n^2+2n)$ ways to choose four $B_{n-2}$'s as required. Observation \ref{ob:4} implies that for any $i\in\{a, b\}$ and $j\in\{c, d\}$, $i_1 i_2 X^{n-2}$ and $X^{n-2} j_1 j_2$ are disjoint. Thus, $|V(a_1 a_2 X^{n-2})\cup V(b_1 b_2 X^{n-2})\cup V(X^{n-2} c_1 c_2)\cup V(X^{n-2} d_1 d_2)|=|V(a_1 a_2 X^{n-2})|+|V(b_1 b_2 X^{n-2})|+|V(X^{n-2} c_1 c_2)|+|V(X^{n-2} d_1 d_2)|=4(n-2)!$. 
The probability $P_3$ that there are four fault-free $B_{n-2}$'s chosen as in Case 3 is $\frac{1}{4}\times 32(n^3-3n^2+2n)p^{4(n-2)!}$.

{\it Case 4.} 2200 or 0022 or 2020 or 0202.

W.l.o.g, assume that the case of 2200.
There are 2 possible cases $a_1=c_1=d_2$, $a_2=c_2=d_1$ and $a_1=c_2=d_1$, $a_2=c_1=d_2$ such that $\{a_1, a_2\}=\{c_1, c_2\}=\{d_1, d_2\}$. W.l.o.g, assume that $a_1=c_1=d_2$, $a_2=c_2=d_1$.
There are $n$ ways to choose $a_1(=c_1=d_2)$ from $N_n$, $n-1$ ways to choose $a_2(=c_2=d_1)$ from $N_n\setminus\{a_1\}$, $n-2$ ways to choose $b_1$ from $N_n\setminus\{c_1, c_2\}$, $n-3$ ways to choose $b_2$ from $N_n\setminus\{c_1, c_2, b_1\}$.
Therefore, there are $4\times 2n(n-1)(n-2)(n-3)=8(n^4-6n^3+11n^2-6n)$ ways to choose four $B_{n-2}$'s as required. Observation \ref{ob:4} implies that for any $i\in\{c, d\}$, $a_1 a_2 X^{n-2}$ and $X^{n-2} i_1 i_2$ are disjoint, and $b_1 b_2 X^{n-2}$ and $X^{n-2} i_1 i_2$ are not disjoint and $V(b_1 b_2 X^{n-2})\cap V(X^{n-2} i_1 i_2)=V(b_1 b_2 X^{n-4} i_1 i_2)$. Thus, $|V(a_1 a_2 X^{n-2})\cup V(b_1 b_2 X^{n-2})\cup V(X^{n-2} c_1 c_2)\cup V(X^{n-2} d_1 d_2)|=|V(a_1 a_2 X^{n-2})|+|V(b_1 b_2 X^{n-2})|+|V(X^{n-2} c_1 c_2)|+|V(X^{n-2} d_1 d_2)|-|V(b_1 b_2 X^{n-2})\cap V(X^{n-2} c_1 c_2)|-|V(b_1 b_2 X^{n-2})\cap V(X^{n-2} c_1 c_2)|=4(n-2)!-2(n-4)!$. 
The probability $P_4$ that there are four fault-free $B_{n-2}$'s chosen as in Case 4 is $\frac{1}{4}\times 8(n^4-6n^3+11n^2-6n)p^{4(n-2)!-2(n-4)!}$.

{\it Case 5.} 2002 or 0220.

W.l.o.g, assume that the case of 2002.
There are 4 possible cases, w.l.o.g, assume that $a_1=c_1$, $a_2=c_2$, $b_1=d_1$ and $b_2=d_2$.
There are $n$ ways to choose $a_1(=c_1)$ from $N_n$, $n-1$ ways to choose $a_2(=c_2)$ from $N_n\setminus\{a_1\}$, $n-2$ ways to choose $b_1(=d_1)$ from $N_n\setminus\{c_1, c_2\}$, $n-3$ ways to choose $b_2(=d_2)$ from $N_n\setminus\{c_1, c_2, b_1\}$.
Therefore, there are $2\times 4n(n-1)(n-2)(n-3)=8(n^4-6n^3+11n^2-6n)$ ways to choose four $B_{n-2}$'s as required. Observation \ref{ob:4} implies that $a_1 a_2 X^{n-2}$ and $X^{n-2} c_1 c_2$ are disjoint, $a_1 a_2 X^{n-2}$ and $X^{n-2} d_1 d_2$ are not disjoint and $V(a_1 a_2 X^{n-2})\cap V(X^{n-2} d_1 d_2)=V(a_1 a_2 X^{n-4} d_1 d_2)$, and $b_1 b_2 X^{n-2}$ and $X^{n-2} c_1 c_2$ are not disjoint and $V(b_1 b_2 X^{n-2})\cap V(X^{n-2} c_1 c_2)=V(b_1 b_2 X^{n-4} c_1 c_2)$, and $b_1 b_2 X^{n-2}$ and $X^{n-2} d_1 d_2$ are disjoint. Thus, $|V(a_1 a_2 X^{n-2})\cup V(b_1 b_2 X^{n-2})\cup V(X^{n-2} c_1 c_2)\cup V(X^{n-2} d_1 d_2)|=|V(a_1 a_2 X^{n-2})|+|V(b_1 b_2 X^{n-2})|+|V(X^{n-2} c_1 c_2)|+|V(X^{n-2} d_1 d_2)|-|V(a_1 a_2 X^{n-2})\cap V(X^{n-2} d_1 d_2)|-|V(b_1 b_2 X^{n-2})\cap V(X^{n-2} c_1 c_2)|=4(n-2)!-2(n-4)!$. 
The probability $P_5$ that there are four fault-free $B_{n-2}$'s chosen as in Case 5 is $\frac{1}{4}\times 8(n^4-6n^3+11n^2-6n)p^{4(n-2)!-2(n-4)!}$.

{\it Case 6.} 2111 or 1211 or 1121 or 1112.

W.l.o.g, assume that the case of 2111.
There are 8 possible cases such that $\{a_1, a_2\}=\{c_1, c_2\}$ and $|\{a_1, a_2\}\cap\{d_1, d_2\}|=1$. W.l.o.g, assume that $a_1=c_1=d_1$, $a_2=c_2$, $a_2\notin\{d_1, d_2\}$.

{\it Case 6.1.} $b_1=c_1$ or $b_2=c_1$.

There are 2 possible cases ($\romannumeral1$) $a_1=c_1=d_1=b_1$, $a_2=c_2$, $a_2\notin\{d_1, d_2\}$, $b_2\notin\{c_1, c_2\}$, $b_2\notin\{d_1, d_2\}$, and ($\romannumeral2$) $a_1=c_1=d_1=b_2$, $a_2=c_2$, $a_2\notin\{d_1, d_2\}$, $b_1\notin\{c_1, c_2\}$, $b_1\notin\{d_1, d_2\}$.
W.l.o.g, assume that the former applies.
There are $n$ ways to choose $a_1(=c_1=d_1=b_1)$ from $N_n$, $n-1$ ways to choose $a_2(=c_2)$ from $N_n\setminus\{a_1\}$, $n-2$ ways to choose $b_2$ from $N_n\setminus\{c_1, c_2\}$, $n-3$ ways to choose $d_2$ from $N_n\setminus\{d_1, a_2, b_2\}$. Therefore, there are $2n(n-1)(n-2)(n-3)=2(n^4-6n^3+11n^2-6n)$ ways to choose four $B_{n-2}$'s as required.

{\it Case 6.2.} $b_1=c_2$ or $b_2=c_2$.

There are 2 possible cases ($\romannumeral1$) $a_1=c_1=d_1$, $a_2=c_2=b_1$, $b_2=d_2$, $a_2\notin\{d_1, d_2\}$, $b_2\notin\{c_1, c_2\}$, $b_1\notin\{d_1, d_2\}$, and ($\romannumeral2$) $a_1=c_1=d_1$, $a_2=c_2=b_2$, $b_1=d_2$, $a_2\notin\{d_1, d_2\}$, $b_1\notin\{c_1, c_2\}$, $b_2\notin\{d_1, d_2\}$.
W.l.o.g, assume that the former applies.
There are $n$ ways to choose $a_1(=c_1=d_1)$ from $N_n$, $n-1$ ways to choose $a_2(=c_2=b_1)$ from $N_n\setminus\{a_1\}$, $n-2$ ways to choose $b_2(=d_2)$ from $N_n\setminus\{c_1, c_2\}$. Therefore, there are $2n(n-1)(n-2)=2(n^3-3n^2+2n)$ ways to choose four $B_{n-2}$'s as required.

In summary, there are $2(n^4-5n^3+8n^2-4n)$ ways to choose four $B_{n-2}$'s as required in Case 6.
Observation \ref{ob:4} implies that for any $i\in\{a, b\}$ and $j\in\{c, d\}$, $i_1 i_2 X^{n-2}$ and $X^{n-2} j_1 j_2$ are disjoint. Thus, $|V(a_1 a_2 X^{n-2})\cup V(b_1 b_2 X^{n-2})\cup V(X^{n-2} c_1 c_2)\cup V(X^{n-2} d_1 d_2)|=|V(a_1 a_2 X^{n-2})|+|V(b_1 b_2 X^{n-2})|+|V(X^{n-2} c_1 c_2)|+|V(X^{n-2} d_1$ $d_2)|=4(n-2)!$.
The probability $P_6$ that there are four fault-free $B_{n-2}$'s chosen as in Case 6 is $\frac{1}{4}\times 4\times 8\times 2(n^4-5n^3+8n^2-4n)p^{4(n-2)!}$.

{\it Case 7.} 2110 or 1201 or 1021 or 0112.

W.l.o.g, assume that the case of 2110.
There are 8 possible cases such that $\{a_1, a_2\}=\{c_1, c_2\}$ and $|\{a_1, a_2\}\cap\{d_1, d_2\}|=1$. W.l.o.g, assume that $a_1=c_1=d_1$, $a_2=c_2$, $a_2\notin\{d_1, d_2\}$.
In this case, there are 2 possible cases ($\romannumeral1$) $a_1=c_1=d_1$, $a_2=c_2=b_1$, $a_2\notin\{d_1, d_2\}$, $b_2\notin\{c_1, c_2\}$, $\{b_1, b_2\}\cap\{d_1, d_2\}=\emptyset$, and ($\romannumeral2$) $a_1=c_1=d_1$, $a_2=c_2=b_2$, $a_2\notin\{d_1, d_2\}$, $b_1\notin\{c_1, c_2\}$, $\{b_1, b_2\}\cap\{d_1, d_2\}=\emptyset$.
W.l.o.g, assume that the former applies.

There are $n$ ways to choose $a_1(=c_1=d_1)$ from $N_n$, $n-1$ ways to choose $a_2(=c_2=b_1)$ from $N_n\setminus\{a_1\}$, $n-2$ ways to choose $b_2$ from $N_n\setminus\{c_1, c_2\}$, $n-3$ ways to choose $d_2$ from $N_n\setminus\{d_1, b_1, b_2\}$.
Therefore, there are $4\times 8\times 2n(n-1)(n-2)(n-3)=64(n^4-6n^3+11n^2-6n)$ ways to choose four $B_{n-2}$'s as required.
Observation \ref{ob:4} implies that for any $i\in\{c, d\}$, $a_1 a_2 X^{n-2}$ and $X^{n-2} i_1 i_2$ are disjoint, and $b_1 b_2 X^{n-2}$ and $X^{n-2} c_1 c_2$ are disjoint, and $b_1 b_2 X^{n-2}$ and $X^{n-2} d_1 d_2$ are not disjoint and $V(b_1 b_2 X^{n-2})\cap V(X^{n-2} d_1 d_2)=V(b_1 b_2 X^{n-4} d_1 d_2)$. Thus, $|V(a_1 a_2 X^{n-2})\cup V(b_1 b_2 X^{n-2})\cup V(X^{n-2} c_1 c_2)\cup V(X^{n-2} d_1 d_2)|=|V(a_1 a_2 X^{n-2})|+|V(b_1 b_2 X^{n-2})|+|V(X^{n-2} c_1 c_2)|+|V(X^{n-2} d_1 d_2)|-|V(b_1 b_2 X^{n-2})\cap V(X^{n-2} d_1 d_2)|=4(n-2)!-(n-4)!$.
The probability $P_7$ that there are four fault-free $B_{n-2}$'s chosen as in Case 7 is $\frac{1}{4}\times 64(n^4-6n^3+11n^2-6n)p^{4(n-2)!-(n-4)!}$.

{\it Case 8.} 2101 or 1210 or 1012 or 0121.

W.l.o.g, assume that the case of 2101.
There are 8 possible cases such that $\{a_1, a_2\}=\{c_1, c_2\}$ and $|\{a_1, a_2\}\cap\{d_1, d_2\}|=1$. W.l.o.g, assume that $a_1=c_1=d_1$, $a_2=c_2$, $a_2\notin\{d_1, d_2\}$.
In this case, there are 2 possible cases ($\romannumeral1$) $a_1=c_1=d_1$, $a_2=c_2$, $b_1=d_2$, $a_2\notin\{d_1, d_2\}$, $b_2\notin\{d_1, d_2\}$, $\{b_1, b_2\}\cap\{c_1, c_2\}=\emptyset$, and ($\romannumeral2$) $a_1=c_1=d_1$, $a_2=c_2$, $b_2=d_2$, $a_2\notin\{d_1, d_2\}$, $b_1\notin\{d_1, d_2\}$, $\{b_1, b_2\}\cap\{c_1, c_2\}=\emptyset$.
W.l.o.g, assume that the former applies.

There are $n$ ways to choose $a_1(=c_1=d_1)$ from $N_n$, $n-1$ ways to choose $a_2(=c_2)$ from $N_n\setminus\{a_1\}$, $n-2$ ways to choose $b_1(=d_2)$ from $N_n\setminus\{c_1, c_2\}$, $n-3$ ways to choose $b_2$ from $N_n\setminus\{b_1, c_1, c_2\}$. Therefore, there are $4\times 8\times 2n(n-1)(n-2)(n-3)=64(n^4-6n^3+11n^2-6n)$ ways to choose four $B_{n-2}$'s as required.
Observation \ref{ob:4} implies that for any $i\in\{c, d\}$, $a_1 a_2 X^{n-2}$ and $X^{n-2} i_1 i_2$ are disjoint, and $b_1 b_2 X^{n-2}$ and $X^{n-2} c_1 c_2$ are not disjoint and $V(b_1 b_2 X^{n-2})\cap V(X^{n-2} c_1 c_2)=V(b_1 b_2 X^{n-4} c_1 c_2)$, and $b_1 b_2 X^{n-2}$ and $X^{n-2} d_1 d_2$ are disjoint. Thus, $|V(a_1 a_2 X^{n-2})\cup V(b_1 b_2 X^{n-2})\cup V(X^{n-2} c_1 c_2)\cup V(X^{n-2} d_1 d_2)|=|V(a_1 a_2 X^{n-2})|+|V(b_1 b_2 X^{n-2})|+|V(X^{n-2} c_1 c_2)|+|V(X^{n-2} d_1 d_2)|-|V(b_1 b_2 X^{n-2})\cap V(X^{n-2} c_1 c_2)|=4(n-2)!-(n-4)!$.
The probability $P_8$ that there are four fault-free $B_{n-2}$'s chosen as in Case 8 is $\frac{1}{4}\times 64(n^4-6n^3+11n^2-6n)p^{4(n-2)!-(n-4)!}$.

{\it Case 9.} 2011 or 0211 or 1102 or 1120.

W.l.o.g, assume that the case of 2011.
There are 8 possible cases such that $\{a_1, a_2\}=\{c_1, c_2\}$ and $|\{b_1, b_2\}\cap\{c_1, c_2\}|=1$. W.l.o.g, assume that $a_1=c_1=b_1$, $a_2=c_2$, $b_2\notin\{c_1, c_2\}$.
In this case, there are 2 possible cases ($\romannumeral1$) $a_1=c_1=b_1$, $a_2=c_2$, $b_2=d_1$, $b_2\notin\{c_1, c_2\}$, $\{a_1, a_2\}\cap\{d_1, d_2\}=\emptyset$, and ($\romannumeral2$) $a_1=c_1=b_1$, $a_2=c_2$, $b_2=d_2$, $b_2\notin\{c_1, c_2\}$, $\{a_1, a_2\}\cap\{d_1, d_2\}=\emptyset$.
W.l.o.g, assume that the former applies.

There are $n$ ways to choose $a_1(=c_1=b_1)$ from $N_n$, $n-1$ ways to choose $a_2(=c_2)$ from $N_n\setminus\{a_1\}$, $n-2$ ways to choose $b_2(=d_1)$ from $N_n\setminus\{c_1, c_2\}$, $n-3$ ways to choose $d_2$ from $N_n\setminus\{d_1, a_1, a_2\}$. Therefore, there are $4\times 8\times 2n(n-1)(n-2)(n-3)=64(n^4-6n^3+11n^2-6n)$ ways to choose four $B_{n-2}$'s as required.
Observation \ref{ob:4} implies that $a_1 a_2 X^{n-2}$ and $X^{n-2} c_1 c_2$ are disjoint, and $b_1 b_2 X^{n-2}$ and $X^{n-2} d_1 d_2$ are not disjoint and $V(a_1 a_2 X^{n-2})\cap V(X^{n-2} d_1 d_2)=V(a_1 a_2 X^{n-4} d_1 d_2)$, and for any $i\in\{c, d\}$, $b_1 b_2 X^{n-2}$ and $X^{n-2} i_1 i_2$ are disjoint.
Thus, $|V(a_1 a_2 X^{n-2})\cup V(b_1 b_2 X^{n-2})\cup V(X^{n-2} c_1 c_2)\cup V(X^{n-2} d_1 d_2)|=|V(a_1 a_2 X^{n-2})|+|V(b_1 b_2 X^{n-2})|+|V(X^{n-2} c_1 c_2)|+|V(X^{n-2} d_1 d_2)|-|V(a_1 a_2 X^{n-2})\cap V(X^{n-2} d_1 d_2)|=4(n-2)!-(n-4)!$.
The probability $P_9$ that there are four fault-free $B_{n-2}$'s chosen as in Case 9 is $\frac{1}{4}\times 64(n^4-6n^3+11n^2-6n)p^{4(n-2)!-(n-4)!}$.

{\it Case 10.} 2100 or 1200 or 0021 or 0012 or 2010 or 0201 or 1020 or 0102 or 2001 or 0210 or 0120 or 1002.

W.l.o.g, assume that the case of 2100.
There are 8 possible cases such that $\{a_1, a_2\}=\{c_1, c_2\}$ and $|\{a_1, a_2\}\cap\{d_1, d_2\}|=1$. W.l.o.g, assume that $a_1=c_1=d_1$, $a_2=c_2$, $a_2\notin\{d_1, d_2\}$.

There are $n$ ways to choose $a_1(=c_1=d_1)$ from $N_n$, $n-1$ ways to choose $a_2(=c_2)$ from $N_n\setminus\{a_1\}$, $n-2$ ways to choose $d_2$ from $N_n\setminus\{d_1, a_2\}$, $n-3$ ways to choose $b_1$ from $N_n\setminus\{c_1, c_2, d_2\}$, $n-4$ ways to choose $b_2$ from $N_n\setminus\{c_1, c_2, d_2, b_1\}$. Therefore, there are $12\times 8n(n-1)(n-2)(n-3)(n-4)=96(n^5-10n^4+35n^3-50n^2+24n)$ ways to choose four $B_{n-2}$'s as required.
Observation \ref{ob:4} implies that $a_1 a_2 X^{n-2}$ and $X^{n-2} c_1 c_2$ are disjoint, and $b_1 b_2 X^{n-2}$ and $X^{n-2} d_1 d_2$ are not disjoint and $V(b_1 b_2 X^{n-2})\cap V(X^{n-2} d_1 d_2)=V(b_1 b_2 X^{n-4} d_1 d_2)$, and for any $i\in\{c, d\}$, $b_1 b_2 X^{n-2}$ and $X^{n-2} i_1 i_2$ are not disjoint and $V(b_1 b_2 X^{n-2})\cap V(X^{n-2} i_1 i_2)=V(b_1 b_2 X^{n-4} i_1 i_2)$.
Thus, $|V(a_1 a_2 X^{n-2})\cup V(b_1 b_2 X^{n-2})\cup V(X^{n-2} c_1 c_2)\cup V(X^{n-2} d_1 d_2)|=|V(a_1 a_2 X^{n-2})|+|V(b_1 b_2 X^{n-2})|+|V(X^{n-2} c_1 c_2)|+|V(X^{n-2} d_1 d_2)|-|V(b_1 b_2 X^{n-2})\cap V(X^{n-2} c_1 c_2)|-|V(b_1 b_2 X^{n-2})\cap V(X^{n-2} d_1 d_2)|=4(n-2)!-2(n-4)!$.
The probability $P_10$ that there are four fault-free $B_{n-2}$'s chosen as in Case 10 is $\frac{1}{4}\times 96(n^5-10n^4+35n^3-50n^2+24n)p^{4(n-2)!-2(n-4)!}$.

{\it Case 11.} 2000 or 0200 or 0020 or 0002.

W.l.o.g, assume that the case of 2000.
There are 2 possible cases. W.l.o.g, assume that $a_1=c_1$, $a_2=c_2$, $\{a_1, a_2\}\cap\{d_1, d_2\}=\emptyset$, $\{b_1, b_2\}\cap\{c_1, c_2\}=\emptyset$ and $\{b_1, b_2\}\cap\{d_1, d_2\}=\emptyset$.

There are $n$ ways to choose $a_1(=c_1)$ from $N_n$, $n-1$ ways to choose $a_2(=c_2)$ from $N_n\setminus\{a_1\}$, $n-2$ ways to choose $d_1$ from $N_n\setminus\{a_1, a_2\}$, $n-3$ ways to choose $d_2$ from $N_n\setminus\{a_1, a_2, d_1\}$, $n-4$ ways to choose $b_1$ from $N_n\setminus\{c_1, c_2, d_1, d_2\}$, $n-4$ ways to choose $b_2$ from $N_n\setminus\{c_1, c_2, d_1, d_2, b_1\}$. Therefore, there are $4\times 2n(n-1)(n-2)(n-3)(n-4)(n-5)=8(n^6-15n^5+85n^4-225n^3+274n^2-120n)$ ways to choose four $B_{n-2}$'s as required.
Observation \ref{ob:4} implies that for any $i\in\{c, d\}$, $a_1 a_2 X^{n-2}$ and $X^{n-2} i_1 i_2$ are disjoint, and $b_1 b_2 X^{n-2}$ and $X^{n-2} i_1 i_2$ are not disjoint and $V(b_1 b_2 X^{n-2})\cap V(X^{n-2} i_1 i_2)=V(b_1 b_2 X^{n-4} i_1 i_2)$.
Thus, $|V(a_1 a_2 X^{n-2})\cup V(b_1 b_2 X^{n-2})\cup V(X^{n-2} c_1 c_2)\cup V(X^{n-2} d_1 d_2)|=|V(a_1 a_2 X^{n-2})|+|V(b_1 b_2 X^{n-2})|+|V(X^{n-2} c_1 c_2)|+|V(X^{n-2} d_1 d_2)|-|V(a_1 a_2 X^{n-2})\cap V(X^{n-2} d_1 d_2)|-|V(b_1 b_2 X^{n-2})\cap V(X^{n-2} c_1 c_2)|-|V(b_1 b_2$ $X^{n-2})\cap V(X^{n-2} d_1 d_2)|=4(n-2)!-3(n-4)!$.
The probability $P_11$ that there are four fault-free $B_{n-2}$'s chosen as in Case 11 is $\frac{1}{4}\times 8(n^6-15n^5+85n^4-225n^3+274n^2-120n)p^{4(n-2)!-3(n-4)!}$.

{\it Case 12.} 1111.

Let $x\in\{1, 2\}$ for any $x\in\{i, j, k, l\}$.

{\it Case 12.1.} $a_i=b_i=c_j=d_j$.

There are 4 possible cases. W.l.o.g, assume that $a_1=b_1=c_1=d_1$, $a_2\notin\{c_1, c_2\}$, $a_2\notin\{d_1, d_2\}$, $b_2\notin\{c_1, c_2\}$ and $b_2\notin\{d_1, d_2\}$. There are $n$ ways to choose $a_1(=b_1=c_1=d_1)$ from $N_n$, $n-1$ ways to choose $a_2$ from $N_n\setminus\{a_1\}$, $n-2$ ways to choose $b_2$ from $N_n\setminus\{b_1, a_2\}$, $n-3$ ways to choose $c_2$ from $N_n\setminus\{c_1, a_2, b_2\}$, $n-4$ ways to choose $d_2$ from $N_n\setminus\{d_1, c_2, a_2, b_2\}$. Therefore, there are $4n(n-1)(n-2)(n-3)(n-4)=4(n^5-10n^4+35n^3-50n^2+24n)$ ways to choose four $B_{n-2}$'s as required.

{\it Case 12.2.} $a_i=b_i=c_j=d_k$ or $a_j=b_k=c_i=d_i$, where $j\neq k$.

There are 8 possible cases. W.l.o.g, assume that $a_1=b_1=c_1=d_2$, $a_2\notin\{c_1, c_2\}$, $a_2\notin\{d_1, d_2\}$, $b_2\notin\{c_1, c_2\}$ and $b_2\notin\{d_1, d_2\}$. There are $n$ ways to choose $a_1(=b_1=c_1=d_2)$ from $N_n$, $n-1$ ways to choose $a_2$ from $N_n\setminus\{a_1\}$, $n-2$ ways to choose $b_2$ from $N_n\setminus\{b_1, a_2\}$, $n-3$ ways to choose $c_2$ from $N_n\setminus\{c_1, a_2, b_2\}$, $n-3$ ways to choose $d_1$ from $N_n\setminus\{d_2, a_2, b_2\}$. Therefore, there are $8n(n-1)(n-2)(n-3)^2=8(n^5-9n^4+29n^3-39n^2+18n)$ ways to choose four $B_{n-2}$'s as required.

{\it Case 12.3.} $a_i=b_j=c_k=d_l$, where $i\neq j$ and $k\neq l$.

There are 4 possible cases. W.l.o.g, assume that $a_1=b_2=c_1=d_2$, $a_2\notin\{c_1, c_2\}$, $a_2\notin\{d_1, d_2\}$, $b_1\notin\{c_1, c_2\}$ and $b_1\notin\{d_1, d_2\}$.

{\it Case 12.3.1.} $a_2\neq b_1$.

There are $n$ ways to choose $a_1(=b_2=c_1=d_2)$ from $N_n$, $n-1$ ways to choose $a_2$ from $N_n\setminus\{a_1\}$, $n-2$ ways to choose $b_1$ from $N_n\setminus\{b_2, a_2\}$, $n-3$ ways to choose $c_2$ from $N_n\setminus\{c_1, a_2, b_1\}$, $n-3$ ways to choose $d_1$ from $N_n\setminus\{d_2, a_2, b_1\}$. Therefore, there are $n(n-1)(n-2)(n-3)^2=n^5-9n^4+29n^3-39n^2+18n$ ways to choose four $B_{n-2}$'s as required.

{\it Case 12.3.2.} $a_2=b_1$. 

There are $n$ ways to choose $a_1(=b_2=c_1=d_2)$ from $N_n$, $n-1$ ways to choose $a_2(=b_1)$ from $N_n\setminus\{a_1\}$, $n-2$ ways to choose $c_2$ from $N_n\setminus\{c_1, a_2\}$, $n-2$ ways to choose $d_1$ from $N_n\setminus\{d_2, a_2\}$. Therefore, there are $n(n-1)(n-2)^2=n^4-5n^3+8n^2-4n$ ways to choose four $B_{n-2}$'s as required.

Therefore, there are $4(n^5-8n^4+24n^3-31n^2+14n)$ ways to choose four $B_{n-2}$'s as required in Case 12.3.

{\it Case 12.4.} $a_i=c_k=d_l$, $b_j=c_l=d_k$ or $a_k=b_l=c_i$, $a_l=b_k=d_j$, where $l\neq k$.

There are 16 possible cases. W.l.o.g, assume that $a_1=c_1=d_2$ and $b_1=c_2=d_1$, $a_2\notin\{c_1, c_2\}$, $a_2\notin\{d_1, d_2\}$, $b_2\notin\{c_1, c_2\}$ and $b_2\notin\{d_1, d_2\}$. There are $n$ ways to choose $a_1(=c_1=d_2)$ from $N_n$, $n-1$ ways to choose $b_1(=c_2=d_1)$ from $N_n\setminus\{c_1\}$, $n-2$ ways to choose $a_2$ from $N_n\setminus\{c_1, c_2\}$, $n-2$ ways to choose $b_2$ from $N_n\setminus\{c_1, c_2\}$. Therefore, there are $16n(n-1)(n-2)^2=16(n^4-5n^3+8n^2-4n)$ ways to choose four $B_{n-2}$'s as required.

{\it Case 12.5.} $x_i=y_j=z_k$ and $\{w_1, w_2\}=(\{y_1, y_2\}\setminus\{y_j\})\cup(\{z_1, z_2\}\setminus\{z_k\})$, where $\{y, z\}=\{a, b\}$ or $\{y, z\}=\{c, d\}$ and $\{x, y, z, w\}=\{a, b, c, d\}$.

There are 64 possible cases. W.l.o.g, assume that $a_1=c_1=d_1$, $b_1=c_2$, $b_2=d_2$, $a_2\notin\{c_1, c_2\}$, $a_2\notin\{d_1, d_2\}$, $b_2\notin\{c_1, c_2\}$ and $b_1\notin\{d_1, d_2\}$. There are $n$ ways to choose $a_1(=c_1=d_1)$ from $N_n$, $n-1$ ways to choose $b_1(=c_2)$ from $N_n\setminus\{c_1\}$, $n-2$ ways to choose $b_2(=d_2)$ from $N_n\setminus\{b_1, d_1\}$, $n-3$ ways to choose $a_2$ from $N_n\setminus\{c_1, c_2, d_2\}$. Therefore, there are $64n(n-1)(n-2)(n-3)=64(n^4-6n^3+11n^2-6n)$ ways to choose four $B_{n-2}$'s as required.

{\it Case 12.6.} $\{a_1, a_2, b_1, b_2\}=\{c_1, c_2, d_1, d_2\}$ and $\{a_1, a_2\}\cap\{b_1, b_2\}=\emptyset$.

There are 16 possible cases. W.l.o.g, assume that $a_1=c_1$, $a_2=d_1$, $b_1=c_2$, $b_2=d_2$, $a_2\notin\{c_1, c_2\}$, $a_1\notin\{d_1, d_2\}$, $b_2\notin\{c_1, c_2\}$ and $b_1\notin\{d_1, d_2\}$. There are $n$ ways to choose $a_1(=c_1)$ from $N_n$, $n-1$ ways to choose $a_2(=d_1)$ from $N_n\setminus\{a_1\}$, $n-2$ ways to choose $b_1(=c_2)$ from $N_n\setminus\{c_1, a_2\}$, $n-3$ ways to choose $b_2(=d_2)$ from $N_n\setminus\{c_1, c_2, d_1\}$. Therefore, there are $16n(n-1)(n-2)(n-3)=16(n^4-6n^3+11n^2-6n)$ ways to choose four $B_{n-2}$'s as required.

In summary, there are $4(4n^5-12n^4-23n^3+93n^2-62n)$ ways to choose four $B_{n-2}$'s as required in Case 12.
Observation \ref{ob:4} implies that for any $i\in\{a, b\}$ and $j\in\{c, d\}$, $i_1 i_2 X^{n-2}$ and $X^{n-2} j_1 j_2$ are disjoint. Thus, $|V(a_1 a_2 X^{n-2})\cup V(b_1 b_2 X^{n-2})\cup V(X^{n-2} c_1 c_2)\cup V(X^{n-2} d_1 d_2)|=|V(a_1 a_2 X^{n-2})|+|V(b_1 b_2 X^{n-2})|+|V(X^{n-2} c_1 c_2)|+|V(X^{n-2} d_1 d_2)|=4(n-2)!$.
The probability $P_12$ that there are four fault-free $B_{n-2}$'s chosen as in Case 12 is $\frac{1}{4}\times 4(4n^5-12n^4-23n^3+93n^2-62n)p^{4(n-2)!}$.

{\it Case 13.} 1110 or 1101 or 1011 or 0111.

W.l.o.g, assume that the case of 1110. Let $x\in\{1, 2\}$ for any $x\in\{i, j, k, l\}$.

{\it Case 13.1.} $a_i=c_k=d_l$, $b_j\in\{c_1, c_2\}\setminus\{c_k\}$ or $a_i=b_j=c_k$, $d_l\in\{a_1, a_2\}\setminus\{a_i\}$.

There are 32 possible cases. W.l.o.g, assume that $a_1=c_1=d_1$, $b_1=c_2$, $a_2\notin\{c_1, c_2\}$, $a_2\notin\{d_1, d_2\}$, $b_2\notin\{c_1, c_2\}$ and $\{b_1, b_2\}\cap\{d_1, d_2\}=\emptyset$. There are $n$ ways to choose $a_1(=c_1=d_1)$ from $N_n$, $n-1$ ways to choose $b_1(=c_2)$ from $N_n\setminus\{c_1\}$, $n-2$ ways to choose $b_2$ from $N_n\setminus\{c_1, c_2\}$, $n-3$ ways to choose $d_2$ from $N_n\setminus\{d_1, b_1, b_2\}$, $n-3$ ways to choose $a_2$ from $N_n\setminus\{c_1, c_2, d_2\}$. Therefore, there are $32n(n-1)(n-2)(n-3)^2=32(n^5-9n^4+29n^3-39n^2+18n)$ ways to choose four $B_{n-2}$'s as required.

{\it Case 13.2.} $a_i=c_k$, $d_l\in\{a_1, a_2\}\setminus\{a_i\}$ and $b_j\in\{c_1, c_2\}\setminus\{c_k\}$.

There are 16 possible cases. W.l.o.g, assume that $a_1=c_1$, $a_2=d_1$, $b_1=c_2$, $a_2\notin\{c_1, c_2\}$, $a_1\notin\{d_1, d_2\}$, $b_2\notin\{c_1, c_2\}$ and $\{b_1, b_2\}\cap\{d_1, d_2\}=\emptyset$. There are $n$ ways to choose $a_1(=c_1)$ from $N_n$, $n-1$ ways to choose $a_2(=d_1)$ from $N_n\setminus\{a_1\}$, $n-2$ ways to choose $b_1(=c_2)$ from $N_n\setminus\{c_1, a_2\}$, $n-3$ ways to choose $b_2$ from $N_n\setminus\{c_1, c_2, d_1\}$, $n-4$ ways to choose $d_2$ from $N_n\setminus\{a_1, d_1, b_1, b_2\}$. Therefore, there are $16n(n-1)(n-2)(n-3)(n-4)=16(n^5-10n^4+35n^3-50n^2+24n)$ ways to choose four $B_{n-2}$'s as required.

In summary, there are $16(3n^5-28n^4+93n^3-128n^2+60n)$ ways to choose four $B_{n-2}$'s as required in Case 13.
Observation \ref{ob:4} implies that for any $i\in\{c, d\}$, $a_1 a_2 X^{n-2}$ and $X^{n-2} i_1 i_2$ are disjoint, and $b_1 b_2 X^{n-2}$ and $X^{n-2} c_1 c_2$ are disjoint, and $b_1 b_2 X^{n-2}$ and $X^{n-2} d_1 d_2$ are not disjoint and $V(b_1 b_2 X^{n-2})\cap V(X^{n-2} d_1 d_2)=V(b_1 b_2 X^{n-4} d_1 d_2)$. Thus, $|V(a_1 a_2 X^{n-2})\cup V(b_1 b_2 X^{n-2})\cup V(X^{n-2} c_1 c_2)\cup V(X^{n-2} d_1 d_2)|=|V(a_1 a_2 X^{n-2})|+|V(b_1 b_2 X^{n-2})|+|V(X^{n-2} c_1 c_2)|+|V(X^{n-2} d_1 d_2)|-|V(b_1 b_2 X^{n-2})\cap V(X^{n-2} d_1 d_2)|=4(n-2)!-(n-4)!$.
The probability $P_13$ that there are four fault-free $B_{n-2}$'s chosen as in Case 13 is $\frac{1}{4}\times 4\times 16(3n^5-28n^4+93n^3-128n^2+60n)p^{4(n-2)!-(n-4)!}$.

{\it Case 14.} 1100 or 0011 or 1010 or 0101.

W.l.o.g, assume that the case of 1100. Let $x\in\{1, 2\}$ for any $x\in\{i, j, k, l\}$.

{\it Case 14.1.} $a_i=c_j=d_j$.

There are 4 possible cases. W.l.o.g, assume that $a_1=c_1=d_1$, $a_2\notin\{c_1, c_2\}$, $a_2\notin\{d_1, d_2\}$, $\{b_1, b_2\}\cap\{c_1, c_2\}=\emptyset$ and $\{b_1, b_2\}\cap\{d_1, d_2\}=\emptyset$.
There are $n$ ways to choose $a_1(=c_1=d_1)$ from $N_n$, $n-1$ ways to choose $c_2$ from $N_n\setminus\{c_1\}$, $n-2$ ways to choose $d_2$ from $N_n\setminus\{d_1, c_2\}$, $n-3$ ways to choose $a_2$ from $N_n\setminus\{c_1, c_2, d_2\}$, $n-3$ ways to choose $b_1$ from $N_n\setminus\{c_1, c_2, d_2\}$, $n-4$ ways to choose $b_2$ from $N_n\setminus\{c_1, c_2, d_2, b_1\}$. Therefore, there are $4n(n-1)(n-2)(n-3)^2(n-4)=4(n^6-13n^5+65n^4-155n^3+174n^2-72n)$ ways to choose four $B_{n-2}$'s as required.

{\it Case 14.2.} $a_i=c_j=d_k$ and $j\neq k$.

There are 4 possible cases. W.l.o.g, assume that $a_1=c_1=d_2$, $a_2\notin\{c_1, c_2\}$, $a_2\notin\{d_1, d_2\}$, $\{b_1, b_2\}\cap\{c_1, c_2\}=\emptyset$ and $\{b_1, b_2\}\cap\{d_1, d_2\}=\emptyset$.

{\it Case 14.2.1.} $c_2\neq d_1$.

There are $n$ ways to choose $a_1(=c_1=d_2)$ from $N_n$, $n-1$ ways to choose $c_2$ from $N_n\setminus\{c_1\}$, $n-2$ ways to choose $d_1$ from $N_n\setminus\{d_2, c_2\}$, $n-3$ ways to choose $a_2$ from $N_n\setminus\{c_1, c_2, d_1\}$, $n-3$ ways to choose $b_1$ from $N_n\setminus\{c_1, c_2, d_1\}$, $n-4$ ways to choose $b_2$ from $N_n\setminus\{c_1, c_2, d_1, b_1\}$. Therefore, there are $n(n-1)(n-2)(n-3)^2(n-4)=n^6-13n^5+65n^4-155n^3+174n^2-72n$ ways to choose four $B_{n-2}$'s as required.

{\it Case 14.2.2.} $c_2=d_1$.

There are $n$ ways to choose $a_1(=c_1=d_2)$ from $N_n$, $n-1$ ways to choose $c_2(=d_1)$ from $N_n\setminus\{c_1\}$, $n-2$ ways to choose $a_2$ from $N_n\setminus\{c_1, c_2\}$, $n-2$ ways to choose $b_1$ from $N_n\setminus\{c_1, c_2\}$, $n-3$ ways to choose $b_2$ from $N_n\setminus\{c_1, c_2, b_1\}$. Therefore, there are $n(n-1)(n-2)^2(n-3)=n^5-8n^4+23n^3-28n^2+12n$ ways to choose four $B_{n-2}$'s as required.

Therefore, there are $4(n^6-12n^5+57n^4-132n^3+146n^2-60n)$ ways to choose four $B_{n-2}$'s as required in Case 14.2.

{\it Case 14.3.} $a_i=c_j$ and $d_k\in\{a_1, a_2\}\setminus\{a_i\}$.

There are 8 possible cases. W.l.o.g, assume that $a_1=c_1$ $a_2=d_1$, $a_2\notin\{c_1, c_2\}$, $a_1\notin\{d_1, d_2\}$, $\{b_1, b_2\}\cap\{c_1, c_2\}=\emptyset$ and $\{b_1, b_2\}\cap\{d_1, d_2\}=\emptyset$.
There are $n$ ways to choose $a_1(=c_1)$ from $N_n$, $n-1$ ways to choose $a_2(=d_1)$ from $N_n\setminus\{a_1\}$, $n-2$ ways to choose $b_1$ from $N_n\setminus\{c_1, d_1\}$, $n-3$ ways to choose $b_2$ from $N_n\setminus\{b_1, c_1, d_1\}$, $n-4$ ways to choose $c_2$ from $N_n\setminus\{c_1, a_2, b_1, b_2\}$, $n-4$ ways to choose $d_2$ from $N_n\setminus\{d_1, a_1, b_1, b_2\}$. Therefore, there are $8n(n-1)(n-2)(n-3)(n-4)^2=8(n^6-14n^5+75n^4-190n^3+224n^2-96n)$ ways to choose four $B_{n-2}$'s as required.

In summary, there are $4(4n^6-53n^5+272n^4-667n^3+768n^2-324n)$ ways to choose four $B_{n-2}$'s as required in Case 14. 
Observation \ref{ob:4} implies that for any $i\in\{c, d\}$, $a_1 a_2 X^{n-2}$ and $X^{n-2} i_1 i_2$ are disjoint, and $b_1 b_2 X^{n-2}$ and $X^{n-2} i_1 i_2$ are not disjoint and $V(b_1 b_2 X^{n-2})\cap V(X^{n-2} i_1 i_2)=V(b_1 b_2 X^{n-4} i_1 i_2)$. Thus, $|V(a_1 a_2 X^{n-2})\cup V(b_1 b_2 X^{n-2})\cup V(X^{n-2} c_1 c_2)\cup V(X^{n-2} d_1 d_2)|=|V(a_1 a_2 X^{n-2})|+|V(b_1 b_2 X^{n-2})|+|V(X^{n-2} c_1 c_2)|+|V(X^{n-2} d_1 d_2)|-|V(b_1 b_2 X^{n-2})\cap V(X^{n-2} c_1 c_2)|-|V(b_1 b_2 X^{n-2})\cap V(X^{n-2} d_1 d_2)|=4(n-2)!-2(n-4)!$.
The probability $P_14$ that there are four fault-free $B_{n-2}$'s chosen as in Case 14 is $\frac{1}{4}\times 4\times 4(4n^6-53n^5+272n^4-667n^3+768n^2-324n)p^{4(n-2)!-2(n-4)!}$.

{\it Case 15.} 1001 or 0110.

W.l.o.g, assume that the case of 1001. Let $x\in\{1, 2\}$ for any $x\in\{i, j, k, l\}$.
There are 16 possible cases. W.l.o.g, assume that $a_1=c_1$, $b_1=d_1$, $a_2\notin\{c_1, c_2\}$, $b_2\notin\{d_1, d_2\}$, $\{a_1, a_2\}\cap\{d_1, d_2\}=\emptyset$ and $\{b_1, b_2\}\cap\{c_1, c_2\}=\emptyset$.

{\it Case 15.1.} $a_2\neq b_2$.

There are $n$ ways to choose $a_1(=c_1)$ from $N_n$, $n-1$ ways to choose $b_1(=d_1)$ from $N_n\setminus\{c_1\}$, $n-2$ ways to choose $a_2$ from $N_n\setminus\{a_1, d_1\}$, $n-3$ ways to choose $b_2$ from $N_n\setminus\{b_1, c_1, a_2\}$, $n-4$ ways to choose $c_2$ from $N_n\setminus\{c_1, a_2, b_1, b_2\}$, $n-4$ ways to choose $d_2$ from $N_n\setminus\{d_1, b_2, a_1, a_2\}$. Therefore, there are $n(n-1)(n-2)(n-3)(n-4)^2=n^6-14n^5+75n^4-190n^3+224n^2-96n$ ways to choose four $B_{n-2}$'s as required.

{\it Case 15.2.} $a_2=b_2$.

There are $n$ ways to choose $a_1(=c_1)$ from $N_n$, $n-1$ ways to choose $b_1(=d_1)$ from $N_n\setminus\{c_1\}$, $n-2$ ways to choose $a_2(=b_2)$ from $N_n\setminus\{a_1, d_1\}$, $n-3$ ways to choose $c_2$ from $N_n\setminus\{c_1, b_1, b_2\}$, $n-3$ ways to choose $d_2$ from $N_n\setminus\{d_1, a_1, a_2\}$. Therefore, there are $n(n-1)(n-2)(n-3)^2=n^5-9n^4+29n^3-39n^2+18n$ ways to choose four $B_{n-2}$'s as required.

In summary, there are $n^6-13n^5+66n^4-161n^3+185n^2-78n$ ways to choose four $B_{n-2}$'s as required in Case 15. 
Observation \ref{ob:4} implies that $a_1 a_2 X^{n-2}$ and $X^{n-2} c_1 c_2$ are disjoint, and $a_1 a_2 X^{n-2}$ and $X^{n-2} d_1 d_2$ are not disjoint and $V(a_1 a_2 X^{n-2})\cap V(X^{n-2} d_1 d_2)=V(a_1 a_2 X^{n-4} d_1 d_2)$, and $b_1 b_2 X^{n-2}$ and $X^{n-2} c_1 c_2$ are not disjoint and $V(b_1 b_2 X^{n-2})\cap V(X^{n-2} c_1 c_2)=V(b_1 b_2 X^{n-4} c_1 c_2)$, and $b_1 b_2 X^{n-2}$ and $X^{n-2} d_1 d_2$ are disjoint. Thus, $|V(a_1 a_2 X^{n-2})\cup V(b_1 b_2 X^{n-2})\cup V(X^{n-2} c_1 c_2)\cup V(X^{n-2} d_1 d_2)|=|V(a_1 a_2 X^{n-2})|+|V(b_1 b_2 X^{n-2})|+|V(X^{n-2} c_1 c_2)|+|V(X^{n-2} d_1 d_2)|-|V(a_1 a_2 X^{n-2})\cap V(X^{n-2} d_1 d_2)|-|V(b_1 b_2 X^{n-2})\cap V(X^{n-2} c_1 c_2)|=4(n-2)!-2(n-4)!$.
The probability $P_15$ that there are four fault-free $B_{n-2}$'s chosen as in Case 15 is $\frac{1}{4}\times 2\times 16(n^6-13n^5+66n^4-161n^3+185n^2-78n)p^{4(n-2)!-2(n-4)!}$.

{\it Case 16.} 1000 or 0100 or 0010 or 0001.

W.l.o.g, assume that the case of 1000.
There are 4 possible cases. W.l.o.g, assume that $a_1=c_1$, $a_2\notin\{c_1, c_2\}$, $\{a_1, a_2\}\cap\{d_1, d_2\}=\emptyset$, $\{b_1, b_2\}\cap\{c_1, c_2\}=\emptyset$, $\{b_1, b_2\}\cap\{d_1, d_2\}=\emptyset$.

{\it Case 16.1.} $a_2\neq b_1$ and $a_2\neq b_2$.

There are $n$ ways to choose $a_1(=c_1)$ from $N_n$, $n-1$ ways to choose $a_2$ from $N_n\setminus\{a_1\}$, $n-2$ ways to choose $c_2$ from $N_n\setminus\{c_1, a_2\}$, $n-3$ ways to choose $b_1$ from $N_n\setminus\{c_1, c_2, a_2\}$, $n-4$ ways to choose $b_2$ from $N_n\setminus\{b_1, c_1, c_2, a_2\}$, $n-4$ ways to choose $d_1$ from $N_n\setminus\{a_1, a_2, b_1, b_2\}$, $n-5$ ways to choose $d_2$ from $N_n\setminus\{a_1, a_2, b_1, b_2, d_1\}$. Therefore, there are $n(n-1)(n-2)(n-3)(n-4)^2(n-5)=n^7-19n^6+145n^5-565n^4+1174n^3-1216n^2+480n$ ways to choose four $B_{n-2}$'s as required.

{\it Case 16.2.} $a_2=b_1$ or $a_2=b_2$.

There are $n$ ways to choose $a_1(=c_1)$ from $N_n$, $n-1$ ways to choose $a_2(=b_1)$ from $N_n\setminus\{a_1\}$, $n-2$ ways to choose $c_2$ from $N_n\setminus\{c_1, a_2\}$, $n-3$ ways to choose $b_2$ from $N_n\setminus\{b_1, c_1, c_2\}$, $n-3$ ways to choose $d_1$ from $N_n\setminus\{a_1, a_2, b_2\}$, $n-4$ ways to choose $d_2$ from $N_n\setminus\{a_1, a_2, b_2, d_1\}$. Therefore, there are $2n(n-1)(n-2)(n-3)^2(n-4)=2(n^6-13n^5+65n^4-155n^3+174n^2-72n)$ ways to choose four $B_{n-2}$'s as required.

In summary, there are $n^7-17n^6+119n^5-435n^4+864n^3-868n^2+336n$ ways to choose four $B_{n-2}$'s as required in Case 16. 
Observation \ref{ob:4} implies that $a_1 a_2 X^{n-2}$ and $X^{n-2} c_1 c_2$ are disjoint, and $a_1 a_2 X^{n-2}$ and $X^{n-2} d_1 d_2$ are not disjoint and $V(a_1 a_2 X^{n-2})\cap V(X^{n-2} d_1 d_2)=V(a_1 a_2 X^{n-4} d_1 d_2)$, and for any $i\in\{c, d\}$, $b_1 b_2 X^{n-2}$ and $X^{n-2} i_1 i_2$ are not disjoint and $V(b_1 b_2 X^{n-2})\cap V(X^{n-2} i_1 i_2)=V(b_1 b_2 X^{n-4} i_1 i_2)$. Thus, $|V(a_1 a_2 X^{n-2})\cup V(b_1 b_2 X^{n-2})\cup V(X^{n-2} c_1 c_2)\cup V(X^{n-2} d_1 d_2)|=|V(a_1 a_2 X^{n-2})|+|V(b_1 b_2 X^{n-2})|+|V(X^{n-2} c_1$ $c_2)|+|V(X^{n-2} d_1 d_2)|-|V(a_1 a_2 X^{n-2})\cap V(X^{n-2} d_1 d_2)|-|V(b_1 b_2 X^{n-2})\cap V(X^{n-2} c_1 c_2)|-|V(b_1 b_2 X^{n-2})\cap V(X^{n-2} d_1 d_2)|=4(n-2)!-3(n-4)!$.
The probability $P_16$ that there are four fault-free $B_{n-2}$'s chosen as in Case 16 is $\frac{1}{4}\times 4\times 4(n^7-17n^6+119n^5-435n^4+864n^3-868n^2+336n)p^{4(n-2)!-3(n-4)!}$.

{\it Case 17.} 0000.

This case refers to $\{a_1, a_2\}\cap\{c_1, c_2\}=\emptyset$, $\{a_1, a_2\}\cap\{d_1, d_2\}=\emptyset$, $\{b_1, b_2\}\cap\{c_1, c_2\}=\emptyset$ and $\{b_1, b_2\}\cap\{d_1, d_2\}=\emptyset$.

{\it Case 17.1.} $\{a_1, a_2\}=\{b_1, b_2\}$ and $c_1\neq d_1$.

There are $n$ ways to choose $a_1(=b_2)$ from $N_n$, $n-1$ ways to choose $a_2(=b_1)$ from $N_n\setminus\{a_1\}$, $n-2$ ways to choose $c_1$ from $N_n\setminus\{a_1, a_2\}$, $n-3$ ways to choose $c_2$ from $N_n\setminus\{c_1, a_1, a_2\}$, $n-3$ ways to choose $d_1$ from $N_n\setminus\{c_1, a_1, a_2\}$, $n-3$ ways to choose $d_2$ from $N_n\setminus\{d_1, a_1, a_2\}$. Therefore, there are $n(n-1)(n-2)(n-3)^3=n^6-12n^5+56n^4-126n^3+135n^2-54n$ ways to choose four $B_{n-2}$'s as required.

{\it Case 17.2.} $\{a_1, a_2\}=\{b_1, b_2\}$ and $c_1=d_1$.

There are $n$ ways to choose $a_1(=b_2)$ from $N_n$, $n-1$ ways to choose $a_2(=b_1)$ from $N_n\setminus\{a_1\}$, $n-2$ ways to choose $c_1(=d_1)$ from $N_n\setminus\{a_1, a_2\}$, $n-3$ ways to choose $c_2$ from $N_n\setminus\{c_1, a_1, a_2\}$, $n-4$ ways to choose $d_2$ from $N_n\setminus\{d_1, a_1, a_2, c_2\}$. Therefore, there are $n(n-1)(n-2)(n-3)(n-4)=n^5-10n^4+35n^3-50n^2+24n$ ways to choose four $B_{n-2}$'s as required.

{\it Case 17.3.} $|\{a_1, a_2\}\cap\{b_1, b_2\}|=1$ and $c_1\neq d_1$.

There are 4 possible cases. W.l.o.g, assume that $a_1=b_1$, $a_2\notin\{b_1, b_2\}$ and $c_1\neq d_1$.
There are $n$ ways to choose $a_1(=b_1)$ from $N_n$, $n-1$ ways to choose $b_2$ from $N_n\setminus\{b_1\}$, $n-2$ ways to choose $a_2$ from $N_n\setminus\{b_1, b_2\}$, $n-3$ ways to choose $c_1$ from $N_n\setminus\{a_1, a_2, b_2\}$, $n-4$ ways to choose $c_2$ from $N_n\setminus\{a_1, a_2, b_2, c_1\}$, $n-4$ ways to choose $d_1$ from $N_n\setminus\{a_1, a_2, b_2, c_1\}$, $n-4$ ways to choose $d_2$ from $N_n\setminus\{a_1, a_2, b_2, d_1\}$. Therefore, there are $4n(n-1)(n-2)(n-3)(n-4)^3=4(n^7-18n^6+131n^5-490n^4+984n^3-992n^2+384n)$ ways to choose four $B_{n-2}$'s as required.

{\it Case 17.4.} $|\{a_1, a_2\}\cap\{b_1, b_2\}|=1$ and $c_1=d_1$.

There are 4 possible cases. W.l.o.g, assume that $a_1=b_1$, $a_2\notin\{b_1, b_2\}$ and $c_1=d_1$.
There are $n$ ways to choose $a_1(=b_1)$ from $N_n$, $n-1$ ways to choose $b_2$ from $N_n\setminus\{b_1\}$, $n-2$ ways to choose $a_2$ from $N_n\setminus\{b_1, b_2\}$, $n-3$ ways to choose $c_1(=d_1)$ from $N_n\setminus\{a_1, a_2, b_2\}$, $n-4$ ways to choose $c_2$ from $N_n\setminus\{a_1, a_2, b_2, c_1\}$, $n-5$ ways to choose $d_2$ from $N_n\setminus\{a_1, a_2, b_2, d_1, c_2\}$. Therefore, there are $n(n-1)(n-2)(n-3)(n-4)(n-5)=4(n^6-15n^5+85n^4-225n^3+274n^2-120n)$ ways to choose four $B_{n-2}$'s as required.

{\it Case 17.5.} $\{a_1, a_2\}\cap\{b_1, b_2\}=\emptyset$ and $c_1\neq d_1$.

There are $n$ ways to choose $a_1$ from $N_n$, $n-1$ ways to choose $a_2$ from $N_n\setminus\{a_1\}$, $n-2$ ways to choose $b_1$ from $N_n\setminus\{a_1, a_2\}$, $n-3$ ways to choose $b_2$ from $N_n\setminus\{a_1, a_2, b_1\}$, $n-4$ ways to choose $c_1$ from $N_n\setminus\{a_1, a_2, b_1, b_2\}$, $n-5$ ways to choose $c_2$ from $N_n\setminus\{a_1, a_2, b_1, b_2, c_1\}$, $n-5$ ways to choose $d_1$ from $N_n\setminus\{a_1, a_2, b_1, b_2, c_1\}$, $n-5$ ways to choose $d_2$ from $N_n\setminus\{a_1, a_2, b_1, b_2, d_1\}$. Therefore, there are $n(n-1)(n-2)(n-3)(n-4)(n-5)^3=n^8-25n^7+260n^6-1450n^5+4649n^4-8485n^3+8050n^2-3000n$ ways to choose four $B_{n-2}$'s as required.

{\it Case 17.6.} $\{a_1, a_2\}\cap\{b_1, b_2\}=\emptyset$ and $c_1=d_1$.

There are $n$ ways to choose $a_1$ from $N_n$, $n-1$ ways to choose $a_2$ from $N_n\setminus\{a_1\}$, $n-2$ ways to choose $b_1$ from $N_n\setminus\{a_1, a_2\}$, $n-3$ ways to choose $b_2$ from $N_n\setminus\{a_1, a_2, b_1\}$, $n-4$ ways to choose $c_1(=d_1)$ from $N_n\setminus\{a_1, a_2, b_1, b_2\}$, $n-5$ ways to choose $c_2$ from $N_n\setminus\{a_1, a_2, b_1, b_2, c_1\}$, $n-6$ ways to choose $d_2$ from $N_n\setminus\{a_1, a_2, b_1, b_2, d_1, c_2\}$. Therefore, there are $n(n-1)(n-2)(n-3)(n-4)(n-5)(n-6)=n^7-21n^6+175n^5-735n^4+1624n^3-1764n^2+720n$ ways to choose four $B_{n-2}$'s as required.

In summary, there are $n^8-20n^7+172n^6-822n^5+2340n^4-3916n^3+3499n^2-1254n$ ways to choose four $B_{n-2}$'s as required in Case 17.
Observation \ref{ob:4} implies that for any $i\in\{a, b\}$ and $j\in\{c, d\}$, $i_1 i_2 X^{n-2}$ and $X^{n-2} j_1 j_2$ are not disjoint and $V(i_1 i_2 X^{n-2})\cap V(X^{n-2} j_1 j_2)=V(i_1 i_2 X^{n-4} j_1 j_2)$. Thus, $|V(a_1 a_2 X^{n-2})\cup V(b_1 b_2 X^{n-2})\cup V(X^{n-2} c_1 c_2)\cup V(X^{n-2} d_1 d_2)|=|V(a_1 a_2 X^{n-2})|+|V(b_1 b_2 X^{n-2})|+|V(X^{n-2} c_1 c_2)|+|V(X^{n-2} d_1$ $d_2)|-|V(a_1 a_2 X^{n-2})\cap V(X^{n-2} c_1 c_2)|-|V(a_1 a_2 X^{n-2})\cap V(X^{n-2} d_1 d_2)|-|V(b_1 b_2 X^{n-2})\cap V(X^{n-2} c_1 c_2)|-|V(b_1 b_2 X^{n-2})\cap V(X^{n-2} d_1 d_2)|=4(n-2)!-4(n-4)!$.
The probability $P_17$ that there are four fault-free $B_{n-2}$'s chosen as in Case 17 is $\frac{1}{4}(n^8-20n^7+172n^6-822n^5+2340n^4-3916n^3+3499n^2-1254n)p^{4(n-2)!-4(n-4)!}$.

{\it* Another way}

In particular, here we find another way to determine the number of distinct $B_{n-2}$'s in Case 17.

When $\{a_1, a_2\}=\{b_1, b_2\}$, there are $\binom{n}{2}$ ways to choose $a_1 a_2 X^{n-2}$ and $b_1 b_2 X^{n-2}$, and $\binom{2\binom{n-2}{2}}{2}$ ways to choose $X^{n-2} c_1 c_2$ and $X^{n-2} d_1 d_2$.

When $|\{a_1, a_2\}\cap\{b_1, b_2\}|=1$, there are $2\binom{n}{2}$ ways to choose one of $a_1 a_2 X^{n-2}$ and $b_1 b_2 X^{n-2}$. Among the total $2\binom{n}{2}$ ways, the cases $\{a_1, a_2\}\cap\{b_1, b_2\}=\emptyset$ and $\{a_1, a_2\}=\{b_1, b_2\}$ need to be excluded, thus there are $2\binom{n}{2}-2\binom{n-2}{2}-2$ possible ways to choose the other one of $a_1 a_2 X^{n-2}$ and $b_1 b_2 X^{n-2}$. Since we have implicitly sorted the values of $a_1 a_2$ and $b_1 b_2$, the number of ways of four distinct $B_{n-2}$'s in this case needs to be divided by 2. For example, choosing $1 2 X_{n-2}$ first and then choosing $1 3 X_{n-2}$ and choosing $1 3 X_{n-2}$ first and then choosing $1 2 X_{n-2}$ are actually the same case. And there are $\binom{2\binom{n-3}{2}}{2}$ ways to choose $X^{n-2} c_1 c_2$ and $X^{n-2} d_1 d_2$.

When $\{a_1, a_2\}\cap\{b_1, b_2\}=\emptyset$, there are $2\binom{n}{2}$ ways to choose one of $a_1 a_2 X^{n-2}$ and $b_1 b_2 X^{n-2}$, and $2\binom{n-2}{2}$ ways to choose the other one of $a_1 a_2 X^{n-2}$ and $b_1 b_2 X^{n-2}$. Since we have implicitly arrange the values of $a_1 a_2$ and $b_1 b_2$, the number of ways to choose four $B_{n-2}$'s as required needs to be divided by 2. And there are $\binom{2\binom{n-4}{2}}{2}$ ways to choose $X^{n-2} c_1 c_2$ and $X^{n-2} d_1 d_2$.

Thus, the probability that there are four fault-free $B_{n-2}$'s chosen as in Case 17 is $(\binom{n}{2}\binom{2\binom{n-2}{2}}{2}+\frac{1}{2}\binom{2\binom{n}{2}}{1}$ $\binom{2\binom{n}{2}-2\binom{n-2}{2}-2}{1}\binom{2\binom{n-3}{2}}{2}+\frac{1}{2}\binom{2\binom{n}{2}}{1}\binom{2\binom{n-2}{2}}{1}\binom{2\binom{n-4}{2}}{2})p^{4(n-2)!-4(n-4)!}=\frac{1}{4}(n^8-20n^7+172n^6-822n^5+2340n^4-3916n^3+3499n^2-1254n)p^{4(n-2)!-4(n-4)!}=P_17$.

By the above computations, $P(2,0,2)=\sum_{i=1}^{17}P_i=\frac{1}{2}(8n^5+8n^4-174n^3+347n^2-189n)p^{4(n-2)!}+16(3n^5-25n^4+75n^3-95n^2+42n)p^{4(n-2)!-(n-4)!}+4(6n^6-73n^5+345n^4-785n^3+849n^2-342n)p^{4(n-2)!-2(n-4)!}+2(2n^7-33n^6+223n^5-785n^4+1503n^3-1462n^2+552n)p^{4(n-2)!-3(n-4)!}+\frac{1}{4}(n^8-20n^7+172n^6-822n^5+2340n^4-3916n^3+3499n^2-1254n)p^{4(n-2)!-4(n-4)!}$.
\end{proof}

\begin{theorem}\label{th9}
$P(2,1,1)=(8n^6-65n^5+229n^4-421n^3+389n^2-140n)p^{4(n-2)!}+\frac{1}{2}(44n^5-291n^4+716n^3-771n^2+302n)p^{4(n-2)!-(n-3)!}+\frac{1}{2}(2n^5-5n^4-10n^3+35n^2-22n)p^{4(n-2)!-2(n-3)!}+(4n^7-46n^6+231n^5-647n^4+1043n^3-891n^2+306n)p^{4(n-2)!-(n-4)!}+(12n^6-128n^5+535n^4-1090n^3+1073n^2-402n)p^{4(n-2)!-(n-3)!-(n-4)!}+(7n^5-59n^4+179n^3-229n^2+102n)p^{4(n-2)!-2(n-3)!-(n-4)!}+(n^4-6n^3+11n^2-6n)p^{4(n-2)!-3(n-3)!-(n-4)!}+\frac{1}{2}(n^8-15n^7+99n^6-377n^5+892n^4-1288n^3+1024n^2-336n)p^{4(n-2)!-2(n-4)!}+\frac{1}{2}(3n^7-45n^6+277n^5-893n^4+1580n^3-1438n^2+516n)p^{4(n-2)!-(n-3)!-2(n-4)!}+\frac{1}{2}(3n^6-38n^5+187n^4-442n^3+494n^2-204n)p^{4(n-2)!-2(n-3)!-2(n-4)!}+\frac{1}{2}(n^5-10n^4+35n^3-50n^2+24n)p^{4(n-2)!-3(n-3)!-2(n-4)!}$.
\end{theorem}

\begin{proof}
Denote by $a_1 a_2 X^{n-2}\in H_{1, 2}$, $b_1 b_2 X^{n-2}\in H_{1, 2}$, $c_1 X^{n-2} c_2\in H_{1, n}$ and $X^{n-2} d_1 d_2\in H_{n-1, n}$ these four $B_{n-2}$'s. Observation \ref{ob:1} implies that $a_1 a_2 X^{n-2}$ and $b_1 b_2 X^{n-2}$ are disjoint.

Since we have implicitly sorted the values of $a_1 a_2$ and $b_1 b_2$, the number of ways to choose four $B_{n-2}$'s as required needs to be divided by 2. It will be applied directly in following analyses.

When $\{a_1, a_2\}\cap\{d_1, d_2\}\neq\emptyset$ and $\{b_1, b_2\}\cap\{d_1, d_2\}\neq\emptyset$, i.e., Case 1, Case 2 and Case 4, we will use the following facts directly in these cases.

(1) $i_1 i_2 X^{n-2}$ and $X^{n-2} d_1 d_2$ are disjoint for any $i\in\{a, b\}$.

(2) For any $\{i, j\}=\{a, b\}$, if ($\romannumeral1$) $c_1\neq a_1$, $c_1\neq b_1$, $c_2\neq d_2$; or ($\romannumeral2$) $c_1=i_1\neq j_1$, $c_2\neq d_2$, $c_2=i_2$; or ($\romannumeral3$) $c_1\neq a_1$, $c_1\neq b_1$, $c_2=d_2$, $c_1=d_1$; or ($\romannumeral4$) $c_1=i_1\neq j_1$, $c_2=d_2$, $c_2=i_2$, $c_1=d_1$, w.l.o.g, assume that the first case applies. Observations \ref{ob:2} and \ref{ob:3} and \ref{ob:4} imply that $k_1 k_2 X^{n-2}$ and $c_1 X^{n-2} c_2$ are disjoint for any $k\in\{a, b\}$, and $X^{n-2} d_1 d_2$ and $c_1 X^{n-2} c_2$ are disjoint. Thus, $|V(a_1 a_2 X^{n-2})\cup V(b_1 b_2 X^{n-2})\cup V(c_1 X^{n-2} c_2)\cup V(X^{n-2} d_1 d_2)|=|V(a_1 a_2 X^{n-2})|+|V(b_1 b_2 X^{n-2})|+|V(c_1 X^{n-2} c_2)|+|V(X^{n-2} d_1 d_2)|=4(n-2)!$. 

(3) For any $\{i, j\}=\{a, b\}$, if ($\romannumeral1$) $c_1=i_1\neq j_1$, $c_2\neq d_2$, $c_2\neq i_2$; or ($\romannumeral2$) $c_1=i_1\neq j_1$, $c_2=d_2$, $c_2\neq i_2$, $c_1=d_1$; or ($\romannumeral3$) $c_1=a_1=b_1$, $c_2\neq d_2$, $c_2=i_2\neq j_2$; or ($\romannumeral4$) $c_1=a_1=b_1$, $c_2=d_2$, $c_2=i_2\neq j_2$, $c_1=d_1$; or ($\romannumeral5$) $c_1\neq a_1$, $c_1\neq b_1$, $c_2=d_2$, $c_1\neq d_1$; ($\romannumeral6$) $c_1=i_1\neq j_1$, $c_2=d_2$, $c_2=i_2$, $c_1\neq d_1$, w.l.o.g, assume that $c_1=a_1\neq b_1$, $c_2\neq d_2$, $c_2\neq a_2$. Observations \ref{ob:2} and \ref{ob:3} and \ref{ob:4} imply that $a_1 a_2 X^{n-2}$ and $c_1 X^{n-2} c_2$ are not disjoint and $V(a_1 a_2 X^{n-2})\cap V(c_1 X^{n-2} c_2)=V(a_1 a_2 X^{n-3} c_2)$, and $b_1 b_2 X^{n-2}$ and $c_1 X^{n-2} c_2$ are disjoint, and $X^{n-2} d_1 d_2$ and $c_1 X^{n-2} c_2$ are disjoint. Thus, $|V(a_1 a_2 X^{n-2})\cup V(b_1 b_2 X^{n-2})\cup V(c_1 X^{n-2} c_2)\cup V(X^{n-2} d_1 d_2)|=|V(a_1 a_2 X^{n-2})|+|V(b_1 b_2 X^{n-2})|+|V(c_1 X^{n-2} c_2)|+|V(X^{n-2} d_1 d_2)|-|V(a_1 a_2 X^{n-2})\cap V(c_1 X^{n-2} c_2)|=4(n-2)!-(n-3)!$. 

(4) For any $\{i, j\}=\{a, b\}$, if ($\romannumeral1$) $c_1=a_1=b_1$, $c_2\neq d_2$, $c_2\neq a_2, c_2\neq b_2$; or ($\romannumeral2$) $c_1=a_1=b_1$, $c_2=d_2$, $c_2\neq a_2, c_2\neq b_2$, $c_1=d_1$; or ($\romannumeral3$) $c_1=i_1\neq j_1$, $c_2=d_2$, $c_2\neq i_2$, $c_1\neq d_1$; or ($\romannumeral4$) $c_1=a_1=b_1$, $c_2=d_2$, $c_2=i_2\neq j_2$, $c_1\neq d_1$, w.l.o.g, assume that the first case applies. Observations \ref{ob:2} and \ref{ob:3} and \ref{ob:4} imply that $i_1 i_2 X^{n-2}$ and $c_1 X^{n-2} c_2$ are not disjoint and $V(i_1 i_2 X^{n-2})\cap V(c_1 X^{n-2} c_2)=V(i_1 i_2 X^{n-3} c_2)$ for any $i\in\{a, b\}$, and $X^{n-2} d_1 d_2$ and $c_1 X^{n-2} c_2$ are disjoint. Thus, $|V(a_1 a_2 X^{n-2})\cup V(b_1 b_2 X^{n-2})\cup V(c_1 X^{n-2} c_2)\cup V(X^{n-2} d_1 d_2)|=|V(a_1 a_2 X^{n-2})|+|V(b_1 b_2 X^{n-2})|+|V(c_1 X^{n-2} c_2)|+|V(X^{n-2} d_1 d_2)|-|V(a_1 a_2 X^{n-2})\cap V(c_1 X^{n-2} c_2)|-|V(b_1 b_2$ $X^{n-2})\cap V(c_1 X^{n-2} c_2)|=4(n-2)!-2(n-3)!$. 

(5) If $c_1=a_1=b_1$, $c_2=d_2$, $c_2\neq a_2, c_2\neq b_2$ and $c_1\neq d_1$, Observations \ref{ob:2} and \ref{ob:3} and \ref{ob:4} imply that $i_1 i_2 X^{n-2}$ and $c_1 X^{n-2} c_2$ are not disjoint and $V(i_1 i_2 X^{n-2})\cap V(c_1 X^{n-2} c_2)=V(i_1 i_2 X^{n-3} c_2)$ for any $i\in\{a, b\}$, and $X^{n-2} d_1 d_2$ and $c_1 X^{n-2} c_2$ are not disjoint and $V(X^{n-2} d_1 d_2)\cap V(c_1 X^{n-2} c_2)=V(c_1 X^{n-3} d_1 d_2)$. Thus, $|V(a_1 a_2 X^{n-2})\cup V(b_1 b_2 X^{n-2})\cup V(c_1 X^{n-2} c_2)\cup V(X^{n-2} d_1 d_2)|=|V(a_1 a_2 X^{n-2})|+|V(b_1 b_2 X^{n-2})|+|V(c_1 X^{n-2} c_2)|+|V(X^{n-2} d_1 d_2)|-|V(a_1 a_2 X^{n-2})\cap V(c_1 X^{n-2} c_2)|-|V(b_1 b_2 X^{n-2})\cap V(c_1 X^{n-2} c_2)|-|V(X^{n-2} d_1$ $d_2)\cap V(c_1 X^{n-2} c_2)|=4(n-2)!-3(n-3)!$. 

We will analyze this problem in 6 cases as follows.

{\it Case 1.} $\{a_1, a_2\}=\{d_1, d_2\}$ and $\{b_1, b_2\}=\{d_1, d_2\}$.

There are 2 possible cases $a_1=b_2=d_1$, $a_2=b_1=d_2$ and $a_1=b_2=d_2$, $a_2=b_1=d_1$ such that $\{a_1, a_2\}=\{b_1, b_2\}=\{d_1, d_2\}$. W.l.o.g, assume that $a_1=b_2=d_1$ and $a_2=b_1=d_2$.

{\it Case 1.1.} $c_1\neq a_1$, $c_1\neq b_1$ and $c_2\neq d_2$.

There are $n$ ways to choose $a_1(=b_2=d_1)$ from $N_n$, $n-1$ ways to choose $a_2(=b_1=d_2)$ from $N_n\setminus\{a_1\}$, $n-2$ ways to choose $c_1$ from $N_n\setminus\{a_1, b_1\}$, $n-2$ ways to choose $c_2$ from $N_n\setminus\{c_1, d_2\}$.
Therefore, there are $n(n-1)(n-2)^2=n^4-5n^3+8n^2-4n$ ways to choose four $B_{n-2}$'s as required.
The probability $P_{1.1}$ that there are four fault-free $B_{n-2}$'s chosen as in Case 1.1 is $\frac{1}{2}(n^4-5n^3+8n^2-4n)p^{4(n-2)!}$.

{\it Case 1.2.} $c_1=a_1\neq b_1$ and $c_2\neq d_2$.

There are $n$ ways to choose $a_1(=b_2=d_1=c_1)$ from $N_n$, $n-1$ ways to choose $a_2(=b_1=d_2)$ from $N_n\setminus\{a_1\}$, $n-2$ ways to choose $c_2$ from $N_n\setminus\{c_1, d_2\}$.
Therefore, there are $n(n-1)(n-2)=n^3-3n^2+2n$ ways to choose four $B_{n-2}$'s as required.
Clearly, in this case $c_2\neq a_2$.
The probability $P_{1.2}$ that there are four fault-free $B_{n-2}$'s chosen as in Case 1.2 is $\frac{1}{2}(n^3-3n^2+2n)p^{4(n-2)!-(n-3)!}$.

{\it Case 1.3.} $c_1=b_1\neq a_1$ and $c_2\neq d_2$.

{\it Case 1.3.1.} $c_2\neq b_2$.

There are $n$ ways to choose $a_1(=b_2=d_1)$ from $N_n$, $n-1$ ways to choose $a_2(=b_1=d_2=c_1)$ from $N_n\setminus\{d_1\}$, $n-2$ ways to choose $c_2$ from $N_n\setminus\{d_2, b_2\}$.
Therefore, there are $n(n-1)(n-2)=n^3-3n^2+2n$ ways to choose four $B_{n-2}$'s as required.
The probability $P_{1.3.1}$ that there are four fault-free $B_{n-2}$'s chosen as in Case 1.3.1 is $\frac{1}{2}(n^3-3n^2+2n)p^{4(n-2)!-(n-3)!}$.

{\it Case 1.3.2.} $c_2=b_2$.

There are $n$ ways to choose $a_1(=b_2=d_1=c_2)$ from $N_n$, $n-1$ ways to choose $a_2(=b_1=d_2=c_1)$ from $N_n\setminus\{d_1\}$.
Therefore, there are $n(n-1)=n^2-n$ ways to choose four $B_{n-2}$'s as required.
The probability $P_{1.3.2}$ that there are four fault-free $B_{n-2}$'s chosen as in Case 1.3.2 is $\frac{1}{2}(n^2-n)p^{4(n-2)!}$.

Therefore, the probability $P_{1.3}$ that there are four fault-free $B_{n-2}$'s chosen as in Case 1.3 is $P_{1.3.1}+P_{1.3.2}=\frac{1}{2}(n^2-n)p^{4(n-2)!}+\frac{1}{2}(n^3-3n^2+2n)p^{4(n-2)!-(n-3)!}$.

{\it Case 1.4.} $c_1\neq a_1$, $c_1\neq b_1$ and $c_2=d_2$.

There are $n$ ways to choose $a_1(=b_2=d_1)$ from $N_n$, $n-1$ ways to choose $a_2(=b_1=d_2=c_2)$ from $N_n\setminus\{a_1\}$, $n-2$ ways to choose $c_1$ from $N_n\setminus\{a_1, b_1\}$.
Therefore, there are $n(n-1)(n-2)=n^3-3n^2+2n$ ways to choose four $B_{n-2}$'s as required.
Clearly, in this case $c_1\neq d_1$.
The probability $P_{1.4}$ that there are four fault-free $B_{n-2}$'s chosen as in Case 1.4 is $\frac{1}{2}(n^3-3n^2+2n)p^{4(n-2)!-(n-3)!}$.

{\it Case 1.5.} $c_1=a_1\neq b_1$ and $c_2=d_2$.

There are $n$ ways to choose $a_1(=b_2=d_1=c_1)$ from $N_n$, $n-1$ ways to choose $a_2(=b_1=d_2=c_2)$ from $N_n\setminus\{a_1\}$.
Therefore, there are $n(n-1)=n^2-n$ ways to choose four $B_{n-2}$'s as required. 
Clearly, in this case $c_2=a_2$ and $c_1=d_1$.
The probability $P_{1.5}$ that there are four fault-free $B_{n-2}$'s chosen as in Case 1.5 is $\frac{1}{2}(n^2-n)p^{4(n-2)!}$.

In summary, the probability $P_1$ that there are four fault-free $B_{n-2}$'s chosen as in Case 1 is $2(\sum_{i=1}^{5}P_{1.i})=(n^4-5n^3+10n^2-6n)p^{4(n-2)!}+3(n^3-3n^2+2n)p^{4(n-2)!-(n-3)!}$.

{\it Case 2.} $\{a_1, a_2\}=\{d_1, d_2\}$, $|\{b_1, b_2\}\cap\{d_1, d_2\}|=1$ or $|\{a_1, a_2\}\cap\{d_1, d_2\}|=1$, $\{b_1, b_2\}=\{d_1, d_2\}$.

W.l.o.g, assume that $\{a_1, a_2\}=\{d_1, d_2\}$ and $|\{b_1, b_2\}\cap\{d_1, d_2\}|=1$.

{\it Case 2.1.} $a_1=d_1=b_1$, $a_2=d_2$ and $b_2\notin\{d_1, d_2\}$.

{\it Case 2.1.1.} $c_1\neq a_1$, $c_1\neq b_1$, $c_2\neq d_2$ and $c_1\neq d_2$.

There are $n$ ways to choose $a_1(=d_1=b_1)$ from $N_n$, $n-1$ ways to choose $a_2(=d_2)$ from $N_n\setminus\{a_1\}$, $n-2$ ways to choose $b_2$ from $N_n\setminus\{d_1, d_2\}$, $n-2$ ways to choose $c_1$ from $N_n\setminus\{a_1, d_2\}$, $n-2$ ways to choose $c_2$ from $N_n\setminus\{c_1, d_2\}$.
Therefore, there are $n(n-1)(n-2)^3=n^5-7n^4+18n^3-20n^2+8n$ ways to choose four $B_{n-2}$'s as required.
The probability $P_{2.1.1}$ that there are four fault-free $B_{n-2}$'s chosen as in Case 2.1.1 is $\frac{1}{2}(n^5-7n^4+18n^3-20n^2+8n)p^{4(n-2)!}$.

{\it Case 2.1.2.} $c_1\neq a_1$, $c_1\neq b_1$, $c_2\neq d_2$ and $c_1=d_2$.

There are $n$ ways to choose $a_1(=d_1=b_1)$ from $N_n$, $n-1$ ways to choose $a_2(=d_2=c_1)$ from $N_n\setminus\{a_1\}$, $n-2$ ways to choose $b_2$ from $N_n\setminus\{d_1, d_2\}$, $n-1$ ways to choose $c_2$ from $N_n\setminus\{d_2\}$.
Therefore, there are $n(n-1)^2(n-2)=n^4-4n^3+5n^2-2n$ ways to choose four $B_{n-2}$'s as required.
The probability $P_{2.1.2}$ that there are four fault-free $B_{n-2}$'s chosen as in Case 2.1.2 is $\frac{1}{2}(n^4-4n^3+5n^2-2n)p^{4(n-2)!}$.

{\it Case 2.1.3.} $c_1=a_1=b_1$, $c_2\neq d_2$ and $c_2\notin\{a_2, b_2\}$.

There are $n$ ways to choose $a_1(=b_1=c_1=d_1)$ from $N_n$, $n-1$ ways to choose $a_2(=d_2)$ from $N_n\setminus\{a_1\}$, $n-2$ ways to choose $b_2$ from $N_n\setminus\{a_1, d_2\}$, $n-3$ ways to choose $c_2$ from $N_n\setminus\{b_2, d_2, c_1\}$.
Therefore, there are $n(n-1)(n-2)(n-3)=n^4-6n^3+11n^2-6n$ ways to choose four $B_{n-2}$'s as required.
The probability $P_{2.1.3}$ that there are four fault-free $B_{n-2}$'s chosen as in Case 2.1.3 is $\frac{1}{2}(n^4-6n^3+11n^2-6n)p^{4(n-2)!-2(n-3)!}$.

{\it Case 2.1.4.} $c_1=a_1=b_1$, $c_2\neq d_2$ and $c_2=b_2$.

There are $n$ ways to choose $a_1(=b_1=c_1=d_1)$ from $N_n$, $n-1$ ways to choose $a_2(=d_2)$ from $N_n\setminus\{a_1\}$, $n-2$ ways to choose $b_2(=c_2)$ from $N_n\setminus\{d_1, d_2\}$.
Therefore, there are $n(n-1)(n-2)=n^3-3n^2+2n$ ways to choose four $B_{n-2}$'s as required.
The probability $P_{2.1.4}$ that there are four fault-free $B_{n-2}$'s chosen as in Case 2.1.4 is $\frac{1}{2}(n^3-3n^2+2n)p^{4(n-2)!-(n-3)!}$.

{\it Case 2.1.5.} $c_1\neq a_1=b_1$ and $c_2=d_2$.

There are $n$ ways to choose $a_1(=d_1=b_1)$ from $N_n$, $n-1$ ways to choose $a_2(=d_2=c_2)$ from $N_n\setminus\{a_1\}$, $n-2$ ways to choose $b_2$ from $N_n\setminus\{d_1, d_2\}$, $n-2$ ways to choose $c_1$ from $N_n\setminus\{a_1, c_2\}$.
Therefore, there are $n(n-1)(n-2)^2=n^4-5n^3+8n^2-4n$ ways to choose four $B_{n-2}$'s as required.
Clearly, in this case $c_1\neq d_1$.
The probability $P_{2.1.5}$ that there are four fault-free $B_{n-2}$'s chosen as in Case 2.1.5 is $\frac{1}{2}(n^4-5n^3+8n^2-4n)p^{4(n-2)!-(n-3)!}$.

{\it Case 2.1.6.} $c_1=a_1=b_1$ and $c_2=d_2$. 

There are $n$ ways to choose $a_1(=b_1=c_1=d_1)$ from $N_n$, $n-1$ ways to choose $a_2(=d_2=c_2)$ from $N_n\setminus\{a_1\}$, $n-2$ ways to choose $b_2$ from $N_n\setminus\{d_1, d_2\}$.
Therefore, there are $n(n-1)(n-2)=n^3-3n^2+2n$ ways to choose four $B_{n-2}$'s as required.
Clearly, in this case $c_2=a_2$, $c_2\neq b_2$ and $c_1\neq d_1$.
The probability $P_{2.1.6}$ that there are four fault-free $B_{n-2}$'s chosen as in Case 2.1.6 is $\frac{1}{2}(n^3-3n^2+2n)p^{4(n-2)!-(n-3)!}$.

Thus, the probability $P_{2.1}$ that there are four fault-free $B_{n-2}$'s chosen as in Case 2.1 is $\sum_{i=1}^{6}P_{2.1.i}=\frac{1}{2}(n^5-6n^4+14n^3-15n^2+6n)p^{4(n-2)!}+\frac{1}{2}(n^4-3n^3+2n^2)p^{4(n-2)!-(n-3)!}+\frac{1}{2}(n^4-6n^3+11n^2-6n)p^{4(n-2)!-2(n-3)!}$.

{\it Case 2.2.} $a_1=d_1$, $a_2=d_2=b_1$ and $b_2\notin\{d_1, d_2\}$.

{\it Case 2.2.1.} $c_1\neq a_1$, $c_1\neq b_1$ and $c_2\neq d_2$.

There are $n$ ways to choose $a_1(=d_1)$ from $N_n$, $n-1$ ways to choose $a_2(=d_2=b_1)$ from $N_n\setminus\{a_1\}$, $n-2$ ways to choose $b_2$ from $N_n\setminus\{d_1, d_2\}$, $n-2$ ways to choose $c_1$ from $N_n\setminus\{a_1, b_1\}$, $n-2$ ways to choose $c_2$ from $N_n\setminus\{c_1, d_2\}$.
Therefore, there are $n(n-1)(n-2)^3=n^5-7n^4+18n^3-20n^2+8n$ ways to choose four $B_{n-2}$'s as required.
The probability $P_{2.2.1}$ that there are four fault-free $B_{n-2}$'s chosen as in Case 2.2.1 is $\frac{1}{2}(n^5-7n^4+18n^3-20n^2+8n)p^{4(n-2)!}$.

{\it Case 2.2.2.} $c_1=a_1\neq b_1$ and $c_2\neq d_2$.

There are $n$ ways to choose $a_1(=d_1=c_1)$ from $N_n$, $n-1$ ways to choose $a_2(=d_2=b_1)$ from $N_n\setminus\{a_1\}$, $n-2$ ways to choose $b_2$ from $N_n\setminus\{d_1, d_2\}$, $n-2$ ways to choose $c_2$ from $N_n\setminus\{d_2, c_1\}$.
Therefore, there are $n(n-1)(n-2)^2=n^4-5n^3+8n^2-4n$ ways to choose four $B_{n-2}$'s as required.
Clearly, in this case $c_2\neq a_2$.
The probability $P_{2.2.2}$ that there are four fault-free $B_{n-2}$'s chosen as in Case 2.2.2 is $\frac{1}{2}(n^4-5n^3+8n^2-4n)p^{4(n-2)!-(n-3)!}$.

{\it Case 2.2.3.} $c_1=b_1\neq a_1$, $c_2\neq d_2$ and $c_2\neq b_2$.

There are $n$ ways to choose $a_1(=d_1)$ from $N_n$, $n-1$ ways to choose $a_2(=d_2=b_1=c_1)$ from $N_n\setminus\{a_1\}$, $n-2$ ways to choose $b_2$ from $N_n\setminus\{d_1, d_2\}$, $n-2$ ways to choose $c_2$ from $N_n\setminus\{c_1, b_2\}$.
Therefore, there are $n(n-1)(n-2)^2=n^4-5n^3+8n^2-4n$ ways to choose four $B_{n-2}$'s as required.
The probability $P_{2.2.3}$ that there are four fault-free $B_{n-2}$'s chosen as in Case 2.2.3 is $\frac{1}{2}(n^4-5n^3+8n^2-4n)p^{4(n-2)!-(n-3)!}$.

{\it Case 2.2.4.} $c_1=b_1\neq a_1$, $c_2\neq d_2$ and $c_2=b_2$.

There are $n$ ways to choose $a_1(=d_1)$ from $N_n$, $n-1$ ways to choose $a_2(=d_2=b_1=c_1)$ from $N_n\setminus\{a_1\}$, $n-2$ ways to choose $b_2(=c_2)$ from $N_n\setminus\{d_1, d_2\}$.
Therefore, there are $n(n-1)(n-2)=n^3-3n^2+2n$ ways to choose four $B_{n-2}$'s as required.
The probability $P_{2.2.4}$ that there are four fault-free $B_{n-2}$'s chosen as in Case 2.2.4 is $\frac{1}{2}(n^3-3n^2+2n)p^{4(n-2)!}$.

{\it Case 2.2.5.} $c_1\neq a_1$, $c_1\neq b_1$ and $c_2=d_2$.

There are $n$ ways to choose $a_1(=d_1)$ from $N_n$, $n-1$ ways to choose $a_2(=d_2=b_1=c_2)$ from $N_n\setminus\{a_1\}$, $n-2$ ways to choose $b_2$ from $N_n\setminus\{d_1, d_2\}$, $n-2$ ways to choose $c_1$ from $N_n\setminus\{a_1, b_1\}$.
Therefore, there are $n(n-1)(n-2)^2=n^4-5n^3+8n^2-4n$ ways to choose four $B_{n-2}$'s as required.
Clearly, in this case $c_1\neq d_1$.
The probability $P_{2.2.5}$ that there are four fault-free $B_{n-2}$'s chosen as in Case 2.2.5 is $\frac{1}{2}(n^4-5n^3+8n^2-4n)p^{4(n-2)!-(n-3)!}$.

{\it Case 2.2.6.} $c_1=a_1\neq b_1$ and $c_2=d_2$.

There are $n$ ways to choose $a_1(=d_1=c_1)$ from $N_n$, $n-1$ ways to choose $a_2(=d_2=b_1=c_2)$ from $N_n\setminus\{a_1\}$, $n-2$ ways to choose $b_2$ from $N_n\setminus\{d_1, d_2\}$.
Therefore, there are $n(n-1)(n-2)=n^3-3n^2+2n$ ways to choose four $B_{n-2}$'s as required.
Clearly, in this case $c_2=a_2$ and $c_1=d_1$.
The probability $P_{2.2.6}$ that there are four fault-free $B_{n-2}$'s chosen as in Case 2.2.6 is $\frac{1}{2}(n^3-3n^2+2n)p^{4(n-2)!}$.

Thus, the probability $P_{2.2}$ that there are four fault-free $B_{n-2}$'s chosen as in Case 2.2 is $\sum_{i=1}^{6}P_{2.2.i}=\frac{1}{2}(n^5-7n^4+20n^3-26n^2+12n)p^{4(n-2)!}+\frac{3}{2}(n^4-5n^3+8n^2-4n)p^{4(n-2)!-(n-3)!}$.

{\it Case 2.3.} $a_1=d_1=b_2$, $a_2=d_2$, $b_1\notin\{d_1, d_2\}$.

{\it Case 2.3.1.} $c_1\neq a_1$, $c_1\neq b_1$, $c_2\neq d_2$ and $c_1\neq d_2$.

There are $n$ ways to choose $a_1(=d_1=b_2)$ from $N_n$, $n-1$ ways to choose $a_2(=d_2)$ from $N_n\setminus\{a_1\}$, $n-2$ ways to choose $b_1$ from $N_n\setminus\{d_1, d_2\}$, $n-3$ ways to choose $c_1$ from $N_n\setminus\{a_1, b_1, d_2\}$, $n-2$ ways to choose $c_2$ from $N_n\setminus\{c_1, d_2\}$.
Therefore, there are $n(n-1)(n-2)^2(n-3)=n^5-8n^4+23n^3-28n^2+12n$ ways to choose four $B_{n-2}$'s as required.
The probability $P_{2.3.1}$ that there are four fault-free $B_{n-2}$'s chosen as in Case 2.3.1 is $\frac{1}{2}(n^5-8n^4+23n^3-28n^2+12n)p^{4(n-2)!}$.

{\it Case 2.3.2.} $c_1\neq a_1$, $c_1\neq b_1$, $c_2\neq d_2$ and $c_1=d_2$.

There are $n$ ways to choose $a_1(=d_1=b_2)$ from $N_n$, $n-1$ ways to choose $a_2(=d_2=c_1)$ from $N_n\setminus\{a_1\}$, $n-2$ ways to choose $b_1$ from $N_n\setminus\{d_1, d_2\}$, $n-1$ ways to choose $c_2$ from $N_n\setminus\{d_2\}$.
Therefore, there are $n(n-1)^2(n-2)=n^4-4n^3+5n^2-2n$ ways to choose four $B_{n-2}$'s as required.
The probability $P_{2.3.2}$ that there are four fault-free $B_{n-2}$'s chosen as in Case 2.3.2 $\frac{1}{2}(n^4-4n^3+5n^2-2n)p^{4(n-2)!}$.

{\it Case 2.3.3.} $c_1=a_1\neq b_1$, $c_2\neq d_2$.

There are $n$ ways to choose $a_1(=b_2=d_1=c_1)$ from $N_n$, $n-1$ ways to choose $a_2(=d_2)$ from $N_n\setminus\{a_1\}$, $n-2$ ways to choose $b_1$ from $N_n\setminus\{d_1, d_2\}$, $n-2$ ways to choose $c_2$ from $N_n\setminus\{d_2, c_1\}$.
Therefore, there are $n(n-1)(n-2)^2=n^4-5n^3+8n^2-4n$ ways to choose four $B_{n-2}$'s as required.
Clearly, in this case $c_2=a_2$.
The probability $P_{2.3.3}$ that there are four fault-free $B_{n-2}$'s chosen as in Case 2.3.3 $\frac{1}{2}(n^4-5n^3+8n^2-4n)p^{4(n-2)!-(n-3)!}$.

{\it Case 2.3.4.} $c_1=b_1\neq a_1$, $c_2\neq d_2$, $c_2\neq b_2$.

There are $n$ ways to choose $a_1(=b_2=d_1)$ from $N_n$, $n-1$ ways to choose $a_2(=d_2)$ from $N_n\setminus\{a_1\}$, $n-2$ ways to choose $b_1(=c_1)$ from $N_n\setminus\{d_1, d_2\}$, $n-3$ ways to choose $c_2$ from $N_n\setminus\{b_2, d_2, c_1\}$.
Therefore, there are $n(n-1)(n-2)(n-3)=n^4-6n^3+11n^2-6n$ ways to choose four $B_{n-2}$'s as required.
The probability $P_{2.3.4}$ that there are four fault-free $B_{n-2}$'s chosen as in Case 2.3.4 $\frac{1}{2}(n^4-6n^3+11n^2-6n)p^{4(n-2)!-(n-3)!}$.

{\it Case 2.3.5.} $c_1=b_1\neq a_1$, $c_2\neq d_2$, $c_2=b_2$.

There are $n$ ways to choose $a_1(=b_2=d_1=c_2)$ from $N_n$, $n-1$ ways to choose $a_2(=d_2)$ from $N_n\setminus\{a_1\}$, $n-2$ ways to choose $b_1(=c_1)$ from $N_n\setminus\{d_1, d_2\}$.
Therefore, there are $n(n-1)(n-2)=n^3-3n^2+2n$ ways to choose four $B_{n-2}$'s as required.
The probability $P_{2.3.5}$ that there are four fault-free $B_{n-2}$'s chosen as in Case 2.3.5 is $\frac{1}{2}(n^3-3n^2+2n)p^{4(n-2)!}$.

{\it Case 2.3.6.} $c_1\neq a_1$, $c_1\neq b_1$ and $c_2=d_2$.

There are $n$ ways to choose $a_1(=d_1=b_2)$ from $N_n$, $n-1$ ways to choose $a_2(=d_2=c_2)$ from $N_n\setminus\{a_1\}$, $n-2$ ways to choose $b_1$ from $N_n\setminus\{d_1, d_2\}$, $n-3$ ways to choose $c_1$ from $N_n\setminus\{a_1, b_1, c_2\}$.
Therefore, there are $n(n-1)(n-2)(n-3)=n^4-6n^3+11n^2-6n$ ways to choose four $B_{n-2}$'s as required.
Clearly, in this case $c_1\neq d_1$.
The probability $P_{2.3.6}$ that there are four fault-free $B_{n-2}$'s chosen as in Case 2.3.6 is $\frac{1}{2}(n^4-6n^3+11n^2-6n)p^{4(n-2)!-(n-3)!}$.

{\it Case 2.3.7.} $c_1=a_1\neq b_1$, $c_2=d_2$.

There are $n$ ways to choose $a_1(=b_2=d_1=c_1)$ from $N_n$, $n-1$ ways to choose $a_2(=d_2=c_2)$ from $N_n\setminus\{a_1\}$, $n-2$ ways to choose $b_1$ from $N_n\setminus\{d_1, d_2\}$.
Therefore, there are $n(n-1)(n-2)=n^3-3n^2+2n$ ways to choose four $B_{n-2}$'s as required.
Clearly, in this case $c_2=a_2$ and $c_1=d_1$.
The probability $P_{2.3.7}$ that there are four fault-free $B_{n-2}$'s chosen as in Case 2.3.7 is $\frac{1}{2}(n^3-3n^2+2n)p^{4(n-2)!}$.

{\it Case 2.3.8.} $c_1=b_1\neq a_1$, $c_2=d_2$.

There are $n$ ways to choose $a_1(=b_2=d_1)$ from $N_n$, $n-1$ ways to choose $a_2(=d_2=c_2)$ from $N_n\setminus\{a_1\}$, $n-2$ ways to choose $b_1(=c_1)$ from $N_n\setminus\{d_1, d_2\}$.
Therefore, there are $n(n-1)(n-2)=n^3-3n^2+2n$ ways to choose four $B_{n-2}$'s as required.
Clearly, in this case $c_2\neq a_2$ and $c_1\neq d_1$.
The probability $P_{2.3.8}$ that there are four fault-free $B_{n-2}$'s chosen as in Case 2.3.8 is $\frac{1}{2}(n^3-3n^2+2n)p^{4(n-2)!-2(n-3)!}$.

Thus, the probability $P_{2.3}$ that there are four fault-free $B_{n-2}$'s chosen as in Case 2.3 is $\sum_{i=1}^{8}P_{2.3.i}=\frac{1}{2}(n^5-7n^4+21n^3-29n^2+14n)p^{4(n-2)!}+\frac{1}{2}(3n^4-17n^3+30n^2-16n)p^{4(n-2)!-(n-3)!}+\frac{1}{2}(n^3-3n^2+2n)p^{4(n-2)!-2(n-3)!}$.

{\it Case 2.4.} $a_1=d_1$, $a_2=d_2=b_2$, $b_1\notin\{d_1, d_2\}$.

{\it Case 2.4.1.} $c_1\neq a_1$, $c_1\neq b_1$, $c_2\neq d_2$ and $c_1\neq d_2$.

There are $n$ ways to choose $a_1(=d_1)$ from $N_n$, $n-1$ ways to choose $a_2(=d_2=b_2)$ from $N_n\setminus\{a_1\}$, $n-2$ ways to choose $b_1$ from $N_n\setminus\{d_1, d_2\}$, $n-3$ ways to choose $c_1$ from $N_n\setminus\{a_1, b_1, d_2\}$, $n-2$ ways to choose $c_2$ from $N_n\setminus\{c_1, d_2\}$.
Therefore, there are $n(n-1)(n-2)^2(n-3)=n^5-8n^4+23n^3-28n^2+12n$ ways to choose four $B_{n-2}$'s as required.
The probability $P_{2.4.1}$ that there are four fault-free $B_{n-2}$'s chosen as in Case 2.4.1 is $\frac{1}{2}(n^5-8n^4+23n^3-28n^2+12n)p^{4(n-2)!}$.

{\it Case 2.4.2.} $c_1\neq a_1$, $c_1\neq b_1$, $c_2\neq d_2$ and $c_1=d_2$.

There are $n$ ways to choose $a_1(=d_1)$ from $N_n$, $n-1$ ways to choose $a_2(=d_2=b_2=c_1)$ from $N_n\setminus\{a_1\}$, $n-2$ ways to choose $b_1$ from $N_n\setminus\{d_1, d_2\}$, $n-1$ ways to choose $c_2$ from $N_n\setminus\{d_2\}$.
Therefore, there are $n(n-1)^2(n-2)=n^4-4n^3+5n^2-2n$ ways to choose four $B_{n-2}$'s as required.
The probability $P_{2.4.2}$ that there are four fault-free $B_{n-2}$'s chosen as in Case 2.4.2 is $\frac{1}{2}(n^4-4n^3+5n^2-2n)p^{4(n-2)!}$.

{\it Case 2.4.3.} $c_1=a_1\neq b_1$, $c_2\neq d_2$ or $c_1=b_1\neq a_1$, $c_2\neq d_2$.

W.l.o.g, assume that the former applies. There are $n$ ways to choose $a_1(=d_1=c_1)$ from $N_n$, $n-1$ ways to choose $a_2(=d_2=b_2)$ from $N_n\setminus\{a_1\}$, $n-2$ ways to choose $b_1$ from $N_n\setminus\{d_1, d_2\}$, $n-2$ ways to choose $c_2$ from $N_n\setminus\{d_2, c_1\}$.
Therefore, there are $n(n-1)(n-2)^2=n^4-5n^3+8n^2-4n$ ways to choose four $B_{n-2}$'s as required.
Clearly, in this case $c_2\neq a_2$.
The probability $P_{2.4.3}$ that there are four fault-free $B_{n-2}$'s chosen as in Case 2.4.3 is $\frac{1}{2}\times 2(n^4-5n^3+8n^2-4n)p^{4(n-2)!-(n-3)!}$.

{\it Case 2.4.4.} $c_1\neq a_1$, $c_1\neq b_1$ and $c_2=d_2$.

There are $n$ ways to choose $a_1(=d_1)$ from $N_n$, $n-1$ ways to choose $a_2(=b_2=d_2=c_2)$ from $N_n\setminus\{a_1\}$, $n-2$ ways to choose $b_1$ from $N_n\setminus\{d_1, d_2\}$, $n-3$ ways to choose $c_1$ from $N_n\setminus\{a_1, b_1, c_2\}$.
Therefore, there are $n(n-1)(n-2)(n-3)=n^4-6n^3+11n^2-6n$ ways to choose four $B_{n-2}$'s as required.
Clearly, in this case $c_1\neq d_1$.
The probability $P_{2.4.4}$ that there are four fault-free $B_{n-2}$'s chosen as in Case 2.4.4 is $\frac{1}{2}(n^4-6n^3+11n^2-6n)p^{4(n-2)!-(n-3)!}$.

{\it Case 2.4.5.} $c_1=a_1\neq b_1$ and $c_2=d_2$.

There are $n$ ways to choose $a_1(=d_1=c_1)$ from $N_n$, $n-1$ ways to choose $a_2(=d_2=b_2=c_2)$ from $N_n\setminus\{a_1\}$, $n-2$ ways to choose $b_1$ from $N_n\setminus\{d_1, d_2\}$.
Therefore, there are $n(n-1)(n-2)=n^3-3n^2+2n$ ways to choose four $B_{n-2}$'s as required.
Clearly, in this case $c_2=a_2$ and $c_1=d_1$.
The probability $P_{2.4.5}$ that there are four fault-free $B_{n-2}$'s chosen as in Case 2.4.5 is $\frac{1}{2}(n^3-3n^2+2n)p^{4(n-2)!}$.

{\it Case 2.4.6.} $c_1=b_1\neq a_1$, $c_2=d_2$.

There are $n$ ways to choose $a_1(=d_1=c_1)$ from $N_n$, $n-1$ ways to choose $a_2(=d_2=b_2=c_2)$ from $N_n\setminus\{a_1\}$, $n-2$ ways to choose $b_1$ from $N_n\setminus\{d_1, d_2\}$.
Therefore, there are $n(n-1)(n-2)=n^3-3n^2+2n$ ways to choose four $B_{n-2}$'s as required.
Clearly, in this case $c_2=a_2$ and $c_1\neq d_1$.
The probability $P_{2.4.6}$ that there are four fault-free $B_{n-2}$'s chosen as in Case 2.4.6 is $\frac{1}{2}(n^3-3n^2+2n)p^{4(n-2)!-(n-3)!}$.

Thus, the probability $P_{2.4}$ that there are four fault-free $B_{n-2}$'s chosen as in Case 2.4 is $\sum_{i=1}^{6}P_{2.4.i}=\frac{1}{2}(n^5-7n^4+20n^3-26n^2+12n)p^{4(n-2)!}+\frac{3}{2}(n^4-5n^3+8n^2-4n)p^{4(n-2)!-(n-3)!}$.

{\it Case 2.5.} $a_1=d_2$, $a_2=d_1=b_1$, $b_2\notin\{d_1, d_2\}$.

{\it Case 2.5.1.} $c_1\neq a_1$, $c_1\neq b_1$ and $c_2\neq d_2$.

There are $n$ ways to choose $a_1(=d_2)$ from $N_n$, $n-1$ ways to choose $a_2(=d_1=b_1)$ from $N_n\setminus\{a_1\}$, $n-2$ ways to choose $b_2$ from $N_n\setminus\{d_1, d_2\}$, $n-2$ ways to choose $c_1$ from $N_n\setminus\{a_1, b_1\}$, $n-2$ ways to choose $c_2$ from $N_n\setminus\{c_1, d_2\}$.
Therefore, there are $n(n-1)(n-2)^3=n^5-7n^4+18n^3-20n^2+8n$ ways to choose four $B_{n-2}$'s as required.
The probability $P_{2.5.1}$ that there are four fault-free $B_{n-2}$'s chosen as in Case 2.5.1 is $\frac{1}{2}(n^5-7n^4+18n^3-20n^2+8n)p^{4(n-2)!}$.

{\it Case 2.5.2.} $c_1=a_1\neq b_1$, $c_2\neq d_2$ and $a_2\neq c_2$.

There are $n$ ways to choose $a_1(=d_2=c_1)$ from $N_n$, $n-1$ ways to choose $a_2(=d_1=b_1)$ from $N_n\setminus\{a_1\}$, $n-2$ ways to choose $b_2$ from $N_n\setminus\{d_1, d_2\}$, $n-2$ ways to choose $c_2$ from $N_n\setminus\{c_1, a_2\}$.
Therefore, there are $n(n-1)(n-2)^2=n^4-5n^3+8n^2-4n$ ways to choose four $B_{n-2}$'s as required.
The probability $P_{2.5.2}$ that there are four fault-free $B_{n-2}$'s chosen as in Case 2.5.2 is $\frac{1}{2}(n^4-5n^3+8n^2-4n)p^{4(n-2)!-(n-3)!}$.

{\it Case 2.5.3.} $c_1=a_1\neq b_1$, $c_2\neq d_2$ and $a_2=c_2$.

There are $n$ ways to choose $a_1(=d_2=c_1)$ from $N_n$, $n-1$ ways to choose $a_2(=d_1=b_1=c_2)$ from $N_n\setminus\{a_1\}$, $n-2$ ways to choose $b_2$ from $N_n\setminus\{d_1, d_2\}$.
Therefore, there are $n(n-1)(n-2)=n^3-3n^2+2n$ ways to choose four $B_{n-2}$'s as required.
The probability $P_{2.5.3}$ that there are four fault-free $B_{n-2}$'s chosen as in Case 2.5.3 is $\frac{1}{2}(n^3-3n^2+2n)p^{4(n-2)!}$.

{\it Case 2.5.4.} $c_1=b_1\neq a_1$, $c_2\neq d_2$ and $b_2\neq c_2$.

There are $n$ ways to choose $a_1(=d_2)$ from $N_n$, $n-1$ ways to choose $a_2(=d_1=b_1=c_1)$ from $N_n\setminus\{a_1\}$, $n-2$ ways to choose $b_2$ from $N_n\setminus\{d_1, d_2\}$, $n-3$ ways to choose $c_2$ from $N_n\setminus\{b_2, d_2, c_1\}$.
Therefore, there are $n(n-1)(n-2)(n-3)=n^4-6n^3+11n^2-6n$ ways to choose four $B_{n-2}$'s as required.
The probability $P_{2.5.4}$ that there are four fault-free $B_{n-2}$'s chosen as in Case 2.5.4 is $\frac{1}{2}(n^4-6n^3+11n^2-6n)p^{4(n-2)!-(n-3)!}$.

{\it Case 2.5.5.} $c_1=b_1\neq a_1$ and $c_2\neq d_2$ and $b_2=c_2$..

There are $n$ ways to choose $a_1(=d_2)$ from $N_n$, $n-1$ ways to choose $a_2(=d_1=b_1=c_1)$ from $N_n\setminus\{a_1\}$, $n-2$ ways to choose $b_2(=c_2)$ from $N_n\setminus\{d_1, d_2\}$.
Therefore, there are $n(n-1)(n-2)=n^3-3n^2+2n$ ways to choose four $B_{n-2}$'s as required.
The probability $P_{2.5.5}$ that there are four fault-free $B_{n-2}$'s chosen as in Case 2.5.5 is $\frac{1}{2}(n^3-3n^2+2n)p^{4(n-2)!}$.

{\it Case 2.5.6.} $c_1\neq a_1$, $c_1\neq b_1$ and $c_2=d_2$.

There are $n$ ways to choose $a_1(=d_2=c_2)$ from $N_n$, $n-1$ ways to choose $a_2(=d_1=b_1)$ from $N_n\setminus\{a_1\}$, $n-2$ ways to choose $b_2$ from $N_n\setminus\{d_1, d_2\}$, $n-2$ ways to choose $c_1$ from $N_n\setminus\{a_1, b_1\}$.
Therefore, there are $n(n-1)(n-2)^2=n^4-5n^3+8n^2-4n$ ways to choose four $B_{n-2}$'s as required.
Clearly, in this case $c_1\neq d_1$.
The probability $P_{2.5.6}$ that there are four fault-free $B_{n-2}$'s chosen as in Case 2.5.6 is $\frac{1}{2}(n^4-5n^3+8n^2-4n)p^{4(n-2)!-(n-3)!}$.

{\it Case 2.5.7.} $c_1=b_1\neq a_1$ and $c_2=d_2$.

There are $n$ ways to choose $a_1(=d_2=c_2)$ from $N_n$, $n-1$ ways to choose $a_2(=d_1=b_1=c_1)$ from $N_n\setminus\{a_1\}$, $n-2$ ways to choose $b_2$ from $N_n\setminus\{d_1, d_2\}$.
Therefore, there are $n(n-1)(n-2)=n^3-3n^2+2n$ ways to choose four $B_{n-2}$'s as required.
Clearly, in this case $c_2\neq b_2$ and $c_1=d_1$.
The probability $P_{2.5.7}$ that there are four fault-free $B_{n-2}$'s chosen as in Case 2.5.7 is $\frac{1}{2}(n^3-3n^2+2n)p^{4(n-2)!-(n-3)!}$.

Thus, the probability $P_{2.7}$ that there are four fault-free $B_{n-2}$'s chosen as in Case 2.7 is $\sum_{i=1}^{6}P_{2.7.i}=\frac{1}{2}(n^5-7n^4+20n^3-26n^2+12n)p^{4(n-2)!}+\frac{3}{2}(n^4-5n^3+8n^2-4n)p^{4(n-2)!-(n-3)!}$.

{\it Case 2.6.} $a_1=d_2=b_1$, $a_2=d_1$ and $b_2\notin\{d_1, d_2\}$.

{\it Case 2.6.1.} $c_1\neq a_1$, $c_1\neq b_1$ and $c_2\neq d_2$.

There are $n$ ways to choose $a_1(=d_2=b_1)$ from $N_n$, $n-1$ ways to choose $a_2(=d_1)$ from $N_n\setminus\{a_1\}$, $n-2$ ways to choose $b_2$ from $N_n\setminus\{d_1, d_2\}$, $n-1$ ways to choose $c_1$ from $N_n\setminus\{a_1\}$, $n-2$ ways to choose $c_2$ from $N_n\setminus\{c_1, d_2\}$.
Therefore, there are $n(n-1)^2(n-2)^2=n^5-6n^4+13n^3-12n^2+4n$ ways to choose four $B_{n-2}$'s as required.
The probability $P_{2.6.1}$ that there are four fault-free $B_{n-2}$'s chosen as in Case 2.6.1 is $\frac{1}{2}(n^5-6n^4+13n^3-12n^2+4n)p^{4(n-2)!}$.

{\it Case 2.6.2.} $c_1=a_1=b_1$, $c_2\neq d_2$ and $c_2\notin\{a_2, b_2\}$.

There are $n$ ways to choose $a_1(=b_1=c_1=d_2)$ from $N_n$, $n-1$ ways to choose $a_2(=d_1)$ from $N_n\setminus\{a_1\}$, $n-2$ ways to choose $b_2$ from $N_n\setminus\{d_1, d_2\}$, $n-3$ ways to choose $c_2$ from $N_n\setminus\{d_2, a_2, b_2\}$.
Therefore, there are $n(n-1)(n-2)(n-3)=n^4-6n^3+11n^2-6n$ ways to choose four $B_{n-2}$'s as required.
The probability $P_{2.6.2}$ that there are four fault-free $B_{n-2}$'s chosen as in Case 2.6.2 is $\frac{1}{2}(n^4-6n^3+11n^2-6n)p^{4(n-2)!-2(n-3)!}$.

{\it Case 2.6.3.} $c_1=a_1=b_1$, $c_2\neq d_2$ and $c_2\in\{a_2, b_2\}$.

There are 2 possible cases, w.l.o.g, assume that $c_1=a_1=b_1$, $c_2\neq d_2$ and $c_2=a_2$.
There are $n$ ways to choose $a_1(=b_1=c_1=d_2)$ from $N_n$, $n-1$ ways to choose $a_2(=d_1=c_2)$ from $N_n\setminus\{a_1\}$, $n-2$ ways to choose $b_2$ from $N_n\setminus\{d_1, d_2\}$.
Therefore, there are $2n(n-1)(n-2)=2(n^3-3n^2+2n)$ ways to choose four $B_{n-2}$'s as required.
The probability $P_{2.6.3}$ that there are four fault-free $B_{n-2}$'s chosen as in Case 2.6.3 is $\frac{1}{2}\times 2(n^3-3n^2+2n)p^{4(n-2)!-(n-3)!}$.

{\it Case 2.6.4.} $c_1\neq a_1=b_1$, $c_2=d_2$ and $c_1\neq d_1$.

There are $n$ ways to choose $a_1(=d_2=b_1=c_2)$ from $N_n$, $n-1$ ways to choose $a_2(=d_1)$ from $N_n\setminus\{a_1\}$, $n-2$ ways to choose $b_2$ from $N_n\setminus\{d_1, d_2\}$, $n-2$ ways to choose $c_1$ from $N_n\setminus\{a_1, d_1\}$.
Therefore, there are $n(n-1)(n-2)^2=n^4-5n^3+8n^2-4n$ ways to choose four $B_{n-2}$'s as required.
The probability $P_{2.6.4}$ that there are four fault-free $B_{n-2}$'s chosen as in Case 2.6.4 is $\frac{1}{2}(n^4-5n^3+8n^2-4n)p^{4(n-2)!-(n-3)!}$.

{\it Case 2.6.5.} $c_1\neq a_1=b_1$, $c_2=d_2$ and $c_1=d_1$.

There are $n$ ways to choose $a_1(=d_2=b_1=c_2)$ from $N_n$, $n-1$ ways to choose $a_2(=d_1=c_1)$ from $N_n\setminus\{a_1\}$, $n-2$ ways to choose $b_2$ from $N_n\setminus\{d_1, d_2\}$.
Therefore, there are $n(n-1)(n-2)=n^3-3n^2+2n$ ways to choose four $B_{n-2}$'s as required.
The probability $P_{2.6.5}$ that there are four fault-free $B_{n-2}$'s chosen as in Case 2.6.5 is $\frac{1}{2}(n^3-3n^2+2n)p^{4(n-2)!}$.

Thus, the probability $P_{2.6}$ that there are four fault-free $B_{n-2}$'s chosen as in Case 2.6 is $\sum_{i=1}^{5}P_{2.6.i}=\frac{1}{2}(n^5-6n^4+14n^3-15n^2+6n)p^{4(n-2)!}+\frac{1}{2}(n^4-3n^3+2n^2)p^{4(n-2)!-(n-3)!}+\frac{1}{2}(n^4-6n^3+11n^2-6n)p^{4(n-2)!-2(n-3)!}$.

{\it Case 2.7.} $a_1=d_2$, $a_2=d_1=b_2$, $b_1\notin\{d_1, d_2\}$.

{\it Case 2.7.1.} $c_1\neq a_1$, $c_1\neq b_1$ and $c_2\neq d_2$.

There are $n$ ways to choose $a_1(=d_2)$ from $N_n$, $n-1$ ways to choose $a_2(=d_1=b_2)$ from $N_n\setminus\{a_1\}$, $n-2$ ways to choose $b_1$ from $N_n\setminus\{d_1, d_2\}$, $n-2$ ways to choose $c_1$ from $N_n\setminus\{a_1, b_1\}$, $n-2$ ways to choose $c_2$ from $N_n\setminus\{c_1, d_2\}$.
Therefore, there are $n(n-1)(n-2)^3=n^5-7n^4+18n^3-20n^2+8n$ ways to choose four $B_{n-2}$'s as required.
The probability $P_{2.7.1}$ that there are four fault-free $B_{n-2}$'s chosen as in Case 2.7.1 is $\frac{1}{2}(n^5-7n^4+18n^3-20n^2+8n)p^{4(n-2)!}$.

{\it Case 2.7.2.} $c_1=a_1\neq b_1$, $c_2\neq d_2$ and $a_2\neq c_2$.

There are $n$ ways to choose $a_1(=d_2=c_1)$ from $N_n$, $n-1$ ways to choose $a_2(=d_1=b_2)$ from $N_n\setminus\{a_1\}$, $n-2$ ways to choose $b_1$ from $N_n\setminus\{d_1, d_2\}$, $n-2$ ways to choose $c_2$ from $N_n\setminus\{d_2, a_2\}$.
Therefore, there are $n(n-1)(n-2)^2=n^4-5n^3+8n^2-4n$ ways to choose four $B_{n-2}$'s as required.
The probability $P_{2.7.2}$ that there are four fault-free $B_{n-2}$'s chosen as in Case 2.7.2 is $\frac{1}{2}(n^4-5n^3+8n^2-4n)p^{4(n-2)!-(n-3)!}$.

{\it Case 2.7.3.} $c_1=a_1\neq b_1$, $c_2\neq d_2$ and $a_2=c_2$.

There are $n$ ways to choose $a_1(=d_2=c_1)$ from $N_n$, $n-1$ ways to choose $a_2(=d_1=b_2=c_2)$ from $N_n\setminus\{a_1\}$, $n-2$ ways to choose $b_1$ from $N_n\setminus\{d_1, d_2\}$.
Therefore, there are $n(n-1)(n-2)=n^3-3n^2+2n$ ways to choose four $B_{n-2}$'s as required.
The probability $P_{2.7.3}$ that there are four fault-free $B_{n-2}$'s chosen as in Case 2.7.3 is $\frac{1}{2}(n^3-3n^2+2n)p^{4(n-2)!}$.

{\it Case 2.7.4.} $c_1=b_1\neq a_1$, $c_2\neq d_2$ and $b_2\neq c_2$.

There are $n$ ways to choose $a_1(=d_2)$ from $N_n$, $n-1$ ways to choose $a_2(=d_1=b_2)$ from $N_n\setminus\{a_1\}$, $n-2$ ways to choose $b_1(=c_1)$ from $N_n\setminus\{d_1, d_2\}$, $n-3$ ways to choose $c_2$ from $N_n\setminus\{b_2, d_2, c_1\}$.
Therefore, there are $n(n-1)(n-2)(n-3)=n^4-6n^3+11n^2-6n$ ways to choose four $B_{n-2}$'s as required.
The probability $P_{2.7.4}$ that there are four fault-free $B_{n-2}$'s chosen as in Case 2.7.4 is $\frac{1}{2}(n^4-6n^3+11n^2-6n)p^{4(n-2)!-(n-3)!}$.

{\it Case 2.7.5.} $c_1=b_1\neq a_1$ and $c_2\neq d_2$ and $b_2=c_2$..

There are $n$ ways to choose $a_1(=d_2)$ from $N_n$, $n-1$ ways to choose $a_2(=d_1=b_2=c_2)$ from $N_n\setminus\{a_1\}$, $n-2$ ways to choose $b_1(=c_1)$ from $N_n\setminus\{d_1, d_2\}$.
Therefore, there are $n(n-1)(n-2)=n^3-3n^2+2n$ ways to choose four $B_{n-2}$'s as required.
The probability $P_{2.7.5}$ that there are four fault-free $B_{n-2}$'s chosen as in Case 2.7.5 is $\frac{1}{2}(n^3-3n^2+2n)p^{4(n-2)!}$.

{\it Case 2.7.6.} $c_1\neq a_1$, $c_1\neq b_1$, $c_2=d_2$ and $c_1\neq d_1$.

There are $n$ ways to choose $a_1(=d_2=c_2)$ from $N_n$, $n-1$ ways to choose $a_2(=d_1=b_2)$ from $N_n\setminus\{a_1\}$, $n-2$ ways to choose $b_1$ from $N_n\setminus\{d_1, d_2\}$, $n-3$ ways to choose $c_1$ from $N_n\setminus\{a_1, b_1, d_1\}$.
Therefore, there are $n(n-1)(n-2)(n-3)=n^4-6n^3+11n^2-6n$ ways to choose four $B_{n-2}$'s as required.
The probability $P_{2.7.6}$ that there are four fault-free $B_{n-2}$'s chosen as in Case 2.7.6 is $\frac{1}{2}(n^4-6n^3+11n^2-6n)p^{4(n-2)!-(n-3)!}$.

{\it Case 2.7.7.} $c_1\neq a_1$, $c_1\neq b_1$, $c_2=d_2$ and $c_1=d_1$.

There are $n$ ways to choose $a_1(=d_2=c_2)$ from $N_n$, $n-1$ ways to choose $a_2(=d_1=b_2=c_1)$ from $N_n\setminus\{a_1\}$, $n-2$ ways to choose $b_1$ from $N_n\setminus\{d_1, d_2\}$.
Therefore, there are $n(n-1)(n-2)=n^3-3n^2+2n$ ways to choose four $B_{n-2}$'s as required.
The probability $P_{2.7.7}$ that there are four fault-free $B_{n-2}$'s chosen as in Case 2.7.7 is $\frac{1}{2}(n^3-3n^2+2n)p^{4(n-2)!}$.

{\it Case 2.7.8.} $c_1=b_1\neq a_1$ and $c_2=d_2$.

There are $n$ ways to choose $a_1(=d_2=c_2)$ from $N_n$, $n-1$ ways to choose $a_2(=d_1=b_2)$ from $N_n\setminus\{a_1\}$, $n-2$ ways to choose $b_1(=c_1)$ from $N_n\setminus\{d_1, d_2\}$.
Therefore, there are $n(n-1)(n-2)=n^3-3n^2+2n$ ways to choose four $B_{n-2}$'s as required.
Clearly, in this case $c_2\neq b_2$ and $c_1\neq d_1$.
The probability $P_{2.7.8}$ that there are four fault-free $B_{n-2}$'s chosen as in Case 2.7.8 is $\frac{1}{2}(n^3-3n^2+2n)p^{4(n-2)!-2(n-3)!}$.

Thus, the probability $P_{2.7}$ that there are four fault-free $B_{n-2}$'s chosen as in Case 2.7 is $\sum_{i=1}^{8}P_{2.7.i}=\frac{1}{2}(n^5-7n^4+21n^3-29n^2+14n)p^{4(n-2)!}+\frac{1}{2}(3n^4-17n^3+30n^2-16n)p^{4(n-2)!-(n-3)!}+\frac{1}{2}(n^3-3n^2+2n)p^{4(n-2)!-2(n-3)!}$.

{\it Case 2.8.} $a_1=d_2=b_2$, $a_2=d_1$, $b_1\notin\{d_1, d_2\}$.

{\it Case 2.8.1.} $c_1\neq a_1$, $c_1\neq b_1$ and $c_2\neq d_2$.

There are $n$ ways to choose $a_1(=d_2)$ from $N_n$, $n-1$ ways to choose $a_2(=d_1=b_1)$ from $N_n\setminus\{a_1\}$, $n-2$ ways to choose $b_2$ from $N_n\setminus\{d_1, d_2\}$, $n-2$ ways to choose $c_1$ from $N_n\setminus\{a_1, b_1\}$, $n-2$ ways to choose $c_2$ from $N_n\setminus\{c_1, d_2\}$.
Therefore, there are $n(n-1)(n-2)^3=n^5-7n^4+18n^3-20n^2+8n$ ways to choose four $B_{n-2}$'s as required.
The probability $P_{2.8.1}$ that there are four fault-free $B_{n-2}$'s chosen as in Case 2.8.1 is $\frac{1}{2}(n^5-7n^4+18n^3-20n^2+8n)p^{4(n-2)!}$.

{\it Case 2.8.2.} $c_1=a_1\neq b_1$, $c_2\neq d_2$ and $a_2\neq c_2$.

There are $n$ ways to choose $a_1(=d_2=b_2=c_1)$ from $N_n$, $n-1$ ways to choose $a_2(=d_1)$ from $N_n\setminus\{a_1\}$, $n-2$ ways to choose $b_1$ from $N_n\setminus\{d_1, d_2\}$, $n-2$ ways to choose $c_2$ from $N_n\setminus\{d_2, a_2\}$.
Therefore, there are $n(n-1)(n-2)^2=n^4-5n^3+8n^2-4n$ ways to choose four $B_{n-2}$'s as required.
The probability $P_{2.8.2}$ that there are four fault-free $B_{n-2}$'s chosen as in Case 2.8.2 is $\frac{1}{2}(n^4-5n^3+8n^2-4n)p^{4(n-2)!-(n-3)!}$.

{\it Case 2.8.3.} $c_1=a_1\neq b_1$, $c_2\neq d_2$ and $a_2=c_2$.

There are $n$ ways to choose $a_1(=d_2=b_2=c_1)$ from $N_n$, $n-1$ ways to choose $a_2(=d_1=c_2)$ from $N_n\setminus\{a_1\}$, $n-2$ ways to choose $b_1$ from $N_n\setminus\{d_1, d_2\}$.
Therefore, there are $n(n-1)(n-2)=n^3-3n^2+2n$ ways to choose four $B_{n-2}$'s as required.
The probability $P_{2.8.3}$ that there are four fault-free $B_{n-2}$'s chosen as in Case 2.8.3 is $\frac{1}{2}(n^3-3n^2+2n)p^{4(n-2)!}$.

{\it Case 2.8.4.} $c_1=b_1\neq a_1$ and $c_2\neq d_2$.

There are $n$ ways to choose $a_1(=d_2)$ from $N_n$, $n-1$ ways to choose $a_2(=d_1=b_1=c_1)$ from $N_n\setminus\{a_1\}$, $n-2$ ways to choose $b_2$ from $N_n\setminus\{d_1, d_2\}$, $n-2$ ways to choose $c_2$ from $N_n\setminus\{d_2, c_1\}$.
Therefore, there are $n(n-1)(n-2)^2=n^4-5n^3+8n^2-4n$ ways to choose four $B_{n-2}$'s as required.
Clearly, in this case $c_2\neq b_2$.
The probability $P_{2.8.4}$ that there are four fault-free $B_{n-2}$'s chosen as in Case 2.8.4 is $\frac{1}{2}(n^4-5n^3+8n^2-4n)p^{4(n-2)!-(n-3)!}$.

{\it Case 2.8.5.} $c_1\neq a_1$, $c_1\neq b_1$, $c_2=d_2$ and $c_1\neq d_1$.

There are $n$ ways to choose $a_1(=d_2=b_2=c_2)$ from $N_n$, $n-1$ ways to choose $a_2(=d_1)$ from $N_n\setminus\{a_1\}$, $n-2$ ways to choose $b_1$ from $N_n\setminus\{d_1, d_2\}$, $n-3$ ways to choose $c_1$ from $N_n\setminus\{a_1, b_1, d_1\}$.
Therefore, there are $n(n-1)(n-2)(n-3)=n^4-6n^3+11n^2-6n$ ways to choose four $B_{n-2}$'s as required.
The probability $P_{2.8.5}$ that there are four fault-free $B_{n-2}$'s chosen as in Case 2.8.5 is $\frac{1}{2}(n^4-6n^3+11n^2-6n)p^{4(n-2)!-(n-3)!}$.

{\it Case 2.8.6.} $c_1\neq a_1$, $c_1\neq b_1$, $c_2=d_2$ and $c_1=d_1$.

There are $n$ ways to choose $a_1(=d_2=b_2=c_2)$ from $N_n$, $n-1$ ways to choose $a_2(=d_1=c_1)$ from $N_n\setminus\{a_1\}$, $n-2$ ways to choose $b_1$ from $N_n\setminus\{d_1, d_2\}$.
Therefore, there are $n(n-1)(n-2)=n^3-3n^2+2n$ ways to choose four $B_{n-2}$'s as required.
The probability $P_{2.8.6}$ that there are four fault-free $B_{n-2}$'s chosen as in Case 2.8.6 is $\frac{1}{2}(n^3-3n^2+2n)p^{4(n-2)!}$.

{\it Case 2.8.7.} $c_1=b_1\neq a_1$ and $c_2=d_2$.

There are $n$ ways to choose $a_1(=d_2=c_2)$ from $N_n$, $n-1$ ways to choose $a_2(=d_1=b_1=c_1)$ from $N_n\setminus\{a_1\}$, $n-2$ ways to choose $b_2$ from $N_n\setminus\{d_1, d_2\}$.
Therefore, there are $n(n-1)(n-2)=n^3-3n^2+2n$ ways to choose four $B_{n-2}$'s as required.
Clearly, in this case $c_2=b_2$ and $c_1\neq d_1$.
The probability $P_{2.8.7}$ that there are four fault-free $B_{n-2}$'s chosen as in Case 2.8.7 is $\frac{1}{2}(n^3-3n^2+2n)p^{4(n-2)!-(n-3)!}$.

Thus, the probability $P_{2.8}$ that there are four fault-free $B_{n-2}$'s chosen as in Case 2.8 is $\sum_{i=1}^{7}P_{2.8.i}=\frac{1}{2}(n^5-7n^4+20n^3-26n^2+12n)p^{4(n-2)!}+\frac{3}{2}(n^4-5n^3+8n^2-4n)p^{4(n-2)!-(n-3)!}$.

In summary, the probability $P_2$ that there are four fault-free $B_{n-2}$'s chosen as in Case 2 is $2\sum_{i=1}^{8}P_{2.i}=2(4n^5-27n^4+75n^3-96n^2+44n)p^{4(n-2)!}+20(n^4-5n^3+8n^2-4n)p^{4(n-2)!-(n-3)!}+2(n^4-5n^3+8n^2-4n)p^{4(n-2)!-2(n-3)!}$.

{\it Case 3.} $\{a_1, a_2\}=\{d_1, d_2\}$, $\{b_1, b_2\}\cap\{d_1, d_2\}=\emptyset$ or $\{a_1, a_2\}\cap\{d_1, d_2\}=\emptyset$, $\{b_1, b_2\}=\{d_1, d_2\}$.

W.l.o.g, assume that $\{a_1, a_2\}=\{d_1, d_2\}$ and $\{b_1, b_2\}\cap\{d_1, d_2\}=\emptyset$.

Observation \ref{ob:1} implies that $a_1 a_2 X^{n-2}$ and $X^{n-2} d_1 d_2$ are disjoint, and $b_1 b_2 X^{n-2}$ and $X^{n-2} d_1 d_2$ are not disjoint and $V(b_1 b_2 X^{n-2})\cap V(X^{n-2} d_1 d_2)=V(b_1 b_2 X^{n-4} d_1 d_2)$.

{\it Case 3.1.} $a_1=d_1$, $a_2=d_2$ and $\{b_1, b_2\}\cap\{d_1, d_2\}=\emptyset$.

{\it Case 3.1.1.} $c_1\neq a_1$, $c_1\neq b_1$, $c_2\neq d_2$ and $c_1\neq d_2$.

There are $n$ ways to choose $a_1(=d_1)$ from $N_n$, $n-1$ ways to choose $a_2(=d_2)$ from $N_n\setminus\{a_1\}$, $n-2$ ways to choose $b_1$ from $N_n\setminus\{d_1, d_2\}$, $n-3$ ways to choose $b_2$ from $N_n\setminus\{d_1, d_2, b_1\}$, $n-3$ ways to choose $c_1$ from $N_n\setminus\{a_1, b_1, d_2\}$, $n-2$ ways to choose $c_2$ from $N_n\setminus\{c_1, d_2\}$.
Therefore, there are $n(n-1)(n-2)^2(n-3)^2=n^6-11n^5+47n^4-97n^3+96n^2-36n$ ways to choose four $B_{n-2}$'s as required.
The probability $P_{3.1.1}$ that there are four fault-free $B_{n-2}$'s chosen as in Case 3.1.1 is $\frac{1}{2}(n^6-11n^5+47n^4-97n^3+96n^2-36n)p^{4(n-2)!-(n-4)!}$.

{\it Case 3.1.2.} $c_1\neq a_1$, $c_1\neq b_1$, $c_2\neq d_2$ and $c_1=d_2$.

There are $n$ ways to choose $a_1(=d_1)$ from $N_n$, $n-1$ ways to choose $a_2(=d_2=c_1)$ from $N_n\setminus\{a_1\}$, $n-2$ ways to choose $b_1$ from $N_n\setminus\{d_1, d_2\}$, $n-3$ ways to choose $b_2$ from $N_n\setminus\{d_1, d_2, b_1\}$, $n-1$ ways to choose $c_2$ from $N_n\setminus\{d_2\}$.
Therefore, there are $n(n-1)^2(n-2)(n-3)=n^5-7n^4+17n^3-17n^2+6n$ ways to choose four $B_{n-2}$'s as required.
The probability $P_{3.1.2}$ that there are four fault-free $B_{n-2}$'s chosen as in Case 3.1.2 is $\frac{1}{2}(n^5-7n^4+17n^3-17n^2+6n)p^{4(n-2)!-(n-4)!}$.

{\it Case 3.1.3.} $c_1=a_1\neq b_1$ and $c_2\neq d_2$.

There are $n$ ways to choose $a_1(=d_1=c_1)$ from $N_n$, $n-1$ ways to choose $a_2(=d_2)$ from $N_n\setminus\{a_1\}$, $n-2$ ways to choose $b_1$ from $N_n\setminus\{d_1, d_2\}$, $n-3$ ways to choose $b_2$ from $N_n\setminus\{d_1, d_2, b_1\}$, $n-2$ ways to choose $c_2$ from $N_n\setminus\{d_2, c_1\}$.
Therefore, there are $n(n-1)(n-2)^2(n-3)=n^5-8n^4+23n^3-28n^2+12n$ ways to choose four $B_{n-2}$'s as required.
Clearly, in this case $c_2\neq a_2$.
The probability $P_{3.1.3}$ that there are four fault-free $B_{n-2}$'s chosen as in Case 3.1.3 is $\frac{1}{2}(n^5-8n^4+23n^3-28n^2+12n)p^{4(n-2)!-(n-3)!-(n-4)!}$.

{\it Case 3.1.4.} $c_1=b_1\neq a_1$, $c_2\neq d_2$ and $c_2\neq b_2$.

There are $n$ ways to choose $a_1(=d_1)$ from $N_n$, $n-1$ ways to choose $a_2(=d_2)$ from $N_n\setminus\{a_1\}$, $n-2$ ways to choose $b_1(=c_1)$ from $N_n\setminus\{d_1, d_2\}$, $n-3$ ways to choose $b_2$ from $N_n\setminus\{d_1, d_2, b_1\}$, $n-3$ ways to choose $c_2$ from $N_n\setminus\{b_2, d_2, c_1\}$.
Therefore, there are $n(n-1)(n-2)(n-3)^2=n^5-9n^4+29n^3-39n^2+18n$ ways to choose four $B_{n-2}$'s as required.
The probability $P_{3.1.4}$ that there are four fault-free $B_{n-2}$'s chosen as in Case 3.1.4 is $\frac{1}{2}(n^5-9n^4+29n^3-39n^2+18n)p^{4(n-2)!-(n-3)!-(n-4)!}$.

{\it Case 3.1.5.} $c_1=b_1\neq a_1$, $c_2\neq d_2$ and $c_2=b_2$.

There are $n$ ways to choose $a_1(=d_1)$ from $N_n$, $n-1$ ways to choose $a_2(=d_2)$ from $N_n\setminus\{a_1\}$, $n-2$ ways to choose $b_1(=c_1)$ from $N_n\setminus\{d_1, d_2\}$, $n-3$ ways to choose $b_2(=c_2)$ from $N_n\setminus\{d_1, d_2, b_1\}$.
Therefore, there are $n(n-1)(n-2)(n-3)=n^4-6n^3+11n^2-6n$ ways to choose four $B_{n-2}$'s as required.
The probability $P_{3.1.5}$ that there are four fault-free $B_{n-2}$'s chosen as in Case 3.1.5 is $\frac{1}{2}(n^4-6n^3+11n^2-6n)p^{4(n-2)!-(n-4)!}$.

{\it Case 3.1.6.} $c_1\neq a_1$, $c_1\neq b_1$ and $c_2=d_2$.

There are $n$ ways to choose $a_1(=d_1)$ from $N_n$, $n-1$ ways to choose $a_2(=d_2=c_2)$ from $N_n\setminus\{a_1\}$, $n-2$ ways to choose $b_1$ from $N_n\setminus\{d_1, d_2\}$, $n-3$ ways to choose $b_2$ from $N_n\setminus\{d_1, d_2, b_1\}$, $n-3$ ways to choose $c_1$ from $N_n\setminus\{a_1, b_1, c_2\}$.
Therefore, there are $n(n-1)(n-2)(n-3)^2=n^5-9n^4+29n^3-39n^2+18n$ ways to choose four $B_{n-2}$'s as required.
Clearly, in this case $c_1\neq d_1$.
The probability $P_{3.1.6}$ that there are four fault-free $B_{n-2}$'s chosen as in Case 3.1.6 is $\frac{1}{2}(n^5-9n^4+29n^3-39n^2+18n)p^{4(n-2)!-(n-3)!-(n-4)!}$.

{\it Case 3.1.7.} $c_1=a_1\neq b_1$ and $c_2=d_2$.

There are $n$ ways to choose $a_1(=d_1=c_1)$ from $N_n$, $n-1$ ways to choose $a_2(=d_2=c_2)$ from $N_n\setminus\{a_1\}$, $n-2$ ways to choose $b_1$ from $N_n\setminus\{d_1, d_2\}$, $n-3$ ways to choose $b_2$ from $N_n\setminus\{d_1, d_2, b_1\}$.
Therefore, there are $n(n-1)(n-2)(n-3)=n^4-6n^3+11n^2-6n$ ways to choose four $B_{n-2}$'s as required.
Clearly, in this case $c_2=a_2$ and $c_1=d_1$.
The probability $P_{3.1.7}$ that there are four fault-free $B_{n-2}$'s chosen as in Case 3.1.7 is $\frac{1}{2}(n^4-6n^3+11n^2-6n)p^{4(n-2)!-(n-4)!}$.

{\it Case 3.1.8.} $c_1=b_1\neq a_1$ and $c_2=d_2$.

There are $n$ ways to choose $a_1(=d_1)$ from $N_n$, $n-1$ ways to choose $a_2(=d_2=c_2)$ from $N_n\setminus\{a_1\}$, $n-2$ ways to choose $b_1(=c_1)$ from $N_n\setminus\{d_1, d_2\}$, $n-3$ ways to choose $b_2$ from $N_n\setminus\{d_1, d_2, b_1\}$.
Therefore, there are $n(n-1)(n-2)(n-3)=n^4-6n^3+11n^2-6n$ ways to choose four $B_{n-2}$'s as required.
Clearly, in this case $c_2\neq b_2$ and $c_1\neq d_1$.
The probability $P_{3.1.8}$ that there are four fault-free $B_{n-2}$'s chosen as in Case 3.1.8 is $\frac{1}{2}(n^4-6n^3+11n^2-6n)p^{4(n-2)!-2(n-3)!-(n-4)!}$.

{\it Case 3.2.} $a_1=d_2$, $a_2=d_1$ and $\{b_1, b_2\}\cap\{d_1, d_2\}=\emptyset$.

{\it Case 3.2.1.} $c_1\neq a_1$, $c_1\neq b_1$ and $c_2\neq d_2$.

There are $n$ ways to choose $a_1(=d_2)$ from $N_n$, $n-1$ ways to choose $a_2(=d_1)$ from $N_n\setminus\{a_1\}$, $n-2$ ways to choose $b_1$ from $N_n\setminus\{d_1, d_2\}$, $n-3$ ways to choose $b_2$ from $N_n\setminus\{d_1, d_2, b_1\}$, $n-2$ ways to choose $c_1$ from $N_n\setminus\{a_1, b_1\}$, $n-2$ ways to choose $c_2$ from $N_n\setminus\{c_1, d_2\}$.
Therefore, there are $n(n-1)(n-2)^3(n-3)=n^6-10n^5+39n^4-74n^3+68n^2-24n$ ways to choose four $B_{n-2}$'s as required.
The probability $P_{3.2.1}$ that there are four fault-free $B_{n-2}$'s chosen as in Case 3.2.1 is $\frac{1}{2}(n^6-10n^5+39n^4-74n^3+68n^2-24n)p^{4(n-2)!-(n-4)!}$.

{\it Case 3.2.2.} $c_1=a_1\neq b_1$, $c_2\neq d_2$ and $c_2\neq a_2$.

There are $n$ ways to choose $a_1(=d_2=c_1)$ from $N_n$, $n-1$ ways to choose $a_2(=d_1)$ from $N_n\setminus\{a_1\}$, $n-2$ ways to choose $b_1$ from $N_n\setminus\{d_1, d_2\}$, $n-3$ ways to choose $b_2$ from $N_n\setminus\{d_1, d_2, b_1\}$, $n-2$ ways to choose $c_2$ from $N_n\setminus\{a_2, d_2\}$.
Therefore, there are $n(n-1)(n-2)^2(n-3)=n^5-8n^4+23n^3-28n^2+12n$ ways to choose four $B_{n-2}$'s as required.
The probability $P_{3.2.2}$ that there are four fault-free $B_{n-2}$'s chosen as in Case 3.2.2 is $\frac{1}{2}(n^5-8n^4+23n^3-28n^2+12n)p^{4(n-2)!-(n-3)!-(n-4)!}$.

{\it Case 3.2.3.} $c_1=a_1\neq b_1$, $c_2\neq d_2$ and $c_2=a_2$.

There are $n$ ways to choose $a_1(=d_2=c_1)$ from $N_n$, $n-1$ ways to choose $a_2(=d_1=c_2)$ from $N_n\setminus\{a_1\}$, $n-2$ ways to choose $b_1$ from $N_n\setminus\{d_1, d_2\}$, $n-3$ ways to choose $b_2$ from $N_n\setminus\{d_1, d_2, b_1\}$.
Therefore, there are $n(n-1)(n-2)(n-3)=n^4-6n^3+11n^2-6n$ ways to choose four $B_{n-2}$'s as required.
The probability $P_{3.2.3}$ that there are four fault-free $B_{n-2}$'s chosen as in Case 3.2.3 is $\frac{1}{2}(n^4-6n^3+11n^2-6n)p^{4(n-2)!-(n-4)!}$.

{\it Case 3.2.4.} $c_1=b_1\neq a_1$, $c_2\neq d_2$ and $c_2\neq b_2$.

There are $n$ ways to choose $a_1(=d_2)$ from $N_n$, $n-1$ ways to choose $a_2(=d_1)$ from $N_n\setminus\{a_1\}$, $n-2$ ways to choose $b_1(=c_1)$ from $N_n\setminus\{d_1, d_2\}$, $n-3$ ways to choose $b_2$ from $N_n\setminus\{d_1, d_2, b_1\}$, $n-3$ ways to choose $c_2$ from $N_n\setminus\{b_2, d_2, c_1\}$.
Therefore, there are $n(n-1)(n-2)(n-3)^2=n^5-9n^4+29n^3-39n^2+18n$ ways to choose four $B_{n-2}$'s as required.
The probability $P_{3.2.4}$ that there are four fault-free $B_{n-2}$'s chosen as in Case 3.2.4 is $\frac{1}{2}(n^5-9n^4+29n^3-39n^2+18n)p^{4(n-2)!-(n-3)!-(n-4)!}$.

{\it Case 3.2.5.} $c_1=b_1\neq a_1$, $c_2\neq d_2$ and $c_2\neq b_2$.

There are $n$ ways to choose $a_1(=d_2)$ from $N_n$, $n-1$ ways to choose $a_2(=d_1)$ from $N_n\setminus\{a_1\}$, $n-2$ ways to choose $b_1(=c_1)$ from $N_n\setminus\{d_1, d_2\}$, $n-3$ ways to choose $b_2(=c_2)$ from $N_n\setminus\{d_1, d_2, b_1\}$.
Therefore, there are $n(n-1)(n-2)(n-3)=n^4-6n^3+11n^2-6n$ ways to choose four $B_{n-2}$'s as required.
The probability $P_{3.2.5}$ that there are four fault-free $B_{n-2}$'s chosen as in Case 3.2.5 is $\frac{1}{2}(n^4-6n^3+11n^2-6n)p^{4(n-2)!-(n-4)!}$.

{\it Case 3.2.6.} $c_1\neq a_1$, $c_1\neq b_1$, $c_2=d_2$ and $c_1\neq d_1$.

There are $n$ ways to choose $a_1(=d_2=c_2)$ from $N_n$, $n-1$ ways to choose $a_2(=d_1)$ from $N_n\setminus\{a_1\}$, $n-2$ ways to choose $b_1$ from $N_n\setminus\{d_1, d_2\}$, $n-3$ ways to choose $b_2$ from $N_n\setminus\{d_1, d_2, b_1\}$, $n-3$ ways to choose $c_1$ from $N_n\setminus\{a_1, b_1, d_1\}$.
Therefore, there are $n(n-1)(n-2)(n-3)^2=n^5-9n^4+29n^3-39n^2+18n$ ways to choose four $B_{n-2}$'s as required.
The probability $P_{3.2.6}$ that there are four fault-free $B_{n-2}$'s chosen as in Case 3.2.6 is $\frac{1}{2}(n^5-9n^4+29n^3-39n^2+18n)p^{4(n-2)!-(n-3)!-(n-4)!}$.

{\it Case 3.2.7.} $c_1\neq a_1$, $c_1\neq b_1$, $c_2=d_2$ and $c_1\neq d_1$.

There are $n$ ways to choose $a_1(=d_2=c_2)$ from $N_n$, $n-1$ ways to choose $a_2(=d_1=c_1)$ from $N_n\setminus\{a_1\}$, $n-2$ ways to choose $b_1$ from $N_n\setminus\{d_1, d_2\}$, $n-3$ ways to choose $b_2$ from $N_n\setminus\{d_1, d_2, b_1\}$.
Therefore, there are $n(n-1)(n-2)(n-3)=n^4-6n^3+11n^2-6n$ ways to choose four $B_{n-2}$'s as required.
The probability $P_{3.2.7}$ that there are four fault-free $B_{n-2}$'s chosen as in Case 3.2.7 is $\frac{1}{2}(n^4-6n^3+11n^2-6n)p^{4(n-2)!-(n-4)!}$.

{\it Case 3.2.8.} $c_1=b_1\neq a_1$ and $c_2=d_2$.

There are $n$ ways to choose $a_1(=d_2=c_2)$ from $N_n$, $n-1$ ways to choose $a_2(=d_1)$ from $N_n\setminus\{a_1\}$, $n-2$ ways to choose $b_1(=c_1)$ from $N_n\setminus\{d_1, d_2\}$, $n-3$ ways to choose $b_2$ from $N_n\setminus\{d_1, d_2, b_1\}$.
Therefore, there are $n(n-1)(n-2)(n-3)=n^4-6n^3+11n^2-6n$ ways to choose four $B_{n-2}$'s as required.
Clearly, in this case $c_2\neq a_2$ and $c_1\neq d_1$.
The probability $P_{3.2.8}$ that there are four fault-free $B_{n-2}$'s chosen as in Case 3.2.8 is $\frac{1}{2}(n^4-6n^3+11n^2-6n)p^{4(n-2)!-2(n-3)!-(n-4)!}$.

In summary, the probability $P_3$ that there are four fault-free $B_{n-2}$'s chosen as in Case 3 is $2(\sum_{i=1}^{8}P_{3.1.i}+\sum_{i=1}^{8}P_{3.2.i})=2(n^6-10n^5+42n^4-92n^3+101n^2-42n)p^{4(n-2)!-(n-4)!}+2(3n^5-26n^4+81n^3-106n^2+48n)p^{4(n-2)!-(n-3)!-(n-4)!}+2(n^4-6n^3+11n^2-6n)p^{4(n-2)!-2(n-3)!-(n-4)!}$.

{\it Case 4.} $|\{a_1, a_2\}\cap\{d_1, d_2\}|=1$ and $|\{b_1, b_2\}\cap\{d_1, d_2\}|=1$.

{\it Case 4.1.} $a_1=b_1=d_1$, $a_2\notin\{d_1, d_2\}$ and $b_2\notin\{d_1, d_2\}$.

{\it Case 4.1.1.} $c_1\neq a_1$, $c_1\neq b_1$, $c_2\neq d_2$ and $c_1\neq d_2$.

There are $n$ ways to choose $a_1(=b_1=d_1)$ from $N_n$, $n-1$ ways to choose $a_2$ from $N_n\setminus\{a_1\}$, $n-2$ ways to choose $b_2$ from $N_n\setminus\{b_1, a_2\}$, $n-3$ ways to choose $d_2$ from $N_n\setminus\{d_1, a_2, b_2\}$, $n-2$ ways to choose $c_1$ from $N_n\setminus\{a_1, d_2\}$, $n-2$ ways to choose $c_2$ from $N_n\setminus\{c_1, d_2\}$. Therefore, there are $n(n-1)(n-2)^3(n-3)=n^6-10n^5+39n^4-74n^3+68n^2-24n$ ways to choose four $B_{n-2}$'s as required. The probability $P_{4.1.1}$ that there are four fault-free $B_{n-2}$'s chosen as in Case 4.1.1 is $\frac{1}{2}(n^6-10n^5+39n^4-74n^3+68n^2-24n)p^{4(n-2)!}$.

{\it Case 4.1.2.} $c_1\neq a_1$, $c_1\neq b_1$, $c_2\neq d_2$ and $c_1=d_2$.

There are $n$ ways to choose $a_1(=b_1=d_1)$ from $N_n$, $n-1$ ways to choose $a_2$ from $N_n\setminus\{a_1\}$, $n-2$ ways to choose $b_2$ from $N_n\setminus\{b_1, a_2\}$, $n-3$ ways to choose $c_1(=d_2)$ from $N_n\setminus\{d_1, a_2, b_2\}$, $n-1$ ways to choose $c_2$ from $N_n\setminus\{d_2\}$. Therefore, there are $n(n-1)^2(n-2)(n-3)=n^5-7n^4+17n^3-17n^2+6n$ ways to choose four $B_{n-2}$'s as required. The probability $P_{4.1.2}$ that there are four fault-free $B_{n-2}$'s chosen as in Case 4.1.2 is $\frac{1}{2}(n^5-7n^4+17n^3-17n^2+6n)p^{4(n-2)!}$.

{\it Case 4.1.3.} $c_1=a_1=b_1$, $c_2\neq d_2$ and $c_2\notin\{a_2, b_2\}$.

There are $n$ ways to choose $a_1(=b_1=c_1=d_1)$ from $N_n$, $n-1$ ways to choose $d_2$ from $N_n\setminus\{d_1\}$, $n-2$ ways to choose $a_2$ from $N_n\setminus\{d_1, d_2\}$, $n-3$ ways to choose $b_2$ from $N_n\setminus\{d_1, d_2, a_2\}$, $n-4$ ways to choose $c_2$ from $N_n\setminus\{c_1, d_2, a_2, b_2\}$.
Therefore, there are $n(n-1)(n-2)(n-3)(n-4)=n^5-10n^4+35n^3-50n^2+24n$ ways to choose four $B_{n-2}$'s as required.
The probability $P_{4.1.3}$ that there are four fault-free $B_{n-2}$'s chosen as in Case 4.1.3 is $\frac{1}{2}(n^5-10n^4+35n^3-50n^2+24n)p^{4(n-2)!-2(n-3)!}$.

{\it Case 4.1.4.} $c_1=a_1=b_1$, $c_2\neq d_2$ and $c_2\in\{a_2, b_2\}$.

Clearly, in this case $a_2\neq b_2$. There are 2 possible cases such taht $c_2\in\{a_2, b_2\}$, w.l.o.g, assume that $c_2=a_2\neq b_2$.

There are $n$ ways to choose $a_1(=b_1=c_1=d_1)$ from $N_n$, $n-1$ ways to choose $d_2$ from $N_n\setminus\{d_1\}$, $n-2$ ways to choose $a_2(=c_2)$ from $N_n\setminus\{d_1, d_2\}$, $n-3$ ways to choose $b_2$ from $N_n\setminus\{d_1, d_2, a_2\}$.
Therefore, there are $2n(n-1)(n-2)(n-3)=2(n^4-6n^3+11n^2-6n)$ ways to choose four $B_{n-2}$'s as required.
The probability $P_{4.1.4}$ that there are four fault-free $B_{n-2}$'s chosen as in Case 4.1.4 is $\frac{1}{2}\times 2(n^4-6n^3+11n^2-6n)p^{4(n-2)!-(n-3)!}$.

{\it Case 4.1.5.} $c_1\neq a_1$, $c_1\neq b_1$ and $c_2=d_2$.

There are $n$ ways to choose $a_1(=b_1=d_1)$ from $N_n$, $n-1$ ways to choose $a_2$ from $N_n\setminus\{a_1\}$, $n-2$ ways to choose $b_2$ from $N_n\setminus\{b_1, a_2\}$, $n-3$ ways to choose $c_2(=d_2)$ from $N_n\setminus\{d_1, a_2, b_2\}$, $n-2$ ways to choose $c_1$ from $N_n\setminus\{a_1, c_2\}$.
Therefore, there are $n(n-1)(n-2)^2(n-3)=n^5-8n^4+23n^3-28n^2+12n$ ways to choose four $B_{n-2}$'s as required.
Clearly, in this case $c_1\neq d_1$.
The probability $P_{4.1.5}$ that there are four fault-free $B_{n-2}$'s chosen as in Case 4.1.5 is $\frac{1}{2}(n^5-8n^4+23n^3-28n^2+12n)p^{4(n-2)!-(n-3)!}$.

{\it Case 4.1.6.} $c_1=a_1=b_1$ and $c_2=d_2$.

There are $n$ ways to choose $a_1(=b_1=c_1=d_1)$ from $N_n$, $n-1$ ways to choose $a_2$ from $N_n\setminus\{a_1\}$, $n-2$ ways to choose $b_2$ from $N_n\setminus\{b_1, a_2\}$, $n-3$ ways to choose $c_2(=d_2)$ from $N_n\setminus\{d_1, a_2, b_2\}$.
Therefore, there are $n(n-1)(n-2)(n-3)=n^4-6n^3+11n^2-6n$ ways to choose four $B_{n-2}$'s as required.
Clearly, in this case $c_2\neq a_2$, $c_2\neq b_2$ and $c_1=d_1$.
The probability $P_{4.1.6}$ that there are four fault-free $B_{n-2}$'s chosen as in Case 4.1.6 is $\frac{1}{2}(n^4-6n^3+11n^2-6n)p^{4(n-2)!-2(n-3)!}$.

Thus, the probability $P_{4.1}$ that there are four fault-free $B_{n-2}$'s chosen as in Case 4.1 is $\sum_{i=1}^{6}P_{4.1.i}=\frac{1}{2}(n^6-9n^5+32n^4-57n^3+51n^2-18n)p^{4(n-2)!}+\frac{1}{2}(n^5-5n^4+5n^3+5n^2-6n)p^{4(n-2)!-(n-3)!}+\frac{1}{2}(n^5-10n^4+35n^3-50n^2+24n)p^{4(n-2)!-2(n-3)!}$.

{\it Case 4.2.} $a_1=b_1=d_2$, $a_2\notin\{d_1, d_2\}$ and $b_2\notin\{d_1, d_2\}$.

{\it Case 4.2.1.} $c_1\neq a_1$, $c_1\neq b_1$ and $c_2\neq d_2$.

There are $n$ ways to choose $a_1(=b_1=d_2)$ from $N_n$,$n-1$ ways to choose $a_2$ from $N_n\setminus\{a_1\}$, $n-2$ ways to choose $b_2$ from $N_n\setminus\{b_1, a_2\}$, $n-3$ ways to choose $d_1$ from $N_n\setminus\{d_2, a_2, b_2\}$, $n-1$ ways to choose $c_1$ from $N_n\setminus\{a_1\}$, $n-2$ ways to choose $c_2$ from $N_n\setminus\{c_1, d_2\}$.
Therefore, there are $n(n-1)^2(n-2)^2(n-3)=n^6-9n^5+31n^4-51n^3+40n^2-12n$ ways to choose four $B_{n-2}$'s as required.
The probability $P_{4.2.1}$ that there are four fault-free $B_{n-2}$'s chosen as in Case 4.2.1 is $\frac{1}{2}(n^6-9n^5+31n^4-51n^3+40n^2-12n)p^{4(n-2)!}$.

{\it Case 4.2.2.} $c_1=a_1=b_1$, $c_2\neq d_2$ and $c_2\notin\{a_2, b_2\}$.

There are $n$ ways to choose $a_1(=b_1=d_2=c_1)$ from $N_n$, $n-1$ ways to choose $d_2$ from $N_n\setminus\{d_1\}$, $n-2$ ways to choose $a_2$ from $N_n\setminus\{d_1, d_2\}$, $n-3$ ways to choose $b_2$ from $N_n\setminus\{d_1, d_2, a_2\}$, $n-1$ ways to choose $c_2$ from $N_n\setminus\{d_2, a_2, b_2\}$.
Therefore, there are $n(n-1)(n-2)(n-3)^2=n^5-9n^4+29n^3-39n^2+18n$ ways to choose four $B_{n-2}$'s as required.
The probability $P_{4.2.2}$ that there are four fault-free $B_{n-2}$'s chosen as in Case 4.2.2 is $\frac{1}{2}(n^5-9n^4+29n^3-39n^2+18n)p^{4(n-2)!-2(n-3)!}$.

{\it Case 4.2.3.} $c_1=a_1=b_1$, $c_2\neq d_2$ and $c_2\in\{a_2, b_2\}$.

Clearly, in this case $a_2\neq b_2$. There are 2 possible cases such taht $c_2\in\{a_2, b_2\}$, w.l.o.g, assume that $c_2=a_2\neq b_2$.

There are $n$ ways to choose $a_1(=b_1=d_2=c_1)$ from $N_n$, $n-1$ ways to choose $d_2$ from $N_n\setminus\{d_1\}$, $n-2$ ways to choose $a_2(=c_2)$ from $N_n\setminus\{d_1, d_2\}$, $n-3$ ways to choose $b_2$ from $N_n\setminus\{d_1, d_2, a_2\}$.
Therefore, there are $2n(n-1)(n-2)(n-3)=2(n^4-6n^3+11n^2-6n)$ ways to choose four $B_{n-2}$'s as required.
The probability $P_{4.2.3}$ that there are four fault-free $B_{n-2}$'s chosen as in Case 4.2.3 is $\frac{1}{2}\times 2(n^4-6n^3+11n^2-6n)p^{4(n-2)!-(n-3)!}$.

{\it Case 4.2.4.} $c_1\neq a_1$, $c_1\neq b_1$, $c_2=d_2$ and $c_1\neq d_1$.

There are $n$ ways to choose $a_1(=b_1=d_2=c_2)$ from $N_n$, $n-1$ ways to choose $d_1$ from $N_n\setminus\{d_2\}$, $n-2$ ways to choose $a_2$ from $N_n\setminus\{d_1, d_2\}$, $n-3$ ways to choose $b_2$ from $N_n\setminus\{d_1, d_2, a_2\}$, $n-2$ ways to choose $c_1$ from $N_n\setminus\{a_1, d_1\}$.
Therefore, there are $n(n-1)(n-2)^2(n-3)=n^5-8n^4+23n^3-28n^2+12n$ ways to choose four $B_{n-2}$'s as required.
The probability $P_{4.2.4}$ that there are four fault-free $B_{n-2}$'s chosen as in Case 4.2.4 is $\frac{1}{2}(n^5-8n^4+23n^3-28n^2+12n)p^{4(n-2)!-(n-3)!}$.

{\it Case 4.2.5.} $c_1\neq a_1$, $c_1\neq b_1$, $c_2=d_2$ and $c_1=d_1$.

There are $n$ ways to choose $a_1(=b_1=d_2=c_2)$ from $N_n$, $n-1$ ways to choose $d_1(=c_1)$ from $N_n\setminus\{d_2\}$, $n-2$ ways to choose $a_2$ from $N_n\setminus\{d_1, d_2\}$, $n-3$ ways to choose $b_2$ from $N_n\setminus\{d_1, d_2, a_2\}$.
Therefore, there are $n(n-1)(n-2)(n-3)=n^4-6n^3+11n^2-6n$ ways to choose four $B_{n-2}$'s as required.
The probability $P_{4.2.5}$ that there are four fault-free $B_{n-2}$'s chosen as in Case 4.2.5 is $\frac{1}{2}(n^4-6n^3+11n^2-6n)p^{4(n-2)!}$.

Thus, the probability $P_{4.2}$ that there are four fault-free $B_{n-2}$'s chosen as in Case 4.2 is $\sum_{i=1}^{5}P_{4.2.i}=\frac{1}{2}(n^6-9n^5+32n^4-57n^3+51n^2-18n)p^{4(n-2)!}+\frac{1}{2}(n^5-6n^4+11n^3-6n^2)p^{4(n-2)!-(n-3)!}+\frac{1}{2}(n^5-9n^4+29n^3-39n^2+18n)p^{4(n-2)!-2(n-3)!}$.

{\it Case 4.3.} $a_2=b_2=d_1$, $a_1\notin\{d_1, d_2\}$, $b_1\notin\{d_1, d_2\}$.

{\it Case 4.3.1.} $c_1\neq a_1$, $c_1\neq b_1$, $c_2\neq d_2$ and $c_1\neq d_2$.

There are $n$ ways to choose $a_2(=b_2=d_1)$ from $N_n$, $n-1$ ways to choose $a_1$ from $N_n\setminus\{a_2\}$, $n-2$ ways to choose $b_1$ from $N_n\setminus\{b_2, a_1\}$, $n-3$ ways to choose $d_2$ from $N_n\setminus\{d_1, a_1, b_1\}$, $n-3$ ways to choose $c_1$ from $N_n\setminus\{a_1, b_1, d_2\}$, $n-2$ ways to choose $c_2$ from $N_n\setminus\{c_1, d_2\}$.
Therefore, there are $n(n-1)(n-2)^2(n-3)^2=n^6-11n^5+47n^4-97n^3+96n^2-36n$ ways to choose four $B_{n-2}$'s as required.
The probability $P_{4.3.1}$ that there are four fault-free $B_{n-2}$'s chosen as in Case 4.3.1 is $\frac{1}{2}(n^6-11n^5+47n^4-97n^3+96n^2-36n)p^{4(n-2)!}$.

{\it Case 4.3.2.} $c_1\neq a_1$, $c_1\neq b_1$, $c_2\neq d_2$ and $c_1=d_2$.

There are $n$ ways to choose $a_2(=b_2=d_1)$ from $N_n$, $n-1$ ways to choose $a_1$ from $N_n\setminus\{a_2\}$, $n-2$ ways to choose $b_1$ from $N_n\setminus\{b_2, a_1\}$, $n-3$ ways to choose $d_2(=c_1)$ from $N_n\setminus\{d_1, a_1, b_1\}$, $n-1$ ways to choose $c_2$ from $N_n\setminus\{d_2\}$.
Therefore, there are $n(n-1)^2(n-2)(n-3)=n^5-7n^4+17n^3-17n^2+6n$ ways to choose four $B_{n-2}$'s as required.
The probability $P_{4.3.2}$ that there are four fault-free $B_{n-2}$'s chosen as in Case 4.3.2 is $\frac{1}{2}(n^5-7n^4+17n^3-17n^2+6n)p^{4(n-2)!}$.

{\it Case 4.3.3.} $c_1=a_1\neq b_1$, $c_2\neq d_2$, $c_2\neq a_2$ or $c_1=b_1\neq a_1$, $c_2\neq d_2$, $c_2\neq b_2$.

W.l.o.g, assume that the former applies.
There are $n$ ways to choose $a_2(=b_2=d_1)$ from $N_n$, $n-1$ ways to choose $d_2$ from $N_n\setminus\{d_1\}$, $n-2$ ways to choose $a_1(=c_1)$ from $N_n\setminus\{d_1, d_2\}$, $n-3$ ways to choose $b_1$ from $N_n\setminus\{d_1, d_2, a_1\}$, $n-3$ ways to choose $c_2$ from $N_n\setminus\{c_1, d_2, a_2\}$.
Therefore, there are $2n(n-1)(n-2)(n-3)^2=2(n^5-9n^4+29n^3-39n^2+18n)$ ways to choose four $B_{n-2}$'s as required.
The probability $P_{4.3.3}$ that there are four fault-free $B_{n-2}$'s chosen as in Case 4.3.3 is $\frac{1}{2}\times 2(n^5-9n^4+29n^3-39n^2+18n)p^{4(n-2)!-(n-3)!}$.

{\it Case 4.3.4.} $c_1=a_1\neq b_1$, $c_2\neq d_2$, $c_2=a_2$ or $c_1=b_1\neq a_1$, $c_2\neq d_2$, $c_2=b_2$.

W.l.o.g, assume that the former applies.
There are $n$ ways to choose $a_2(=b_2=d_1=c_2)$ from $N_n$, $n-1$ ways to choose $d_2$ from $N_n\setminus\{d_1\}$, $n-2$ ways to choose $a_1(=c_1)$ from $N_n\setminus\{d_1, d_2\}$, $n-3$ ways to choose $b_1$ from $N_n\setminus\{d_1, d_2, a_1\}$.
Therefore, there are $2n(n-1)(n-2)(n-3)=2(n^4-6n^3+11n^2-6n)$ ways to choose four $B_{n-2}$'s as required.
The probability $P_{4.3.4}$ that there are four fault-free $B_{n-2}$'s chosen as in Case 4.3.4 is $\frac{1}{2}\times 2(n^4-6n^3+11n^2-6n)p^{4(n-2)!}$.

{\it Case 4.3.5.} $c_1\neq a_1$, $c_1\neq b_1$, $c_2=d_2$ and $c_1\neq d_1$.

There are $n$ ways to choose $a_2(=b_2=d_1)$ from $N_n$, $n-1$ ways to choose $d_2(=c_2)$ from $N_n\setminus\{d_1\}$, $n-2$ ways to choose $a_1$ from $N_n\setminus\{d_1, d_2\}$, $n-3$ ways to choose $b_1$ from $N_n\setminus\{d_1, d_2, a_1\}$, $n-4$ ways to choose $c_1$ from $N_n\setminus\{a_1, b_1, c_2, d_1\}$. 
Therefore, there are $n(n-1)(n-2)(n-3)(n-4)=n^5-10n^4+35n^3-50n^2+24n$ ways to choose four $B_{n-2}$'s as required.
The probability $P_{4.3.5}$ that there are four fault-free $B_{n-2}$'s chosen as in Case 4.3.5 is $\frac{1}{2}(n^5-10n^4+35n^3-50n^2+24n)p^{4(n-2)!-(n-3)!}$.

{\it Case 4.3.6.} $c_1\neq a_1$, $c_1\neq b_1$, $c_2=d_2$ and $c_1=d_1$.

There are $n$ ways to choose $a_2(=b_2=d_1=c_1)$ from $N_n$, $n-1$ ways to choose $d_2(=c_2)$ from $N_n\setminus\{d_1\}$, $n-2$ ways to choose $a_1$ from $N_n\setminus\{d_1, d_2\}$, $n-3$ ways to choose $b_1$ from $N_n\setminus\{d_1, d_2, a_1\}$. 
Therefore, there are $n(n-1)(n-2)(n-3)=n^4-6n^3+11n^2-6n$ ways to choose four $B_{n-2}$'s as required.
The probability $P_{4.3.6}$ that there are four fault-free $B_{n-2}$'s chosen as in Case 4.3.6 is $\frac{1}{2}(n^4-6n^3+11n^2-6n)p^{4(n-2)!}$.

{\it Case 4.3.7.} $c_1=a_1\neq b_1$, $c_2=d_2$ or $c_1=b_1\neq a_1$, $c_2=d_2$.

W.l.o.g, assume that the former applies.
There are $n$ ways to choose $a_2(=b_2=d_1)$ from $N_n$, $n-1$ ways to choose $a_1$ from $N_n\setminus\{a_2\}$, $n-2$ ways to choose $b_1(=c_1)$ from $N_n\setminus\{b_2, a_1\}$, $n-3$ ways to choose $d_2(=c_2)$ from $N_n\setminus\{d_1, a_1, b_1\}$. 
Therefore, there are $2n(n-1)(n-2)(n-3)=2(n^4-6n^3+11n^2-6n)$ ways to choose four $B_{n-2}$'s as required.
Clearly, in this case $c_2\neq a_2$ and $c_1\neq d_1$.
The probability $P_{4.3.7}$ that there are four fault-free $B_{n-2}$'s chosen as in Case 4.3.7 is $\frac{1}{2}\times 2(n^4-6n^3+11n^2-6n)p^{4(n-2)!-2(n-3)!}$.

Thus, the probability $P_{4.3}$ that there are four fault-free $B_{n-2}$'s chosen as in Case 4.3 is $\sum_{i=1}^{7}P_{4.3.i}=\frac{1}{2}(n^6-10n^5+43n^4-98n^3+112n^2-48n)p^{4(n-2)!}+\frac{1}{2}(3n^5-28n^4+93n^3-128n^2+60n)p^{4(n-2)!-(n-3)!}+(n^4-6n^3+11n^2-6n)p^{4(n-2)!-2(n-3)!}$.

{\it Case 4.4.} $a_2=b_2=d_2$, $a_1\notin\{d_1, d_2\}$, $b_1\notin\{d_1, d_2\}$.

{\it Case 4.4.1.} $c_1\neq a_1$, $c_1\neq b_1$, $c_2\neq d_2$ and $c_1\neq d_2$.

There are $n$ ways to choose $a_2(=b_2=d_2)$ from $N_n$, $n-1$ ways to choose $a_1$ from $N_n\setminus\{a_2\}$, $n-2$ ways to choose $b_1$ from $N_n\setminus\{b_2, a_1\}$, $n-3$ ways to choose $d_1$ from $N_n\setminus\{d_2, a_1, b_1\}$, $n-3$ ways to choose $c_1$ from $N_n\setminus\{a_1, b_1, d_2\}$, $n-2$ ways to choose $c_2$ from $N_n\setminus\{c_1, d_2\}$.
Therefore, there are $n(n-1)(n-2)^2(n-3)^2=n^6-11n^5+47n^4-97n^3+96n^2-36n$ ways to choose four $B_{n-2}$'s as required.
The probability $P_{4.4.1}$ that there are four fault-free $B_{n-2}$'s chosen as in Case 4.4.1 is $\frac{1}{2}(n^6-11n^5+47n^4-97n^3+96n^2-36n)p^{4(n-2)!}$.

{\it Case 4.4.2.} $c_1\neq a_1$, $c_1\neq b_1$, $c_2\neq d_2$ and $c_1=d_2$.

There are $n$ ways to choose $a_2(=b_2=d_2=c_1)$ from $N_n$, $n-1$ ways to choose $a_1$ from $N_n\setminus\{a_2\}$, $n-2$ ways to choose $b_1$ from $N_n\setminus\{b_2, a_1\}$, $n-3$ ways to choose $d_1$ from $N_n\setminus\{d_2, a_1, b_1\}$, $n-1$ ways to choose $c_2$ from $N_n\setminus\{d_2\}$.
Therefore, there are $n(n-1)^2(n-2)(n-3)=n^5-7n^4+17n^3-17n^2+6n$ ways to choose four $B_{n-2}$'s as required.
The probability $P_{4.4.2}$ that there are four fault-free $B_{n-2}$'s chosen as in Case 4.4.2 is $\frac{1}{2}(n^5-7n^4+17n^3-17n^2+6n)p^{4(n-2)!}$.

{\it Case 4.4.3.} $c_1=a_1\neq b_1$, $c_2\neq d_2$ or $c_1=b_1\neq a_1$, $c_2\neq d_2$.

W.l.o.g, assume that the former applies. There are $n$ ways to choose $a_2(=b_2=d_2)$ from $N_n$, $n-1$ ways to choose $a_1(=c_1)$ from $N_n\setminus\{a_2\}$, $n-2$ ways to choose $b_1$ from $N_n\setminus\{b_2, a_1\}$, $n-3$ ways to choose $d_1$ from $N_n\setminus\{d_2, a_1, b_1\}$, $n-2$ ways to choose $c_2$ from $N_n\setminus\{c_1, d_2\}$.
Therefore, there are $2n(n-1)(n-2)^2(n-3)=2(n^5-8n^4+23n^3-28n^2+12n)$ ways to choose four $B_{n-2}$'s as required.
Clearly, in this case $c_2\neq a_2$.
The probability $P_{4.4.3}$ that there are four fault-free $B_{n-2}$'s chosen as in Case 4.4.3 is $\frac{1}{2}\times 2(n^5-8n^4+23n^3-28n^2+12n)p^{4(n-2)!-(n-3)!}$.

{\it Case 4.4.4.} $c_1\neq a_1$, $c_1\neq b_1$, $c_2=d_2$ and $c_1\neq d_1$.

There are $n$ ways to choose $a_2(=b_2=d_2=c_2)$ from $N_n$, $n-1$ ways to choose $a_1$ from $N_n\setminus\{a_2\}$, $n-2$ ways to choose $b_1$ from $N_n\setminus\{a_1, b_2\}$, $n-3$ ways to choose $d_1$ from $N_n\setminus\{d_2, a_1, b_1\}$, $n-4$ ways to choose $c_1$ from $N_n\setminus\{a_1, b_1, c_2, d_1\}$. 
Therefore, there are $n(n-1)(n-2)(n-3)(n-4)=n^5-10n^4+35n^3-50n^2+24n$ ways to choose four $B_{n-2}$'s as required.
The probability $P_{4.4.4}$ that there are four fault-free $B_{n-2}$'s chosen as in Case 4.4.4 is $\frac{1}{2}(n^5-10n^4+35n^3-50n^2+24n)p^{4(n-2)!-(n-3)!}$.

{\it Case 4.4.5.} $c_1\neq a_1$, $c_1\neq b_1$, $c_2=d_2$ and $c_1=d_1$.

There are $n$ ways to choose $a_2(=b_2=d_2=c_2)$ from $N_n$, $n-1$ ways to choose $a_1$ from $N_n\setminus\{a_2\}$, $n-2$ ways to choose $b_1$ from $N_n\setminus\{a_1, b_2\}$, $n-3$ ways to choose $d_1(=c_1)$ from $N_n\setminus\{d_2, a_1, b_1\}$.
Therefore, there are $n(n-1)(n-2)(n-3)=n^4-6n^3+11n^2-6n$ ways to choose four $B_{n-2}$'s as required.
The probability $P_{4.4.5}$ that there are four fault-free $B_{n-2}$'s chosen as in Case 4.4.5 is $\frac{1}{2}(n^4-6n^3+11n^2-6n)p^{4(n-2)!}$.

{\it Case 4.4.6.} $c_1=a_1\neq b_1$, $c_2=d_2$ or $c_1=b_1\neq a_1$, $c_2=d_2$.

W.l.o.g, assume that the former applies. There are $n$ ways to choose $a_2(=b_2=d_2=c_2)$ from $N_n$, $n-1$ ways to choose $a_1(=c_1)$ from $N_n\setminus\{a_2\}$, $n-2$ ways to choose $b_1$ from $N_n\setminus\{b_2, a_1\}$, $n-3$ ways to choose $d_1$ from $N_n\setminus\{d_2, a_1, b_1\}$. 
Therefore, there are $2n(n-1)(n-2)(n-3)=2(n^4-6n^3+11n^2-6n)$ ways to choose four $B_{n-2}$'s as required.
Clearly, in this case $c_2=a_2$ and $c_1\neq d_1$.
The probability $P_{4.4.6}$ that there are four fault-free $B_{n-2}$'s chosen as in Case 4.4.6 is $\frac{1}{2}\times 2(n^4-6n^3+11n^2-6n)p^{4(n-2)!-(n-3)!}$.

Thus, the probability $P_{4.4}$ that there are four fault-free $B_{n-2}$'s chosen as in Case 4.4 is $\sum_{i=1}^{6}P_{4.4.i}=\frac{1}{2}(n^6-10n^5+41n^4-86n^3+90n^2-36n)p^{4(n-2)!}+\frac{3}{2}(n^5-8n^4+23n^3-28n^2+12n)p^{4(n-2)!-(n-3)!}$.

{\it Case 4.5.} $a_1=b_2=d_1$, $a_2\notin\{d_1, d_2\}$, $b_1\notin\{d_1, d_2\}$ or $a_2=b_1=d_1$, $a_1\notin\{d_1, d_2\}$, $b_2\notin\{d_1, d_2\}$.

W.l.o.g, assume that the former applies.

{\it Case 4.5.1.} $c_1\neq a_1$, $c_1\neq b_1$, $c_2\neq d_2$ and $c_1\neq d_2$.

There are $n$ ways to choose $a_1(=b_2=d_1)$ from $N_n$, $n-1$ ways to choose $d_2$ from $N_n\setminus\{d_1\}$, $n-2$ ways to choose $a_2$ from $N_n\setminus\{d_1, d_2\}$, $n-2$ ways to choose $b_1$ from $N_n\setminus\{d_1, d_2\}$, $n-3$ ways to choose $c_1$ from $N_n\setminus\{a_1, b_1, d_2\}$, $n-2$ ways to choose $c_2$ from $N_n\setminus\{c_1, d_2\}$.
Therefore, there are $n(n-1)(n-2)^3(n-3)=n^6-10n^5+39n^4-74n^3+68n^2-24n$ ways to choose four $B_{n-2}$'s as required.
The probability $P_{4.5.1}$ that there are four fault-free $B_{n-2}$'s chosen as in Case 4.5.1 is $\frac{1}{2}(n^6-10n^5+39n^4-74n^3+68n^2-24n)p^{4(n-2)!}$.

{\it Case 4.5.2.} $c_1\neq a_1$, $c_1\neq b_1$, $c_2\neq d_2$ and $c_1=d_2$.

There are $n$ ways to choose $a_1(=b_2=d_1)$ from $N_n$, $n-1$ ways to choose $d_2(=c_1)$ from $N_n\setminus\{d_1\}$, $n-2$ ways to choose $a_2$ from $N_n\setminus\{d_1, d_2\}$, $n-2$ ways to choose $b_1$ from $N_n\setminus\{d_1, d_2\}$, $n-1$ ways to choose $c_2$ from $N_n\setminus\{d_2\}$.
Therefore, there are $n(n-1)^2(n-2)^2=n^5-6n^4+13n^3-12n^2+4n$ ways to choose four $B_{n-2}$'s as required.
The probability $P_{4.5.2}$ that there are four fault-free $B_{n-2}$'s chosen as in Case 4.5.2 is $\frac{1}{2}(n^5-6n^4+13n^3-12n^2+4n)p^{4(n-2)!}$.

{\it Case 4.5.3.} $c_1=a_1\neq b_1$, $c_2\neq d_2$, $c_2\neq a_2$ or $c_1=b_1\neq a_1$, $c_2\neq d_2$, $c_2\neq b_2$.

W.l.o.g, assume that the former applies.
There are $n$ ways to choose $a_1(=b_2=d_1=c_1)$ from $N_n$, $n-1$ ways to choose $d_2$ from $N_n\setminus\{d_1\}$, $n-2$ ways to choose $a_2$ from $N_n\setminus\{d_1, d_2\}$, $n-2$ ways to choose $b_1$ from $N_n\setminus\{d_1, d_2\}$, $n-3$ ways to choose $c_2$ from $N_n\setminus\{c_1, d_2, a_2\}$.
Therefore, there are $2n(n-1)(n-2)^2(n-3)=2(n^5-8n^4+23n^3-28n^2+12n)$ ways to choose four $B_{n-2}$'s as required.
The probability $P_{4.5.3}$ that there are four fault-free $B_{n-2}$'s chosen as in Case 4.5.3 is $\frac{1}{2}\times 2(n^5-8n^4+23n^3-28n^2+12n)p^{4(n-2)!-(n-3)!}$.

{\it Case 4.5.4.} $c_1=a_1\neq b_1$, $c_2\neq d_2$, $c_2=a_2$ or $c_1=b_1\neq a_1$, $c_2\neq d_2$, $c_2=b_2$.

W.l.o.g, assume that the former applies.
There are $n$ ways to choose $a_1(=b_2=d_1=c_1)$ from $N_n$, $n-1$ ways to choose $d_2$ from $N_n\setminus\{d_1\}$, $n-2$ ways to choose $a_2(=c_2)$ from $N_n\setminus\{d_1, d_2\}$, $n-2$ ways to choose $b_1$ from $N_n\setminus\{d_1, d_2\}$.
Therefore, there are $2n(n-1)(n-2)^2=2(n^4-5n^3+8n^2-4n)$ ways to choose four $B_{n-2}$'s as required.
The probability $P_{4.5.4}$ that there are four fault-free $B_{n-2}$'s chosen as in Case 4.5.4 is $\frac{1}{2}\times 2(n^4-5n^3+8n^2-4n)p^{4(n-2)!}$.

{\it Case 4.5.5.} $c_1\neq a_1$, $c_1\neq b_1$, $c_2=d_2$.

There are $n$ ways to choose $a_1(=b_2=d_1)$ from $N_n$, $n-1$ ways to choose $d_2(=c_2)$ from $N_n\setminus\{d_1\}$, $n-2$ ways to choose $a_2$ from $N_n\setminus\{d_1, d_2\}$, $n-2$ ways to choose $b_1$ from $N_n\setminus\{d_1, d_2\}$, $n-3$ ways to choose $c_1$ from $N_n\setminus\{a_1, b_1, c_2\}$.
Therefore, there are $n(n-1)(n-2)^2(n-3)=n^5-8n^4+23n^3-28n^2+12n$ ways to choose four $B_{n-2}$'s as required.
Clearly, in this case $c_1\neq d_1$.
The probability $P_{4.5.5}$ that there are four fault-free $B_{n-2}$'s chosen as in Case 4.5.5 is $\frac{1}{2}(n^5-8n^4+23n^3-28n^2+12n)p^{4(n-2)!-(n-3)!}$.

{\it Case 4.5.6.} $c_1=a_1\neq b_1$ and $c_2=d_2$.

There are $n$ ways to choose $a_1(=b_2=d_1=c_1)$ from $N_n$, $n-1$ ways to choose $d_2(=c_2)$ from $N_n\setminus\{d_1\}$, $n-2$ ways to choose $a_2$ from $N_n\setminus\{d_1, d_2\}$, $n-2$ ways to choose $b_1$ from $N_n\setminus\{d_1, d_2\}$.
Therefore, there are $n(n-1)(n-2)^2=n^4-5n^3+8n^2-4n$ ways to choose four $B_{n-2}$'s as required.
Clearly, in this case $c_2\neq a_2$ and $c_1=d_1$.
The probability $P_{4.5.6}$ that there are four fault-free $B_{n-2}$'s chosen as in Case 4.5.6 is $\frac{1}{2}(n^4-5n^3+8n^2-4n)p^{4(n-2)!-(n-3)!}$.

{\it Case 4.5.7.} $c_1=b_1\neq a_1$ and $c_2=d_2$.

There are $n$ ways to choose $a_1(=b_2=d_1)$ from $N_n$, $n-1$ ways to choose $d_2(=c_2)$ from $N_n\setminus\{a_1\}$, $n-2$ ways to choose $a_2$ from $N_n\setminus\{d_1, d_2\}$, $n-2$ ways to choose $b_1(=c_1)$ from $N_n\setminus\{d_1, d_2\}$.
Therefore, there are $n(n-1)(n-2)^2=n^4-5n^3+8n^2-4n$ ways to choose four $B_{n-2}$'s as required.
Clearly, in this case $c_2\neq b_2$ and $c_1\neq d_1$.
The probability $P_{4.5.7}$ that there are four fault-free $B_{n-2}$'s chosen as in Case 4.5.7 is $\frac{1}{2}(n^4-5n^3+8n^2-4n)p^{4(n-2)!-2(n-3)!}$.

Thus, the probability $P_{4.5}$ that there are four fault-free $B_{n-2}$'s chosen as in Case 4.5 is $2\sum_{i=1}^{7}P_{4.5.i}=(n^6-9n^5+35n^4-71n^3+72n^2-28n)p^{4(n-2)!}+(3n^5-23n^4+64n^3-76n^2+32n)p^{4(n-2)!-(n-3)!}+(n^4-5n^3+8n^2-4n)p^{4(n-2)!-2(n-3)!}$.

{\it Case 4.6.} $a_1=b_2=d_2$, $a_2\notin\{d_1, d_2\}$, $b_1\notin\{d_1, d_2\}$ or $a_2=b_1=d_2$, $a_1\notin\{d_1, d_2\}$, $b_2\notin\{d_1, d_2\}$.

W.l.o.g, assume that the former applies.

{\it Case 4.6.1.} $c_1\neq a_1$, $c_1\neq b_1$ and $c_2\neq d_2$.

There are $n$ ways to choose $a_1(=b_2=d_2)$ from $N_n$, $n-1$ ways to choose $d_1$ from $N_n\setminus\{d_2\}$, $n-2$ ways to choose $a_2$ from $N_n\setminus\{d_1, d_2\}$, $n-2$ ways to choose $b_1$ from $N_n\setminus\{d_1, d_2\}$, $n-2$ ways to choose $c_1$ from $N_n\setminus\{a_1, b_1\}$, $n-2$ ways to choose $c_2$ from $N_n\setminus\{c_1, d_2\}$.
Therefore, there are $n(n-1)(n-2)^4=n^6-9n^5+32n^4-56n^3+48n^2-16n$ ways to choose four $B_{n-2}$'s as required.
The probability $P_{4.6.1}$ that there are four fault-free $B_{n-2}$'s chosen as in Case 4.6.1 is $\frac{1}{2}(n^6-9n^5+32n^4-56n^3+48n^2-16n)p^{4(n-2)!}$.

{\it Case 4.6.2.} $c_1=a_1\neq b_1$, $c_2\neq d_2$ and $c_2\neq a_2$.

There are $n$ ways to choose $a_1(=b_2=d_2=c_1)$ from $N_n$, $n-1$ ways to choose $d_1$ from $N_n\setminus\{d_2\}$, $n-2$ ways to choose $a_2$ from $N_n\setminus\{d_1, d_2\}$, $n-2$ ways to choose $b_1$ from $N_n\setminus\{d_1, d_2\}$, $n-2$ ways to choose $c_2$ from $N_n\setminus\{d_2, a_2\}$.
Therefore, there are $n(n-1)(n-2)^3=n^5-7n^4+18n^3-20n^2+8n$ ways to choose four $B_{n-2}$'s as required.
The probability $P_{4.6.2}$ that there are four fault-free $B_{n-2}$'s chosen as in Case 4.6.2 is $\frac{1}{2}(n^5-7n^4+18n^3-20n^2+8n)p^{4(n-2)!-(n-3)!}$.

{\it Case 4.6.3.} $c_1=a_1\neq b_1$, $c_2\neq d_2$ and $c_2=a_2$.

There are $n$ ways to choose $a_1(=b_2=d_2=c_1)$ from $N_n$, $n-1$ ways to choose $d_1$ from $N_n\setminus\{d_2\}$, $n-2$ ways to choose $a_2(=c_2)$ from $N_n\setminus\{d_1, d_2\}$, $n-2$ ways to choose $b_1$ from $N_n\setminus\{d_1, d_2\}$.
Therefore, there are $n(n-1)(n-2)^2=n^4-5n^3+8n^2-4n$ ways to choose four $B_{n-2}$'s as required.
The probability $P_{4.6.3}$ that there are four fault-free $B_{n-2}$'s chosen as in Case 4.6.3 is $\frac{1}{2}(n^4-5n^3+8n^2-4n)p^{4(n-2)!}$.

{\it Case 4.6.4.} $c_1=b_1\neq a_1$ and $c_2\neq d_2$.

There are $n$ ways to choose $a_1(=b_2=d_2)$ from $N_n$, $n-1$ ways to choose $d_1$ from $N_n\setminus\{d_2\}$, $n-2$ ways to choose $a_2$ from $N_n\setminus\{d_1, d_2\}$, $n-2$ ways to choose $b_1(=c_1)$ from $N_n\setminus\{d_1, d_2\}$, $n-2$ ways to choose $c_2$ from $N_n\setminus\{d_2, c_1\}$.
Therefore, there are $n(n-1)(n-2)^3=n^5-7n^4+18n^3-20n^2+8n$ ways to choose four $B_{n-2}$'s as required.
Clearly, in this case $c_2\neq b_2$.
The probability $P_{4.6.4}$ that there are four fault-free $B_{n-2}$'s chosen as in Case 4.6.4 is $\frac{1}{2}(n^5-7n^4+18n^3-20n^2+8n)p^{4(n-2)!-(n-3)!}$.

{\it Case 4.6.5.} $c_1\neq a_1$, $c_1\neq b_1$, $c_2=d_2$ and $c_1\neq d_1$.

There are $n$ ways to choose $a_1(=b_2=d_2=c_2)$ from $N_n$, $n-1$ ways to choose $d_1$ from $N_n\setminus\{d_2\}$, $n-2$ ways to choose $a_2$ from $N_n\setminus\{d_1, d_2\}$, $n-2$ ways to choose $b_1$ from $N_n\setminus\{d_1, d_2\}$, $n-3$ ways to choose $c_1$ from $N_n\setminus\{a_1, b_1, d_1\}$.
Therefore, there are $n(n-1)(n-2)^2(n-3)=n^5-8n^4+23n^3-28n^2+12n$ ways to choose four $B_{n-2}$'s as required.
The probability $P_{4.6.5}$ that there are four fault-free $B_{n-2}$'s chosen as in Case 4.6.5 is $\frac{1}{2}(n^5-8n^4+23n^3-28n^2+12n)p^{4(n-2)!-(n-3)!}$.

{\it Case 4.6.6.} $c_1\neq a_1$, $c_1\neq b_1$, $c_2=d_2$ and $c_1=d_1$.

There are $n$ ways to choose $a_1(=b_2=d_2=c_2)$ from $N_n$, $n-1$ ways to choose $d_1(=c_1)$ from $N_n\setminus\{d_2\}$, $n-2$ ways to choose $a_2$ from $N_n\setminus\{d_1, d_2\}$, $n-2$ ways to choose $b_1$ from $N_n\setminus\{d_1, d_2\}$.
Therefore, there are $n(n-1)(n-2)^2=n^4-5n^3+8n^2-4n$ ways to choose four $B_{n-2}$'s as required.
The probability $P_{4.6.6}$ that there are four fault-free $B_{n-2}$'s chosen as in Case 4.6.6 is $\frac{1}{2}(n^4-5n^3+8n^2-4n)p^{4(n-2)!}$.

{\it Case 4.6.7.} $c_1=b_1\neq a_1$, $c_2=d_2$.

There are $n$ ways to choose $a_1(=b_2=d_2=c_2)$ from $N_n$, $n-1$ ways to choose $d_1$ from $N_n\setminus\{d_2\}$, $n-2$ ways to choose $a_2$ from $N_n\setminus\{d_1, d_2\}$, $n-2$ ways to choose $b_1(=c_1)$ from $N_n\setminus\{d_1, d_2\}$.
Therefore, there are $n(n-1)(n-2)^2=n^4-5n^3+8n^2-4n$ ways to choose four $B_{n-2}$'s as required.
The probability $P_{4.6.7}$ that there are four fault-free $B_{n-2}$'s chosen as in Case 4.6.7 is $\frac{1}{2}(n^4-5n^3+8n^2-4n)p^{4(n-2)!-(n-3)!}$.

Thus, the probability $P_{4.6}$ that there are four fault-free $B_{n-2}$'s chosen as in Case 4.6 is $2\sum_{i=1}^{7}P_{4.6.i}=(n^6-9n^5+34n^4-66n^3+64n^2-24n)p^{4(n-2)!}+3(n^5-7n^4+18n^3-20n^2+8n)p^{4(n-2)!-(n-3)!}$.

{\it Case 4.7.} $a_1=d_1$, $b_1=d_2$, $a_2\notin\{d_1, d_2\}$, $b_2\notin\{d_1, d_2\}$ or $a_1=d_2$, $b_1=d_1$, $a_2\notin\{d_1, d_2\}$, $b_2\notin\{d_1, d_2\}$.

W.l.o.g, assume that the former applies.

{\it Case 4.7.1.} $c_1\neq a_1$, $c_1\neq b_1$ and $c_2\neq d_2$.

There are $n$ ways to choose $a_1(=d_1)$ from $N_n$, $n-1$ ways to choose $b_1(=d_2)$ from $N_n\setminus\{d_1\}$, $n-2$ ways to choose $a_2$ from $N_n\setminus\{d_1, d_2\}$, $n-2$ ways to choose $b_2$ from $N_n\setminus\{d_1, d_2\}$, $n-2$ ways to choose $c_1$ from $N_n\setminus\{a_1, b_1\}$, $n-2$ ways to choose $c_2$ from $N_n\setminus\{c_1, d_2\}$.
Therefore, there are $n(n-1)(n-2)^4=n^6-9n^5+32n^4-56n^3+48n^2-16n$ ways to choose four $B_{n-2}$'s as required.
The probability $P_{4.7.1}$ that there are four fault-free $B_{n-2}$'s chosen as in Case 4.7.1 is $\frac{1}{2}(n^6-9n^5+32n^4-56n^3+48n^2-16n)p^{4(n-2)!}$.

{\it Case 4.7.2.} $c_1=a_1\neq b_1$, $c_2\neq d_2$ and $c_2\neq a_2$.

There are $n$ ways to choose $a_1(=d_1=c_1)$ from $N_n$, $n-1$ ways to choose $b_1(=d_2)$ from $N_n\setminus\{d_1\}$, $n-2$ ways to choose $a_2$ from $N_n\setminus\{d_1, d_2\}$, $n-2$ ways to choose $b_2$ from $N_n\setminus\{d_1, d_2\}$, $n-3$ ways to choose $c_2$ from $N_n\setminus\{c_1, d_2, a_2\}$.
Therefore, there are $n(n-1)(n-2)^2(n-3)=n^5-8n^4+23n^3-28n^2+12n$ ways to choose four $B_{n-2}$'s as required.
The probability $P_{4.7.2}$ that there are four fault-free $B_{n-2}$'s chosen as in Case 4.7.2 is $\frac{1}{2}(n^5-8n^4+23n^3-28n^2+12n)p^{4(n-2)!-(n-3)!}$.

{\it Case 4.7.3.} $c_1=a_1\neq b_1$, $c_2\neq d_2$ and $c_2=a_2$.

There are $n$ ways to choose $a_1(=d_1=c_1)$ from $N_n$, $n-1$ ways to choose $b_1(=d_2)$ from $N_n\setminus\{d_1\}$, $n-2$ ways to choose $a_2(=c_2)$ from $N_n\setminus\{d_1, d_2\}$, $n-2$ ways to choose $b_2$ from $N_n\setminus\{d_1, d_2\}$.
Therefore, there are $n(n-1)(n-2)^2=n^4-5n^3+8n^2-4n$ ways to choose four $B_{n-2}$'s as required.
The probability $P_{4.7.3}$ that there are four fault-free $B_{n-2}$'s chosen as in Case 4.7.3 is $\frac{1}{2}(n^4-5n^3+8n^2-4n)p^{4(n-2)!}$.

{\it Case 4.7.4.} $c_1=b_1\neq a_1$, $c_2\neq d_2$ and $c_2\neq b_2$.

There are $n$ ways to choose $a_1(=d_1)$ from $N_n$, $n-1$ ways to choose $b_1(=d_2=c_1)$ from $N_n\setminus\{d_1\}$, $n-2$ ways to choose $a_2$ from $N_n\setminus\{d_1, d_2\}$, $n-2$ ways to choose $b_2$ from $N_n\setminus\{d_1, d_2\}$, $n-2$ ways to choose $c_2$ from $N_n\setminus\{d_2, b_2\}$.
Therefore, there are $n(n-1)(n-2)^3=n^5-7n^4+18n^3-20n^2+8n$ ways to choose four $B_{n-2}$'s as required.
The probability $P_{4.7.4}$ that there are four fault-free $B_{n-2}$'s chosen as in Case 4.7.4 is $\frac{1}{2}(n^5-7n^4+18n^3-20n^2+8n)p^{4(n-2)!-(n-3)!}$.

{\it Case 4.7.5.} $c_1=b_1\neq a_1$, $c_2\neq d_2$ and $c_2=b_2$.

There are $n$ ways to choose $a_1(=d_1)$ from $N_n$, $n-1$ ways to choose $b_1(=d_2=c_1)$ from $N_n\setminus\{d_1\}$, $n-2$ ways to choose $a_2$ from $N_n\setminus\{d_1, d_2\}$, $n-2$ ways to choose $b_2(=c_2)$ from $N_n\setminus\{d_1, d_2\}$.
Therefore, there are $n(n-1)(n-2)^2=n^4-5n^3+8n^2-4n$ ways to choose four $B_{n-2}$'s as required.
The probability $P_{4.7.5}$ that there are four fault-free $B_{n-2}$'s chosen as in Case 4.7.5 is $\frac{1}{2}(n^4-5n^3+8n^2-4n)p^{4(n-2)!}$.

{\it Case 4.7.6.} $c_1\neq a_1$, $c_1\neq b_1$ and $c_2=d_2$.

There are $n$ ways to choose $a_1(=d_1)$ from $N_n$, $n-1$ ways to choose $b_1(=d_2=c_2)$ from $N_n\setminus\{d_1\}$, $n-2$ ways to choose $a_2$ from $N_n\setminus\{d_1, d_2\}$, $n-2$ ways to choose $b_2$ from $N_n\setminus\{d_1, d_2\}$, $n-2$ ways to choose $c_1$ from $N_n\setminus\{a_1, b_1\}$.
Therefore, there are $n(n-1)(n-2)^3=n^5-7n^4+18n^3-20n^2+8n$ ways to choose four $B_{n-2}$'s as required.
Clearly, in this case $c_1\neq d_1$.
The probability $P_{4.7.6}$ that there are four fault-free $B_{n-2}$'s chosen as in Case 4.7.6 is $\frac{1}{2}(n^5-7n^4+18n^3-20n^2+8n)p^{4(n-2)!-(n-3)!}$.

{\it Case 4.7.7.} $c_1=a_1\neq b_1$ and $c_2=d_2$.

There are $n$ ways to choose $a_1(=d_1=c_1)$ from $N_n$, $n-1$ ways to choose $b_1(=d_2=c_2)$ from $N_n\setminus\{d_1\}$, $n-2$ ways to choose $a_2$ from $N_n\setminus\{d_1, d_2\}$, $n-2$ ways to choose $b_2$ from $N_n\setminus\{d_1, d_2\}$.
Therefore, there are $n(n-1)(n-2)^2=n^4-5n^3+8n^2-4n$ ways to choose four $B_{n-2}$'s as required.
Clearly, in this case $c_2\neq a_2$ and $c_1=d_1$.
The probability $P_{4.7.7}$ that there are four fault-free $B_{n-2}$'s chosen as in Case 4.7.7 is $\frac{1}{2}(n^4-5n^3+8n^2-4n)p^{4(n-2)!-(n-3)!}$.

Thus, the probability $P_{4.7}$ that there are four fault-free $B_{n-2}$'s chosen as in Case 4.7 is $2\sum_{i=1}^{7}P_{4.7.i}=(n^6-9n^5+34n^4-66n^3+64n^2-24n)p^{4(n-2)!}+3(n^5-7n^4+18n^3-20n^2+8n)p^{4(n-2)!-(n-3)!}$.

{\it Case 4.8.} $a_1=d_2$, $b_2=d_1$, $a_2\notin\{d_1, d_2\}$, $b_1\notin\{d_1, d_2\}$ or $a_2=d_1$, $b_1=d_2$, $a_1\notin\{d_1, d_2\}$, $b_2\notin\{d_1, d_2\}$.

W.l.o.g, assume that the former applies.

{\it Case 4.8.1.} $c_1\neq a_1$, $c_1\neq b_1$ and $c_2\neq d_2$.

There are $n$ ways to choose $a_1(=d_2)$ from $N_n$, $n-1$ ways to choose $b_2(=d_1)$ from $N_n\setminus\{d_2\}$, $n-2$ ways to choose $a_2$ from $N_n\setminus\{d_1, d_2\}$, $n-2$ ways to choose $b_1$ from $N_n\setminus\{d_1, d_2\}$, $n-2$ ways to choose $c_1$ from $N_n\setminus\{a_1, b_1\}$, $n-2$ ways to choose $c_2$ from $N_n\setminus\{c_1, d_2\}$.
Therefore, there are $n(n-1)(n-2)^4=n^6-9n^5+32n^4-56n^3+48n^2-16n$ ways to choose four $B_{n-2}$'s as required.
The probability $P_{4.8.1}$ that there are four fault-free $B_{n-2}$'s chosen as in Case 4.8.1 is $\frac{1}{2}(n^6-9n^5+32n^4-56n^3+48n^2-16n)p^{4(n-2)!}$.

{\it Case 4.8.2.} $c_1=a_1\neq b_1$, $c_2\neq d_2$ and $c_2\neq a_2$.

There are $n$ ways to choose $a_1(=d_2=c_1)$ from $N_n$, $n-1$ ways to choose $b_2(=d_1)$ from $N_n\setminus\{d_2\}$, $n-2$ ways to choose $a_2$ from $N_n\setminus\{d_1, d_2\}$, $n-2$ ways to choose $b_1$ from $N_n\setminus\{d_1, d_2\}$, $n-2$ ways to choose $c_2$ from $N_n\setminus\{d_2, a_2\}$.
Therefore, there are $n(n-1)(n-2)^3=n^5-7n^4+18n^3-20n^2+8n$ ways to choose four $B_{n-2}$'s as required.
The probability $P_{4.8.2}$ that there are four fault-free $B_{n-2}$'s chosen as in Case 4.8.2 is $\frac{1}{2}(n^5-7n^4+18n^3-20n^2+8n)p^{4(n-2)!-(n-3)!}$.

{\it Case 4.8.3.} $c_1=a_1\neq b_1$, $c_2\neq d_2$ and $c_2=a_2$.

There are $n$ ways to choose $a_1(=d_2=c_1)$ from $N_n$, $n-1$ ways to choose $b_2(=d_1)$ from $N_n\setminus\{d_2\}$, $n-2$ ways to choose $a_2(=c_2)$ from $N_n\setminus\{d_1, d_2\}$, $n-2$ ways to choose $b_1$ from $N_n\setminus\{d_1, d_2\}$.
Therefore, there are $n(n-1)(n-2)^2=n^4-5n^3+8n^2-4n$ ways to choose four $B_{n-2}$'s as required.
The probability $P_{4.8.3}$ that there are four fault-free $B_{n-2}$'s chosen as in Case 4.8.3 is $\frac{1}{2}(n^4-5n^3+8n^2-4n)p^{4(n-2)!}$.

{\it Case 4.8.4.} $c_1=b_1\neq a_1$, $c_2\neq d_2$ and $c_2\neq b_2$.

There are $n$ ways to choose $a_1(=d_2)$ from $N_n$, $n-1$ ways to choose $b_2(=d_1)$ from $N_n\setminus\{d_2\}$, $n-2$ ways to choose $a_2$ from $N_n\setminus\{d_1, d_2\}$, $n-2$ ways to choose $b_1(=c_1)$ from $N_n\setminus\{d_1, d_2\}$, $n-3$ ways to choose $c_2$ from $N_n\setminus\{c_1, d_2, b_2\}$.
Therefore, there are $n(n-1)(n-2)^2(n-3)=n^5-8n^4+23n^3-28n^2+12n$ ways to choose four $B_{n-2}$'s as required.
The probability $P_{4.8.4}$ that there are four fault-free $B_{n-2}$'s chosen as in Case 4.8.4 is $\frac{1}{2}(n^5-8n^4+23n^3-28n^2+12n)p^{4(n-2)!-(n-3)!}$.

{\it Case 4.8.5.} $c_1=b_1\neq a_1$, $c_2\neq d_2$ and $c_2=b_2$.

There are $n$ ways to choose $a_1(=d_2)$ from $N_n$, $n-1$ ways to choose $b_2(=d_1=c_2)$ from $N_n\setminus\{d_2\}$, $n-2$ ways to choose $a_2$ from $N_n\setminus\{d_1, d_2\}$, $n-2$ ways to choose $b_1(=c_1)$ from $N_n\setminus\{d_1, d_2\}$.
Therefore, there are $n(n-1)(n-2)^2=n^4-5n^3+8n^2-4n$ ways to choose four $B_{n-2}$'s as required.
The probability $P_{4.8.5}$ that there are four fault-free $B_{n-2}$'s chosen as in Case 4.8.5 is $\frac{1}{2}(n^4-5n^3+8n^2-4n)p^{4(n-2)!}$.

{\it Case 4.8.6.} $c_1\neq a_1$, $c_1\neq b_1$, $c_2=d_2$ and $c_1\neq d_1$.

There are $n$ ways to choose $a_1(=d_2=c_2)$ from $N_n$, $n-1$ ways to choose $b_2(=d_1)$ from $N_n\setminus\{d_2\}$, $n-2$ ways to choose $a_2$ from $N_n\setminus\{d_1, d_2\}$, $n-2$ ways to choose $b_1$ from $N_n\setminus\{d_1, d_2\}$, $n-3$ ways to choose $c_1$ from $N_n\setminus\{a_1, b_1, d_1\}$.
Therefore, there are $n(n-1)(n-2)^2(n-3)=n^5-8n^4+23n^3-28n^2+12n$ ways to choose four $B_{n-2}$'s as required.
The probability $P_{4.8.6}$ that there are four fault-free $B_{n-2}$'s chosen as in Case 4.8.6 is $\frac{1}{2}(n^5-8n^4+23n^3-28n^2+12n)p^{4(n-2)!-(n-3)!}$.

{\it Case 4.8.7.} $c_1\neq a_1$, $c_1\neq b_1$, $c_2=d_2$ and $c_1=d_1$.

There are $n$ ways to choose $a_1(=d_2=c_2)$ from $N_n$, $n-1$ ways to choose $b_2(=d_1=c_1)$ from $N_n\setminus\{d_2\}$, $n-2$ ways to choose $a_2$ from $N_n\setminus\{d_1, d_2\}$, $n-2$ ways to choose $b_1$ from $N_n\setminus\{d_1, d_2\}$.
Therefore, there are $n(n-1)(n-2)^2=n^4-5n^3+8n^2-4n$ ways to choose four $B_{n-2}$'s as required.
The probability $P_{4.8.7}$ that there are four fault-free $B_{n-2}$'s chosen as in Case 4.8.7 is $\frac{1}{2}(n^4-5n^3+8n^2-4n)p^{4(n-2)!}$.

{\it Case 4.8.8.} $c_1=b_1\neq a_1$ and $c_2=d_2$.

There are $n$ ways to choose $a_1(=d_2=c_2)$ from $N_n$, $n-1$ ways to choose $b_2(=d_1)$ from $N_n\setminus\{d_2\}$, $n-2$ ways to choose $a_2$ from $N_n\setminus\{d_1, d_2\}$, $n-2$ ways to choose $b_1(=c_1)$ from $N_n\setminus\{d_1, d_2\}$.
Therefore, there are $n(n-1)(n-2)^2=n^4-5n^3+8n^2-4n$ ways to choose four $B_{n-2}$'s as required.
Clearly, in this case $c_2\neq b_2$ and $c_1\neq d_1$.
The probability $P_{4.8.8}$ that there are four fault-free $B_{n-2}$'s chosen as in Case 4.8.8 is $\frac{1}{2}(n^4-5n^3+8n^2-4n)p^{4(n-2)!-2(n-3)!}$.

Thus, the probability $P_{4.8}$ that there are four fault-free $B_{n-2}$'s chosen as in Case 4.8 is $2\sum_{i=1}^{8}P_{4.8.i}=(n^6-9n^5+35n^4-71n^3+72n^2-28n)p^{4(n-2)!}+(3n^5-23n^4+64n^3-76n^2+32n)p^{4(n-2)!-(n-3)!}+(n^4-5n^3+8n^2-4n)p^{4(n-2)!-2(n-3)!}$.

{\it Case 4.9.} $a_1=d_1$, $b_2=d_2$, $a_2\notin\{d_1, d_2\}$, $b_1\notin\{d_1, d_2\}$ or $a_2=d_2$, $b_1=d_1$, $a_1\notin\{d_1, d_2\}$, $b_2\notin\{d_1, d_2\}$.

W.l.o.g, assume that the former applies.

{\it Case 4.9.1.} $c_1\neq a_1$, $c_1\neq b_1$, $c_2\neq d_2$ and $c_1\neq d_2$.

There are $n$ ways to choose $a_1(=d_1)$ from $N_n$, $n-1$ ways to choose $b_2(=d_2)$ from $N_n\setminus\{d_1\}$, $n-2$ ways to choose $a_2$ from $N_n\setminus\{d_1, d_2\}$, $n-2$ ways to choose $b_1$ from $N_n\setminus\{d_1, d_2\}$, $n-3$ ways to choose $c_1$ from $N_n\setminus\{a_1, b_1, d_2\}$, $n-2$ ways to choose $c_2$ from $N_n\setminus\{c_1, d_2\}$.
Therefore, there are $n(n-1)(n-2)^3(n-3)=n^6-10n^5+39n^4-74n^3+68n^2-24n$ ways to choose four $B_{n-2}$'s as required.
The probability $P_{4.9.1}$ that there are four fault-free $B_{n-2}$'s chosen as in Case 4.9.1 is $\frac{1}{2}(n^6-10n^5+39n^4-74n^3+68n^2-24n)p^{4(n-2)!}$.

{\it Case 4.9.2.} $c_1\neq a_1$, $c_1\neq b_1$, $c_2\neq d_2$ and $c_1=d_2$.

There are $n$ ways to choose $a_1(=d_1)$ from $N_n$, $n-1$ ways to choose $b_2(=d_2=c_1)$ from $N_n\setminus\{d_1\}$, $n-2$ ways to choose $a_2$ from $N_n\setminus\{d_1, d_2\}$, $n-2$ ways to choose $b_1$ from $N_n\setminus\{d_1, d_2\}$, $n-1$ ways to choose $c_2$ from $N_n\setminus\{d_2\}$.
Therefore, there are $n(n-1)^2(n-2)^2=n^5-6n^4+13n^3-12n^2+4n$ ways to choose four $B_{n-2}$'s as required.
The probability $P_{4.9.2}$ that there are four fault-free $B_{n-2}$'s chosen as in Case 4.9.2 is $\frac{1}{2}(n^5-6n^4+13n^3-12n^2+4n)p^{4(n-2)!}$.

{\it Case 4.9.3.} $c_1=a_1\neq b_1$, $c_2\neq d_2$ and $c_2\neq a_2$.

There are $n$ ways to choose $a_1(=d_1=c_1)$ from $N_n$, $n-1$ ways to choose $b_2(=d_2)$ from $N_n\setminus\{d_1\}$, $n-2$ ways to choose $a_2$ from $N_n\setminus\{d_1, d_2\}$, $n-2$ ways to choose $b_1$ from $N_n\setminus\{d_1, d_2\}$, $n-3$ ways to choose $c_2$ from $N_n\setminus\{c_1, d_2, a_2\}$.
Therefore, there are $n(n-1)(n-2)^2(n-3)=n^5-8n^4+23n^3-28n^2+12n$ ways to choose four $B_{n-2}$'s as required.
The probability $P_{4.9.3}$ that there are four fault-free $B_{n-2}$'s chosen as in Case 4.9.3 is $\frac{1}{2}(n^5-8n^4+23n^3-28n^2+12n)p^{4(n-2)!-(n-3)!}$.

{\it Case 4.9.4.} $c_1=a_1\neq b_1$, $c_2\neq d_2$ and $c_2=a_2$.

There are $n$ ways to choose $a_1(=d_1=c_1)$ from $N_n$, $n-1$ ways to choose $b_2(=d_2)$ from $N_n\setminus\{d_1\}$, $n-2$ ways to choose $a_2(=c_2)$ from $N_n\setminus\{d_1, d_2\}$, $n-2$ ways to choose $b_1$ from $N_n\setminus\{d_1, d_2\}$.
Therefore, there are $n(n-1)(n-2)^2=n^4-5n^3+8n^2-4n$ ways to choose four $B_{n-2}$'s as required.
The probability $P_{4.9.4}$ that there are four fault-free $B_{n-2}$'s chosen as in Case 4.9.4 is $\frac{1}{2}(n^4-5n^3+8n^2-4n)p^{4(n-2)!}$.

{\it Case 4.9.5.} $c_1=b_1\neq a_1$ and $c_2\neq d_2$.

There are $n$ ways to choose $a_1(=d_1)$ from $N_n$, $n-1$ ways to choose $b_2(=d_2)$ from $N_n\setminus\{d_1\}$, $n-2$ ways to choose $a_2$ from $N_n\setminus\{d_1, d_2\}$, $n-2$ ways to choose $b_1(=c_1)$ from $N_n\setminus\{d_1, d_2\}$, $n-2$ ways to choose $c_2$ from $N_n\setminus\{c_1, d_2\}$.
Therefore, there are $n(n-1)(n-2)^3=n^5-7n^4+18n^3-20n^2+8n$ ways to choose four $B_{n-2}$'s as required.
Clearly, in this case $c_2\neq b_2$.
The probability $P_{4.9.5}$ that there are four fault-free $B_{n-2}$'s chosen as in Case 4.9.5 is $\frac{1}{2}(n^5-7n^4+18n^3-20n^2+8n)p^{4(n-2)!-(n-3)!}$.

{\it Case 4.9.6.} $c_1\neq a_1$, $c_1\neq b_1$ and $c_2=d_2$.

There are $n$ ways to choose $a_1(=d_1)$ from $N_n$, $n-1$ ways to choose $b_2(=d_2=c_2)$ from $N_n\setminus\{d_1\}$, $n-2$ ways to choose $a_2$ from $N_n\setminus\{d_1, d_2\}$, $n-2$ ways to choose $b_1$ from $N_n\setminus\{d_1, d_2\}$, $n-3$ ways to choose $c_1$ from $N_n\setminus\{a_1, b_1, c_2\}$.
Therefore, there are $n(n-1)(n-2)^2(n-3)=n^5-8n^4+23n^3-28n^2+12n$ ways to choose four $B_{n-2}$'s as required.
Clearly, in this case $c_1\neq d_1$.
The probability $P_{4.9.6}$ that there are four fault-free $B_{n-2}$'s chosen as in Case 4.9.6 is $\frac{1}{2}(n^5-8n^4+23n^3-28n^2+12n)p^{4(n-2)!-(n-3)!}$.

{\it Case 4.9.7.} $c_1=a_1\neq b_1$, $c_2=d_2$ or $c_1=b_1\neq a_1$, $c_2=d_2$.

W.l.o.g, assume that the former applies. There are $n$ ways to choose $a_1(=d_1=c_1)$ from $N_n$, $n-1$ ways to choose $b_2(=d_2=c_2)$ from $N_n\setminus\{d_1\}$, $n-2$ ways to choose $a_2$ from $N_n\setminus\{d_1, d_2\}$, $n-2$ ways to choose $b_1$ from $N_n\setminus\{d_1, d_2\}$.
Therefore, there are $2n(n-1)(n-2)^2=2(n^4-5n^3+8n^2-4n)$ ways to choose four $B_{n-2}$'s as required.
Clearly, in this case $c_2\neq a_2$ and $c_1=d_1$.
The probability $P_{4.9.7}$ that there are four fault-free $B_{n-2}$'s chosen as in Case 4.9.7 is $\frac{1}{2}\times 2(n^4-5n^3+8n^2-4n)p^{4(n-2)!-(n-3)!}$.

Thus, the probability $P_{4.9}$ that there are four fault-free $B_{n-2}$'s chosen as in Case 4.9 is $2\sum_{i=1}^{7}P_{4.9.i}=(n^6-9n^5+34n^4-66n^3+64n^2-24n)p^{4(n-2)!}+3(n^5-7n^4+18n^3-20n^2+8n)p^{4(n-2)!-(n-3)!}$.

{\it Case 4.10.} $a_2=d_1$, $b_2=d_2$, $a_1\notin\{d_1, d_2\}$, $b_1\notin\{d_1, d_2\}$ or $a_2=d_2$, $b_2=d_1$, $a_1\notin\{d_1, d_2\}$, $b_1\notin\{d_1, d_2\}$.

W.l.o.g, assume that the former applies.

{\it Case 4.10.1.} $c_1\neq a_1$, $c_1\neq b_1$, $c_2\neq d_2$, $a_1\neq b_1$ and $c_1\neq d_2$.

There are $n$ ways to choose $a_2(=d_1)$ from $N_n$, $n-1$ ways to choose $b_2(=d_2)$ from $N_n\setminus\{d_1\}$, $n-2$ ways to choose $a_1$ from $N_n\setminus\{d_1, d_2\}$, $n-3$ ways to choose $b_1$ from $N_n\setminus\{d_1, d_2, a_1\}$, $n-3$ ways to choose $c_1$ from $N_n\setminus\{a_1, b_1, d_2\}$, $n-2$ ways to choose $c_2$ from $N_n\setminus\{c_1, d_2\}$.
Therefore, there are $n(n-1)(n-2)^2(n-3)^2=n^6-11n^5+47n^4-97n^3+96n^2-36n$ ways to choose four $B_{n-2}$'s as required.
The probability $P_{4.10.1}$ that there are four fault-free $B_{n-2}$'s chosen as in Case 4.10.1 is $\frac{1}{2}(n^6-11n^5+47n^4-97n^3+96n^2-36n)p^{4(n-2)!}$.

{\it Case 4.10.2.} $c_1\neq a_1$, $c_1\neq b_1$, $c_2\neq d_2$, $a_1\neq b_1$ and $c_1=d_2$.

There are $n$ ways to choose $a_2(=d_1)$ from $N_n$, $n-1$ ways to choose $b_2(=d_2=c_1)$ from $N_n\setminus\{d_1\}$, $n-2$ ways to choose $a_1$ from $N_n\setminus\{d_1, d_2\}$, $n-3$ ways to choose $b_1$ from $N_n\setminus\{d_1, d_2, a_1\}$, $n-1$ ways to choose $c_2$ from $N_n\setminus\{d_2\}$.
Therefore, there are $n(n-1)^2(n-2)(n-3)=n^5-7n^4+17n^3-17n^2+6n$ ways to choose four $B_{n-2}$'s as required.
The probability $P_{4.10.2}$ that there are four fault-free $B_{n-2}$'s chosen as in Case 4.10.2 is $\frac{1}{2}(n^5-7n^4+17n^3-17n^2+6n)p^{4(n-2)!}$.

{\it Case 4.10.3.} $c_1\neq a_1=b_1$, $c_2\neq d_2$ and $c_1\neq d_2$.

There are $n$ ways to choose $a_2(=d_1)$ from $N_n$, $n-1$ ways to choose $b_2(=d_2)$ from $N_n\setminus\{d_1\}$, $n-2$ ways to choose $a_1(=b_1)$ from $N_n\setminus\{d_1, d_2\}$, $n-2$ ways to choose $c_1$ from $N_n\setminus\{a_1, d_2\}$, $n-2$ ways to choose $c_2$ from $N_n\setminus\{c_1, d_2\}$.
Therefore, there are $n(n-1)(n-2)^3=n^5-7n^4+18n^3-20n^2+8n$ ways to choose four $B_{n-2}$'s as required.
The probability $P_{4.10.3}$ that there are four fault-free $B_{n-2}$'s chosen as in Case 4.10.3 is $\frac{1}{2}(n^5-7n^4+18n^3-20n^2+8n)p^{4(n-2)!}$.

{\it Case 4.10.4.} $c_1\neq a_1=b_1$, $c_2\neq d_2$ and $c_1=d_2$.

There are $n$ ways to choose $a_2(=d_1)$ from $N_n$, $n-1$ ways to choose $b_2(=d_2=c_1)$ from $N_n\setminus\{d_1\}$, $n-2$ ways to choose $a_1(=b_1)$ from $N_n\setminus\{d_1, d_2\}$,  $n-1$ ways to choose $c_2$ from $N_n\setminus\{d_2\}$.
Therefore, there are $n(n-1)^2(n-2)=n^4-4n^3+5n^2-2n$ ways to choose four $B_{n-2}$'s as required.
The probability $P_{4.10.4}$ that there are four fault-free $B_{n-2}$'s chosen as in Case 4.10.4 is $\frac{1}{2}(n^4-4n^3+5n^2-2n)p^{4(n-2)!}$.

{\it Case 4.10.5.} $c_1=a_1\neq b_1$, $c_2\neq d_2$ and $c_2\neq a_2$.

There are $n$ ways to choose $a_2(=d_1)$ from $N_n$, $n-1$ ways to choose $b_2(=d_2)$ from $N_n\setminus\{d_1\}$, $n-2$ ways to choose $a_1(=c_1)$ from $N_n\setminus\{d_1, d_2\}$, $n-3$ ways to choose $b_1$ from $N_n\setminus\{d_1, d_2, a_1\}$, $n-3$ ways to choose $c_2$ from $N_n\setminus\{c_1, d_2, a_2\}$.
Therefore, there are $n(n-1)(n-2)(n-3)^2=n^5-9n^4+29n^3-39n^2+18n$ ways to choose four $B_{n-2}$'s as required.
The probability $P_{4.10.5}$ that there are four fault-free $B_{n-2}$'s chosen as in Case 4.10.5 is $\frac{1}{2}(n^5-9n^4+29n^3-39n^2+18n)p^{4(n-2)!-(n-3)!}$.

{\it Case 4.10.6.} $c_1=a_1\neq b_1$, $c_2\neq d_2$ and $c_2=a_2$.

There are $n$ ways to choose $a_2(=d_1=c_2)$ from $N_n$, $n-1$ ways to choose $b_2(=d_2)$ from $N_n\setminus\{d_1\}$, $n-2$ ways to choose $a_1(=c_1)$ from $N_n\setminus\{d_1, d_2\}$, $n-3$ ways to choose $b_1$ from $N_n\setminus\{d_1, d_2, a_1\}$.
Therefore, there are $n(n-1)(n-2)(n-3)=n^4-6n^3+11n^2-6n$ ways to choose four $B_{n-2}$'s as required.
The probability $P_{4.10.6}$ that there are four fault-free $B_{n-2}$'s chosen as in Case 4.10.6 is $\frac{1}{2}(n^4-6n^3+11n^2-6n)p^{4(n-2)!}$.

{\it Case 4.10.7.} $c_1=b_1\neq a_1$, $c_2\neq d_2$.

There are $n$ ways to choose $a_2(=d_1)$ from $N_n$, $n-1$ ways to choose $b_2(=d_2)$ from $N_n\setminus\{d_1\}$, $n-2$ ways to choose $a_1$ from $N_n\setminus\{d_1, d_2\}$, $n-3$ ways to choose $b_1(=c_1)$ from $N_n\setminus\{d_1, d_2, a_1\}$, $n-2$ ways to choose $c_2$ from $N_n\setminus\{c_1, d_2\}$.
Therefore, there are $n(n-1)(n-2)^2(n-3)=n^5-8n^4+23n^3-28n^2+12n$ ways to choose four $B_{n-2}$'s as required.
Clearly, in this case $c_2\neq b_2$.
The probability $P_{4.10.7}$ that there are four fault-free $B_{n-2}$'s chosen as in Case 4.10.7 is $\frac{1}{2}(n^5-8n^4+23n^3-28n^2+12n)p^{4(n-2)!-(n-3)!}$.

{\it Case 4.10.8.} $c_1=a_1=b_1$, $c_2\neq d_2$ and $c_2\neq a_2$.

There are $n$ ways to choose $a_2(=d_1)$ from $N_n$, $n-1$ ways to choose $b_2(=d_2)$ from $N_n\setminus\{d_1\}$, $n-2$ ways to choose $a_1(=b_1=c_1)$ from $N_n\setminus\{d_1, d_2\}$, $n-3$ ways to choose $c_2$ from $N_n\setminus\{c_1, d_2, a_2\}$.
Therefore, there are $n(n-1)(n-2)(n-3)=n^4-6n^3+11n^2-6n$ ways to choose four $B_{n-2}$'s as required.
Clearly, in this case $c_2\neq b_2$.
The probability $P_{4.10.8}$ that there are four fault-free $B_{n-2}$'s chosen as in Case 4.10.8 is $\frac{1}{2}(n^4-6n^3+11n^2-6n)p^{4(n-2)!-2(n-3)!}$.

{\it Case 4.10.9.} $c_1=a_1=b_1$, $c_2\neq d_2$ and $c_2=a_2$.

There are $n$ ways to choose $a_2(=d_1=c_2)$ from $N_n$, $n-1$ ways to choose $b_2(=d_2)$ from $N_n\setminus\{d_1\}$, $n-2$ ways to choose $a_1(=b_1=c_1)$ from $N_n\setminus\{d_1, d_2\}$.
Therefore, there are $n(n-1)(n-2)=n^3-3n^2+2n$ ways to choose four $B_{n-2}$'s as required.
Clearly, in this case $c_2\neq b_2$.
The probability $P_{4.10.9}$ that there are four fault-free $B_{n-2}$'s chosen as in Case 4.10.9 is $\frac{1}{2}(n^3-3n^2+2n)p^{4(n-2)!-(n-3)!}$.

{\it Case 4.10.10.} $c_1\neq a_1$, $c_1\neq b_1$, $c_2=d_2$, $a_1\neq b_1$ and $c_1\neq d_1$.

There are $n$ ways to choose $a_2(=d_1)$ from $N_n$, $n-1$ ways to choose $b_2(=d_2=c_2)$ from $N_n\setminus\{d_1\}$, $n-2$ ways to choose $a_1$ from $N_n\setminus\{d_1, d_2\}$, $n-3$ ways to choose $b_1$ from $N_n\setminus\{d_1, d_2, a_1\}$, $n-4$ ways to choose $c_1$ from $N_n\setminus\{a_1, b_1, c_2, d_1\}$.
Therefore, there are $n(n-1)(n-2)(n-3)(n-4)=n^5-10n^4+35n^3-50n^2+24n$ ways to choose four $B_{n-2}$'s as required.
The probability $P_{4.10.10}$ that there are four fault-free $B_{n-2}$'s chosen as in Case 4.10.10 is $\frac{1}{2}(n^5-10n^4+35n^3-50n^2+24n)p^{4(n-2)!-(n-3)!}$.

{\it Case 4.10.11.} $c_1\neq a_1$, $c_1\neq b_1$, $c_2=d_2$, $a_1\neq b_1$ and $c_1=d_1$.

There are $n$ ways to choose $a_2(=d_1=c_1)$ from $N_n$, $n-1$ ways to choose $b_2(=d_2=c_2)$ from $N_n\setminus\{d_1\}$, $n-2$ ways to choose $a_1$ from $N_n\setminus\{d_1, d_2\}$, $n-3$ ways to choose $b_1$ from $N_n\setminus\{d_1, d_2, a_1\}$.
Therefore, there are $n(n-1)(n-2)(n-3)=n^4-6n^3+11n^2-6n$ ways to choose four $B_{n-2}$'s as required.
The probability $P_{4.10.11}$ that there are four fault-free $B_{n-2}$'s chosen as in Case 4.10.11 is $\frac{1}{2}(n^4-6n^3+11n^2-6n)p^{4(n-2)!}$.

{\it Case 4.10.12} $c_1\neq a_1=b_1$ $c_2=d_2$ and $c_1\neq d_1$.

There are $n$ ways to choose $a_2(=d_1)$ from $N_n$, $n-1$ ways to choose $b_2(=d_2=c_2)$ from $N_n\setminus\{d_1\}$, $n-2$ ways to choose $a_1(=b_1)$ from $N_n\setminus\{d_1, d_2\}$, $n-3$ ways to choose $c_1$ from $N_n\setminus\{a_1, c_2, d_1\}$.
Therefore, there are $n(n-1)(n-2)(n-3)=n^4-6n^3+11n^2-6n$ ways to choose four $B_{n-2}$'s as required.
The probability $P_{4.10.12}$ that there are four fault-free $B_{n-2}$'s chosen as in Case 4.10.12 is $\frac{1}{2}(n^4-6n^3+11n^2-6n)p^{4(n-2)!-(n-3)!}$.

{\it Case 4.10.13} $c_1\neq a_1=b_1$ $c_2=d_2$ and $c_1=d_1$.

There are $n$ ways to choose $a_2(=d_1=c_1)$ from $N_n$, $n-1$ ways to choose $b_2(=d_2=c_2)$ from $N_n\setminus\{d_1\}$, $n-2$ ways to choose $a_1(=b_1)$ from $N_n\setminus\{d_1, d_2\}$.
Therefore, there are $n(n-1)(n-2)=n^3-3n^2+2n$ ways to choose four $B_{n-2}$'s as required.
The probability $P_{4.10.13}$ that there are four fault-free $B_{n-2}$'s chosen as in Case 4.10.13 is $\frac{1}{2}(n^3-3n^2+2n)p^{4(n-2)!}$.

{\it Case 4.10.14.} $c_1=a_1\neq b_1$ and $c_2=d_2$.

There are $n$ ways to choose $a_2(=d_1)$ from $N_n$, $n-1$ ways to choose $b_2(=d_2=c_2)$ from $N_n\setminus\{d_1\}$, $n-2$ ways to choose $a_1(=c_1)$ from $N_n\setminus\{d_1, d_2\}$, $n-3$ ways to choose $b_1$ from $N_n\setminus\{d_1, d_2, a_1\}$.
Therefore, there are $n(n-1)(n-2)(n-3)=n^4-6n^3+11n^2-6n$ ways to choose four $B_{n-2}$'s as required.
Clearly, in this case $c_2\neq a_2$ and $c_1\neq d_1$.
The probability $P_{4.10.14}$ that there are four fault-free $B_{n-2}$'s chosen as in Case 4.10.14 is $\frac{1}{2}(n^4-6n^3+11n^2-6n)p^{4(n-2)!-2(n-3)!}$.

{\it Case 4.10.15.} $c_1=b_1\neq a_1$ and $c_2=d_2$.

There are $n$ ways to choose $a_2(=d_1)$ from $N_n$, $n-1$ ways to choose $b_2(=d_2=c_2)$ from $N_n\setminus\{d_1\}$, $n-2$ ways to choose $a_1$ from $N_n\setminus\{d_1, d_2\}$, $n-3$ ways to choose $b_1(=c_1)$ from $N_n\setminus\{d_1, d_2, a_1\}$.
Therefore, there are $n(n-1)(n-2)(n-3)=n^4-6n^3+11n^2-6n$ ways to choose four $B_{n-2}$'s as required.
Clearly, in this case $c_2=b_2$ and $c_1\neq d_1$.
The probability $P_{4.10.15}$ that there are four fault-free $B_{n-2}$'s chosen as in Case 4.10.15 is $\frac{1}{2}(n^4-6n^3+11n^2-6n)p^{4(n-2)!-(n-3)!}$.

{\it Case 4.10.16.} $c_1=a_1=b_1$ and $c_2\neq d_2$.

There are $n$ ways to choose $a_2(=d_1)$ from $N_n$, $n-1$ ways to choose $b_2(=d_2=c_2)$ from $N_n\setminus\{d_1\}$, $n-2$ ways to choose $a_1(=b_1=c_1)$ from $N_n\setminus\{d_1, d_2\}$.
Therefore, there are $n(n-1)(n-2)=n^3-3n^2+2n$ ways to choose four $B_{n-2}$'s as required.
Clearly, in this case $c_2\neq a_2$, $c_2=b_2$ and $c_1\neq d_1$.
The probability $P_{4.10.16}$ that there are four fault-free $B_{n-2}$'s chosen as in Case 4.10.16 is $\frac{1}{2}(n^3-3n^2+2n)p^{4(n-2)!-2(n-3)!}$.

Thus, the probability $P_{4.10}$ that there are four fault-free $B_{n-2}$'s chosen as in Case 4.10 is $2\sum_{i=1}^{16}P_{4.10.i}=(n^6-9n^5+36n^4-77n^3+83n^2-34n)p^{4(n-2)!}+(3n^5-25n^4+76n^3-98n^2+44n)p^{4(n-2)!-(n-3)!}+(2n^4-11n^3+19n^2-10n)p^{4(n-2)!-2(n-3)!}$.

In summary, the probability $P_4$ that there are four fault-free $B_{n-2}$'s chosen as in Case 4 is $\sum_{i=1}^{10}P_{4.i}=(8n^6-73n^5+282n^4-566n^3+571n^2-222n)p^{4(n-2)!}+\frac{1}{2}(44n^5-331n^4+910n^3-1073n^2+450n)p^{4(n-2)!-(n-3)!}+\frac{1}{2}(2n^5-9n^4+10n^3+3n^2-6n)p^{4(n-2)!-2(n-3)!}$.

{\it Case 5.} $|\{a_1, a_2\}\cap\{d_1, d_2\}|=1$, $\{b_1, b_2\}\cap\{d_1, d_2\}=\emptyset$ or $\{a_1, a_2\}\cap\{d_1, d_2\}=\emptyset$, $|\{b_1, b_2\}\cap\{d_1, d_2\}|=1$.

W.l.o.g, assume that the former applies.

Observation \ref{ob:4} implies that $a_1 a_2 X^{n-2}$ and $X^{n-2} d_1 d_2$ are disjoint, and $b_1 b_2 X^{n-2}$ and $X^{n-2} d_1 d_2$ are not disjoint and $V(b_1 b_2 X^{n-2})\cap V(X^{n-2} d_1 d_2)=V(b_1 b_2 X^{n-4} d_1 d_2)$ in this case.

{\it Case 5.1.} $a_1=d_1$, $a_2\notin\{d_1, d_2\}$ and $\{b_1, b_2\}\cap\{d_1, d_2\}=\emptyset$.

{\it Case 5.1.1.} $c_1\neq a_1$, $c_1\neq b_1$, $c_2\neq d_2$ and $c_1\neq d_2$.

There are $n$ ways to choose $a_1(=d_1)$ from $N_n$, $n-1$ ways to choose $d_2$ from $N_n\setminus\{d_1\}$, $n-2$ ways to choose $a_2$ from $N_n\setminus\{d_1, d_2\}$, $n-2$ ways to choose $b_1$ from $N_n\setminus\{d_1, d_2\}$, $n-3$ ways to choose $b_2$ from $N_n\setminus\{d_1, d_2, b_1\}$, $n-3$ ways to choose $c_1$ from $N_n\setminus\{a_1, b_1, d_2\}$, $n-2$ ways to choose $c_2$ from $N_n\setminus\{c_1, d_2\}$.
Therefore, there are $n(n-1)(n-2)^3(n-3)^2=n^7-13n^6+69n^5-191n^4+290n^3-228n^2+72n$ ways to choose four $B_{n-2}$'s as required.
The probability $P_{5.1.1}$ that there are four fault-free $B_{n-2}$'s chosen as in Case 5.1.1 is $\frac{1}{2}(n^7-13n^6+69n^5-191n^4+290n^3-228n^2+72n)p^{4(n-2)!-(n-4)!}$.

{\it Case 5.1.2.} $c_1\neq a_1$, $c_1\neq b_1$, $c_2\neq d_2$ and $c_1=d_2$.

There are $n$ ways to choose $a_1(=d_1)$ from $N_n$, $n-1$ ways to choose $d_2(=c_1)$ from $N_n\setminus\{d_1\}$, $n-2$ ways to choose $a_2$ from $N_n\setminus\{d_1, d_2\}$, $n-2$ ways to choose $b_1$ from $N_n\setminus\{d_1, d_2\}$, $n-3$ ways to choose $b_2$ from $N_n\setminus\{d_1, d_2, b_1\}$, $n-1$ ways to choose $c_2$ from $N_n\setminus\{d_2\}$.
Therefore, there are $n(n-1)^2(n-2)^2(n-3)=n^6-9n^5+31n^4-51n^3+40n^2-12n$ ways to choose four $B_{n-2}$'s as required.
The probability $P_{5.1.2}$ that there are four fault-free $B_{n-2}$'s chosen as in Case 5.1.2 is $\frac{1}{2}(n^6-9n^5+31n^4-51n^3+40n^2-12n)p^{4(n-2)!-(n-4)!}$.

{\it Case 5.1.3.} $c_1=a_1\neq b_1$, $c_2\neq d_2$, $c_2\neq a_2$ or $c_1=b_1\neq a_1$, $c_2\neq d_2$, $c_2\neq b_2$.

W.l.o.g, assume that the former applies. There are $n$ ways to choose $a_1(=d_1=c_1)$ from $N_n$, $n-1$ ways to choose $d_2$ from $N_n\setminus\{d_1\}$, $n-2$ ways to choose $a_2$ from $N_n\setminus\{d_1, d_2\}$, $n-2$ ways to choose $b_1$ from $N_n\setminus\{d_1, d_2\}$, $n-3$ ways to choose $b_2$ from $N_n\setminus\{d_1, d_2, b_1\}$, $n-3$ ways to choose $c_2$ from $N_n\setminus\{d_2, c_1, a_2\}$.
Therefore, there are $2n(n-1)(n-2)^2(n-3)^2=2(n^6-11n^5+47n^4-97n^3+96n^2-36n)$ ways to choose four $B_{n-2}$'s as required.
The probability $P_{5.1.3}$ that there are four fault-free $B_{n-2}$'s chosen as in Case 5.1.3 is $\frac{1}{2}\times 2(n^6-11n^5+47n^4-97n^3+96n^2-36n)p^{4(n-2)!-(n-3)!-(n-4)!}$.

{\it Case 5.1.4.} $c_1=a_1\neq b_1$, $c_2\neq d_2$, $c_2=a_2$ or $c_1=b_1\neq a_1$, $c_2\neq d_2$, $c_2=b_2$.

W.l.o.g, assume that the former applies. There are $n$ ways to choose $a_1(=d_1=c_1)$ from $N_n$, $n-1$ ways to choose $d_2$ from $N_n\setminus\{d_1\}$, $n-2$ ways to choose $a_2(=c_2)$ from $N_n\setminus\{d_1, d_2\}$, $n-2$ ways to choose $b_1$ from $N_n\setminus\{d_1, d_2\}$, $n-3$ ways to choose $b_2$ from $N_n\setminus\{d_1, d_2, b_1\}$.
Therefore, there are $2n(n-1)(n-2)^2(n-3)=2(n^5-8n^4+23n^3-28n^2+12n)$ ways to choose four $B_{n-2}$'s as required.
The probability $P_{5.1.4}$ that there are four fault-free $B_{n-2}$'s chosen as in Case 5.1.4 is $\frac{1}{2}\times 2(n^5-8n^4+23n^3-28n^2+12n)p^{4(n-2)!-(n-4)!}$.

{\it Case 5.1.5.} $c_1\neq a_1$, $c_1\neq b_1$ and $c_2=d_2$.

There are $n$ ways to choose $a_1(=d_1)$ from $N_n$, $n-1$ ways to choose $d_2(=c_2)$ from $N_n\setminus\{d_1\}$, $n-2$ ways to choose $a_2$ from $N_n\setminus\{d_1, d_2\}$, $n-2$ ways to choose $b_1$ from $N_n\setminus\{d_1, d_2\}$, $n-3$ ways to choose $b_2$ from $N_n\setminus\{d_1, d_2, b_1\}$, $n-3$ ways to choose $c_1$ from $N_n\setminus\{a_1, b_1, c_2\}$.
Therefore, there are $n(n-1)(n-2)^2(n-3)^2=n^6-11n^5+47n^4-97n^3+96n^2-36n$ ways to choose four $B_{n-2}$'s as required.
Clearly, in this case $c_1\neq d_1$.
The probability $P_{5.1.5}$ that there are four fault-free $B_{n-2}$'s chosen as in Case 5.1.5 is $\frac{1}{2}(n^6-11n^5+47n^4-97n^3+96n^2-36n)p^{4(n-2)!-(n-3)!-(n-4)!}$.

{\it Case 5.1.6.} $c_1=a_1\neq b_1$ and $c_2=d_2$.

There are $n$ ways to choose $a_1(=d_1=c_1)$ from $N_n$, $n-1$ ways to choose $d_2(=c_2)$ from $N_n\setminus\{d_1\}$, $n-2$ ways to choose $a_2$ from $N_n\setminus\{d_1, d_2\}$, $n-2$ ways to choose $b_1$ from $N_n\setminus\{d_1, d_2\}$, $n-3$ ways to choose $b_2$ from $N_n\setminus\{d_1, d_2, b_1\}$.
Therefore, there are $n(n-1)(n-2)^2(n-3)=n^5-8n^4+23n^3-28n^2+12n$ ways to choose four $B_{n-2}$'s as required.
Clearly, in this case $c_2\neq a_2$ and $c_1=d_1$.
The probability $P_{5.1.6}$ that there are four fault-free $B_{n-2}$'s chosen as in Case 5.1.6 is $\frac{1}{2}(n^5-8n^4+23n^3-28n^2+12n)p^{4(n-2)!-(n-3)!-(n-4)!}$.

{\it Case 5.1.7.} $c_1=b_1\neq a_1$ and $c_2=d_2$.

There are $n$ ways to choose $a_1(=d_1=c_1)$ from $N_n$, $n-1$ ways to choose $d_2(=c_2)$ from $N_n\setminus\{d_1\}$, $n-2$ ways to choose $a_2$ from $N_n\setminus\{d_1, d_2\}$, $n-2$ ways to choose $b_1$ from $N_n\setminus\{d_1, d_2\}$, $n-3$ ways to choose $b_2$ from $N_n\setminus\{d_1, d_2, b_1\}$.
Therefore, there are $n(n-1)(n-2)^2(n-3)=n^5-8n^4+23n^3-28n^2+12n$ ways to choose four $B_{n-2}$'s as required.
Clearly, in this case $c_2\neq b_2$ and $c_1\neq d_1$.
The probability $P_{5.1.7}$ that there are four fault-free $B_{n-2}$'s chosen as in Case 5.1.7 is $\frac{1}{2}(n^5-8n^4+23n^3-28n^2+12n)p^{4(n-2)!-2(n-3)!-(n-4)!}$.

Thus, the probability $P_{5.1}$ that there are four fault-free $B_{n-2}$'s chosen as in Case 5.1 is $\sum_{i=1}^{7}P_{5.1.i}=\frac{1}{2}(n^7-12n^6+62n^5-176n^4+285n^3-244n^2+84n)p^{4(n-2)!-(n-4)!}+\frac{1}{2}(3n^6-32n^5+133n^4-268n^3+260n^2-96n)p^{4(n-2)!-(n-3)!-(n-4)!}+\frac{1}{2}(n^5-8n^4+23n^3-28n^2+12n)p^{4(n-2)!-2(n-3)!-(n-4)!}$.

{\it Case 5.2.} $a_1=d_2$, $a_2=\notin\{d_1, d_2\}$ and $\{b_1, b_2\}\cap\{d_1, d_2\}=\emptyset$.

{\it Case 5.2.1.} $c_1\neq a_1$, $c_1\neq b_1$ and $c_2\neq d_2$.

There are $n$ ways to choose $a_1(=d_2)$ from $N_n$, $n-1$ ways to choose $d_1$ from $N_n\setminus\{d_2\}$, $n-2$ ways to choose $a_2$ from $N_n\setminus\{d_1, d_2\}$, $n-2$ ways to choose $b_1$ from $N_n\setminus\{d_1, d_2\}$, $n-3$ ways to choose $b_2$ from $N_n\setminus\{d_1, d_2, b_1\}$, $n-2$ ways to choose $c_1$ from $N_n\setminus\{a_1, b_1\}$, $n-2$ ways to choose $c_2$ from $N_n\setminus\{c_1, d_2\}$.
Therefore, there are $n(n-1)(n-2)^4(n-3)=n^7-12n^6+59n^5-152n^4+216n^3-160n^2+48n$ ways to choose four $B_{n-2}$'s as required.
The probability $P_{5.2.1}$ that there are four fault-free $B_{n-2}$'s chosen as in Case 5.2.1 is $\frac{1}{2}(n^7-12n^6+59n^5-152n^4+216n^3-160n^2+48n)p^{4(n-2)!-(n-4)!}$.

{\it Case 5.2.2.} $c_1=a_1\neq b_1$, $c_2\neq d_2$ and $c_2\neq a_2$.

There are $n$ ways to choose $a_1(=d_2=c_1)$ from $N_n$, $n-1$ ways to choose $d_1$ from $N_n\setminus\{d_2\}$, $n-2$ ways to choose $a_2$ from $N_n\setminus\{d_1, d_2\}$, $n-2$ ways to choose $b_1$ from $N_n\setminus\{d_1, d_2\}$, $n-3$ ways to choose $b_2$ from $N_n\setminus\{d_1, d_2, b_1\}$, $n-2$ ways to choose $c_2$ from $N_n\setminus\{d_2, a_2\}$.
Therefore, there are $n(n-1)(n-2)^3(n-3)=n^6-10n^5+39n^4-74n^3+68n^2-24n$ ways to choose four $B_{n-2}$'s as required.
The probability $P_{5.2.2}$ that there are four fault-free $B_{n-2}$'s chosen as in Case 5.2.2 is $\frac{1}{2}(n^6-10n^5+39n^4-74n^3+68n^2-24n)p^{4(n-2)!-(n-3)!-(n-4)!}$.

{\it Case 5.2.3.} $c_1=a_1\neq b_1$, $c_2\neq d_2$ and $c_2=a_2$.

There are $n$ ways to choose $a_1(=d_2=c_1)$ from $N_n$, $n-1$ ways to choose $d_1$ from $N_n\setminus\{d_2\}$, $n-2$ ways to choose $a_2(=c_2)$ from $N_n\setminus\{d_1, d_2\}$, $n-2$ ways to choose $b_1$ from $N_n\setminus\{d_1, d_2\}$, $n-3$ ways to choose $b_2$ from $N_n\setminus\{d_1, d_2, b_1\}$.
Therefore, there are $n(n-1)(n-2)^2(n-3)=n^5-8n^4+23n^3-28n^2+12n$ ways to choose four $B_{n-2}$'s as required.
The probability $P_{5.2.3}$ that there are four fault-free $B_{n-2}$'s chosen as in Case 5.2.3 is $\frac{1}{2}(n^5-8n^4+23n^3-28n^2+12n)p^{4(n-2)!-(n-4)!}$.

{\it Case 5.2.4.} $c_1=b_1\neq a_1$, $c_2\neq d_2$ and $c_2\neq b_2$.

There are $n$ ways to choose $a_1(=d_2)$ from $N_n$, $n-1$ ways to choose $d_1$ from $N_n\setminus\{d_2\}$, $n-2$ ways to choose $a_2$ from $N_n\setminus\{d_1, d_2\}$, $n-2$ ways to choose $b_1(=c_1)$ from $N_n\setminus\{d_1, d_2\}$, $n-3$ ways to choose $b_2$ from $N_n\setminus\{d_1, d_2, b_1\}$, $n-3$ ways to choose $c_2$ from $N_n\setminus\{d_2, c_1, b_2\}$.
Therefore, there are $n(n-1)(n-2)^2(n-3)^2=n^6-11n^5+47n^4-97n^3+96n^2-36n$ ways to choose four $B_{n-2}$'s as required.
The probability $P_{5.2.4}$ that there are four fault-free $B_{n-2}$'s chosen as in Case 5.2.4 is $\frac{1}{2}(n^6-11n^5+47n^4-97n^3+96n^2-36n)p^{4(n-2)!-(n-3)!-(n-4)!}$.

{\it Case 5.2.5.} $c_1=b_1\neq a_1$, $c_2\neq d_2$ and $c_2=b_2$.

There are $n$ ways to choose $a_1(=d_2)$ from $N_n$, $n-1$ ways to choose $d_1$ from $N_n\setminus\{d_2\}$, $n-2$ ways to choose $a_2$ from $N_n\setminus\{d_1, d_2\}$, $n-2$ ways to choose $b_1(=c_1)$ from $N_n\setminus\{d_1, d_2\}$, $n-3$ ways to choose $b_2(=c_2)$ from $N_n\setminus\{d_1, d_2, b_1\}$.
Therefore, there are $n(n-1)(n-2)^2(n-3)=n^5-8n^4+23n^3-28n^2+12n$ ways to choose four $B_{n-2}$'s as required.
The probability $P_{5.2.5}$ that there are four fault-free $B_{n-2}$'s chosen as in Case 5.2.5 is $\frac{1}{2}(n^5-8n^4+23n^3-28n^2+12n)p^{4(n-2)!-(n-4)!}$.

{\it Case 5.2.6.} $c_1\neq a_1$, $c_1\neq b_1$, $c_2=d_2$ and $c_1\neq d_1$.

There are $n$ ways to choose $a_1(=d_2=c_2)$ from $N_n$, $n-1$ ways to choose $d_1$ from $N_n\setminus\{d_2\}$, $n-2$ ways to choose $a_2$ from $N_n\setminus\{d_1, d_2\}$, $n-2$ ways to choose $b_1$ from $N_n\setminus\{d_1, d_2\}$, $n-3$ ways to choose $b_2$ from $N_n\setminus\{d_1, d_2, b_1\}$, $n-3$ ways to choose $c_1$ from $N_n\setminus\{a_1, b_1, d_1\}$.
Therefore, there are $n(n-1)(n-2)^2(n-3)^2=n^6-11n^5+47n^4-97n^3+96n^2-36n$ ways to choose four $B_{n-2}$'s as required.
The probability $P_{5.2.6}$ that there are four fault-free $B_{n-2}$'s chosen as in Case 5.2.6 is $\frac{1}{2}(n^6-11n^5+47n^4-97n^3+96n^2-36n)p^{4(n-2)!-(n-3)!-(n-4)!}$.

{\it Case 5.2.7.} $c_1\neq a_1$, $c_1\neq b_1$, $c_2=d_2$ and $c_1=d_1$.

There are $n$ ways to choose $a_1(=d_2=c_2)$ from $N_n$, $n-1$ ways to choose $d_1(=c_1)$ from $N_n\setminus\{d_2\}$, $n-2$ ways to choose $a_2$ from $N_n\setminus\{d_1, d_2\}$, $n-2$ ways to choose $b_1$ from $N_n\setminus\{d_1, d_2\}$, $n-3$ ways to choose $b_2$ from $N_n\setminus\{d_1, d_2, b_1\}$.
Therefore, there are $n(n-1)(n-2)^2(n-3)=n^5-8n^4+23n^3-28n^2+12n$ ways to choose four $B_{n-2}$'s as required.
The probability $P_{5.2.7}$ that there are four fault-free $B_{n-2}$'s chosen as in Case 5.2.7 is $\frac{1}{2}(n^5-8n^4+23n^3-28n^2+12n)p^{4(n-2)!-(n-4)!}$.

{\it Case 5.2.8.} $c_1=b_1\neq a_1$ and $c_2=d_2$.

There are $n$ ways to choose $a_1(=d_2=c_2)$ from $N_n$, $n-1$ ways to choose $d_1$ from $N_n\setminus\{d_2\}$, $n-2$ ways to choose $a_2$ from $N_n\setminus\{d_1, d_2\}$, $n-2$ ways to choose $b_1(=c_1)$ from $N_n\setminus\{d_1, d_2\}$, $n-3$ ways to choose $b_2$ from $N_n\setminus\{d_1, d_2, b_1\}$.
Therefore, there are $n(n-1)(n-2)^2(n-3)=n^5-8n^4+23n^3-28n^2+12n$ ways to choose four $B_{n-2}$'s as required.
Clearly, in this case $c_2\neq b_2$ and $c_1\neq d_1$.
The probability $P_{5.2.8}$ that there are four fault-free $B_{n-2}$'s chosen as in Case 5.2.8 is $\frac{1}{2}(n^5-8n^4+23n^3-28n^2+12n)p^{4(n-2)!-2(n-3)!-(n-4)!}$.

Thus, the probability $P_{5.2}$ that there are four fault-free $B_{n-2}$'s chosen as in Case 5.2 is $\sum_{i=1}^{8}P_{5.2.i}=\frac{1}{2}(n^7-12n^6+62n^5-176n^4+285n^3-244n^2+84n)p^{4(n-2)!-(n-4)!}+\frac{1}{2}(3n^6-32n^5+133n^4-268n^3+260n^2-96n)p^{4(n-2)!-(n-3)!-(n-4)!}+\frac{1}{2}(n^5-8n^4+23n^3-28n^2+12n)p^{4(n-2)!-2(n-3)!-(n-4)!}$.

{\it Case 5.3.} $a_2=d_1$, $a_1\notin\{d_1, d_2\}$, $\{b_1, b_2\}\cap\{d_1, d_2\}=\emptyset$.

{\it Case 5.3.1.} $c_1\neq a_1$, $c_1\neq b_1$, $c_2\neq d_2$, $a_1\neq b_1$ and $c_1\neq d_2$.

There are $n$ ways to choose $a_2(=d_1)$ from $N_n$, $n-1$ ways to choose $d_2$ from $N_n\setminus\{d_1\}$, $n-2$ ways to choose $a_1$ from $N_n\setminus\{d_1, d_2\}$, $n-3$ ways to choose $b_1$ from $N_n\setminus\{d_1, d_2, a_1\}$, $n-3$ ways to choose $b_2$ from $N_n\setminus\{d_1, d_2, b_1\}$, $n-3$ ways to choose $c_1$ from $N_n\setminus\{a_1, b_1, d_2\}$, $n-2$ ways to choose $c_2$ from $N_n\setminus\{c_1, d_2\}$.
Therefore, there are $n(n-1)(n-2)^2(n-3)^3=n^7-14n^6+80n^5-238n^4+387n^3-324n^2+108n$ ways to choose four $B_{n-2}$'s as required.
The probability $P_{5.3.1}$ that there are four fault-free $B_{n-2}$'s chosen as in Case 5.3.1 is $\frac{1}{2}(n^7-14n^6+80n^5-238n^4+387n^3-324n^2+108n)p^{4(n-2)!-(n-4)!}$.

{\it Case 5.3.2.} $c_1\neq a_1$, $c_1\neq b_1$, $c_2\neq d_2$, $a_1\neq b_1$ and $c_1=d_2$.

There are $n$ ways to choose $a_2(=d_1)$ from $N_n$, $n-1$ ways to choose $d_2(=c_1)$ from $N_n\setminus\{d_1\}$, $n-2$ ways to choose $a_1$ from $N_n\setminus\{d_1, d_2\}$, $n-3$ ways to choose $b_1$ from $N_n\setminus\{d_1, d_2, a_1\}$, $n-3$ ways to choose $b_2$ from $N_n\setminus\{d_1, d_2, b_1\}$, $n-1$ ways to choose $c_2$ from $N_n\setminus\{d_2\}$.
Therefore, there are $n(n-1)^2(n-2)(n-3)^2=n^6-10n^5+38n^4-68n^3+57n^2-18n$ ways to choose four $B_{n-2}$'s as required.
The probability $P_{5.3.2}$ that there are four fault-free $B_{n-2}$'s chosen as in Case 5.3.2 is $\frac{1}{2}(n^6-10n^5+38n^4-68n^3+57n^2-18n)p^{4(n-2)!-(n-4)!}$.

{\it Case 5.3.3.} $c_1\neq a_1=b_1$, $c_2\neq d_2$ and $c_1\neq d_2$.

There are $n$ ways to choose $a_2(=d_1)$ from $N_n$, $n-1$ ways to choose $d_2$ from $N_n\setminus\{d_1\}$, $n-2$ ways to choose $a_1(=b_1)$ from $N_n\setminus\{d_1, d_2\}$, $n-3$ ways to choose $b_2$ from $N_n\setminus\{d_1, d_2, b_1\}$, $n-2$ ways to choose $c_1$ from $N_n\setminus\{a_1, d_2\}$, $n-2$ ways to choose $c_2$ from $N_n\setminus\{c_1, d_2\}$.
Therefore, there are $n(n-1)(n-2)^3(n-3)=n^6-10n^5+39n^4-74n^3+68n^2-24n$ ways to choose four $B_{n-2}$'s as required.
The probability $P_{5.3.3}$ that there are four fault-free $B_{n-2}$'s chosen as in Case 5.3.3 is $\frac{1}{2}(n^6-10n^5+39n^4-74n^3+68n^2-24n)p^{4(n-2)!-(n-4)!}$.

{\it Case 5.3.4.} $c_1\neq a_1=b_1$, $c_2\neq d_2$ and $c_1=d_2$.

There are $n$ ways to choose $a_2(=d_1)$ from $N_n$, $n-1$ ways to choose $d_2(=c_1)$ from $N_n\setminus\{d_1\}$, $n-2$ ways to choose $a_1(=b_1)$ from $N_n\setminus\{d_1, d_2\}$, $n-3$ ways to choose $b_2$ from $N_n\setminus\{d_1, d_2, b_1\}$, $n-1$ ways to choose $c_2$ from $N_n\setminus\{d_2\}$.
Therefore, there are $n(n-1)^2(n-2)(n-3)=n^5-7n^4+17n^3-17n^2+6n$ ways to choose four $B_{n-2}$'s as required.
The probability $P_{5.3.4}$ that there are four fault-free $B_{n-2}$'s chosen as in Case 5.3.4 is $\frac{1}{2}(n^5-7n^4+17n^3-17n^2+6n)p^{4(n-2)!-(n-4)!}$.

{\it Case 5.3.5.} $c_1=a_1\neq b_1$, $c_2\neq d_2$, $c_2\neq a_2$ or $c_1=b_1\neq a_1$, $c_2\neq d_2$, $c_2\neq b_2$.

W.l.o.g, assume that the former applies. There are $n$ ways to choose $a_2(=d_1)$ from $N_n$, $n-1$ ways to choose $d_2$ from $N_n\setminus\{d_1\}$, $n-2$ ways to choose $a_1(=c_1)$ from $N_n\setminus\{d_1, d_2\}$, $n-3$ ways to choose $b_1$ from $N_n\setminus\{d_1, d_2, a_1\}$, $n-3$ ways to choose $b_2$ from $N_n\setminus\{d_1, d_2, b_1\}$, $n-3$ ways to choose $c_2$ from $N_n\setminus\{c_1, d_2, a_2\}$.
Therefore, there are $n(n-1)(n-2)(n-3)^3=n^6-12n^5+56n^4-126n^3+135n^2-54n$ ways to choose four $B_{n-2}$'s as required.
The probability $P_{5.3.5}$ that there are four fault-free $B_{n-2}$'s chosen as in Case 5.3.5 is $\frac{1}{2}\times 2(n^6-12n^5+56n^4-126n^3+135n^2-54n)p^{4(n-2)!-(n-3)!-(n-4)!}$.

{\it Case 5.3.6.} $c_1=a_1\neq b_1$, $c_2\neq d_2$, $c_2=a_2$ or $c_1=b_1\neq a_1$, $c_2\neq d_2$, $c_2=b_2$.

W.l.o.g, assume that the former applies. There are $n$ ways to choose $a_2(=d_1=c_2)$ from $N_n$, $n-1$ ways to choose $d_2$ from $N_n\setminus\{d_1\}$, $n-2$ ways to choose $a_1(=c_1)$ from $N_n\setminus\{d_1, d_2\}$, $n-3$ ways to choose $b_1$ from $N_n\setminus\{d_1, d_2, a_1\}$, $n-3$ ways to choose $b_2$ from $N_n\setminus\{d_1, d_2, b_1\}$.
Therefore, there are $n(n-1)(n-2)(n-3)^2=n^5-9n^4+29n^3-39n^2+18n$ ways to choose four $B_{n-2}$'s as required.
The probability $P_{5.3.6}$ that there are four fault-free $B_{n-2}$'s chosen as in Case 5.3.6 is $\frac{1}{2}\times 2(n^5-9n^4+29n^3-39n^2+18n)p^{4(n-2)!-(n-4)!}$.

{\it Case 5.3.7.} $c_1=a_1=b_1$, $c_2\neq d_2$ and $c_2\notin\{a_2, b_2\}$.

There are $n$ ways to choose $a_2(=d_1)$ from $N_n$, $n-1$ ways to choose $d_2$ from $N_n\setminus\{d_1\}$, $n-2$ ways to choose $a_1(=b_1=c_1)$ from $N_n\setminus\{d_1, d_2\}$, $n-3$ ways to choose $b_2$ from $N_n\setminus\{d_1, d_2, b_1\}$, $n-4$ ways to choose $c_2$ from $N_n\setminus\{c_1, d_2, a_2, b_2\}$.
Therefore, there are $n(n-1)(n-2)(n-3)(n-4)=n^5-10n^4+35n^3-50n^2+24n$ ways to choose four $B_{n-2}$'s as required.
The probability $P_{5.3.7}$ that there are four fault-free $B_{n-2}$'s chosen as in Case 5.3.7 is $\frac{1}{2}(n^5-10n^4+35n^3-50n^2+24n)p^{4(n-2)!-2(n-3)!-(n-4)!}$.

{\it Case 5.3.8.} $c_1=a_1=b_1$, $c_2\neq d_2$ and $c_2\in\{a_2, b_2\}$.

Clearly, in this case $a_2\neq b_2$. There are 2 possible cases such taht $c_2\in\{a_2, b_2\}$, w.l.o.g, assume that $c_2=a_2\neq b_2$.

There are $n$ ways to choose $a_2(=d_1=c_2)$ from $N_n$, $n-1$ ways to choose $d_2$ from $N_n\setminus\{d_1\}$, $n-2$ ways to choose $a_1(=b_1=c_1)$ from $N_n\setminus\{d_1, d_2\}$, $n-3$ ways to choose $b_2$ from $N_n\setminus\{d_1, d_2, b_1\}$.
Therefore, there are $2n(n-1)(n-2)(n-3)=2(n^4-6n^3+11n^2-6n)$ ways to choose four $B_{n-2}$'s as required.
The probability $P_{5.3.8}$ that there are four fault-free $B_{n-2}$'s chosen as in Case 5.3.8 is $\frac{1}{2}\times 2(n^4-6n^3+11n^2-6n)p^{4(n-2)!-(n-3)!-(n-4)!}$.

{\it Case 5.3.9.} $c_1\neq a_1$, $c_1\neq b_1$, $c_2=d_2$, $a_1\neq b_1$ and $c_1\neq d_1$.

There are $n$ ways to choose $a_2(=d_1)$ from $N_n$, $n-1$ ways to choose $d_2(=c_2)$ from $N_n\setminus\{d_1\}$, $n-2$ ways to choose $a_1$ from $N_n\setminus\{d_1, d_2\}$, $n-3$ ways to choose $b_1$ from $N_n\setminus\{d_1, d_2, a_1\}$, $n-3$ ways to choose $b_2$ from $N_n\setminus\{d_1, d_2, b_1\}$, $n-3$ ways to choose $c_1$ from $N_n\setminus\{a_1, b_1, c_2, d_1\}$.
Therefore, there are $n(n-1)(n-2)(n-3)^2(n-4)=n^6-13n^5+65n^4-155n^3+174n^2-72n$ ways to choose four $B_{n-2}$'s as required.
The probability $P_{5.3.9}$ that there are four fault-free $B_{n-2}$'s chosen as in Case 5.3.9 is $\frac{1}{2}(n^6-13n^5+65n^4-155n^3+174n^2-72n)p^{4(n-2)!-(n-3)!-(n-4)!}$.

{\it Case 5.3.10.} $c_1\neq a_1$, $c_1\neq b_1$, $c_2=d_2$, $a_1\neq b_1$ and $c_1=d_1$.

There are $n$ ways to choose $a_2(=d_1=c_1)$ from $N_n$, $n-1$ ways to choose $d_2(=c_2)$ from $N_n\setminus\{d_1\}$, $n-2$ ways to choose $a_1$ from $N_n\setminus\{d_1, d_2\}$, $n-3$ ways to choose $b_1$ from $N_n\setminus\{d_1, d_2, a_1\}$, $n-3$ ways to choose $b_2$ from $N_n\setminus\{d_1, d_2, b_1\}$.
Therefore, there are $n(n-1)(n-2)(n-3)^2=n^5-9n^4+29n^3-39n^2+18n$ ways to choose four $B_{n-2}$'s as required.
The probability $P_{5.3.10}$ that there are four fault-free $B_{n-2}$'s chosen as in Case 5.3.10 is $\frac{1}{2}(n^5-9n^4+29n^3-39n^2+18n)p^{4(n-2)!-(n-4)!}$.

{\it Case 5.3.11.} $c_1\neq a_1=b_1$, $c_2=d_2$ and $c_1\neq d_1$.

There are $n$ ways to choose $a_2(=d_1)$ from $N_n$, $n-1$ ways to choose $d_2(=c_2)$ from $N_n\setminus\{d_1\}$, $n-2$ ways to choose $a_1(=b_1)$ from $N_n\setminus\{d_1, d_2\}$, $n-3$ ways to choose $b_2$ from $N_n\setminus\{d_1, d_2, b_1\}$, $n-2$ ways to choose $c_1$ from $N_n\setminus\{a_1, c_2, d_1\}$.
Therefore, there are $n(n-1)(n-2)(n-3)^2=n^5-9n^4+29n^3-39n^2+18n$ ways to choose four $B_{n-2}$'s as required.
The probability $P_{5.3.11}$ that there are four fault-free $B_{n-2}$'s chosen as in Case 5.3.11 is $\frac{1}{2}(n^5-9n^4+29n^3-39n^2+18n)p^{4(n-2)!-(n-3)!-(n-4)!}$.

{\it Case 5.3.12.} $c_1\neq a_1=b_1$, $c_2=d_2$ and $c_1=d_1$.

There are $n$ ways to choose $a_2(=d_1=c_1)$ from $N_n$, $n-1$ ways to choose $d_2(=c_2)$ from $N_n\setminus\{d_1\}$, $n-2$ ways to choose $a_1(=b_1)$ from $N_n\setminus\{d_1, d_2\}$, $n-3$ ways to choose $b_2$ from $N_n\setminus\{d_1, d_2, b_1\}$.
Therefore, there are $n(n-1)(n-2)(n-3)=n^4-6n^3+11n^2-6n$ ways to choose four $B_{n-2}$'s as required.
The probability $P_{5.3.12}$ that there are four fault-free $B_{n-2}$'s chosen as in Case 5.3.12 is $\frac{1}{2}(n^4-6n^3+11n^2-6n)p^{4(n-2)!-(n-4)!}$.

{\it Case 5.3.13.} $c_1=a_1\neq b_1$, $c_2=d_2$ or $c_1=b_1\neq a_1$, $c_2=d_2$.

W.l.o.g, assume that the former applies. There are $n$ ways to choose $a_2(=d_1)$ from $N_n$, $n-1$ ways to choose $d_2(=c_2)$ from $N_n\setminus\{d_1\}$, $n-2$ ways to choose $a_1(=c_1)$ from $N_n\setminus\{d_1, d_2\}$, $n-3$ ways to choose $b_1$ from $N_n\setminus\{d_1, d_2, a_1\}$, $n-3$ ways to choose $b_2$ from $N_n\setminus\{d_1, d_2, b_1\}$.
Therefore, there are $n(n-1)(n-2)(n-3)^2=n^5-9n^4+29n^3-39n^2+18n$ ways to choose four $B_{n-2}$'s as required.
Clearly, in this case $c_2\neq a_2$ and $c_1\neq d_1$.
The probability $P_{5.3.13}$ that there are four fault-free $B_{n-2}$'s chosen as in Case 5.3.13 is $\frac{1}{2}\times 2(n^5-9n^4+29n^3-39n^2+18n)p^{4(n-2)!-2(n-3)!-(n-4)!}$.

{\it Case 5.3.14.} $c_1=a_1=b_1$ and $c_2\neq d_2$.

There are $n$ ways to choose $a_2(=d_1)$ from $N_n$, $n-1$ ways to choose $d_2(=c_2)$ from $N_n\setminus\{d_1\}$, $n-2$ ways to choose $a_1(=b_1=c_1)$ from $N_n\setminus\{d_1, d_2\}$, $n-3$ ways to choose $b_2$ from $N_n\setminus\{d_1, d_2, b_1\}$.
Therefore, there are $n(n-1)(n-2)(n-3)=n^4-6n^3+11n^2-6n$ ways to choose four $B_{n-2}$'s as required.
Clearly, in this case $c_2\neq a_2$, $c_2\neq b_2$ and $c_1\neq d_1$.
The probability $P_{5.3.14}$ that there are four fault-free $B_{n-2}$'s chosen as in Case 5.3.14 is $\frac{1}{2}(n^4-6n^3+11n^2-6n)p^{4(n-2)!-3(n-3)!-(n-4)!}$.

Thus, the probability $P_{5.3}$ that there are four fault-free $B_{n-2}$'s chosen as in Case 5.3 is $\sum_{i=1}^{14}P_{5.3.i}=\frac{1}{2}(n^7-12n^6+64n^5-194n^4+343n^3-322n^2+120n)p^{4(n-2)!-(n-4)!}+\frac{1}{2}(3n^6-36n^5+170n^4-390n^3+427n^2-174n)p^{4(n-2)!-(n-3)!-(n-4)!}+\frac{1}{2}(3n^5-28n^4+93n^3-128n^2+60n)p^{4(n-2)!-2(n-3)!-(n-4)!}+\frac{1}{2}(n^4-6n^3+11n^2-6n)p^{4(n-2)!-3(n-3)!-(n-4)!}$.

{\it Case 5.4.} $a_2=d_2$, $a_1\notin\{d_1, d_2\}$, $\{b_1, b_2\}\cap\{d_1, d_2\}=\emptyset$.

{\it Case 5.4.1.} $c_1\neq a_1$, $c_1\neq b_1$, $c_2\neq d_2$, $a_1\neq b_1$ and $c_1\neq d_2$.

There are $n$ ways to choose $a_2(=d_2)$ from $N_n$, $n-1$ ways to choose $d_1$ from $N_n\setminus\{d_2\}$, $n-2$ ways to choose $a_1$ from $N_n\setminus\{d_1, d_2\}$, $n-3$ ways to choose $b_1$ from $N_n\setminus\{d_1, d_2, a_1\}$, $n-3$ ways to choose $b_2$ from $N_n\setminus\{d_1, d_2, b_1\}$, $n-3$ ways to choose $c_1$ from $N_n\setminus\{a_1, b_1, d_2\}$, $n-2$ ways to choose $c_2$ from $N_n\setminus\{c_1, d_2\}$.
Therefore, there are $n(n-1)(n-2)^2(n-3)^3=n^7-14n^6+80n^5-238n^4+387n^3-324n^2+108n$ ways to choose four $B_{n-2}$'s as required.
The probability $P_{5.4.1}$ that there are four fault-free $B_{n-2}$'s chosen as in Case 5.4.1 is $\frac{1}{2}(n^7-14n^6+80n^5-238n^4+387n^3-324n^2+108n)p^{4(n-2)!-(n-4)!}$.

{\it Case 5.4.2.} $c_1\neq a_1$, $c_1\neq b_1$, $c_2\neq d_2$, $a_1\neq b_1$ and $c_1=d_2$.

There are $n$ ways to choose $a_2(=d_2=c_1)$ from $N_n$, $n-1$ ways to choose $d_1$ from $N_n\setminus\{d_2\}$, $n-2$ ways to choose $a_1$ from $N_n\setminus\{d_1, d_2\}$, $n-3$ ways to choose $b_1$ from $N_n\setminus\{d_1, d_2, a_1\}$, $n-3$ ways to choose $b_2$ from $N_n\setminus\{d_1, d_2, b_1\}$, $n-1$ ways to choose $c_2$ from $N_n\setminus\{d_2\}$.
Therefore, there are $n(n-1)^2(n-2)(n-3)^2=n^6-10n^5+38n^4-68n^3+57n^2-18n$ ways to choose four $B_{n-2}$'s as required.
The probability $P_{5.4.2}$ that there are four fault-free $B_{n-2}$'s chosen as in Case 5.4.2 is $\frac{1}{2}(n^6-10n^5+38n^4-68n^3+57n^2-18n)p^{4(n-2)!-(n-4)!}$.

{\it Case 5.4.3.} $c_1\neq a_1=b_1$, $c_2\neq d_2$ and $c_1\neq d_2$.

There are $n$ ways to choose $a_2(=d_2)$ from $N_n$, $n-1$ ways to choose $d_1$ from $N_n\setminus\{d_2\}$, $n-2$ ways to choose $a_1(=b_1)$ from $N_n\setminus\{d_1, d_2\}$, $n-3$ ways to choose $b_2$ from $N_n\setminus\{d_1, d_2, b_1\}$, $n-2$ ways to choose $c_1$ from $N_n\setminus\{a_1, d_2\}$, $n-2$ ways to choose $c_2$ from $N_n\setminus\{c_1, d_2\}$.
Therefore, there are $n(n-1)(n-2)^3(n-3)=n^6-10n^5+39n^4-74n^3+68n^2-24n$ ways to choose four $B_{n-2}$'s as required.
The probability $P_{5.4.3}$ that there are four fault-free $B_{n-2}$'s chosen as in Case 5.4.3 is $\frac{1}{2}(n^6-10n^5+39n^4-74n^3+68n^2-24n)p^{4(n-2)!-(n-4)!}$.

{\it Case 5.4.4.} $c_1\neq a_1=b_1$, $c_2\neq d_2$ and $c_1=d_2$.

There are $n$ ways to choose $a_2(=d_2=c_1)$ from $N_n$, $n-1$ ways to choose $d_1$ from $N_n\setminus\{d_1\}$, $n-2$ ways to choose $a_1(=b_1)$ from $N_n\setminus\{d_1, d_2\}$, $n-3$ ways to choose $b_2$ from $N_n\setminus\{d_1, d_2, b_1\}$, $n-1$ ways to choose $c_2$ from $N_n\setminus\{d_2\}$.
Therefore, there are $n(n-1)^2(n-2)(n-3)=n^5-7n^4+17n^3-17n^2+6n$ ways to choose four $B_{n-2}$'s as required.
The probability $P_{5.4.4}$ that there are four fault-free $B_{n-2}$'s chosen as in Case 5.4.4 is $\frac{1}{2}(n^5-7n^4+17n^3-17n^2+6n)p^{4(n-2)!-(n-4)!}$.

{\it Case 5.4.5.} $c_1=a_1\neq b_1$ and $c_2\neq d_2$.

There are $n$ ways to choose $a_2(=d_2)$ from $N_n$, $n-1$ ways to choose $d_1$ from $N_n\setminus\{d_2\}$, $n-2$ ways to choose $a_1(=c_1)$ from $N_n\setminus\{d_1, d_2\}$, $n-3$ ways to choose $b_1$ from $N_n\setminus\{d_1, d_2, a_1\}$, $n-3$ ways to choose $b_2$ from $N_n\setminus\{d_1, d_2, b_1\}$, $n-2$ ways to choose $c_2$ from $N_n\setminus\{c_1, d_2\}$.
Therefore, there are $n(n-1)(n-2)^2(n-3)^2=n^6-11n^5+47n^4-97n^3+96n^2-36n$ ways to choose four $B_{n-2}$'s as required.
Clearly, in this case $c_2\neq a_2$.
The probability $P_{5.4.5}$ that there are four fault-free $B_{n-2}$'s chosen as in Case 5.4.5 is $\frac{1}{2}(n^6-11n^5+47n^4-97n^3+96n^2-36n)p^{4(n-2)!-(n-3)!-(n-4)!}$.

{\it Case 5.4.6.} $c_1=b_1\neq a_1$, $c_2\neq d_2$ and $c_2\neq b_2$.

There are $n$ ways to choose $a_2(=d_2)$ from $N_n$, $n-1$ ways to choose $d_1$ from $N_n\setminus\{d_2\}$, $n-2$ ways to choose $a_1$ from $N_n\setminus\{d_1, d_2\}$, $n-3$ ways to choose $b_1(=c_1)$ from $N_n\setminus\{d_1, d_2, a_1\}$, $n-3$ ways to choose $b_2$ from $N_n\setminus\{d_1, d_2, b_1\}$, $n-2$ ways to choose $c_2$ from $N_n\setminus\{c_1, d_2, b_2\}$.
Therefore, there are $n(n-1)(n-2)(n-3)^3=n^6-12n^5+56n^4-126n^3+135n^2-54n$ ways to choose four $B_{n-2}$'s as required.
The probability $P_{5.4.6}$ that there are four fault-free $B_{n-2}$'s chosen as in Case 5.4.6 is $\frac{1}{2}(n^6-12n^5+56n^4-126n^3+135n^2-54n)p^{4(n-2)!-(n-3)!-(n-4)!}$.

{\it Case 5.4.7.} $c_1=b_1\neq a_1$, $c_2\neq d_2$ and $c_2=b_2$.

There are $n$ ways to choose $a_2(=d_2)$ from $N_n$, $n-1$ ways to choose $d_1$ from $N_n\setminus\{d_2\}$, $n-2$ ways to choose $a_1$ from $N_n\setminus\{d_1, d_2\}$, $n-3$ ways to choose $b_1(=c_1)$ from $N_n\setminus\{d_1, d_2, a_1\}$, $n-3$ ways to choose $b_2(=c_2)$ from $N_n\setminus\{d_1, d_2, b_1\}$.
Therefore, there are $n(n-1)(n-2)(n-3)^2=n^5-9n^4+29n^3-39n^2+18n$ ways to choose four $B_{n-2}$'s as required.
The probability $P_{5.4.7}$ that there are four fault-free $B_{n-2}$'s chosen as in Case 5.4.7 is $\frac{1}{2}(n^5-9n^4+29n^3-39n^2+18n)p^{4(n-2)!-(n-4)!}$.

{\it Case 5.4.8.} $c_1=a_1=b_1$, $c_2\neq d_2$ and $c_2\neq b_2$.

There are $n$ ways to choose $a_2(=d_2)$ from $N_n$, $n-1$ ways to choose $d_1$ from $N_n\setminus\{d_2\}$, $n-2$ ways to choose $a_1(=b_1=c_1)$ from $N_n\setminus\{d_1, d_2\}$, $n-3$ ways to choose $b_2$ from $N_n\setminus\{d_1, d_2, b_1\}$, $n-3$ ways to choose $c_2$ from $N_n\setminus\{c_1, d_2, b_2\}$.
Therefore, there are $n(n-1)(n-2)(n-3)^2=n^5-9n^4+29n^3-39n^2+18n$ ways to choose four $B_{n-2}$'s as required.
Clearly, in this case $c_2\neq a_2$.
The probability $P_{5.4.8}$ that there are four fault-free $B_{n-2}$'s chosen as in Case 5.4.8 is $\frac{1}{2}(n^5-9n^4+29n^3-39n^2+18n)p^{4(n-2)!-2(n-3)!-(n-4)!}$.

{\it Case 5.4.9.} $c_1=a_1=b_1$, $c_2\neq d_2$ and $c_2=b_2$.

There are $n$ ways to choose $a_2(=d_2)$ from $N_n$, $n-1$ ways to choose $d_1$ from $N_n\setminus\{d_2\}$, $n-2$ ways to choose $a_1(=b_1=c_1)$ from $N_n\setminus\{d_1, d_2\}$, $n-3$ ways to choose $b_2(=c_2)$ from $N_n\setminus\{d_1, d_2, b_1\}$.
Therefore, there are $n(n-1)(n-2)(n-3)=n^4-6n^3+11n^2-6n$ ways to choose four $B_{n-2}$'s as required.
Clearly, in this case $c_2\neq a_2$.
The probability $P_{5.4.9}$ that there are four fault-free $B_{n-2}$'s chosen as in Case 5.4.9 is $\frac{1}{2}(n^4-6n^3+11n^2-6n)p^{4(n-2)!-(n-3)!-(n-4)!}$.

{\it Case 5.4.10.} $c_1\neq a_1$, $c_1\neq b_1$, $c_2=d_2$, $a_1\neq b_1$ and $c_1\neq d_1$.

There are $n$ ways to choose $a_2(=d_2=c_2)$ from $N_n$, $n-1$ ways to choose $d_1$ from $N_n\setminus\{d_2\}$, $n-2$ ways to choose $a_1$ from $N_n\setminus\{d_1, d_2\}$, $n-3$ ways to choose $b_1$ from $N_n\setminus\{d_1, d_2, a_1\}$, $n-3$ ways to choose $b_2$ from $N_n\setminus\{d_1, d_2, b_1\}$, $n-4$ ways to choose $c_1$ from $N_n\setminus\{a_1, b_1, c_2, d_1\}$.
Therefore, there are $n(n-1)(n-2)(n-3)^2(n-4)=n^6-13n^5+65n^4-155n^3+174n^2-72n$ ways to choose four $B_{n-2}$'s as required.
The probability $P_{5.4.10}$ that there are four fault-free $B_{n-2}$'s chosen as in Case 5.4.10 is $\frac{1}{2}(n^6-13n^5+65n^4-155n^3+174n^2-72n)p^{4(n-2)!-(n-3)!-(n-4)!}$.

{\it Case 5.4.11.} $c_1\neq a_1$, $c_1\neq b_1$, $c_2=d_2$, $a_1\neq b_1$ and $c_1=d_1$.

There are $n$ ways to choose $a_2(=d_2=c_2)$ from $N_n$, $n-1$ ways to choose $d_1(=c_1)$ from $N_n\setminus\{d_2\}$, $n-2$ ways to choose $a_1$ from $N_n\setminus\{d_1, d_2\}$, $n-3$ ways to choose $b_1$ from $N_n\setminus\{d_1, d_2, a_1\}$, $n-3$ ways to choose $b_2$ from $N_n\setminus\{d_1, d_2, b_1\}$.
Therefore, there are $n(n-1)(n-2)(n-3)^2=n^5-9n^4+29n^3-39n^2+18n$ ways to choose four $B_{n-2}$'s as required.
The probability $P_{5.4.11}$ that there are four fault-free $B_{n-2}$'s chosen as in Case 5.4.11 is $\frac{1}{2}(n^5-9n^4+29n^3-39n^2+18n)p^{4(n-2)!-(n-4)!}$.

{\it Case 5.4.12.} $c_1\neq a_1=b_1$, $c_2=d_2$ and $c_1\neq d_1$.

There are $n$ ways to choose $a_2(=d_2=c_2)$ from $N_n$, $n-1$ ways to choose $d_1$ from $N_n\setminus\{d_2\}$, $n-2$ ways to choose $a_1(=b_1)$ from $N_n\setminus\{d_1, d_2\}$, $n-3$ ways to choose $b_2$ from $N_n\setminus\{d_1, d_2, b_1\}$, $n-3$ ways to choose $c_1$ from $N_n\setminus\{a_1, d_2, d_1\}$.
Therefore, there are $n(n-1)(n-2)(n-3)^2=n^5-9n^4+29n^3-39n^2+18n$ ways to choose four $B_{n-2}$'s as required.
The probability $P_{5.4.12}$ that there are four fault-free $B_{n-2}$'s chosen as in Case 5.4.12 is $\frac{1}{2}(n^5-9n^4+29n^3-39n^2+18n)p^{4(n-2)!-(n-3)!-(n-4)!}$.

{\it Case 5.4.13.} $c_1\neq a_1=b_1$, $c_2=d_2$ and $c_1=d_1$.

There are $n$ ways to choose $a_2(=d_2=c_2)$ from $N_n$, $n-1$ ways to choose $d_1(=c_1)$ from $N_n\setminus\{d_2\}$, $n-2$ ways to choose $a_1(=b_1)$ from $N_n\setminus\{d_1, d_2\}$, $n-3$ ways to choose $b_2$ from $N_n\setminus\{d_1, d_2, b_1\}$.
Therefore, there are $n(n-1)(n-2)(n-3)=n^4-6n^3+11n^2-6n$ ways to choose four $B_{n-2}$'s as required.
The probability $P_{5.4.13}$ that there are four fault-free $B_{n-2}$'s chosen as in Case 5.4.13 is $\frac{1}{2}(n^4-6n^3+11n^2-6n)p^{4(n-2)!-(n-4)!}$.

{\it Case 5.4.14.} $c_1=a_1\neq b_1$, $c_2=d_2$ or $c_1=b_1\neq a_1$, $c_2=d_2$.

There are $n$ ways to choose $a_2(=d_2=c_2)$ from $N_n$, $n-1$ ways to choose $d_1$ from $N_n\setminus\{d_2\}$, $n-2$ ways to choose $a_1(=c_1)$ from $N_n\setminus\{d_1, d_2\}$, $n-3$ ways to choose $b_1$ from $N_n\setminus\{d_1, d_2, a_1\}$, $n-3$ ways to choose $b_2$ from $N_n\setminus\{d_1, d_2, b_1\}$.
Therefore, there are $n(n-1)(n-2)(n-3)^2=n^5-9n^4+29n^3-39n^2+18n$ ways to choose four $B_{n-2}$'s as required.
Clearly, in this case $c_2=a_2$ and $c_1\neq d_1$.
The probability $P_{5.4.14}$ that there are four fault-free $B_{n-2}$'s chosen as in Case 5.4.14 is $\frac{1}{2}(n^5-9n^4+29n^3-39n^2+18n)p^{4(n-2)!-(n-3)!-(n-4)!}$.

{\it Case 5.4.15.} $c_1=a_1\neq b_1$, $c_2=d_2$ or $c_1=b_1\neq a_1$, $c_2=d_2$.

There are $n$ ways to choose $a_2(=d_2=c_2)$ from $N_n$, $n-1$ ways to choose $d_1$ from $N_n\setminus\{d_2\}$, $n-2$ ways to choose $a_1(=c_1)$ from $N_n\setminus\{d_1, d_2\}$, $n-3$ ways to choose $b_1$ from $N_n\setminus\{d_1, d_2, a_1\}$, $n-3$ ways to choose $b_2$ from $N_n\setminus\{d_1, d_2, b_1\}$.
Therefore, there are $n(n-1)(n-2)(n-3)^2=n^5-9n^4+29n^3-39n^2+18n$ ways to choose four $B_{n-2}$'s as required.
Clearly, in this case $c_2\neq b_2$ and $c_1\neq d_1$.
The probability $P_{5.4.15}$ that there are four fault-free $B_{n-2}$'s chosen as in Case 5.4.15 is $\frac{1}{2}(n^5-9n^4+29n^3-39n^2+18n)p^{4(n-2)!-2(n-3)!-(n-4)!}$.

{\it Case 5.4.16.} $c_1=a_1=b_1$ and $c_2\neq d_2$.

There are $n$ ways to choose $a_2(=d_1)$ from $N_n$, $n-1$ ways to choose $d_2(=c_2)$ from $N_n\setminus\{d_1\}$, $n-2$ ways to choose $a_1(=b_1=c_1)$ from $N_n\setminus\{d_1, d_2\}$, $n-3$ ways to choose $b_2$ from $N_n\setminus\{d_1, d_2, b_1\}$.
Therefore, there are $n(n-1)(n-2)(n-3)=n^4-6n^3+11n^2-6n$ ways to choose four $B_{n-2}$'s as required.
Clearly, in this case $c_2=a_2$, $c_2\neq b_2$ and $c_1\neq d_1$.
The probability $P_{5.4.16}$ that there are four fault-free $B_{n-2}$'s chosen as in Case 5.4.16 is $\frac{1}{2}(n^4-6n^3+11n^2-6n)p^{4(n-2)!-2(n-3)!-(n-4)!}$.

Thus, the probability $P_{5.4}$ that there are four fault-free $B_{n-2}$'s chosen as in Case 5.4 is $\sum_{i=1}^{16}P_{5.4.i}=\frac{1}{2}(n^7-12n^6+63n^5-185n^4+314n^3-283n^2+102n)p^{4(n-2)!-(n-4)!}+\frac{1}{2}(3n^6-34n^5+151n^4-326n^3+338n^2-132n)p^{4(n-2)!-(n-3)!-(n-4)!}+\frac{1}{2}(2n^5-17n^4+52n^3-67n^2+30n)p^{4(n-2)!-2(n-3)!-(n-4)!}$.

In summary, the probability $P_5$ that there are four fault-free $B_{n-2}$'s chosen as in Case 5 is $2(\sum_{i=1}^{4}P_{5.i})=(4n^7-48n^6+251n^5-731n^4+1227n^3-1093n^2+390n)p^{4(n-2)!-(n-4)!}+(12n^6-134n^5+587n^4-1252n^3+1285n^2-498n)p^{4(n-2)!-(n-3)!-(n-4)!}+(7n^5-61n^4+191n^3-251n^2+114n)p^{4(n-2)!-2(n-3)!-(n-4)!}+(n^4-6n^3+11n^2-6n)p^{4(n-2)!-3(n-3)!-(n-4)!}$.

{\it Case 6.} $\{a_1, a_2\}\cap\{d_1, d_2\}=\emptyset$ and $\{b_1, b_2\}\cap\{d_1, d_2\}=\emptyset$.

Observation \ref{ob:2} implies that for any $i\in\{a, b\}$, $i_1 i_2 X^{n-2}$ and $X^{n-2} d_1 d_2$ are not disjoint and $V(i_1 i_2 X^{n-2})\cap V(X^{n-2} d_1 d_2)=V(i_1 i_2 X^{n-4} d_1 d_2)$.

{\it Case 6.1.} $c_1\neq a_1$, $c_1\neq b_1$, $c_2\neq d_2$, $a_1\neq b_1$ and $c_1\neq d_2$.

There are $n$ ways to choose $d_1$ from $N_n$, $n-1$ ways to choose $d_2$ from $N_n\setminus\{d_1\}$, $n-2$ ways to choose $a_1$ from $N_n\setminus\{d_1, d_2\}$, $n-3$ ways to choose $a_2$ from $N_n\setminus\{d_1, d_2, a_1\}$, $n-3$ ways to choose $b_1$ from $N_n\setminus\{d_1, d_2, a_1\}$, $n-3$ ways to choose $b_2$ from $N_n\setminus\{d_1, d_2, b_1\}$, $n-3$ ways to choose $c_1$ from $N_n\setminus\{a_1, b_1, d_2\}$, $n-2$ ways to choose $c_2$ from $N_n\setminus\{c_1, d_2\}$.
Therefore, there are $n(n-1)(n-2)^2(n-3)^4=n^8-17n^7+122n^6-478n^5+1101n^4-1485n^3+1080n^2-324n$ ways to choose four $B_{n-2}$'s as required.
The probability $P_{6.1}$ that there are four fault-free $B_{n-2}$'s chosen as in Case 6.1 is $\frac{1}{2}(n^8-17n^7+122n^6-478n^5+1101n^4-1485n^3+1080n^2-324n)p^{4(n-2)!-2(n-4)!}$.

{\it Case 6.2.} $c_1\neq a_1$, $c_1\neq b_1$, $c_2\neq d_2$, $a_1\neq b_1$ and $c_1=d_2$.

There are $n$ ways to choose $d_1$ from $N_n$, $n-1$ ways to choose $d_2(=c_1)$ from $N_n\setminus\{d_1\}$, $n-2$ ways to choose $a_1$ from $N_n\setminus\{d_1, d_2\}$, $n-3$ ways to choose $a_2$ from $N_n\setminus\{d_1, d_2, a_1\}$, $n-3$ ways to choose $b_1$ from $N_n\setminus\{d_1, d_2, a_1\}$, $n-3$ ways to choose $b_2$ from $N_n\setminus\{d_1, d_2, b_1\}$, $n-1$ ways to choose $c_2$ from $N_n\setminus\{d_2\}$.
Therefore, there are $n(n-1)^2(n-2)(n-3)^3=n^7-13n^6+68n^5-182n^4+261n^3-189n^2+54n$ ways to choose four $B_{n-2}$'s as required.
The probability $P_{6.2}$ that there are four fault-free $B_{n-2}$'s chosen as in Case 6.2 is $\frac{1}{2}(n^7-13n^6+68n^5-182n^4+261n^3-189n^2+54n)p^{4(n-2)!-2(n-4)!}$.

{\it Case 6.3.} $c_1\neq a_1=b_1$, $c_2\neq d_2$ and $c_1\neq d_2$.

There are $n$ ways to choose $d_1$ from $N_n$, $n-1$ ways to choose $d_2$ from $N_n\setminus\{d_1\}$, $n-2$ ways to choose $a_1(=b_1)$ from $N_n\setminus\{d_1, d_2\}$, $n-3$ ways to choose $a_2$ from $N_n\setminus\{d_1, d_2, a_1\}$, $n-4$ ways to choose $b_2$ from $N_n\setminus\{d_1, d_2, b_1, a_2\}$, $n-2$ ways to choose $c_1$ from $N_n\setminus\{a_1, d_2\}$, $n-2$ ways to choose $c_2$ from $N_n\setminus\{c_1, d_2\}$.
Therefore, there are $n(n-1)(n-2)^3(n-3)(n-4)=n^7-14n^6+79n^5-230n^4+364n^3-296n^2+96n$ ways to choose four $B_{n-2}$'s as required.
The probability $P_{6.3}$ that there are four fault-free $B_{n-2}$'s chosen as in Case 6.3 is $\frac{1}{2}(n^7-14n^6+79n^5-230n^4+364n^3-296n^2+96n)p^{4(n-2)!-2(n-4)!}$.

{\it Case 6.4.} $c_1\neq a_1=b_1$, $c_2\neq d_2$ and $c_1=d_2$.

There are $n$ ways to choose $d_1$ from $N_n$, $n-1$ ways to choose $d_2(=c_1)$ from $N_n\setminus\{d_1\}$, $n-2$ ways to choose $a_1(=b_1)$ from $N_n\setminus\{d_1, d_2\}$, $n-3$ ways to choose $a_2$ from $N_n\setminus\{d_1, d_2, a_1\}$, $n-4$ ways to choose $b_2$ from $N_n\setminus\{d_1, d_2, b_1, a_2\}$, $n-1$ ways to choose $c_2$ from $N_n\setminus\{d_2\}$.
Therefore, there are $n(n-1)^2(n-2)(n-3)(n-4)=n^6-11n^5+45n^4-85n^3+74n^2-24n$ ways to choose four $B_{n-2}$'s as required.
The probability $P_{6.4}$ that there are four fault-free $B_{n-2}$'s chosen as in Case 6.4 is $\frac{1}{2}(n^6-11n^5+45n^4-85n^3+74n^2-24n)p^{4(n-2)!-2(n-4)!}$.

{\it Case 6.5.} $c_1=a_1\neq b_1$, $c_2\neq d_2$, $c_2\neq a_2$ or $c_1=b_1\neq a_1$, $c_2\neq d_2$, $c_2\neq b_2$.

W.l.o.g, assume that the former applies.
There are $n$ ways to choose $d_1$ from $N_n$, $n-1$ ways to choose $d_2$ from $N_n\setminus\{d_1\}$, $n-2$ ways to choose $a_1(=c_1)$ from $N_n\setminus\{d_1, d_2\}$, $n-3$ ways to choose $a_2$ from $N_n\setminus\{d_1, d_2, a_1\}$, $n-3$ ways to choose $b_1$ from $N_n\setminus\{d_1, d_2, a_1\}$, $n-3$ ways to choose $b_2$ from $N_n\setminus\{d_1, d_2, b_1\}$, $n-3$ ways to choose $c_2$ from $N_n\setminus\{c_1, d_2, a_2\}$.
Therefore, there are $2n(n-1)(n-2)(n-3)^4=2(n^7-15n^6+92n^5-294n^4+513n^3-459n^2+162n)$ ways to choose four $B_{n-2}$'s as required.
The probability $P_{6.5}$ that there are four fault-free $B_{n-2}$'s chosen as in Case 6.5 is $\frac{1}{2}\times 2(n^7-15n^6+92n^5-294n^4+513n^3-459n^2+162n)p^{4(n-2)!-(n-3)!-2(n-4)!}$.

{\it Case 6.6.} $c_1=a_1\neq b_1$, $c_2\neq d_2$, $c_2=a_2$ or $c_1=b_1\neq a_1$, $c_2\neq d_2$, $c_2=b_2$.

W.l.o.g, assume that the former applies.
There are $n$ ways to choose $d_1$ from $N_n$, $n-1$ ways to choose $d_2$ from $N_n\setminus\{d_1\}$, $n-2$ ways to choose $a_1(=c_1)$ from $N_n\setminus\{d_1, d_2\}$, $n-3$ ways to choose $a_2(=c_2)$ from $N_n\setminus\{d_1, d_2, a_1\}$, $n-3$ ways to choose $b_1$ from $N_n\setminus\{d_1, d_2, a_1\}$, $n-3$ ways to choose $b_2$ from $N_n\setminus\{d_1, d_2, b_1\}$.
Therefore, there are $2n(n-1)(n-2)(n-3)^3=2(n^6-12n^5+56n^4-126n^3+135n^2-54n)$ ways to choose four $B_{n-2}$'s as required.
The probability $P_{6.6}$ that there are four fault-free $B_{n-2}$'s chosen as in Case 6.6 is $\frac{1}{2}\times 2(n^6-12n^5+56n^4-126n^3+135n^2-54n)p^{4(n-2)!-2(n-4)!}$.

{\it Case 6.7.} $c_1=a_1=b_1$, $c_2\neq d_2$ and $c_2\notin\{a_2, b_2\}$.

There are $n$ ways to choose $d_1$ from $N_n$, $n-1$ ways to choose $d_2$ from $N_n\setminus\{d_1\}$, $n-2$ ways to choose $a_1(=b_1=c_1)$ from $N_n\setminus\{d_1, d_2\}$, $n-3$ ways to choose $a_2$ from $N_n\setminus\{d_1, d_2, a_1\}$, $n-4$ ways to choose $b_2$ from $N_n\setminus\{d_1, d_2, b_1, a_2\}$, $n-4$ ways to choose $c_2$ from $N_n\setminus\{c_1, d_2, a_2, b_2\}$.
Therefore, there are $n(n-1)(n-2)(n-3)(n-4)^2=n^6-14n^5+75n^4-190n^3+224n^2-96n$ ways to choose four $B_{n-2}$'s as required.
The probability $P_{6.7}$ that there are four fault-free $B_{n-2}$'s chosen as in Case 6.7 is $\frac{1}{2}(n^6-14n^5+75n^4-190n^3+224n^2-96n)p^{4(n-2)!-2(n-3)!-2(n-4)!}$.

{\it Case 6.8.} $c_1=a_1=b_1$, $c_2\neq d_2$ and $c_2\in\{a_2, b_2\}$.

Clearly, in this case $a_2\neq b_2$. There are 2 possible cases such taht $c_2\in\{a_2, b_2\}$, w.l.o.g, assume that $c_2=a_2\neq b_2$.

There are $n$ ways to choose $d_1$ from $N_n$, $n-1$ ways to choose $d_2$ from $N_n\setminus\{d_1\}$, $n-2$ ways to choose $a_1(=b_1=c_1)$ from $N_n\setminus\{d_1, d_2\}$, $n-3$ ways to choose $a_2(=c_2)$ from $N_n\setminus\{d_1, d_2, a_1\}$, $n-4$ ways to choose $b_2$ from $N_n\setminus\{d_1, d_2, b_1, a_2\}$.
Therefore, there are $2n(n-1)(n-2)(n-3)(n-4)=2(n^5-10n^4+35n^3-50n^2+24n)$ ways to choose four $B_{n-2}$'s as required.
The probability $P_{6.8}$ that there are four fault-free $B_{n-2}$'s chosen as in Case 6.8 is $\frac{1}{2}\times 2(n^5-10n^4+35n^3-50n^2+24n)p^{4(n-2)!-(n-3)!-2(n-4)!}$.

{\it Case 6.9.} $c_1\neq a_1$, $c_1\neq b_1$, $c_2=d_2$, $a_1\neq b_1$ and $c_1\neq d_1$.

There are $n$ ways to choose $d_1$ from $N_n$, $n-1$ ways to choose $d_2(=c_2)$ from $N_n\setminus\{d_1\}$, $n-2$ ways to choose $a_1$ from $N_n\setminus\{d_1, d_2\}$, $n-3$ ways to choose $a_2$ from $N_n\setminus\{d_1, d_2, a_1\}$, $n-3$ ways to choose $b_1$ from $N_n\setminus\{d_1, d_2, a_1\}$, $n-3$ ways to choose $b_2$ from $N_n\setminus\{d_1, d_2, b_1\}$, $n-4$ ways to choose $c_1$ from $N_n\setminus\{a_1, b_1, c_2, d_1\}$.
Therefore, there are $n(n-1)(n-2)(n-3)^3(n-4)=n^7-16n^6+104n^5-350n^4+639n^3-594n^2+216n$ ways to choose four $B_{n-2}$'s as required.
The probability $P_{6.9}$ that there are four fault-free $B_{n-2}$'s chosen as in Case 6.9 is $\frac{1}{2}(n^7-16n^6+104n^5-350n^4+639n^3-594n^2+216n)p^{4(n-2)!-(n-3)!-2(n-4)!}$.

{\it Case 6.10.} $c_1\neq a_1$, $c_1\neq b_1$, $c_2=d_2$, $a_1\neq b_1$ and $c_1=d_1$.

There are $n$ ways to choose $d_1(=c_1)$ from $N_n$, $n-1$ ways to choose $d_2(=c_2)$ from $N_n\setminus\{d_1\}$, $n-2$ ways to choose $a_1$ from $N_n\setminus\{d_1, d_2\}$, $n-3$ ways to choose $a_2$ from $N_n\setminus\{d_1, d_2, a_1\}$, $n-3$ ways to choose $b_1$ from $N_n\setminus\{d_1, d_2, a_1\}$, $n-3$ ways to choose $b_2$ from $N_n\setminus\{d_1, d_2, b_1\}$.
Therefore, there are $n(n-1)(n-2)(n-3)^3=n^6-12n^5+56n^4-126n^3+135n^2-54n$ ways to choose four $B_{n-2}$'s as required.
The probability $P_{6.10}$ that there are four fault-free $B_{n-2}$'s chosen as in Case 6.10 is $\frac{1}{2}(n^6-12n^5+56n^4-126n^3+135n^2-54n)p^{4(n-2)!-2(n-4)!}$.

{\it Case 6.11.} $c_1\neq a_1=b_1$, $c_2=d_2$ and $c_1\neq d_1$.

There are $n$ ways to choose $d_1$ from $N_n$, $n-1$ ways to choose $d_2(=c_2)$ from $N_n\setminus\{d_1\}$, $n-2$ ways to choose $a_1(=b_1)$ from $N_n\setminus\{d_1, d_2\}$, $n-3$ ways to choose $a_2$ from $N_n\setminus\{d_1, d_2, a_1\}$, $n-4$ ways to choose $b_2$ from $N_n\setminus\{d_1, d_2, b_1, a_2\}$, $n-3$ ways to choose $c_1$ from $N_n\setminus\{a_1, c_2, d_1\}$.
Therefore, there are $n(n-1)(n-2)(n-3)^2(n-4)=n^6-13n^5+65n^4-155n^3+174n^2-72n$ ways to choose four $B_{n-2}$'s as required.
The probability $P_{6.11}$ that there are four fault-free $B_{n-2}$'s chosen as in Case 6.11 is $\frac{1}{2}(n^6-13n^5+65n^4-155n^3+174n^2-72n)p^{4(n-2)!-(n-3)!-2(n-4)!}$.

{\it Case 6.12.} $c_1\neq a_1=b_1$, $c_2=d_2$ and $c_1=d_1$.

There are $n$ ways to choose $d_1(=c_1)$ from $N_n$, $n-1$ ways to choose $d_2(=c_2)$ from $N_n\setminus\{d_1\}$, $n-2$ ways to choose $a_1(=b_1)$ from $N_n\setminus\{d_1, d_2\}$, $n-3$ ways to choose $a_2$ from $N_n\setminus\{d_1, d_2, a_1\}$, $n-4$ ways to choose $b_2$ from $N_n\setminus\{d_1, d_2, b_1, a_2\}$.
Therefore, there are $n(n-1)(n-2)(n-3)(n-4)=n^5-10n^4+35n^3-50n^2+24n$ ways to choose four $B_{n-2}$'s as required.
The probability $P_{6.12}$ that there are four fault-free $B_{n-2}$'s chosen as in Case 6.12 is $\frac{1}{2}(n^5-10n^4+35n^3-50n^2+24n)p^{4(n-2)!-2(n-4)!}$.

{\it Case 6.13.} $c_1=a_1\neq b_1$, $c_2=d_2$ or $c_1=b_1\neq a_1$, $c_2=d_2$.

W.l.o.g, assume that the former applies. There are $n$ ways to choose $d_1$ from $N_n$, $n-1$ ways to choose $d_2(=c_2)$ from $N_n\setminus\{d_1\}$, $n-2$ ways to choose $a_1(=c_1)$ from $N_n\setminus\{d_1, d_2\}$, $n-3$ ways to choose $a_2$ from $N_n\setminus\{d_1, d_2, a_1\}$, $n-3$ ways to choose $b_1$ from $N_n\setminus\{d_1, d_2, a_1\}$, $n-3$ ways to choose $b_2$ from $N_n\setminus\{d_1, d_2, b_1\}$.
Therefore, there are $2n(n-1)(n-2)(n-3)^3=2(n^6-12n^5+56n^4-126n^3+135n^2-54n)$ ways to choose four $B_{n-2}$'s as required.
Clearly, in this case $c_2\neq a_2$ and $c_1\neq d_1$.
The probability $P_{6.13}$ that there are four fault-free $B_{n-2}$'s chosen as in Case 6.13 is $\frac{1}{2}\times 2(n^6-12n^5+56n^4-126n^3+135n^2-54n)p^{4(n-2)!-2(n-3)!-2(n-4)!}$.

{\it Case 6.14.} $c_1=a_1=b_1$ and $c_2\neq d_2$.

There are $n$ ways to choose $d_1$ from $N_n$, $n-1$ ways to choose $d_2(=c_2)$ from $N_n\setminus\{d_1\}$, $n-2$ ways to choose $a_1(=b_1=c_1)$ from $N_n\setminus\{d_1, d_2\}$, $n-3$ ways to choose $a_2$ from $N_n\setminus\{d_1, d_2, a_1\}$, $n-4$ ways to choose $b_2$ from $N_n\setminus\{d_1, d_2, b_1, a_2\}$.
Therefore, there are $n(n-1)(n-2)(n-3)(n-4)=n^5-10n^4+35n^3-50n^2+24n$ ways to choose four $B_{n-2}$'s as required.
Clearly, in this case $c_2\neq a_2$, $c_2\neq b_2$ and $c_1\neq d_1$.
The probability $P_{6.14}$ that there are four fault-free $B_{n-2}$'s chosen as in Case 6.14 is $\frac{1}{2}(n^5-10n^4+35n^3-50n^2+24n)p^{4(n-2)!-3(n-3)!-2(n-4)!}$.

In summary, the probability $P_6$ that there are four fault-free $B_{n-2}$'s chosen as in Case 6 is $\sum_{i=1}^{14}P_{6.i}=\frac{1}{2}(n^8-15n^7+99n^6-377n^5+892n^4-1288n^3+1024n^2-336n)p^{4(n-2)!-2(n-4)!}+\frac{1}{2}(3n^7-45n^6+277n^5-893n^4+1580n^3-1438n^2+516n)p^{4(n-2)!-(n-3)!-2(n-4)!}+\frac{1}{2}(3n^6-38n^5+187n^4-442n^3+494n^2-204n)p^{4(n-2)!-2(n-3)!-}$ $^{2(n-4)!}+\frac{1}{2}(n^5-10n^4+35n^3-50n^2+24n)p^{4(n-2)!-3(n-3)!-2(n-4)!}$.

By the above computations, $P(2,1,1)=\sum_{i=1}^{6}P_i=(8n^6-65n^5+229n^4-421n^3+389n^2-140n)p^{4(n-2)!}+\frac{1}{2}(44n^5-291n^4+716n^3-771n^2+302n)p^{4(n-2)!-(n-3)!}+\frac{1}{2}(2n^5-5n^4-10n^3+35n^2-22n)p^{4(n-2)!-2(n-3)!}+(4n^7-46n^6+231n^5-647n^4+1043n^3-891n^2+306n)p^{4(n-2)!-(n-4)!}+(12n^6-128n^5+535n^4-1090n^3+1073n^2-402n)p^{4(n-2)!-(n-3)!-(n-4)!}+(7n^5-59n^4+179n^3-229n^2+102n)p^{4(n-2)!-2(n-3)!-(n-4)!}+(n^4-6n^3+11n^2-6n)p^{4(n-2)!-3(n-3)!-(n-4)!}+\frac{1}{2}(n^8-15n^7+99n^6-377n^5+892n^4-1288n^3+1024n^2-336n)p^{4(n-2)!-2(n-4)!}+\frac{1}{2}(3n^7-45n^6+277n^5-893n^4+1580n^3-1438n^2+516n)p^{4(n-2)!-(n-3)!-2(n-4)!}+\frac{1}{2}(3n^6-38n^5+187n^4-442n^3+494n^2-204n)p^{4(n-2)!-2(n-3)!-2(n-4)!}+\frac{1}{2}(n^5-10n^4+35n^3-50n^2+24n)p^{4(n-2)!-3(n-3)!-2(n-4)!}$.
\end{proof}

\begin{theorem}\label{th10}
$P(1,1,2)=(8n^6-65n^5+229n^4-421n^3+389n^2-140n)p^{4(n-2)!}+\frac{1}{2}(44n^5-291n^4+716n^3-771n^2+302n)p^{4(n-2)!-(n-3)!}+\frac{1}{2}(2n^5-5n^4-10n^3+35n^2-22n)p^{4(n-2)!-2(n-3)!}+(4n^7-46n^6+231n^5-647n^4+1043n^3-891n^2+306n)p^{4(n-2)!-(n-4)!}+(12n^6-128n^5+535n^4-1090n^3+1073n^2-402n)p^{4(n-2)!-(n-3)!-(n-4)!}+(7n^5-59n^4+179n^3-229n^2+102n)p^{4(n-2)!-2(n-3)!-(n-4)!}+(n^4-6n^3+11n^2-6n)p^{4(n-2)!-3(n-3)!-(n-4)!}+\frac{1}{2}(n^8-15n^7+99n^6-377n^5+892n^4-1288n^3+1024n^2-336n)p^{4(n-2)!-2(n-4)!}+\frac{1}{2}(3n^7-45n^6+277n^5-893n^4+1580n^3-1438n^2+516n)p^{4(n-2)!-(n-3)!-2(n-4)!}+\frac{1}{2}(3n^6-38n^5+187n^4-442n^3+494n^2-204n)p^{4(n-2)!-2(n-3)!-2(n-4)!}+\frac{1}{2}(n^5-10n^4+35n^3-50n^2+24n)p^{4(n-2)!-3(n-3)!-2(n-4)!}$.
\end{theorem}
\begin{proof}
This scenario is similar to Theorem \ref{th9}, it is easy to get that $P(1,1,2)=P(2,1,1)=(8n^6-65n^5+229n^4-421n^3+389n^2-140n)p^{4(n-2)!}+\frac{1}{2}(44n^5-291n^4+716n^3-771n^2+302n)p^{4(n-2)!-(n-3)!}+\frac{1}{2}(2n^5-5n^4-10n^3+35n^2-22n)p^{4(n-2)!-2(n-3)!}+(4n^7-46n^6+231n^5-647n^4+1043n^3-891n^2+306n)p^{4(n-2)!-(n-4)!}+(12n^6-128n^5+535n^4-1090n^3+1073n^2-402n)p^{4(n-2)!-(n-3)!-(n-4)!}+(7n^5-59n^4+179n^3-229n^2+102n)p^{4(n-2)!-2(n-3)!-(n-4)!}+(n^4-6n^3+11n^2-6n)p^{4(n-2)!-3(n-3)!-(n-4)!}+\frac{1}{2}(n^8-15n^7+99n^6-377n^5+892n^4-1288n^3+1024n^2-336n)p^{4(n-2)!-2(n-4)!}+\frac{1}{2}(3n^7-45n^6+277n^5-893n^4+1580n^3-1438n^2+516n)p^{4(n-2)!-(n-3)!-2(n-4)!}+\frac{1}{2}(3n^6-38n^5+187n^4-442n^3+494n^2-204n)p^{4(n-2)!-2(n-3)!-2(n-4)!}+\frac{1}{2}(n^5-10n^4+35n^3-50n^2+24n)p^{4(n-2)!-3(n-3)!-2(n-4)!}$.
\end{proof}

\begin{theorem}\label{th11}
$P(1,2,1)=(2n^7-17n^6+65n^5-139n^4+173n^3-118n^2+34n)p^{4(n-2)!}+(8n^6-61n^5+199n^4-347n^3+315n^2-114n)p^{4(n-2)!-(n-3)!}+(n^6-32n^4+112n^3-143n^2+62n)p^{4(n-2)!-2(n-3)!}+2(n^5-8n^4+23n^3-28n^2+12n)p^{4(n-2)!-3(n-3)!}+\frac{1}{2}(n^8-12n^7+65n^6-207n^5+416n^4-525n^3+382n^2-120n)p^{4(n-2)!-(n-4)!}+2(n^7-12n^6+61n^5-171n^4+280n^3-249n^2+90n)p^{4(n-2)!-(n-3)!-(n-4)!}+(2n^6-25n^5+121n^4-281n^3+309n^2-126n)p^{4(n-2)!-2(n-3)!-(n-4)!}$.
\end{theorem}
\begin{proof}
Denote by $a_1 a_2 X^{n-2}\in H_{1, 2}$, $b_1 X^{n-2} b_2\in H_{1, n}$, $c_1 X^{n-2} c_2\in H_{1, n}$ and $X^{n-2} d_1 d_2\in H_{n-1, n}$ these four $B_{n-2}$'s. Observation \ref{ob:1} implies that $b_1 X^{n-2} b_2$ and $c_1 X^{n-2} c_2$ are disjoint.

Since we have implicitly sorted the values of $b_1 b_2$ and $c_1 c_2$. The number of ways to choose four $B_{n-2}$'s as required needs to be divided by 2. It will be applied directly in following analyses.

We will analyze this problem in 9 cases as follows.

{\it Case 1.} $\{a_1, b_1\}\cap\{d_1, d_2\}=\emptyset$.

Observation \ref{ob:4} implies that $a_1 a_2 X^{n-2}$ and $X^{n-2} d_1 d_2$ are not disjoint, and $V(a_1 a_2 X^{n-2})\cap V(X^{n-2} d_1 d_2)=V(a_1 a_2 X^{n-4} d_1 d_2)$.

{\it Case 1.1.} $b_1\neq a_1$, $c_1\neq a_1$, $b_2\neq d_2$ and $c_2\neq d_2$.

{\it Case 1.1.1.} $b_1\neq c_1$, $b_1\neq d_2$ and $c_1\neq d_2$.

There are $n$ ways to choose $d_1$ from $N_n$, $n-1$ ways to choose $d_2$ from $N_n\setminus\{d_1\}$, $n-2$ ways to choose $a_1$ from $N_n\setminus\{d_1, d_2\}$, $n-3$ ways to choose $a_2$ from $N_n\setminus\{d_1, d_2, a_1\}$, $n-2$ ways to choose $b_1$ from $N_n\setminus\{a_1, d_2\}$, $n-3$ ways to choose $c_1$ from $N_n\setminus\{a_1, b_1, d_2\}$, $n-2$ ways to choose $b_2$ from $N_n\setminus\{b_1, d_2\}$, $n-2$ ways to choose $c_2$ from $N_n\setminus\{c_1, d_2\}$.
Therefore, there are $n(n-1)(n-2)^4(n-3)^2=n^8-15n^7+95n^6-329n^5+672n^4-808n^3+528n^2-144n$ ways to choose four $B_{n-2}$'s as required.

{\it Case 1.1.2.} $b_1\neq c_1$, $b_1\neq d_2=c_1$ or $b_1\neq c_1$, $c_1\neq d_2=b_1$.

W.l.o.g, assume that $b_1\neq d_2=c_1$. There are $n$ ways to choose $d_1$ from $N_n$, $n-1$ ways to choose $d_2(=c_1)$ from $N_n\setminus\{d_1\}$, $n-2$ ways to choose $a_1$ from $N_n\setminus\{d_1, d_2\}$, $n-3$ ways to choose $a_2$ from $N_n\setminus\{d_1, d_2, a_1\}$, $n-2$ ways to choose $b_1$ from $N_n\setminus\{a_1, d_2\}$, $n-2$ ways to choose $b_2$ from $N_n\setminus\{b_1, d_2\}$, $n-1$ ways to choose $c_2$ from $N_n\setminus\{d_2\}$.
Therefore, there are $2n(n-1)^2(n-2)^3(n-3)=2(n^7-11n^6+49n^5-113n^4+142n^3-92n^2+24n)$ ways to choose four $B_{n-2}$'s as required.

{\it Case 1.1.3.} $b_1=c_1$ and $b_1\neq d_2$.

There are $n$ ways to choose $d_1$ from $N_n$, $n-1$ ways to choose $d_2$ from $N_n\setminus\{d_1\}$, $n-2$ ways to choose $a_1$ from $N_n\setminus\{d_1, d_2\}$, $n-3$ ways to choose $a_2$ from $N_n\setminus\{d_1, d_2, a_1\}$, $n-2$ ways to choose $b_1(=c_1)$ from $N_n\setminus\{b_1, d_2\}$, $n-2$ ways to choose $b_2$ from $N_n\setminus\{b_1, d_2\}$, $n-3$ ways to choose $c_2$ from $N_n\setminus\{c_1, b_2, d_2\}$.
Therefore, there are $n(n-1)(n-2)^3(n-3)^2=n^7-13n^6+69n^5-191n^4+290n^3-228n^2+72n$ ways to choose four $B_{n-2}$'s as required.

{\it Case 1.1.4.} $b_1=c_1$ and $b_1=d_2$.

There are $n$ ways to choose $d_1$ from $N_n$, $n-1$ ways to choose $d_2(=b_1=c_1)$ from $N_n\setminus\{d_1\}$, $n-2$ ways to choose $a_1$ from $N_n\setminus\{d_1, d_2\}$, $n-3$ ways to choose $a_2$ from $N_n\setminus\{d_1, d_2, a_1\}$, $n-1$ ways to choose $b_2$ from $N_n\setminus\{d_2\}$, $n-2$ ways to choose $c_2$ from $N_n\setminus\{b_2, d_2\}$.
Therefore, there are $n(n-1)^2(n-2)^2(n-3)=n^6-9n^5+31n^4-51n^3+40n^2-12n$ ways to choose four $B_{n-2}$'s as required.

In summary, there are $n^8-12n^7+61n^6-171n^5+286n^4-285n^3+156n^2-36n$ ways to choose four $B_{n-2}$'s as required in Case 1.1.
Observations \ref{ob:2} and \ref{ob:3} imply that for any $i\in\{b, c\}$, $a_1 a_2 X^{n-2}$ and $i_1 X^{n-2} i_2$ are disjoint, and $X^{n-2} d_1 d_2$ and $i_1 X^{n-2} i_2$ are disjoint. Thus, $|V(a_1 a_2 X^{n-2})\cup V(b_1 X^{n-2} b_2)\cup V(c_1 X^{n-2} c_2)\cup V(X^{n-2} d_1 d_2)|=|V(a_1 a_2 X^{n-2})|+|V(b_1 X^{n-2} b_2)|+|V(c_1 X^{n-2} c_2)|+|V(X^{n-2} d_1 d_2)|-|V(a_1 a_2 X^{n-2})\cap V(X^{n-2} d_1 d_2)|=4(n-2)!-(n-4)!$.
The probability $P_{1.1}$ that there are four fault-free $B_{n-2}$'s chosen as in Case 1.1 is $\frac{1}{2}(n^8-12n^7+61n^6-171n^5+286n^4-285n^3+156n^2-36n)p^{4(n-2)!-(n-4)!}$.

{\it Case 1.2.} $b_1\neq c_1$, $b_1\neq a_1$, $c_1\neq a_1$, $i_2\neq j_2=d_2$ or $b_2\neq c_2$, $b_2\neq d_2$, $c_2\neq d_2$, $i_1\neq j_1=a_1$, where $\{i, j\}=\{b, c\}$.

There are 4 possible cases, w.l.o.g, assume that $b_1\neq c_1$, $b_1\neq a_1$, $c_1\neq a_1$, $b_2\neq c_2=d_2$.

{\it Case 1.2.1.} $c_1\neq d_1$.

If $b_1\neq d_2$, there are $n$ ways to choose $d_1$ from $N_n$, $n-1$ ways to choose $d_2(=c_2)$ from $N_n\setminus\{d_1\}$, $n-2$ ways to choose $a_1$ from $N_n\setminus\{d_1, d_2\}$, $n-3$ ways to choose $a_2$ from $N_n\setminus\{d_1, d_2, a_1\}$, $n-3$ ways to choose $c_1$ from $N_n\setminus\{a_1, c_2, d_1\}$, $n-3$ ways to choose $b_1$ from $N_n\setminus\{a_1, c_1, d_2\}$, $n-2$ ways to choose $b_2$ from $N_n\setminus\{b_1, d_2\}$.
Therefore, there are $n(n-1)(n-2)^2(n-3)^3=n^7-14n^6+80n^5-238n^4+387n^3-324n^2+108n$ ways to choose four $B_{n-2}$'s as required.

If $b_1=d_2$, there are $n$ ways to choose $d_1$ from $N_n$, $n-1$ ways to choose $d_2(=c_2=b_1)$ from $N_n\setminus\{d_1\}$, $n-2$ ways to choose $a_1$ from $N_n\setminus\{d_1, d_2\}$, $n-3$ ways to choose $a_2$ from $N_n\setminus\{d_1, d_2, a_1\}$, $n-3$ ways to choose $c_1$ from $N_n\setminus\{a_1, c_2, d_1\}$, $n-1$ ways to choose $b_1$ from $N_n\setminus\{b_1\}$.
Therefore, there are $n(n-1)^2(n-2)(n-3)^2=n^6-10n^5+38n^4-68n^3+57n^2-18n$ ways to choose four $B_{n-2}$'s as required.

In summary, there are $n^7-13n^6+70n^5-200n^4+319n^3-267n^2+90n$ ways to choose four $B_{n-2}$'s as required in Case 1.2.1.
Observations \ref{ob:2} and \ref{ob:3} imply that for any $i\in\{b, c\}$, $a_1 a_2 X^{n-2}$ and $i_1 X^{n-2} i_2$ are disjoint, and $X^{n-2} d_1 d_2$ and $b_1 X^{n-2} b_2$ are disjoint, and $X^{n-2} d_1 d_2$ and $c_1 X^{n-2} c_2$ are not disjoint and $V(X^{n-2} d_1 d_2)\cap V(c_1 X^{n-2} c_2)=V(c_1 X^{n-3} d_1 d_2)$. Thus, $|V(a_1 a_2 X^{n-2})\cup V(b_1 X^{n-2} b_2)\cup V(c_1 X^{n-2} c_2)\cup V(X^{n-2} d_1 d_2)|=|V(a_1 a_2 X^{n-2})|+|V(b_1 X^{n-2} b_2)|+|V(c_1 X^{n-2} c_2)|+|V(X^{n-2} d_1 d_2)|-|V(X^{n-2} d_1 d_2)\cap V(c_1 X^{n-2} c_2)|-|V(a_1 a_2$ $X^{n-2})\cap V(X^{n-2} d_1 d_2)|=4(n-2)!-(n-3)!-(n-4)!$. 
The probability $P_{1.2.1}$ that there are four fault-free $B_{n-2}$'s chosen as in Case 1.2.1 is $\frac{1}{2}(n^7-13n^6+70n^5-200n^4+319n^3-267n^2+90n)p^{4(n-2)!-(n-3)!-(n-4)!}$.

{\it Case 1.2.2.} $c_1=d_1$.

If $b_1\neq d_2$, there are $n$ ways to choose $c_1(=d_1)$ from $N_n$, $n-1$ ways to choose $d_2(=c_2)$ from $N_n\setminus\{d_1\}$, $n-2$ ways to choose $a_1$ from $N_n\setminus\{d_1, d_2\}$, $n-3$ ways to choose $a_2$ from $N_n\setminus\{d_1, d_2, a_1\}$, $n-3$ ways to choose $b_1$ from $N_n\setminus\{a_1, c_1, d_2\}$, $n-2$ ways to choose $b_2$ from $N_n\setminus\{b_1, d_2\}$.
Therefore, there are $n(n-1)(n-2)^2(n-3)^2=n^6-11n^5+47n^4-97n^3+96n^2-36n$ ways to choose four $B_{n-2}$'s as required.

If $b_1=d_2$, there are $n$ ways to choose $c_1(=d_1)$ from $N_n$, $n-1$ ways to choose $d_2(=c_2=b_1)$ from $N_n\setminus\{d_1\}$, $n-2$ ways to choose $a_1$ from $N_n\setminus\{d_1, d_2\}$, $n-3$ ways to choose $a_2$ from $N_n\setminus\{d_1, d_2, a_1\}$, $n-1$ ways to choose $b_2$ from $N_n\setminus\{b_1\}$.
Therefore, there are $n(n-1)^2(n-2)(n-3)=n^5-7n^4+17n^3-17n^2+6n$ ways to choose four $B_{n-2}$'s as required.

In summary, there are $n^6-10n^5+40n^4-80n^3+79n^2-30n$ ways to choose four $B_{n-2}$'s as required in Case 1.2.2.
Observations \ref{ob:2} and \ref{ob:3} imply that for any $i\in\{b, c\}$, $a_1 a_2 X^{n-2}$ and $i_1 X^{n-2} i_2$ are disjoint, and $X^{n-2} d_1 d_2$ and $i_1 X^{n-2} i_2$ are disjoint. Thus, $|V(a_1 a_2 X^{n-2})\cup V(b_1 X^{n-2} b_2)\cup V(c_1 X^{n-2} c_2)\cup V(X^{n-2} d_1 d_2)|=|V(a_1 a_2 X^{n-2})|+|V(b_1 X^{n-2} b_2)|+|V(c_1 X^{n-2} c_2)|+|V(X^{n-2} d_1 d_2)|-|V(a_1 a_2 X^{n-2})\cap V(X^{n-2} d_1 d_2)|=4(n-2)!-(n-4)!$.
The probability $P_{1.2.2}$ that there are four fault-free $B_{n-2}$'s chosen as in Case 1.2.2 is $\frac{1}{2}(n^6-10n^5+40n^4-80n^3+79n^2-30n)p^{4(n-2)!-(n-4)!}$.

Thus, the probability $P_{1.2}$ that there are four fault-free $B_{n-2}$'s chosen as in Case 1.2 is $4(P_{1.2.1}+P_{1.2.2})=2(n^6-10n^5+40n^4-80n^3+79n^2-30n)p^{4(n-2)!-(n-4)!}+2(n^7-13n^6+70n^5-200n^4+319n^3-267n^2+90n)p^{4(n-2)!-(n-3)!-(n-4)!}$.

{\it Case 1.3.} $b_1\neq c_1$, $b_1\neq a_1$, $c_1\neq a_1$, $b_2=c_2=d_2$ or $b_2\neq c_2$, $b_2\neq d_2$, $c_2\neq d_2$, $b_1=c_1=a_1$.

W.l.o.g, assume that the former applies.

{\it Case 1.3.1.} $b_1\neq d_1$ and $c_1\neq d_1$.

There are $n$ ways to choose $d_1$ from $N_n$, $n-1$ ways to choose $d_2(=b_2=c_2)$ from $N_n\setminus\{d_1\}$, $n-2$ ways to choose $a_1$ from $N_n\setminus\{d_1, d_2\}$, $n-3$ ways to choose $a_2$ from $N_n\setminus\{d_1, d_2, a_1\}$, $n-3$ ways to choose $b_1$ from $N_n\setminus\{a_1, b_2, d_1\}$, $n-4$ ways to choose $c_1$ from $N_n\setminus\{a_1, b_1, c_2, d_1\}$.
Therefore, there are $n(n-1)(n-2)(n-3)^2(n-4)=n^6-13n^5+65n^4-155n^3+174n^2-72n$ ways to choose four $B_{n-2}$'s as required.
Observations \ref{ob:2} and \ref{ob:3} imply that for any $i\in\{b, c\}$, $a_1 a_2 X^{n-2}$ and $i_1 X^{n-2} i_2$ are disjoint, and $X^{n-2} d_1 d_2$ and $i_1 X^{n-2} i_2$ are not disjoint and $V(X^{n-2} d_1 d_2)\cap V(i_1 X^{n-2} i_2)=V(i_1 X^{n-3} d_1 d_2)$. Thus, $|V(a_1 a_2 X^{n-2})\cup V(b_1 X^{n-2} b_2)\cup V(c_1 X^{n-2} c_2)\cup V(X^{n-2} d_1 d_2)|=|V(a_1 a_2 X^{n-2})|+|V(b_1 X^{n-2} b_2)|+|V(c_1 X^{n-2} c_2)|+|V(X^{n-2} d_1 d_2)|-|V(X^{n-2} d_1 d_2)\cap V(b_1 X^{n-2}$ $b_2)|-|V(X^{n-2} d_1 d_2)\cap V(c_1 X^{n-2} c_2)|-|V(a_1 a_2 X^{n-2})\cap V(X^{n-2} d_1 d_2)|=4(n-2)!-2(n-3)!-(n-4)!$. 
The probability $P_{1.3.1}$ that there are four fault-free $B_{n-2}$'s chosen as in Case 1.3.1 is $\frac{1}{2}(n^6-13n^5+65n^4-155n^3+174n^2-72n)p^{4(n-2)!-2(n-3)!-(n-4)!}$.

{\it Case 1.3.2.} $b_1=d_1\neq c_1$ or $c_1=d_1\neq b_1$.

W.l.o.g, assume that $b_1=d_1\neq c_1$.

There are $n$ ways to choose $d_1(=b_1)$ from $N_n$, $n-1$ ways to choose $d_2(=b_2=c_2)$ from $N_n\setminus\{d_1\}$, $n-2$ ways to choose $a_1$ from $N_n\setminus\{d_1, d_2\}$, $n-3$ ways to choose $a_2$ from $N_n\setminus\{d_1, d_2, a_1\}$, $n-3$ ways to choose $c_1$ from $N_n\setminus\{a_1, b_1, c_2\}$.
Therefore, there are $2n(n-1)(n-2)(n-3)^2=2(n^5-9n^4+29n^3-39n^2+18n)$ ways to choose four $B_{n-2}$'s as required.
Observations \ref{ob:2} and \ref{ob:3} imply that for any $i\in\{b, c\}$, $a_1 a_2 X^{n-2}$ and $i_1 X^{n-2} i_2$ are disjoint, and $X^{n-2} d_1 d_2$ and $b_1 X^{n-2} b_2$ are disjoint, and $X^{n-2} d_1 d_2$ and $c_1 X^{n-2} c_2$ are not disjoint and $V(X^{n-2} d_1 d_2)\cap V(c_1 X^{n-2} c_2)=V(c_1 X^{n-3} d_1 d_2)$. Thus, $|V(a_1 a_2 X^{n-2})\cup V(b_1 X^{n-2} b_2)\cup V(c_1 X^{n-2} c_2)\cup V(X^{n-2} d_1 d_2)|=|V(a_1 a_2 X^{n-2})|+|V(b_1 X^{n-2} b_2)|+|V(c_1 X^{n-2} c_2)|+|V(X^{n-2} d_1 d_2)|-|V(X^{n-2} d_1 d_2)\cap V(c_1 X^{n-2} c_2)|-|V(a_1 a_2$ $X^{n-2})\cap V(X^{n-2} d_1 d_2)|=4(n-2)!-(n-3)!-(n-4)!$. 
The probability $P_{1.3.2}$ that there are four fault-free $B_{n-2}$'s chosen as in Case 1.3.2 is $\frac{1}{2}\times 2(n^5-9n^4+29n^3-39n^2+18n)p^{4(n-2)!-(n-3)!-(n-4)!}$.

Thus, the probability $P_{1.3}$ that there are four fault-free $B_{n-2}$'s chosen as in Case 1.3 is $2(P_{1.3.1}+P_{1.3.2})=2(n^5-9n^4+29n^3-39n^2+18n)p^{4(n-2)!-(n-3)!-(n-4)!}+(n^6-13n^5+65n^4-155n^3+174n^2-72n)p^{4(n-2)!-2(n-3)!-(n-4)!}+$.

{\it Case 1.4.} $b_1\neq c_1=a_1$, $b_2\neq c_2=d_2$ or $c_1\neq b_1=a_1$, $c_2\neq b_2=d_2$.

W.l.o.g, assume that the former applies.

{\it Case 1.4.1.} $b_1\neq d_2$.

There are $n$ ways to choose $d_1$ from $N_n$, $n-1$ ways to choose $d_2(=c_2)$ from $N_n\setminus\{d_1\}$, $n-2$ ways to choose $a_1(=c_1)$ from $N_n\setminus\{d_1, d_2\}$, $n-3$ ways to choose $a_2$ from $N_n\setminus\{d_1, d_2, a_1\}$, $n-2$ ways to choose $b_1$ from $N_n\setminus\{a_1, d_2\}$, $n-2$ ways to choose $b_2$ from $N_n\setminus\{d_2, b_1\}$.
Therefore, there are $n(n-1)(n-2)^3(n-3)=n^6-10n^5+39n^4-74n^3+68n^2-24n$ ways to choose four $B_{n-2}$'s as required.

{\it Case 1.4.2.} $b_1=d_2$.

There are $n$ ways to choose $d_1$ from $N_n$, $n-1$ ways to choose $d_2(=c_2=b_1)$ from $N_n\setminus\{d_1\}$, $n-2$ ways to choose $a_1(=c_1)$ from $N_n\setminus\{d_1, d_2\}$, $n-3$ ways to choose $a_2$ from $N_n\setminus\{d_1, d_2, a_1\}$, $n-1$ ways to choose $b_2$ from $N_n\setminus\{d_2\}$.
Therefore, there are $n(n-1)^2(n-2)(n-3)=n^5-7n^4+17n^3-17n^2+6n$ ways to choose four $B_{n-2}$'s as required.

In summary, there are $n^6-9n^5+32n^4-57n^3+51n^2-18n$ ways to choose four $B_{n-2}$'s as required in Case 1.4.
Observations \ref{ob:2} and \ref{ob:3} imply that $a_1 a_2 X^{n-2}$ and $b_1 X^{n-2} b_2$ are disjoint, and $a_1 a_2 X^{n-2}$ and $c_1 X^{n-2} c_2$ are not disjoint and $V(a_1 a_2 X^{n-2})\cap V(c_1 X^{n-2} c_2)=V(a_1 a_2 X^{n-3} c_2)$, and $X^{n-2} d_1 d_2$ and $b_1 X^{n-2} b_2$ are disjoint, and $X^{n-2} d_1 d_2$ and $c_1 X^{n-2} c_2$ are not disjoint and $V(X^{n-2} d_1 d_2)\cap V(c_1 X^{n-2} c_2)=V(c_1 X^{n-3} d_1 d_2)$. And $V(a_1 a_2 X^{n-2})\cap V(c_1 X^{n-2} c_2)\cap V(X^{n-2} d_1 d_2)=V(a_1 a_2 X^{n-2} d_1 d_2)=V(a_1 a_2 X^{n-2})\cap V(X^{n-2} d_1 d_2)$.
Thus, $|V(a_1 a_2 X^{n-2})\cup V(b_1 X^{n-2} b_2)\cup V(c_1 X^{n-2} c_2)\cup V(X^{n-2} d_1 d_2)|=|V(a_1 a_2 X^{n-2})|+|V(b_1 X^{n-2} b_2)|+|V(c_1 X^{n-2} c_2)|+|V(X^{n-2} d_1 d_2)|-|V(a_1 a_2 X^{n-2})\cap V(c_1 X^{n-2} c_2)|-|V(X^{n-2} d_1 d_2)\cap V(c_1 X^{n-2} c_2)|-|V(a_1 a_2$ $X^{n-2})\cap V(X^{n-2} d_1 d_2)|+|V(a_1 a_2 X^{n-2})\cap V(c_1 X^{n-2} c_2)\cap V(X^{n-2} d_1 d_2)|=4(n-2)!-2(n-3)!$. 
The probability $P_{1.4}$ that there are four fault-free $B_{n-2}$'s chosen as in Case 1.4 is $\frac{1}{2}\times 2(n^6-9n^5+32n^4-57n^3+51n^2-18n)p^{4(n-2)!-2(n-3)!}$.

{\it Case 1.5.} $b_1\neq c_1=a_1$, $c_2\neq b_2=d_2$ or $c_1\neq b_1=a_1$, $b_2\neq c_2=d_2$.

W.l.o.g, assume that the former applies.

{\it Case 1.5.1.} $a_2\neq c_2$ and $b_1\neq d_1$.

There are $n$ ways to choose $d_1$ from $N_n$, $n-1$ ways to choose $d_2(=b_2)$ from $N_n\setminus\{d_1\}$, $n-2$ ways to choose $a_1(=c_1)$ from $N_n\setminus\{d_1, d_2\}$, $n-3$ ways to choose $a_2$ from $N_n\setminus\{d_1, d_2, a_1\}$, $n-3$ ways to choose $b_1$ from $N_n\setminus\{a_1, b_2, d_1\}$, $n-3$ ways to choose $c_2$ from $N_n\setminus\{d_2, c_1, a_2\}$.
Therefore, there are $n(n-1)(n-2)(n-3)^3=n^6-12n^5+56n^4-126n^3+135n^2-54n$ ways to choose four $B_{n-2}$'s as required.
Observations \ref{ob:2} and \ref{ob:3} imply that $a_1 a_2 X^{n-2}$ and $b_1 X^{n-2} b_2$ are disjoint, and $a_1 a_2 X^{n-2}$ and $c_1 X^{n-2} c_2$ are not disjoint and $V(a_1 a_2 X^{n-2})\cap V(c_1 X^{n-2} c_2)=V(a_1 a_2 X^{n-3} c_2)$, and $X^{n-2} d_1 d_2$ and $b_1 X^{n-2} b_2$ are not disjoint and $V(X^{n-2} d_1 d_2)\cap V(b_1 X^{n-2} b_2)=V(b_1 X^{n-3} d_1 d_2)$, and $X^{n-2} d_1 d_2$ and $c_1 X^{n-2} c_2$ are disjoint. Thus, $|V(a_1 a_2 X^{n-2})\cup V(b_1 X^{n-2} b_2)\cup V(c_1 X^{n-2} c_2)\cup V(X^{n-2} d_1 d_2)|=|V(a_1 a_2 X^{n-2})|+|V(b_1 X^{n-2} b_2)|+|V(c_1 X^{n-2} c_2)|+|V(X^{n-2} d_1 d_2)|-|V(a_1 a_2 X^{n-2})\cap V(c_1 X^{n-2} c_2)|-|V(X^{n-2} d_1 d_2)\cap V(b_1 X^{n-2} b_2)|-|V(a_1 a_2$ $X^{n-2})\cap V(X^{n-2} d_1 d_2)|=4(n-2)!-2(n-3)!-(n-4)!$. 
The probability $P_{1.5.1}$ that there are four fault-free $B_{n-2}$'s chosen as in Case 1.5.1 is $\frac{1}{2}(n^6-12n^5+56n^4-126n^3+135n^2-54n)p^{4(n-2)!-2(n-3)!-(n-4)!}$.

{\it Case 1.5.2.} $a_2\neq c_2$, $b_1=d_1$ or $a_2=c_2$, $b_1\neq d_1$.

W.l.o.g, assume that $a_2\neq c_2$ and $b_1=d_1$.

There are $n$ ways to choose $d_1(=b_1)$ from $N_n$, $n-1$ ways to choose $d_2(=b_2)$ from $N_n\setminus\{d_1\}$, $n-2$ ways to choose $a_1(=c_1)$ from $N_n\setminus\{d_1, d_2\}$, $n-3$ ways to choose $a_2$ from $N_n\setminus\{d_1, d_2, a_1\}$, $n-3$ ways to choose $c_2$ from $N_n\setminus\{d_2, c_1, a_2\}$.
Therefore, there are $2n(n-1)(n-2)(n-3)^2=2(n^5-9n^4+29n^3-39n^2+18n)$ ways to choose four $B_{n-2}$'s as required.
Observations \ref{ob:2} and \ref{ob:3} imply that $a_1 a_2 X^{n-2}$ and $b_1 X^{n-2} b_2$ are disjoint, and $a_1 a_2 X^{n-2}$ and $c_1 X^{n-2} c_2$ are not disjoint and $V(a_1 a_2 X^{n-2})\cap V(c_1 X^{n-2} c_2)=V(a_1 a_2 X^{n-3} c_2)$, and for any $i\in\{b, c\}$, $X^{n-2} d_1 d_2$ and $i_1 X^{n-2} i_2$ are disjoint. Thus, $|V(a_1 a_2 X^{n-2})\cup V(b_1 X^{n-2} b_2)\cup V(c_1 X^{n-2} c_2)\cup V(X^{n-2} d_1 d_2)|=|V(a_1 a_2 X^{n-2})|+|V(b_1 X^{n-2} b_2)|+|V(c_1 X^{n-2} c_2)|+|V(X^{n-2} d_1 d_2)|-|V(a_1 a_2 X^{n-2})\cap V(c_1 X^{n-2} c_2)|-|V(a_1 a_2$ $X^{n-2})\cap V(X^{n-2} d_1 d_2)|=4(n-2)!-(n-3)!-(n-4)!$. 
The probability $P_{1.5.2}$ that there are four fault-free $B_{n-2}$'s chosen as in Case 1.5.2 is $\frac{1}{2}\times 2(n^5-9n^4+29n^3-39n^2+18n)p^{4(n-2)!-(n-3)!-(n-4)!}$.

{\it Case 1.5.3.} $a_2=c_2$ and $b_1=d_1$.

There are $n$ ways to choose $d_1(=b_1)$ from $N_n$, $n-1$ ways to choose $d_2(=b_2)$ from $N_n\setminus\{d_1\}$, $n-2$ ways to choose $a_1(=c_1)$ from $N_n\setminus\{d_1, d_2\}$, $n-3$ ways to choose $a_2(=c_2)$ from $N_n\setminus\{d_1, d_2, a_1\}$.
Therefore, there are $n(n-1)(n-2)(n-3)=n^4-6n^3+11n^2-6n$ ways to choose four $B_{n-2}$'s as required.
Observations \ref{ob:2} and \ref{ob:3} imply that for any $i\in\{b, c\}$, $a_1 a_2 X^{n-2}$ and $i_1 X^{n-2} i_2$ are disjoint, and $X^{n-2} d_1 d_2$ and $i_1 X^{n-2} i_2$ are disjoint. Thus, $|V(a_1 a_2 X^{n-2})\cup V(b_1 X^{n-2} b_2)\cup V(c_1 X^{n-2} c_2)\cup V(X^{n-2} d_1 d_2)|=|V(a_1 a_2 X^{n-2})|+|V(b_1 X^{n-2} b_2)|+|V(c_1 X^{n-2} c_2)|+|V(X^{n-2} d_1 d_2)|-|V(a_1 a_2 X^{n-2})\cap V(X^{n-2} d_1 d_2)|=4(n-2)!-(n-4)!$. 
The probability $P_{1.5.3}$ that there are four fault-free $B_{n-2}$'s chosen as in Case 1.5.3 is $\frac{1}{2}(n^4-6n^3+11n^2-6n)p^{4(n-2)!-(n-4)!}$.

Thus, the probability $P_{1.5}$ that there are four fault-free $B_{n-2}$'s chosen as in Case 1.5 is $2(P_{1.5.1}+P_{1.5.2}+P_{1.5.3})=(n^4-6n^3+11n^2-6n)p^{4(n-2)!-(n-4)!}+2(n^5-9n^4+29n^3-39n^2+18n)p^{4(n-2)!-(n-3)!-(n-4)!}+(n^6-12n^5+56n^4-126n^3+135n^2-54n)p^{4(n-2)!-2(n-3)!-(n-4)!}$.

{\it Case 1.6.} For $\{i, j\}=\{b, c\}$, $i_1\neq j_1=a_1$, $b_2=c_2\neq d_2$ or $i_2\neq j_2=d_2$, $b_1=c_1\neq a_1$.

There are 4 possible cases, w.l.o.g, assume that $b_1\neq c_1=a_1$ and $b_2=c_2\neq d_2$.

{\it Case 1.6.1.} $a_2\neq c_2$.

There are $n$ ways to choose $d_1$ from $N_n$, $n-1$ ways to choose $d_2$ from $N_n\setminus\{d_1\}$, $n-2$ ways to choose $a_1(=c_1)$ from $N_n\setminus\{d_1, d_2\}$, $n-3$ ways to choose $a_2$ from $N_n\setminus\{d_1, d_2, a_1\}$, $n-3$ ways to choose $b_2(=c_2)$ from $N_n\setminus\{c_1, d_2, a_2\}$, $n-2$ ways to choose $b_1$ from $N_n\setminus\{a_1, b_2\}$.
Therefore, there are $n(n-1)(n-2)^2(n-3)^2=n^6-11n^5+47n^4-97n^3+96n^2-36n$ ways to choose four $B_{n-2}$'s as required.
Observations \ref{ob:2} and \ref{ob:3} imply that $a_1 a_2 X^{n-2}$ and $b_1 X^{n-2} b_2$ are disjoint, and $a_1 a_2 X^{n-2}$ and $c_1 X^{n-2} c_2$ are not disjoint and $V(a_1 a_2 X^{n-2})\cap V(c_1 X^{n-2} c_2)=V(a_1 a_2 X^{n-3} c_2)$, and for any $i\in\{b, c\}$, $X^{n-2} d_1 d_2$ and $i_1 X^{n-2} i_2$ are disjoint. Thus, $|V(a_1 a_2 X^{n-2})\cup V(b_1 X^{n-2} b_2)\cup V(c_1 X^{n-2} c_2)\cup V(X^{n-2} d_1 d_2)|=|V(a_1 a_2 X^{n-2})|+|V(b_1 X^{n-2} b_2)|+|V(c_1 X^{n-2} c_2)|+|V(X^{n-2} d_1 d_2)|-|V(a_1 a_2 X^{n-2})\cap V(c_1 X^{n-2} c_2)|-|V(a_1 a_2 X^{n-2})\cap V(X^{n-2} d_1 d_2)|=4(n-2)!-(n-3)!-(n-4)!$. 
The probability $P_{1.6.1}$ that there are four fault-free $B_{n-2}$'s chosen as in Case 1.6.1 is $\frac{1}{2}(n^6-11n^5+47n^4-97n^3+96n^2-36n)p^{4(n-2)!-(n-3)!-(n-4)!}$.

{\it Case 1.6.2.} $a_2=c_2$.

There are $n$ ways to choose $d_1$ from $N_n$, $n-1$ ways to choose $d_2$ from $N_n\setminus\{d_1\}$, $n-2$ ways to choose $a_1(=c_1)$ from $N_n\setminus\{d_1, d_2\}$, $n-3$ ways to choose $a_2(=b_2=c_2)$ from $N_n\setminus\{d_1, d_2, a_1\}$, $n-2$ ways to choose $b_1$ from $N_n\setminus\{a_1, b_2\}$.
Therefore, there are $n(n-1)(n-2)^2(n-3)=n^5-8n^4+23n^3-28n^2+12n$ ways to choose four $B_{n-2}$'s as required.
Observations \ref{ob:2} and \ref{ob:3} imply that for any $i\in\{b, c\}$, $a_1 a_2 X^{n-2}$ and $i_1 X^{n-2} i_2$ are disjoint, and $X^{n-2} d_1 d_2$ and $i_1 X^{n-2} i_2$ are disjoint. Thus, $|V(a_1 a_2 X^{n-2})\cup V(b_1 X^{n-2} b_2)\cup V(c_1 X^{n-2} c_2)\cup V(X^{n-2} d_1 d_2)|=|V(a_1 a_2 X^{n-2})|+|V(b_1 X^{n-2} b_2)|+|V(c_1 X^{n-2} c_2)|+|V(X^{n-2} d_1 d_2)|-|V(a_1 a_2 X^{n-2})\cap V(X^{n-2} d_1 d_2)|=4(n-2)!-(n-4)!$. 
The probability $P_{1.6.2}$ that there are four fault-free $B_{n-2}$'s chosen as in Case 1.6.2 is $\frac{1}{2}(n^5-8n^4+23n^3-28n^2+12n)p^{4(n-2)!-(n-4)!}$.

Thus, the probability $P_1.6$ that there are four fault-free $B_{n-2}$'s chosen as in Case 1.6 is $4(P_{1.6.1}+P_{1.6.2})=2(n^5-8n^4+23n^3-28n^2+12n)p^{4(n-2)!-(n-4)!}+2(n^6-11n^5+47n^4-97n^3+96n^2-36n)p^{4(n-2)!-(n-3)!-(n-4)!}$.

{\it Case 1.7.} For $\{i, j\}=\{b, c\}$, $i_1\neq j_1=a_1$, $b_2=c_2=d_2$ or $i_2\neq j_2=d_2$, $b_1=c_1=a_1$.

There are 4 possible cases, w.l.o.g, assume that $b_1\neq c_1=a_1$, $b_2=c_2=d_2$.

{\it Case 1.7.1.} $b_1\neq d_1$.

There are $n$ ways to choose $d_1$ from $N_n$, $n-1$ ways to choose $b_2(=c_2=d_2)$ from $N_n\setminus\{d_1\}$, $n-2$ ways to choose $a_1(=c_1)$ from $N_n\setminus\{d_1, d_2\}$, $n-3$ ways to choose $a_2$ from $N_n\setminus\{d_1, d_2, a_1\}$, $n-3$ ways to choose $b_1$ from $N_n\setminus\{a_1, b_2, d_1\}$.
Therefore, there are $n(n-1)(n-2)(n-3)^2=n^5-9n^4+29n^3-39n^2+18n$ ways to choose four $B_{n-2}$'s as required.
Observations \ref{ob:2} and \ref{ob:3} imply that $a_1 a_2 X^{n-2}$ and $b_1 X^{n-2} b_2$ are disjoint, and $a_1 a_2 X^{n-2}$ and $c_1 X^{n-2} c_2$ are not disjoint and $V(a_1 a_2 X^{n-2})\cap V(c_1 X^{n-2} c_2)=V(a_1 a_2 X^{n-3} c_2)$, and for any $i\in\{b, c\}$, $X^{n-2} d_1 d_2$ and $i_1 X^{n-2} i_2$ are not disjoint and $V(X^{n-2} d_1 d_2)\cap V(i_1 X^{n-2} i_2)=V(i_1 X^{n-3} d_1 d_2)$. And $V(a_1 a_2 X^{n-2})\cap V(c_1 X^{n-2} c_2)\cap V(X^{n-2} d_1 d_2)=V(a_1 a_2 X^{n-2} d_1 d_2)=V(a_1 a_2 X^{n-2})\cap V(X^{n-2} d_1 d_2)$.
Thus, $|V(a_1 a_2 X^{n-2})\cup V(b_1 X^{n-2} b_2)\cup V(c_1 X^{n-2} c_2)\cup V(X^{n-2} d_1 d_2)|=|V(a_1 a_2 X^{n-2})|+|V(b_1 X^{n-2} b_2)|+|V(c_1 X^{n-2} c_2)|+|V(X^{n-2} d_1 d_2)|-|V(a_1 a_2 X^{n-2})\cap V(c_1 X^{n-2} c_2)|-|V(X^{n-2} d_1 d_2)\cap V(b_1 X^{n-2} b_2)|-|V(X^{n-2} d_1$ $d_2)\cap V(c_1 X^{n-2} c_2)|-|V(a_1 a_2 X^{n-2})\cap V(X^{n-2} d_1 d_2)|+|V(a_1 a_2 X^{n-2})\cap V(c_1 X^{n-2} c_2)\cap V(X^{n-2} d_1 d_2)|=4(n-2)!-3(n-3)!$. 
The probability $P_{1.7.1}$ that there are four fault-free $B_{n-2}$'s chosen as in Case 1.7.1 is $\frac{1}{2}(n^5-9n^4+29n^3-39n^2+18n)p^{4(n-2)!-3(n-3)!}$.

{\it Case 1.7.2.} $b_1=d_1$.

There are $n$ ways to choose $d_1(=b_1)$ from $N_n$, $n-1$ ways to choose $b_2(=c_2=d_2)$ from $N_n\setminus\{d_1\}$, $n-2$ ways to choose $a_1(=c_1)$ from $N_n\setminus\{d_1, d_2\}$, $n-3$ ways to choose $a_2$ from $N_n\setminus\{d_1, d_2, a_1\}$.
Therefore, there are $n(n-1)(n-2)(n-3)=n^4-6n^3+11n^2-6n$ ways to choose four $B_{n-2}$'s as required.
Observations \ref{ob:2} and \ref{ob:3} imply that $a_1 a_2 X^{n-2}$ and $b_1 X^{n-2} b_2$ are disjoint, and $a_1 a_2 X^{n-2}$ and $c_1 X^{n-2} c_2$ are not disjoint and $V(a_1 a_2 X^{n-2})\cap V(c_1 X^{n-2} c_2)=V(a_1 a_2 X^{n-3} c_2)$, and $X^{n-2} d_1 d_2$ and $b_1 X^{n-2} b_2$ are disjoint, and $X^{n-2} d_1 d_2$ and $c_1 X^{n-2} c_2$ are not disjoint and $V(X^{n-2} d_1 d_2)\cap V(c_1 X^{n-2} c_2)=V(c_1 X^{n-3} d_1 d_2)$. And $V(a_1 a_2 X^{n-2})\cap V(c_1 X^{n-2} c_2)\cap V(X^{n-2} d_1 d_2)=V(a_1 a_2 X^{n-2} d_1 d_2)=V(a_1 a_2 X^{n-2})\cap V(X^{n-2} d_1 d_2)$.
Thus, $|V(a_1 a_2 X^{n-2})\cup V(b_1 X^{n-2} b_2)\cup V(c_1 X^{n-2} c_2)\cup V(X^{n-2} d_1 d_2)|=|V(a_1 a_2 X^{n-2})|+|V(b_1 X^{n-2} b_2)|+|V(c_1 X^{n-2} c_2)|+|V(X^{n-2} d_1 d_2)|-|V(a_1 a_2 X^{n-2})\cap V(c_1 X^{n-2} c_2)|-|V(X^{n-2} d_1 d_2)\cap V(c_1 X^{n-2} c_2)|-|V(a_1 a_2$ $X^{n-2})\cap V(X^{n-2} d_1 d_2)|+|V(a_1 a_2 X^{n-2})\cap V(c_1 X^{n-2} c_2)\cap V(X^{n-2} d_1 d_2)|=4(n-2)!-2(n-3)!$. 
The probability $P_{1.7.2}$ that there are four fault-free $B_{n-2}$'s chosen as in Case 1.7.2 is $\frac{1}{2}(n^4-6n^3+11n^2-6n)p^{4(n-2)!-2(n-3)!}$.

Thus, the probability $P_{1.7}$ that there are four fault-free $B_{n-2}$'s chosen as in Case 1.7 is $4(P_{1.7.1}+P_{1.7.2})=2(n^4-6n^3+11n^2-6n)p^{4(n-2)!-2(n-3)!}+2(n^5-9n^4+29n^3-39n^2+18n)p^{4(n-2)!-3(n-3)!}$.

Therefore, the probability $P_1$ that there are four fault-free $B_{n-2}$'s chosen as in Case 1 is $\sum_{i=1}^{7}P_{1.i}=(n^6-9n^5+34n^4-69n^3+73n^2-30n)p^{4(n-2)!-2(n-3)!}+2(n^5-9n^4+29n^3-39n^2+18n)p^{4(n-2)!-3(n-3)!}+\frac{1}{2}(n^8-12n^7+65n^6-207n^5+416n^4-525n^3+382n^2-120n)p^{4(n-2)!-(n-4)!}+2(n^7-12n^6+61n^5-171n^4+280n^3-249n^2+90n)p^{4(n-2)!-(n-3)!-(n-4)!}+(2n^6-25n^5+121n^4-281n^3+309n^2-126n)p^{4(n-2)!-2(n-3)!-(n-4)!}$.

{\it Case 2.} $\{a_1, a_2\}\cap\{d_1, d_2\}\neq\emptyset$, $b_1\neq a_1$, $c_1\neq a_1$, $b_2\neq d_2$ and $c_2\neq d_2$.

{\it Case 2.1.} $a_1=d_1$ and $a_2=d_2$.

{\it Case 2.1.1.} $b_1\neq c_1$, $b_1\neq d_2$ and $c_1\neq d_2$.

There are $n$ ways to choose $a_1(=d_1)$ from $N_n$, $n-1$ ways to choose $a_2(=d_2)$ from $N_n\setminus\{a_1\}$, $n-2$ ways to choose $b_1$ from $N_n\setminus\{a_1, d_2\}$, $n-3$ ways to choose $c_1$ from $N_n\setminus\{a_1, d_2, b_1\}$, $n-2$ ways to choose $b_2$ from $N_n\setminus\{b_1, d_2\}$, $n-2$ ways to choose $c_2$ from $N_n\setminus\{c_1, d_2\}$.
Therefore, there are $n(n-1)(n-2)^3(n-3)=n^6-10n^5+39n^4-74n^3+68n^2-24n$ ways to choose four $B_{n-2}$'s as required. 

{\it Case 2.1.2.} $b_1\neq c_1$, $b_1\neq d_2=c_1$ or $b_1\neq c_1$, $c_1\neq d_2=b_1$.

W.l.o.g, assume that $b_1\neq c_1$ and $b_1\neq d_2=c_1$.
There are $n$ ways to choose $a_1(=d_1)$ from $N_n$, $n-1$ ways to choose $a_2(=d_2=c_1)$ from $N_n\setminus\{a_1\}$, $n-2$ ways to choose $b_1$ from $N_n\setminus\{a_1, d_2\}$, $n-2$ ways to choose $b_2$ from $N_n\setminus\{b_1, d_2\}$, $n-1$ ways to choose $c_2$ from $N_n\setminus\{d_2\}$.
Therefore, there are $2n(n-1)^2(n-2)^2=2(n^5-6n^4+13n^3-12n^2+4n)$ ways to choose four $B_{n-2}$'s as required.

{\it Case 2.1.3.} $b_1=c_1$ and $b_1\neq d_2$.

There are $n$ ways to choose $a_1(=d_1)$ from $N_n$, $n-1$ ways to choose $a_2(=d_2)$ from $N_n\setminus\{a_1\}$, $n-2$ ways to choose $b_1(=c_1)$ from $N_n\setminus\{a_1, d_2\}$, $n-2$ ways to choose $b_2$ from $N_n\setminus\{b_1, d_2\}$, $n-3$ ways to choose $c_2$ from $N_n\setminus\{c_1, b_2, d_2\}$.
Therefore, there are $n(n-1)(n-2)^2(n-3)=n^5-8n^4+23n^3-28n^2+12n$ ways to choose four $B_{n-2}$'s as required.

{\it Case 2.1.4.} $b_1=c_1$ and $b_1=d_2$.

There are $n$ ways to choose $a_1(=d_1)$ from $N_n$, $n-1$ ways to choose $a_2(=d_2=b_1=c_1)$ from $N_n\setminus\{a_1\}$, $n-1$ ways to choose $b_2$ from $N_n\setminus\{d_2\}$, $n-2$ ways to choose $c_2$ from $N_n\setminus\{b_2, d_2\}$.
Therefore, there are $n(n-1)^2(n-2)=n^4-4n^3+5n^2-2n$ ways to choose four $B_{n-2}$'s as required.

{\it Case 2.2.} $a_1=d_2$ and $a_2=d_1$.

{\it Case 2.2.1.} $b_1\neq c_1$.

There are $n$ ways to choose $a_1(=d_2)$ from $N_n$, $n-1$ ways to choose $a_2(=d_1)$ from $N_n\setminus\{a_1\}$, $n-1$ ways to choose $b_1$ from $N_n\setminus\{a_1\}$, $n-2$ ways to choose $c_1$ from $N_n\setminus\{a_1, b_1\}$, $n-2$ ways to choose $b_2$ from $N_n\setminus\{b_1, d_2\}$, $n-2$ ways to choose $c_2$ from $N_n\setminus\{c_1, d_2\}$.
Therefore, there are $n(n-1)^2(n-2)^3=n^6-8n^5+25n^4-38n^3+28n^2-8n$ ways to choose four $B_{n-2}$'s as required. 

{\it Case 2.2.2.} $b_1=c_1$.

There are $n$ ways to choose $a_1(=d_2)$ from $N_n$, $n-1$ ways to choose $a_2(=d_1)$ from $N_n\setminus\{a_1\}$, $n-1$ ways to choose $b_1(=c_1)$ from $N_n\setminus\{a_1\}$, $n-2$ ways to choose $b_2$ from $N_n\setminus\{b_1, d_2\}$, $n-3$ ways to choose $c_2$ from $N_n\setminus\{c_1, d_2, b_2\}$.
Therefore, there are $n(n-1)^2(n-2)(n-3)=n^5-7n^4+17n^3-17n^2+6n$ ways to choose four $B_{n-2}$'s as required. 

{\it Case 2.3.} $a_1=d_1$, $a_2\notin\{d_1, d_2\}$ or $a_2=d_1$, $a_1\notin\{d_1, d_2\}$ or $a_2=d_2$, $a_1\notin\{d_1, d_2\}$.

W.l.o.g, assume that $a_1=d_1$ and $a_2\notin\{d_1, d_2\}$.

{\it Case 2.3.1.} $b_1\neq c_1$, $b_1\neq d_2$ and $c_1\neq d_2$.

There are $n$ ways to choose $a_1(=d_1)$ from $N_n$, $n-1$ ways to choose $d_2$ from $N_n\setminus\{d_1\}$, $n-2$ ways to choose $a_2$ from $N_n\setminus\{d_1, d_2\}$, $n-2$ ways to choose $b_1$ from $N_n\setminus\{a_1, d_2\}$, $n-3$ ways to choose $c_1$ from $N_n\setminus\{a_1, b_1, d_2\}$, $n-2$ ways to choose $b_2$ from $N_n\setminus\{b_1, d_2\}$, $n-2$ ways to choose $c_2$ from $N_n\setminus\{c_1, d_2\}$.
Therefore, there are $3n(n-1)(n-2)^4(n-3)=3(n^7-12n^6+59n^5-152n^4+216n^3-160n^2+48n)$ ways to choose four $B_{n-2}$'s as required. 

{\it Case 2.3.2.} $b_1\neq c_1$, $b_1\neq d_2=c_1$ or $b_1\neq c_1$, $c_1\neq d_2=b_1$.

W.l.o.g, assume that $b_1\neq d_2=c_1$. There are $n$ ways to choose $a_1(=d_1)$ from $N_n$, $n-1$ ways to choose $d_2(=c_1)$ from $N_n\setminus\{d_1\}$, $n-2$ ways to choose $a_2$ from $N_n\setminus\{d_1, d_2\}$, $n-2$ ways to choose $b_1$ from $N_n\setminus\{a_1, d_2\}$, $n-2$ ways to choose $b_2$ from $N_n\setminus\{b_1, d_2\}$, $n-1$ ways to choose $c_2$ from $N_n\setminus\{d_2\}$.
Therefore, there are $3\times 2n(n-1)^2(n-2)^3=6(n^6-8n^5+25n^4-38n^3+28n^2-8n)$ ways to choose four $B_{n-2}$'s as required. 

{\it Case 2.3.3.} $b_1=c_1$ and $b_1\neq d_2$.

There are $n$ ways to choose $a_1(=d_1)$ from $N_n$, $n-1$ ways to choose $d_2$ from $N_n\setminus\{d_1\}$, $n-2$ ways to choose $a_2$ from $N_n\setminus\{d_1, d_2\}$, $n-2$ ways to choose $b_1(=c_1)$ from $N_n\setminus\{a_1, d_2\}$, $n-2$ ways to choose $b_2$ from $N_n\setminus\{b_1, d_2\}$, $n-3$ ways to choose $c_2$ from $N_n\setminus\{c_1, b_2, d_2\}$.
Therefore, there are $3n(n-1)(n-2)^3(n-3)=3(n^6-10n^5+39n^4-74n^3+68n^2-24n)$ ways to choose four $B_{n-2}$'s as required.

{\it Case 2.3.4.} $b_1=c_1$ and $b_1=d_2$.

There are $n$ ways to choose $a_1(=d_1)$ from $N_n$, $n-1$ ways to choose $d_2(=b_1=c_1)$ from $N_n\setminus\{d_1\}$, $n-2$ ways to choose $a_2$ from $N_n\setminus\{d_1, d_2\}$, $n-1$ ways to choose $b_2$ from $N_n\setminus\{d_2\}$, $n-2$ ways to choose $c_2$ from $N_n\setminus\{b_2, d_2\}$.
Therefore, there are $3n(n-1)^2(n-2)^2=3(n^5-6n^4+13n^3-12n^2+4n)$ ways to choose four $B_{n-2}$'s as required.

{\it Case 2.4.} $a_1=d_2$, $a_2\notin\{d_1, d_2\}$.

{\it Case 2.4.1.} $b_1\neq c_1$.

There are $n$ ways to choose $a_1(=d_2)$ from $N_n$, $n-1$ ways to choose $d_1$ from $N_n\setminus\{d_2\}$, $n-2$ ways to choose $a_2$ from $N_n\setminus\{d_1, d_2\}$, $n-1$ ways to choose $b_1$ from $N_n\setminus\{a_1\}$, $n-2$ ways to choose $c_1$ from $N_n\setminus\{a_1, b_1\}$, $n-2$ ways to choose $b_2$ from $N_n\setminus\{b_1, d_2\}$, $n-2$ ways to choose $c_2$ from $N_n\setminus\{c_1, d_2\}$.
Therefore, there are $n(n-1)^2(n-2)^4=n^7-10n^6+41n^5-88n^4+104n^3-64n^2+16n$ ways to choose four $B_{n-2}$'s as required. 

{\it Case 2.4.2.} $b_1=c_1$.

There are $n$ ways to choose $a_1(=d_2)$ from $N_n$, $n-1$ ways to choose $d_1$ from $N_n\setminus\{d_2\}$, $n-2$ ways to choose $a_2$ from $N_n\setminus\{d_1, d_2\}$, $n-1$ ways to choose $b_1(=c_1)$ from $N_n\setminus\{a_1\}$, $n-2$ ways to choose $b_2$ from $N_n\setminus\{b_1, d_2\}$, $n-3$ ways to choose $c_2$ from $N_n\setminus\{c_1, d_2, b_2\}$.
Therefore, there are $n(n-1)^2(n-2)^2(n-3)=n^6-9n^5+31n^4-51n^3+40n^2-12n$ ways to choose four $B_{n-2}$'s as required. 

In summary, there are $4n^7-34n^6+120n^5-226n^4+240n^3-136n^2+32n$ ways to choose four $B_{n-2}$'s as required in Case 2.
Observations \ref{ob:2}, \ref{ob:3} and \ref{ob:4} imply that $a_1 a_2 X^{n-2}$, $b_1 X^{n-2} b_2$, $c_1 X^{n-2} c_2$ and $X^{n-2} d_1 d_2$ are pairwise disjoint.
Thus, $|V(a_1 a_2 X^{n-2})\cup V(b_1 X^{n-2} b_2)\cup V(c_1 X^{n-2} c_2)\cup V(X^{n-2} d_1 d_2)|=|V(a_1 a_2 X^{n-2})|+|V(b_1 X^{n-2} b_2)|+|V(c_1 X^{n-2} c_2)|+|V(X^{n-2} d_1 d_2)|=4(n-2)!$. 
The probability $P_2$ that there are four fault-free $B_{n-2}$'s chosen as in Case 2 is $\frac{1}{2}(4n^7-34n^6+120n^5-226n^4+240n^3-136n^2+32n)p^{4(n-2)!}$.

{\it Case 3.} $\{a_1, b_1\}\cap\{d_1, d_2\}\neq\emptyset$, $b_1\neq c_1$, $b_1\neq a_1$, $c_1\neq a_1$, $i_2\neq j_2=d_2$ for any $\{i, j\}=\{b, c\}$.

There are 2 possible cases, w.l.o.g, assume that $i=b$ and $j=c$. 
The following facts will be applied directly in this case.

($\romannumeral1$) Observation \ref{ob:4} implies that $a_1 a_2 X^{n-2}$ and $X^{n-2} d_1 d_2$ are disjoint. 

($\romannumeral2$) If $c_1\neq d_1$, Observations \ref{ob:2} and \ref{ob:3} imply that for any $i\in\{b, c\}$, $a_1 a_2 X^{n-2}$ and $i_1 X^{n-2} i_2$ are disjoint, and $X^{n-2} d_1 d_2$ and $b_1 X^{n-2} b_2$ are disjoint, and $X^{n-2} d_1 d_2$ and $c_1 X^{n-2} c_2$ are not disjoint and $V(X^{n-2} d_1 d_2)\cap V(c_1 X^{n-2} c_2)=V(c_1 X^{n-3} d_1 d_2)$. Thus, $|V(a_1 a_2 X^{n-2})\cup V(b_1 X^{n-2} b_2)\cup V(c_1 X^{n-2} c_2)\cup V(X^{n-2} d_1 d_2)|=|V(a_1 a_2 X^{n-2})|+|V(b_1 X^{n-2} b_2)|+|V(c_1 X^{n-2} c_2)|+|V(X^{n-2} d_1 d_2)|-|V(X^{n-2} d_1 d_2)\cap V(c_1 X^{n-2} c_2)|=4(n-2)!-(n-3)!$ in this case. 

($\romannumeral3$) If $c_1=d_1$, Observations \ref{ob:2} and \ref{ob:3} imply that for any $i\in\{b, c\}$, $a_1 a_2 X^{n-2}$ and $i_1 X^{n-2} i_2$ are disjoint, and $X^{n-2} d_1 d_2$ and $i_1 X^{n-2} i_2$ are disjoint. Thus, $|V(a_1 a_2 X^{n-2})\cup V(b_1 X^{n-2} b_2)\cup V(c_1 X^{n-2} c_2)\cup V(X^{n-2} d_1 d_2)|=|V(a_1 a_2 X^{n-2})|+|V(b_1 X^{n-2} b_2)|+|V(c_1 X^{n-2} c_2)|+|V(X^{n-2} d_1 d_2)|=4(n-2)!$ in this case. 

{\it Case 3.1.} $a_1=d_1$ and $a_2=d_2$.

{\it Case 3.1.1.} $b_1\neq d_2$.

There are $n$ ways to choose $a_1(=d_1)$ from $N_n$, $n-1$ ways to choose $a_2(=d_2=c_2)$ from $N_n\setminus\{a_1\}$, $n-2$ ways to choose $b_1$ from $N_n\setminus\{a_1, d_2\}$, $n-3$ ways to choose $c_1$ from $N_n\setminus\{a_1, b_1, c_2\}$, $n-2$ ways to choose $b_2$ from $N_n\setminus\{b_1, d_2\}$.
Therefore, there are $n(n-1)(n-2)^2(n-3)=n^5-8n^4+23n^3-28n^2+12n$ ways to choose four $B_{n-2}$'s as required.

{\it Case 3.1.2.} $b_1=d_2$.

There are $n$ ways to choose $a_1(=d_1)$ from $N_n$, $n-1$ ways to choose $a_2(=d_2=c_2=b_1)$ from $N_n\setminus\{a_1\}$, $n-2$ ways to choose $c_1$ from $N_n\setminus\{a_1, b_1\}$, $n-1$ ways to choose $b_2$ from $N_n\setminus\{d_2\}$.
Therefore, there are $n(n-1)^2(n-2)=n^4-4n^3+5n^2-2n$ ways to choose four $B_{n-2}$'s as required.

In summary, there are $n^5-7n^4+19n^3-23n^2+10n$ ways to choose four $B_{n-2}$'s as required in Case 3.1.
Clearly, in this case $c_1\neq d_1$. By ($\romannumeral1$) and ($\romannumeral2$), the probability $P_{3.1}$ that there are four fault-free $B_{n-2}$'s chosen as in Case 3.1 is $\frac{1}{2}(n^5-7n^4+19n^3-23n^2+10n)p^{4(n-2)!-(n-3)!}$.

{\it Case 3.2.} $a_1=d_2$ and $a_2=d_1$.

{\it Case 3.2.1.} $c_1\neq d_1$.

There are $n$ ways to choose $a_1(=d_2=c_2)$ from $N_n$, $n-1$ ways to choose $a_2(=d_1)$ from $N_n\setminus\{a_1\}$, $n-2$ ways to choose $c_1$ from $N_n\setminus\{a_1, d_1\}$, $n-2$ ways to choose $b_1$ from $N_n\setminus\{a_1, c_1\}$, $n-2$ ways to choose $b_2$ from $N_n\setminus\{b_1, d_2\}$.
Therefore, there are $n(n-1)(n-2)^3=n^5-7n^4+18n^3-20n^2+8n$ ways to choose four $B_{n-2}$'s as required.
By ($\romannumeral1$) and ($\romannumeral2$), the probability $P_{3.2.1}$ that there are four fault-free $B_{n-2}$'s chosen as in Case 3.2.1 is $\frac{1}{2}(n^5-7n^4+18n^3-20n^2+8n)p^{4(n-2)!-(n-3)!}$.

{\it Case 3.2.2.} $c_1=d_1$.

There are $n$ ways to choose $a_1(=d_2=c_2)$ from $N_n$, $n-1$ ways to choose $a_2(=d_1=c_1)$ from $N_n\setminus\{a_1\}$, $n-2$ ways to choose $b_1$ from $N_n\setminus\{a_1, c_1\}$, $n-2$ ways to choose $b_2$ from $N_n\setminus\{b_1, d_2\}$.
Therefore, there are $n(n-1)(n-2)^2=n^4-5n^3+8n^2-4n$ ways to choose four $B_{n-2}$'s as required.
By ($\romannumeral1$) and ($\romannumeral3$), the probability $P_{3.2.2}$ that there are four fault-free $B_{n-2}$'s chosen as in Case 3.2.2 is $\frac{1}{2}(n^4-5n^3+8n^2-4n)p^{4(n-2)!}$.

Thus, the probability $P_{3.2}$ that there are four fault-free $B_{n-2}$'s chosen as in Case 3.2 is $P_{3.2.1}+P_{3.2.2}=\frac{1}{2}(n^4-5n^3+8n^2-4n)p^{4(n-2)!}+\frac{1}{2}(n^5-7n^4+18n^3-20n^2+8n)p^{4(n-2)!-(n-3)!}$.

{\it Case 3.3.} $a_1=d_1$ and $a_2\notin\{d_1, d_2\}$.

{\it Case 3.3.1.} $b_1\neq d_2$.

There are $n$ ways to choose $a_1(=d_1)$ from $N_n$, $n-1$ ways to choose $d_2(=c_2)$ from $N_n\setminus\{d_1\}$, $n-2$ ways to choose $a_2$ from $N_n\setminus\{d_1, d_2\}$, $n-2$ ways to choose $b_1$ from $N_n\setminus\{a_1, d_2\}$, $n-3$ ways to choose $c_1$ from $N_n\setminus\{a_1, b_1, c_2\}$, $n-2$ ways to choose $b_2$ from $N_n\setminus\{b_1, d_2\}$.
Therefore, there are $n(n-1)(n-2)^3(n-3)=(n^6-10n^5+39n^4-74n^3+68n^2-24n)$ ways to choose four $B_{n-2}$'s as required. 

{\it Case 3.3.2.}  $b_1=d_2$.

There are $n$ ways to choose $a_1(=d_1)$ from $N_n$, $n-1$ ways to choose $d_2(=c_2=b_1)$ from $N_n\setminus\{d_1\}$, $n-2$ ways to choose $a_2$ from $N_n\setminus\{d_1, d_2\}$, $n-2$ ways to choose $c_1$ from $N_n\setminus\{a_1, b_1\}$, $n-1$ ways to choose $b_2$ from $N_n\setminus\{d_2\}$.
Therefore, there are $n(n-1)^2(n-2)^2=(n^5-6n^4+13n^3-12n^2+4n)$ ways to choose four $B_{n-2}$'s as required. 

In summary, there are $n^6-9n^5+33n^4-61n^3+56n^2-20n$ ways to choose four $B_{n-2}$'s as required in Case 3.1.
Clearly, in this case $c_1\neq d_1$. By ($\romannumeral1$) and ($\romannumeral2$), the probability $P_{3.3}$ that there are four fault-free $B_{n-2}$'s chosen as in Case 3.3 is $\frac{1}{2}(n^6-9n^5+33n^4-61n^3+56n^2-20n)p^{4(n-2)!-(n-3)!}$.

{\it Case 3.4.} $a_1=d_2$ and $a_2\notin\{d_1, d_2\}$.

{\it Case 3.4.1.} $c_1\neq d_1$.

There are $n$ ways to choose $a_1(=d_2=c_2)$ from $N_n$, $n-1$ ways to choose $d_1$ from $N_n\setminus\{d_2\}$, $n-2$ ways to choose $a_2$ from $N_n\setminus\{d_1, d_2\}$, $n-2$ ways to choose $c_1$ from $N_n\setminus\{a_1, d_1\}$, $n-2$ ways to choose $b_1$ from $N_n\setminus\{a_1, c_1\}$, $n-2$ ways to choose $b_2$ from $N_n\setminus\{b_1, d_2\}$.
Therefore, there are $n(n-1)(n-2)^4=n^6-9n^5+32n^4-56n^3+48n^2-16n$ ways to choose four $B_{n-2}$'s as required. 
By ($\romannumeral1$) and ($\romannumeral2$), the probability $P_{3.4.1}$ that there are four fault-free $B_{n-2}$'s chosen as in Case 3.4.1 is $\frac{1}{2}(n^6-9n^5+32n^4-56n^3+48n^2-16n)p^{4(n-2)!-(n-3)!}$.

{\it Case 3.4.2.} $c_1=d_1$.

There are $n$ ways to choose $a_1(=d_2=c_2)$ from $N_n$, $n-1$ ways to choose $d_1(=c_1)$ from $N_n\setminus\{d_2\}$, $n-2$ ways to choose $a_2$ from $N_n\setminus\{d_1, d_2\}$, $n-2$ ways to choose $b_1$ from $N_n\setminus\{a_1, c_1\}$, $n-2$ ways to choose $b_2$ from $N_n\setminus\{b_1, d_2\}$.
Therefore, there are $n(n-1)(n-2)^3=n^5-7n^4+18n^3-20n^2+8n$ ways to choose four $B_{n-2}$'s as required. 
By ($\romannumeral1$) and ($\romannumeral3$), the probability $P_{3.4.2}$ that there are four fault-free $B_{n-2}$'s chosen as in Case 3.4.2 is $\frac{1}{2}(n^5-7n^4+18n^3-20n^2+8n)p^{4(n-2)!}$.

Thus, the probability $P_{3.4}$ that there are four fault-free $B_{n-2}$'s chosen as in Case 3.4 is $P_{3.4.1}+P_{3.4.2}=\frac{1}{2}(n^5-7n^4+18n^3-20n^2+8n)p^{4(n-2)!}+\frac{1}{2}(n^6-9n^5+32n^4-56n^3+48n^2-16n)p^{4(n-2)!-(n-3)!}$.

{\it Case 3.5.} $a_2=d_1$ and $a_1\notin\{d_1, d_2\}$.

{\it Case 3.5.1.} $c_1\neq d_1$.

If $b_1\neq d_2$, there are $n$ ways to choose $a_2(=d_1)$ from $N_n$, $n-1$ ways to choose $d_2(=c_2)$ from $N_n\setminus\{d_1\}$, $n-2$ ways to choose $a_1$ from $N_n\setminus\{d_1, d_2\}$, $n-3$ ways to choose $c_1$ from $N_n\setminus\{a_1, c_2, d_1\}$, $n-3$ ways to choose $b_1$ from $N_n\setminus\{a_1, c_1, d_2\}$, $n-2$ ways to choose $b_2$ from $N_n\setminus\{b_1, d_2\}$.
Therefore, there are $n(n-1)(n-2)^2(n-3)^2=n^6-11n^5+47n^4-97n^3+96n^2-36n$ ways to choose four $B_{n-2}$'s as required. 

If $b_1=d_2$, there are $n$ ways to choose $a_2(=d_1)$ from $N_n$, $n-1$ ways to choose $d_2(=c_2=b_1)$ from $N_n\setminus\{d_1\}$, $n-2$ ways to choose $a_1$ from $N_n\setminus\{d_1, d_2\}$, $n-3$ ways to choose $c_1$ from $N_n\setminus\{a_1, c_2, d_1\}$, $n-1$ ways to choose $b_2$ from $N_n\setminus\{d_2\}$.
Therefore, there are $n(n-1)^2(n-2)(n-3)=n^5-7n^4+17n^3-17n^2+6n$ ways to choose four $B_{n-2}$'s as required. 

In summary, there are $n^6-10n^5+40n^4-80n^3+79n^2-30n$ ways to choose four $B_{n-2}$'s as required in Case 3.5.1.
By ($\romannumeral1$) and ($\romannumeral2$), the probability $P_{3.5.1}$ that there are four fault-free $B_{n-2}$'s chosen as in Case 3.5.1 is $\frac{1}{2}(n^6-10n^5+40n^4-80n^3+79n^2-30n)p^{4(n-2)!-(n-3)!}$.

{\it Case 3.5.2.} $c_1=d_1$.

If $b_1\neq d_2$, there are $n$ ways to choose $a_2(=d_1=c_1)$ from $N_n$, $n-1$ ways to choose $d_2(=c_2)$ from $N_n\setminus\{d_1\}$, $n-2$ ways to choose $a_1$ from $N_n\setminus\{d_1, d_2\}$, $n-3$ ways to choose $b_1$ from $N_n\setminus\{a_1, c_1, d_2\}$, $n-2$ ways to choose $b_2$ from $N_n\setminus\{b_1, d_2\}$.
Therefore, there are $n(n-1)(n-2)^2(n-3)=n^5-8n^4+23n^3-28n^2+12n$ ways to choose four $B_{n-2}$'s as required. 

If $b_1=d_2$, there are $n$ ways to choose $a_2(=d_1=c_1)$ from $N_n$, $n-1$ ways to choose $d_2(=c_2=b_1)$ from $N_n\setminus\{d_1\}$, $n-2$ ways to choose $a_1$ from $N_n\setminus\{d_1, d_2\}$, $n-1$ ways to choose $b_2$ from $N_n\setminus\{d_2\}$.
Therefore, there are $n(n-1)^2(n-2)=n^4-4n^3+5n^2-2n$ ways to choose four $B_{n-2}$'s as required. 

In summary, there are $n^5-7n^4+19n^3-23n^2+10n$ ways to choose four $B_{n-2}$'s as required in Case 3.5.2.
By ($\romannumeral1$) and ($\romannumeral3$), the probability $P_{3.5.2}$ that there are four fault-free $B_{n-2}$'s chosen as in Case 3.5.2 is $\frac{1}{2}(n^5-7n^4+19n^3-23n^2+10n)p^{4(n-2)!}$.

Thus, the probability $P_{3.5}$ that there are four fault-free $B_{n-2}$'s chosen as in Case 3.5 is $P_{3.5.1}+P_{3.5.2}=\frac{1}{2}(n^5-7n^4+19n^3-23n^2+10n)p^{4(n-2)!}+\frac{1}{2}(n^6-10n^5+40n^4-80n^3+79n^2-30n)p^{4(n-2)!-(n-3)!}$.

{\it Case 3.6.} $a_2=d_2$ and $a_1\notin\{d_1, d_2\}$.

This scenario is similar to Case 3.3, the probability $P_{3.6}$ that there are four fault-free $B_{n-2}$'s chosen as in Case 3.5 is $\frac{1}{2}(n^6-9n^5+33n^4-61n^3+56n^2-20n)p^{4(n-2)!-(n-3)!}$.

In summary, the probability $P_3$ that there are four fault-free $B_{n-2}$'s chosen as in Case 3 is $2\sum_{i=1}^{6}P_{3.i}=(2n^5-13n^4+32n^3-35n^2+14n)p^{4(n-2)!}+(4n^6-35n^5+124n^4-221n^3+196n^2-68n)p^{4(n-2)!-(n-3)!}$.

{\it Case 4.} $b_2\neq c_2$, $b_2\neq d_2$, $c_2\neq d_2$, $i_1\neq j_1=a_1$ for any $\{i, j\}=\{b, c\}$.

This scenario is similar to Case 3, except that $a_1=d_1$ and $a_2\notin\{d_1, d_2\}$ in this case is similar to the case of $a_2=d_1$ and $a_1\notin\{d_1, d_2\}$ in Case 3.5.

In summary, the probability $P_4$ that there are four fault-free $B_{n-2}$'s chosen as in Case 4 is $2(\sum_{i=1}^{6}P_{3.i}-P_{3.3}+P_{3.5})=(3n^5-20n^4+51n^3-58n^2+24n)p^{4(n-2)!}+(4n^6-36n^5+131n^4-240n^3+219n^2-78n)p^{4(n-2)!-(n-3)!}$.

{\it Case 5.} $\{a_1, a_2\}\cap\{d_1, d_2\}\neq\emptyset$, $b_1\neq c_1$, $b_1\neq a_1$, $c_1\neq a_1$, $b_2=c_2=d_2$ or $\{a_1, a_2\}\cap\{d_1, d_2\}\neq\emptyset$, $b_2\neq c_2$, $b_2\neq d_2$, $c_2\neq d_2$, $b_1=c_1=a_1$.

W.l.o.g, assume that the former applies. And the following facts will be applied directly in this case.

($\romannumeral1$) Observation \ref{ob:4} implies that $a_1 a_2 X^{n-2}$ and $X^{n-2} d_1 d_2$ are disjoint. 

($\romannumeral2$) If $b_1\neq d_1$ and $c_1\neq d_1$, Observations \ref{ob:2} and \ref{ob:3} imply that for any $i\in\{b, c\}$, $a_1 a_2 X^{n-2}$ and $i_1 X^{n-2} i_2$ are disjoint, and $X^{n-2} d_1 d_2$ and $i_1 X^{n-2} i_2$ are not disjoint and $V(X^{n-2} d_1 d_2)\cap V(i_1 X^{n-2} i_2)=V(i_1 X^{n-3} d_1 d_2)$. Thus, $|V(a_1 a_2 X^{n-2})\cup V(b_1 X^{n-2} b_2)\cup V(c_1 X^{n-2} c_2)\cup V(X^{n-2} d_1 d_2)|=|V(a_1 a_2 X^{n-2})|+|V(b_1 X^{n-2} b_2)|+|V(c_1 X^{n-2} c_2)|+|V(X^{n-2} d_1 d_2)|-|V(X^{n-2} d_1 d_2)\cap V(b_1 X^{n-2} b_2)|-|V(X^{n-2} d_1 d_2)\cap V(c_1 X^{n-2} c_2)|=4(n-2)!-2(n-3)!$ in this case. 

($\romannumeral3$) If $b_1=d_1\neq c_1$ or $c_1=d_1\neq b_1$, w.l.o.g, assume that the former applies. Observations \ref{ob:2} and \ref{ob:3} imply that for any $i\in\{b, c\}$, $a_1 a_2 X^{n-2}$ and $i_1 X^{n-2} i_2$ are disjoint, and $X^{n-2} d_1 d_2$ and $b_1 X^{n-2} b_2$ are disjoint, and $X^{n-2} d_1 d_2$ and $c_1 X^{n-2} c_2$ are not disjoint and $V(X^{n-2} d_1 d_2)\cap V(c_1 X^{n-2} c_2)=V(c_1 X^{n-3} d_1 d_2)$. Thus, $|V(a_1 a_2 X^{n-2})\cup V(b_1 X^{n-2} b_2)\cup V(c_1 X^{n-2} c_2)\cup V(X^{n-2} d_1 d_2)|=|V(a_1 a_2 X^{n-2})|+|V(b_1 X^{n-2} b_2)|+|V(c_1 X^{n-2} c_2)|+|V(X^{n-2} d_1 d_2)|-|V(X^{n-2} d_1 d_2)\cap V(c_1 X^{n-2} c_2)|=4(n-2)!-(n-3)!$ in this case. 

{\it Case 5.1.} $a_1=d_1$ and $a_2=d_2$.

There are $n$ ways to choose $a_1(=d_1)$ from $N_n$, $n-1$ ways to choose $a_2(=d_2=b_2=c_2)$ from $N_n\setminus\{a_1\}$, $n-2$ ways to choose $b_1$ from $N_n\setminus\{a_1, b_2\}$, $n-3$ ways to choose $c_1$ from $N_n\setminus\{a_1, b_1, c_2\}$.
Therefore, there are $n(n-1)(n-2)(n-3)=n^4-6n^3+11n^2-6n$ ways to choose four $B_{n-2}$'s as required. 
Clearly, in this case $b_1\neq d_1$ and $c_1\neq d_1$. By ($\romannumeral1$) and ($\romannumeral2$), the probability $P_{5.1}$ that there are four fault-free $B_{n-2}$'s chosen as in Case 5.1 is $\frac{1}{2}(n^4-6n^3+11n^2-6n)p^{4(n-2)!-2(n-3)!}$.

{\it Case 5.2.} $a_1=d_2$ and $a_2=d_1$.

{\it Case 5.2.1.} $b_1\neq d_1$ and $c_1\neq d_1$.

There are $n$ ways to choose $a_1(=d_2=b_2=c_2)$ from $N_n$, $n-1$ ways to choose $a_2(=d_1)$ from $N_n\setminus\{a_1\}$, $n-2$ ways to choose $b_1$ from $N_n\setminus\{a_1, d_1\}$, $n-3$ ways to choose $c_1$ from $N_n\setminus\{a_1, b_1, d_1\}$.
Therefore, there are $n(n-1)(n-2)(n-3)=n^4-6n^3+11n^2-6n$ ways to choose four $B_{n-2}$'s as required. 
By ($\romannumeral1$) and ($\romannumeral2$), the probability $P_{5.2.1}$ that there are four fault-free $B_{n-2}$'s chosen as in Case 5.2.1 is $\frac{1}{2}(n^4-6n^3+11n^2-6n)p^{4(n-2)!-2(n-3)!}$.

{\it Case 5.2.2.} $b_1=d_1\neq c_1$ or $c_1=d_1\neq b_1$.

W.l.o.g, assume that $b_1=d_1\neq c_1$.

There are $n$ ways to choose $a_1(=d_2=b_2=c_2)$ from $N_n$, $n-1$ ways to choose $a_2(=d_1=b_1)$ from $N_n\setminus\{a_1\}$, $n-2$ ways to choose $c_1$ from $N_n\setminus\{a_1, b_1\}$.
Therefore, there are $2n(n-1)(n-2)=2(n^3-3n^2+2n)$ ways to choose four $B_{n-2}$'s as required. 
By ($\romannumeral1$) and ($\romannumeral3$), the probability $P_{5.2.2}$ that there are four fault-free $B_{n-2}$'s chosen as in Case 5.2.2 is $\frac{1}{2}\times 2(n^3-3n^2+2n)p^{4(n-2)!-(n-3)!}$.

Thus, the probability $P_{5.2}$ that there are four fault-free $B_{n-2}$'s chosen as in Case 5.2 is $P_{5.2.1}+P_{5.2.2}=(n^3-3n^2+2n)p^{4(n-2)!-(n-3)!}+\frac{1}{2}(n^4-6n^3+11n^2-6n)p^{4(n-2)!-2(n-3)!}$.

{\it Case 5.3.} $a_1=d_1$ and $a_2\notin\{d_1, d_2\}$.

There are $n$ ways to choose $a_1(=d_1)$ from $N_n$, $n-1$ ways to choose $b_2(=c_2=d_2)$ from $N_n\setminus\{d_1\}$, $n-2$ ways to choose $a_2$ from $N_n\setminus\{d_1, d_2\}$, $n-2$ ways to choose $b_1$ from $N_n\setminus\{a_1, b_2\}$, $n-3$ ways to choose $c_1$ from $N_n\setminus\{a_1, b_1, c_2\}$.
Therefore, there are $n(n-1)(n-2)^2(n-3)=n^5-8n^4+23n^3-28n^2+12n$ ways to choose four $B_{n-2}$'s as required. 
Clearly, in this case $b_1\neq d_1$ and $c_1\neq d_1$. By ($\romannumeral1$) and ($\romannumeral2$), the probability $P_{5.3}$ that there are four fault-free $B_{n-2}$'s chosen as in Case 5.3 is $\frac{1}{2}(n^5-8n^4+23n^3-28n^2+12n)p^{4(n-2)!-2(n-3)!}$.

{\it Case 5.4.} $a_1=d_2$ and $a_2\notin\{d_1, d_2\}$.

{\it Case 5.4.1.} $b_1\neq d_1$ and $c_1\neq d_1$.

There are $n$ ways to choose $a_1(=d_2=b_2=c_2)$ from $N_n$, $n-1$ ways to choose $d_1$ from $N_n\setminus\{d_2\}$, $n-2$ ways to choose $a_2$ from $N_n\setminus\{d_1, d_2\}$, $n-2$ ways to choose $b_1$ from $N_n\setminus\{a_1, d_1\}$, $n-3$ ways to choose $c_1$ from $N_n\setminus\{a_1, b_1, d_1\}$.
Therefore, there are $n(n-1)(n-2)^2(n-3)=n^5-8n^4+23n^3-28n^2+12n$ ways to choose four $B_{n-2}$'s as required. 
By ($\romannumeral1$) and ($\romannumeral2$), the probability $P_{5.4.1}$ that there are four fault-free $B_{n-2}$'s chosen as in Case 5.4.1 is $\frac{1}{2}(n^5-8n^4+23n^3-28n^2+12n)p^{4(n-2)!-2(n-3)!}$.

{\it Case 5.4.2.} $b_1=d_1\neq c_1$ or $c_1=d_1\neq b_1$.

W.l.o.g, assume that $b_1=d_1\neq c_1$.

There are $n$ ways to choose $a_1(=d_2=b_2=c_2)$ from $N_n$, $n-1$ ways to choose $d_1(=b_1)$ from $N_n\setminus\{d_2\}$, $n-2$ ways to choose $a_2$ from $N_n\setminus\{d_1, d_2\}$, $n-2$ ways to choose $c_1$ from $N_n\setminus\{a_1, b_1\}$.
Therefore, there are $2n(n-1)(n-2)^2=2(n^4-5n^3+8n^2-4n)$ ways to choose four $B_{n-2}$'s as required. 
By ($\romannumeral1$) and ($\romannumeral3$), the probability $P_{5.4.2}$ that there are four fault-free $B_{n-2}$'s chosen as in Case 5.4.2 is $\frac{1}{2}\times 2(n^4-5n^3+8n^2-4n)p^{4(n-2)!-(n-3)!}$.

Thus, the probability $P_{5.4}$ that there are four fault-free $B_{n-2}$'s chosen as in Case 5.4 is $P_{5.4.2}+P_{5.4.2}=(n^4-5n^3+8n^2-4n)p^{4(n-2)!-(n-3)!}+\frac{1}{2}(n^5-8n^4+23n^3-28n^2+12n)p^{4(n-2)!-2(n-3)!}$.

{\it Case 5.5.} $a_2=d_1$ and $a_1\notin\{d_1, d_2\}$ or $a_2=d_2$ and $a_1\notin\{d_1, d_2\}$.

W.l.o.g, assume that the former applies.

{\it Case 5.5.1.} $b_1\neq d_1$ and $c_1\neq d_1$.

There are $n$ ways to choose $a_2(=d_1)$ from $N_n$, $n-1$ ways to choose $d_2(=b_2=c_2)$ from $N_n\setminus\{d_1\}$, $n-2$ ways to choose $a_1$ from $N_n\setminus\{d_1, d_2\}$, $n-3$ ways to choose $b_1$ from $N_n\setminus\{a_1, d_1, b_2\}$, $n-4$ ways to choose $c_1$ from $N_n\setminus\{a_1, b_1, d_1, c_2\}$.
Therefore, there are $n(n-1)(n-2)(n-3)(n-4)=n^5-10n^4+35n^3-50n^2+24n$ ways to choose four $B_{n-2}$'s as required. 
By ($\romannumeral1$) and ($\romannumeral2$), the probability $P_{5.5.1}$ that there are four fault-free $B_{n-2}$'s chosen as in Case 5.5.1 is $\frac{1}{2}(n^5-10n^4+35n^3-50n^2+24n)p^{4(n-2)!-2(n-3)!}$.

{\it Case 5.5.2.} $b_1=d_1\neq c_1$ or $c_1=d_1\neq b_1$.

W.l.o.g, assume that $b_1=d_1\neq c_1$.

There are $n$ ways to choose $a_2(=d_1=b_1)$ from $N_n$, $n-1$ ways to choose $d_2(=b_2=c_2)$ from $N_n\setminus\{d_1\}$, $n-2$ ways to choose $a_1$ from $N_n\setminus\{d_1, d_2\}$, $n-3$ ways to choose $c_1$ from $N_n\setminus\{a_1, b_1, c_2\}$.
Therefore, there are $2n(n-1)(n-2)(n-3)=2(n^4-6n^3+11n^2-6n)$ ways to choose four $B_{n-2}$'s as required. 
By ($\romannumeral1$) and ($\romannumeral3$), the probability $P_{5.5.2}$ that there are four fault-free $B_{n-2}$'s chosen as in Case 5.5.2 is $\frac{1}{2}\times 2(n^4-6n^3+11n^2-6n)p^{4(n-2)!-(n-3)!}$.

Thus, the probability $P_{5.5}$ that there are four fault-free $B_{n-2}$'s chosen as in Case 5.5 is $2(P_{5.5.1}+P_{5.5.2})=2(n^4-6n^3+11n^2-6n)p^{4(n-2)!-(n-3)!}+(n^5-10n^4+35n^3-50n^2+24n)p^{4(n-2)!-2(n-3)!}$.

In summary, the probability $P_5$ that there are four fault-free $B_{n-2}$'s chosen as in Case 5 is $2\sum_{i=1}^{5}P_{5.i}=2(3n^4-16n^3+27n^2-14n)p^{4(n-2)!-(n-3)!}+2(2n^5-17n^4+52n^3-67n^2+30n)p^{4(n-2)!-2(n-3)!}$.

{\it Case 6.} $\{a_1, a_2\}\cap\{d_1, d_2\}\neq\emptyset$, $b_1\neq c_1=a_1$, $b_2\neq c_2=d_2$ or $\{a_1, a_2\}\cap\{d_1, d_2\}\neq\emptyset$, $c_1\neq b_1=a_1$, $c_2\neq b_2=d_2$.

W.l.o.g, assume that the former applies. And the following facts will be applied directly in this case.

($\romannumeral1$) Observation \ref{ob:4} implies that $a_1 a_2 X^{n-2}$ and $X^{n-2} d_1 d_2$ are disjoint. 

($\romannumeral2$) If $c_2\neq a_2$ and $c_1\neq d_1$, Observations \ref{ob:2} and \ref{ob:3} imply that $a_1 a_2 X^{n-2}$ and $b_1 X^{n-2} b_2$ are disjoint, and $a_1 a_2 X^{n-2}$ and $c_1 X^{n-2} c_2$ are not disjoint and $V(a_1 a_2 X^{n-2})\cap V(c_1 X^{n-2} c_2)=V(a_1 a_2 X^{n-3} c_2)$, and $X^{n-2} d_1 d_2$ and $b_1 X^{n-2} b_2$ are disjoint, and $X^{n-2} d_1 d_2$ and $c_1 X^{n-2} c_2$ are not disjoint and $V(X^{n-2} d_1 d_2)\cap V(c_1 X^{n-2} c_2)=V(c_1 X^{n-3} d_1 d_2)$. Thus, $|V(a_1 a_2 X^{n-2})\cup V(b_1 X^{n-2} b_2)\cup V(c_1 X^{n-2} c_2)\cup V(X^{n-2} d_1 d_2)|=|V(a_1 a_2 X^{n-2})|+|V(b_1 X^{n-2} b_2)|+|V(c_1 X^{n-2} c_2)|+|V(X^{n-2} d_1 d_2)|-|V(a_1 a_2 X^{n-2})\cap V(c_1 X^{n-2} c_2)|-|V(X^{n-2} d_1 d_2)\cap V(c_1 X^{n-2}$ $c_2)|=4(n-2)!-2(n-3)!$ in this case. 

($\romannumeral3$) If $c_2\neq a_2$, $c_1=d_1$ or $c_2=a_2$, $c_1\neq d_1$, w.l.o.g, assume that the former applies. Observations \ref{ob:2} and \ref{ob:3} imply that $a_1 a_2 X^{n-2}$ and $b_1 X^{n-2} b_2$ are disjoint, and $a_1 a_2 X^{n-2}$ and $c_1 X^{n-2} c_2$ are not disjoint and $V(a_1 a_2 X^{n-2})\cap V(c_1 X^{n-2} c_2)=V(a_1 a_2 X^{n-3} c_2)$, and for any $i\in\{b, c\}$, $X^{n-2} d_1 d_2$ and $i_1 X^{n-2} i_2$ are disjoint. Thus, $|V(a_1 a_2 X^{n-2})\cup V(b_1 X^{n-2} b_2)\cup V(c_1 X^{n-2} c_2)\cup V(X^{n-2} d_1 d_2)|=|V(a_1 a_2 X^{n-2})|+|V(b_1 X^{n-2} b_2)|+|V(c_1 X^{n-2} c_2)|+|V(X^{n-2} d_1 d_2)|-|V(a_1 a_2 X^{n-2})\cap V(c_1 X^{n-2} c_2)|=4(n-2)!-(n-3)!$ in this case. 

($\romannumeral4$) If $c_2=a_2$ and $c_1=d_1$. Observations \ref{ob:2} and \ref{ob:3} imply that for any $i\in\{b, c\}$, $a_1 a_2 X^{n-2}$ and $i_1 X^{n-2} i_2$ are disjoint, and $X^{n-2} d_1 d_2$ and $i_1 X^{n-2} i_2$ are disjoint. Thus, $|V(a_1 a_2 X^{n-2})\cup V(b_1 X^{n-2} b_2)\cup V(c_1 X^{n-2} c_2)\cup V(X^{n-2} d_1 d_2)|=|V(a_1 a_2 X^{n-2})|+|V(b_1 X^{n-2} b_2)|+|V(c_1 X^{n-2} c_2)|+|V(X^{n-2} d_1 d_2)|=4(n-2)!$ in this case. 

{\it Case 6.1.} $a_1=d_1$ and $a_2=d_2$.

{\it Case 6.1.1.} $b_1\neq d_2$.

There are $n$ ways to choose $c_1(=a_1=d_1)$ from $N_n$, $n-1$ ways to choose $a_2(=d_2=c_2)$ from $N_n\setminus\{a_1\}$, $n-2$ ways to choose $b_1$ from $N_n\setminus\{a_1, d_2\}$, $n-2$ ways to choose $b_2$ from $N_n\setminus\{b_1, d_2\}$.
Therefore, there are $n(n-1)(n-2)^2=n^4-5n^3+8n^2-4n$ distinct $B_{n-2}$'s in this case.

{\it Case 6.1.2.} $b_1=d_2$.

There are $n$ ways to choose $c_1(=a_1=d_1)$ from $N_n$, $n-1$ ways to choose $a_2(=d_2=c_2=b_1)$ from $N_n\setminus\{a_1\}$, $n-1$ ways to choose $b_2$ from $N_n\setminus\{d_2\}$.
Therefore, there are $n(n-1)^2=n^3-2n^2+n$ distinct $B_{n-2}$'s in this case.

In summary, there are $n^4-4n^3+6n^2-3n$ ways to choose four $B_{n-2}$'s as required in Case 6.1.
Clearly, in this case $c_2=a_2$ and $c_1=d_1$. By ($\romannumeral1$) and ($\romannumeral4$), the probability $P_{6.1}$ that there are four fault-free $B_{n-2}$'s chosen as in Case 6.1 is $\frac{1}{2}(n^4-4n^3+6n^2-3n)p^{4(n-2)!}$.

{\it Case 6.2.} $a_1=d_1$, $a_2\notin\{d_1, d_2\}$ or $a_2=d_2$, $a_1\notin\{d_1, d_2\}$.

W.l.o.g, assume that $a_1=d_1$ and $a_2\notin\{d_1, d_2\}$.

{\it Case 6.2.1.} $b_1\neq d_2$.

There are $n$ ways to choose $c_1(=a_1=d_1)$ from $N_n$, $n-1$ ways to choose $d_2(=c_2)$ from $N_n\setminus\{d_1\}$, $n-2$ ways to choose $a_2$ from $N_n\setminus\{d_1, d_2\}$, $n-2$ ways to choose $b_1$ from $N_n\setminus\{a_1, d_2\}$, $n-2$ ways to choose $b_2$ from $N_n\setminus\{b_1, d_2\}$.
Therefore, there are $n(n-1)(n-2)^3=n^5-7n^4+18n^3-20n^2+8n$ distinct $B_{n-2}$'s in this case.

{\it Case 6.2.2.} $b_1=d_2$.

There are $n$ ways to choose $c_1(=a_1=d_1)$ from $N_n$, $n-1$ ways to choose $d_2(=c_2=b_1)$ from $N_n\setminus\{d_1\}$, $n-2$ ways to choose $a_2$ from $N_n\setminus\{d_1, d_2\}$, $n-1$ ways to choose $b_2$ from $N_n\setminus\{d_2\}$.
Therefore, there are $n(n-1)^2(n-2)=n^4-4n^3+5n^2-2n$ distinct $B_{n-2}$'s in this case.

In summary, there are $2(n^5-6n^4+14n^3-15n^2+6n)$ ways to choose four $B_{n-2}$'s as required in Case 6.2.
Clearly, in this case $c_2\neq a_2$ and $c_1=d_1$. By ($\romannumeral1$) and ($\romannumeral3$), the probability $P_{6.2}$ that there are four fault-free $B_{n-2}$'s chosen as in Case 6.2 is $\frac{1}{2}\times 2(n^5-6n^4+14n^3-15n^2+6n)p^{4(n-2)!-(n-3)!}$.

{\it Case 6.3.} $a_2=d_1$ and $a_1\notin\{d_1, d_2\}$.

{\it Case 6.3.1.} $b_1\neq d_2$.

There are $n$ ways to choose $a_2(=d_1)$ from $N_n$, $n-1$ ways to choose $d_2(=c_2)$ from $N_n\setminus\{d_1\}$, $n-2$ ways to choose $a_1(=c_1)$ from $N_n\setminus\{d_1, d_2\}$, $n-2$ ways to choose $b_1$ from $N_n\setminus\{a_1, d_2\}$, $n-2$ ways to choose $b_2$ from $N_n\setminus\{b_1, d_2\}$.
Therefore, there are $n(n-1)(n-2)^3=n^5-7n^4+18n^3-20n^2+8n$ distinct $B_{n-2}$'s in this case.

{\it Case 6.3.2.} $b_1=d_2$.

There are $n$ ways to choose $a_2(=d_1)$ from $N_n$, $n-1$ ways to choose $d_2(=c_2=b_1)$ from $N_n\setminus\{d_1\}$, $n-2$ ways to choose $a_1(=c_1)$ from $N_n\setminus\{d_1, d_2\}$, $n-1$ ways to choose $b_2$ from $N_n\setminus\{d_2\}$.
Therefore, there are $n(n-1)^2(n-2)=n^4-4n^3+5n^2-2n$ distinct $B_{n-2}$'s in this case.

In summary, there are $n^5-6n^4+14n^3-15n^2+6n$ ways to choose four $B_{n-2}$'s as required in Case 6.3.
Clearly, in this case $c_2\neq a_2$ and $c_1\neq d_1$. By ($\romannumeral1$) and ($\romannumeral2$), the probability $P_{6.3}$ that there are four fault-free $B_{n-2}$'s chosen as in Case 6.3 is $\frac{1}{2}(n^5-6n^4+14n^3-15n^2+6n)p^{4(n-2)!-2(n-3)!}$.

In summary, the probability $P_6$ that there are four fault-free $B_{n-2}$'s chosen as in Case 6 is $2\sum_{i=1}^{3}P_{6.i}=(n^4-4n^3+6n^2-3n)p^{4(n-2)!}+2(n^5-6n^4+14n^3-15n^2+6n)p^{4(n-2)!-(n-3)!}+(n^5-6n^4+14n^3-15n^2+6n)p^{4(n-2)!-2(n-3)!}$.

{\it Case 7.} $\{a_1, a_2\}\cap\{d_1, d_2\}\neq\emptyset$, $b_1\neq c_1=a_1$, $c_2\neq b_2=d_2$ or $\{a_1, a_2\}\cap\{d_1, d_2\}\neq\emptyset$, $c_1\neq b_1=a_1$, $b_2\neq c_2=d_2$.

W.l.o.g, assume that the former applies. And the following facts will be applied directly in this case.

($\romannumeral1$) Observation \ref{ob:4} implies that $a_1 a_2 X^{n-2}$ and $X^{n-2} d_1 d_2$ are disjoint. 

($\romannumeral2$) If $c_2\neq a_2$ and $b_1\neq d_1$, Observations \ref{ob:2} and \ref{ob:3} imply that $a_1 a_2 X^{n-2}$ and $b_1 X^{n-2} b_2$ are disjoint, and $a_1 a_2 X^{n-2}$ and $c_1 X^{n-2} c_2$ are not disjoint and $V(a_1 a_2 X^{n-2})\cap V(c_1 X^{n-2} c_2)=V(a_1 a_2 X^{n-3} c_2)$, and $X^{n-2} d_1 d_2$ and $b_1 X^{n-2} b_2$ are not disjoint and $V(X^{n-2} d_1 d_2)\cap V(b_1 X^{n-2} b_2)=V(b_1 X^{n-3} d_1 d_2)$, and $X^{n-2} d_1 d_2$ and $c_1 X^{n-2} c_2$ are disjoint. Thus, $|V(a_1 a_2 X^{n-2})\cup V(b_1 X^{n-2} b_2)\cup V(c_1 X^{n-2} c_2)\cup V(X^{n-2} d_1 d_2)|=|V(a_1 a_2 X^{n-2})|+|V(b_1 X^{n-2} b_2)|+|V(c_1 X^{n-2} c_2)|+|V(X^{n-2} d_1 d_2)|-|V(a_1 a_2 X^{n-2})\cap V(c_1 X^{n-2} c_2)|-|V(X^{n-2} d_1 d_2)\cap V(b_1 X^{n-2}$ $b_2)|=4(n-2)!-2(n-3)!$ in this case. 

($\romannumeral3$) If $c_2\neq a_2$, $b_1=d_1$ or $c_2=a_2$, $b_1\neq d_1$, w.l.o.g, assume that the former applies. Observations \ref{ob:2} and \ref{ob:3} imply that $a_1 a_2 X^{n-2}$ and $b_1 X^{n-2} b_2$ are disjoint, and $a_1 a_2 X^{n-2}$ and $c_1 X^{n-2} c_2$ are not disjoint and $V(a_1 a_2 X^{n-2})\cap V(c_1 X^{n-2} c_2)=V(a_1 a_2 X^{n-3} c_2)$, and for any $i\in\{b, c\}$, $X^{n-2} d_1 d_2$ and $i_1 X^{n-2} i_2$ are disjoint. Thus, $|V(a_1 a_2 X^{n-2})\cup V(b_1 X^{n-2} b_2)\cup V(c_1 X^{n-2} c_2)\cup V(X^{n-2} d_1 d_2)|=|V(a_1 a_2 X^{n-2})|+|V(b_1 X^{n-2} b_2)|+|V(c_1 X^{n-2} c_2)|+|V(X^{n-2} d_1 d_2)|-|V(a_1 a_2 X^{n-2})\cap V(c_1 X^{n-2} c_2)|=4(n-2)!-(n-3)!$ in this case. 

($\romannumeral4$) If $c_2=a_2$ and $b_1=d_1$. Observations \ref{ob:2} and \ref{ob:3} imply that for any $i\in\{b, c\}$, $a_1 a_2 X^{n-2}$ and $i_1 X^{n-2} i_2$ are disjoint, and $X^{n-2} d_1 d_2$ and $i_1 X^{n-2} i_2$ are disjoint. Thus, $|V(a_1 a_2 X^{n-2})\cup V(b_1 X^{n-2} b_2)\cup V(c_1 X^{n-2} c_2)\cup V(X^{n-2} d_1 d_2)|=|V(a_1 a_2 X^{n-2})|+|V(b_1 X^{n-2} b_2)|+|V(c_1 X^{n-2} c_2)|+|V(X^{n-2} d_1 d_2)|=4(n-2)!$ in this case. 

{\it Case 7.1.} $a_1=d_1$ and $a_2=d_2$.

There are $n$ ways to choose $c_1(=a_1=d_1)$ from $N_n$, $n-1$ ways to choose $a_2(=d_2=b_2)$ from $N_n\setminus\{a_1\}$, $n-2$ ways to choose $b_1$ from $N_n\setminus\{a_1, b_2\}$, $n-2$ ways to choose $c_2$ from $N_n\setminus\{c_1, d_2\}$.
Therefore, there are $n(n-1)(n-2)^2=n^4-5n^3+8n^2-4n$ ways to choose four $B_{n-2}$'s as required.
Clearly, in this case $c_2\neq a_2$ and $b_1\neq d_1$. By ($\romannumeral1$) and ($\romannumeral2$), the probability $P_{7.1}$ that there are four fault-free $B_{n-2}$'s chosen as in Case 7.1 is $\frac{1}{2}(n^4-5n^3+8n^2-4n)p^{4(n-2)!-2(n-3)!}$.

{\it Case 7.2.} $a_1=d_2$ and $a_2=d_1$.

{\it Case 7.2.1.} $c_2\neq a_2$ and $b_1\neq d_1$.

There are $n$ ways to choose $a_1(=d_2=c_1=b_2)$ from $N_n$, $n-1$ ways to choose $a_2(=d_1)$ from $N_n\setminus\{a_1\}$, $n-2$ ways to choose $c_2$ from $N_n\setminus\{c_1, a_2\}$, $n-2$ ways to choose $b_1$ from $N_n\setminus\{b_2, d_1\}$.
Therefore, there are $n(n-1)(n-2)^2=n^4-5n^3+8n^2-4n$ ways to choose four $B_{n-2}$'s as required.
By ($\romannumeral1$) and ($\romannumeral2$), the probability $P_{7.2.1}$ that there are four fault-free $B_{n-2}$'s chosen as in Case 7.2.1 is $\frac{1}{2}(n^4-5n^3+8n^2-4n)p^{4(n-2)!-2(n-3)!}$.

{\it Case 7.2.2.} $c_2\neq a_2$, $b_1=d_1$ or $c_2=a_2$, $b_1\neq d_1$.

W.l.o.g, assume that the former applies.

There are $n$ ways to choose $a_1(=d_2=c_1=b_2)$ from $N_n$, $n-1$ ways to choose $a_2(=d_1=b_2)$ from $N_n\setminus\{a_1\}$, $n-2$ ways to choose $c_2$ from $N_n\setminus\{c_1, a_2\}$.
Therefore, there are $2n(n-1)(n-2)=2(n^3-3n^2+2n)$ ways to choose four $B_{n-2}$'s as required.
By ($\romannumeral1$) and ($\romannumeral3$), the probability $P_{7.2.2}$ that there are four fault-free $B_{n-2}$'s chosen as in Case 7.2.2 is $\frac{1}{2}\times 2(n^3-3n^2+2n)p^{4(n-2)!-(n-3)!}$.

{\it Case 7.2.3.} $c_2=a_2$ and $b_1=d_1$.

There are $n$ ways to choose $a_1(=d_2=c_1=b_2)$ from $N_n$, $n-1$ ways to choose $a_2(=d_1=b_1=c_2)$ from $N_n\setminus\{a_1\}$.
Therefore, there are $n(n-1)=n^2-n$ ways to choose four $B_{n-2}$'s as required.
By ($\romannumeral1$) and ($\romannumeral4$), the probability $P_{7.2.3}$ that there are four fault-free $B_{n-2}$'s chosen as in Case 7.2.3 is $\frac{1}{2}(n^2-n)p^{4(n-2)!}$.

Thus, the probability $P_{7.2}$ that there are four fault-free $B_{n-2}$'s chosen as in Case 7.2 is $P_{7.2.1}+P_{7.2.2}+P_{7.2.3}=\frac{1}{2}(n^2-n)p^{4(n-2)!}+(n^3-3n^2+2n)p^{4(n-2)!-(n-3)!}+\frac{1}{2}(n^4-5n^3+8n^2-4n)p^{4(n-2)!-2(n-3)!}$.

{\it Case 7.3.} $a_1=d_1$ and $a_2\notin\{d_1, d_2\}$ or $a_2=d_2$ and $a_1\notin\{d_1, d_2\}$.

W.l.o.g, assume that the former applies. 

{\it Case 7.3.1.} $c_2\neq a_2$.

There are $n$ ways to choose $c_1(=a_1=d_1)$ from $N_n$, $n-1$ ways to choose $d_2(=b_2)$ from $N_n\setminus\{d_1\}$, $n-2$ ways to choose $a_2$ from $N_n\setminus\{d_1, d_2\}$, $n-2$ ways to choose $b_1$ from $N_n\setminus\{a_1, b_2\}$, $n-3$ ways to choose $c_2$ from $N_n\setminus\{c_1, d_2, a_2\}$.
Therefore, there are $n(n-1)(n-2)^2(n-3)=n^5-8n^4+23n^3-28n^2+12n
$ ways to choose four $B_{n-2}$'s as required.
Clearly, in this case $b_1\neq d_1$. By ($\romannumeral1$) and ($\romannumeral2$), the probability $P_{7.3.1}$ that there are four fault-free $B_{n-2}$'s chosen as in Case 7.3.1 is $\frac{1}{2}(n^5-8n^4+23n^3-28n^2+12n)p^{4(n-2)!-2(n-3)!}$.

{\it Case 7.3.2.} $c_2=a_2$.

There are $n$ ways to choose $c_1(=a_1=d_1)$ from $N_n$, $n-1$ ways to choose $d_2(=b_2)$ from $N_n\setminus\{d_1\}$, $n-2$ ways to choose $a_2(=c_2)$ from $N_n\setminus\{d_1, d_2\}$, $n-2$ ways to choose $b_1$ from $N_n\setminus\{a_1, b_2\}$.
Therefore, there are $n(n-1)(n-2)^2=n^4-5n^3+8n^2-4n
$ ways to choose four $B_{n-2}$'s as required.
Clearly, in this case $b_1\neq d_1$. By ($\romannumeral1$) and ($\romannumeral3$), the probability $P_{7.3.2}$ that there are four fault-free $B_{n-2}$'s chosen as in Case 7.3.2 is $\frac{1}{2}(n^4-5n^3+8n^2-4n)p^{4(n-2)!-(n-3)!}$.

Thus, the probability $P_{7.3}$ that there are four fault-free $B_{n-2}$'s chosen as in Case 7.3 is $2(P_{7.3.1}+P_{7.3.2})=(n^4-5n^3+8n^2-4n)p^{4(n-2)!-(n-3)!}+(n^5-8n^4+23n^3-28n^2+12n)p^{4(n-2)!-2(n-3)!}$.

{\it Case 7.4.} $a_1=d_2$ and $a_2\notin\{d_1, d_2\}$.

{\it Case 7.4.1.} $c_2\neq a_2$ and $b_1\neq d_1$.

There are $n$ ways to choose $a_1(=d_2=c_1=b_2)$ from $N_n$, $n-1$ ways to choose $d_1$ from $N_n\setminus\{d_2\}$, $n-2$ ways to choose $a_2$ from $N_n\setminus\{d_1, d_2\}$, $n-2$ ways to choose $c_2$ from $N_n\setminus\{c_1, a_2\}$, $n-2$ ways to choose $b_1$ from $N_n\setminus\{b_2, d_1\}$.
Therefore, there are $n(n-1)(n-2)^3=n^5-7n^4+18n^3-20n^2+8n$ ways to choose four $B_{n-2}$'s as required.
By ($\romannumeral1$) and ($\romannumeral2$), the probability $P_{7.4.1}$ that there are four fault-free $B_{n-2}$'s chosen as in Case 7.4.1 is $\frac{1}{2}(n^5-7n^4+18n^3-20n^2+8n)p^{4(n-2)!-2(n-3)!}$.

{\it Case 7.4.2.} $c_2\neq a_2$, $b_1=d_1$ or $c_2=a_2$, $b_1\neq d_1$.

W.l.o.g, assume that the former applies.

There are $n$ ways to choose $a_1(=d_2=c_1=b_2)$ from $N_n$, $n-1$ ways to choose $d_1(=b_1)$ from $N_n\setminus\{d_2\}$, $n-2$ ways to choose $a_2$ from $N_n\setminus\{d_1, d_2\}$, $n-2$ ways to choose $c_2$ from $N_n\setminus\{c_1, a_2\}$.
Therefore, there are $2n(n-1)(n-2)^2=2(n^4-5n^3+8n^2-4n)$ ways to choose four $B_{n-2}$'s as required.
By ($\romannumeral1$) and ($\romannumeral3$), the probability $P_{7.4.2}$ that there are four fault-free $B_{n-2}$'s chosen as in Case 7.4.2 is $\frac{1}{2}\times 2(n^4-5n^3+8n^2-4n)p^{4(n-2)!-(n-3)!}$.

{\it Case 7.4.3.} $c_2=a_2$ and $b_1=d_1$.

There are $n$ ways to choose $a_1(=d_2=c_1=b_2)$ from $N_n$, $n-1$ ways to choose $d_1(=b_1)$ from $N_n\setminus\{d_2\}$, $n-2$ ways to choose $a_2(=c_2)$ from $N_n\setminus\{d_1, d_2\}$.
Therefore, there are $n(n-1)(n-2)=n^3-3n^2+2n$ ways to choose four $B_{n-2}$'s as required.
By ($\romannumeral1$) and ($\romannumeral4$), the probability $P_{7.4.3}$ that there are four fault-free $B_{n-2}$'s chosen as in Case 7.4.3 is $\frac{1}{2}(n^3-3n^2+2n)p^{4(n-2)!}$.

Thus, the probability $P_{7.4}$ that there are four fault-free $B_{n-2}$'s chosen as in Case 7.4 is $P_{7.4.1}+P_{7.4.2}+P_{7.4.3}=\frac{1}{2}(n^3-3n^2+2n)p^{4(n-2)!}+(n^4-5n^3+8n^2-4n)p^{4(n-2)!-(n-3)!}+\frac{1}{2}(n^5-7n^4+18n^3-20n^2+8n)p^{4(n-2)!-2(n-3)!}$.

{\it Case 7.5.} $a_2=d_1$ and $a_1\notin\{d_1, d_2\}$.

{\it Case 7.5.1.} $c_2\neq a_2$ and $b_1\neq d_1$.

There are $n$ ways to choose $a_2(=d_1)$ from $N_n$, $n-1$ ways to choose $d_2(=b_2)$ from $N_n\setminus\{d_1\}$, $n-2$ ways to choose $a_1(=c_1)$ from $N_n\setminus\{d_1, d_2\}$, $n-3$ ways to choose $b_1$ from $N_n\setminus\{b_2, a_1, d_1\}$, $n-3$ ways to choose $c_2$ from $N_n\setminus\{c_1, d_2, a_2\}$.
Therefore, there are $n(n-1)(n-2)(n-3)^2=n^5-9n^4+29n^3-39n^2+18n$ ways to choose four $B_{n-2}$'s as required.
By ($\romannumeral1$) and ($\romannumeral2$), the probability $P_{7.5.1}$ that there are four fault-free $B_{n-2}$'s chosen as in Case 7.5.1 is $\frac{1}{2}(n^5-9n^4+29n^3-39n^2+18n)p^{4(n-2)!-2(n-3)!}$.

{\it Case 7.5.2.} $c_2\neq a_2$, $b_1=d_1$ or $c_2=a_2$, $b_1\neq d_1$.

W.l.o.g, assume that the former applies.

There are $n$ ways to choose $a_2(=d_1=b_1)$ from $N_n$, $n-1$ ways to choose $d_2(=b_2)$ from $N_n\setminus\{d_1\}$, $n-2$ ways to choose $a_1(=c_1)$ from $N_n\setminus\{d_1, d_2\}$, $n-3$ ways to choose $c_2$ from $N_n\setminus\{c_1, d_2, a_2\}$.
Therefore, there are $2n(n-1)(n-2)(n-3)=2(n^4-6n^3+11n^2-6n)$ ways to choose four $B_{n-2}$'s as required.
By ($\romannumeral1$) and ($\romannumeral3$), the probability $P_{7.5.2}$ that there are four fault-free $B_{n-2}$'s chosen as in Case 7.5.2 is $\frac{1}{2}\times 2(n^4-6n^3+11n^2-6n)p^{4(n-2)!-(n-3)!}$.

{\it Case 7.5.3.} $c_2=a_2$ and $b_1=d_1$.

There are $n$ ways to choose $a_2(=d_1=b_1=c_2)$ from $N_n$, $n-1$ ways to choose $d_2(=b_2)$ from $N_n\setminus\{d_1\}$, $n-2$ ways to choose $a_1(=c_1)$ from $N_n\setminus\{d_1, d_2\}$.
Therefore, there are $n(n-1)(n-2)=n^3-3n^2+2n$ ways to choose four $B_{n-2}$'s as required.
By ($\romannumeral1$) and ($\romannumeral4$), the probability $P_{7.5.3}$ that there are four fault-free $B_{n-2}$'s chosen as in Case 7.5.3 is $\frac{1}{2}(n^3-3n^2+2n)p^{4(n-2)!}$.

Thus, the probability $P_{7.5}$ that there are four fault-free $B_{n-2}$'s chosen as in Case 7.5 is $P_{7.5.1}+P_{7.5.2}+P_{7.5.3}=\frac{1}{2}(n^3-3n^2+2n)p^{4(n-2)!}+(n^4-6n^3+11n^2-6n)p^{4(n-2)!-(n-3)!}+\frac{1}{2}(n^5-9n^4+29n^3-39n^2+18n)p^{4(n-2)!-2(n-3)!}$.

In summary, the probability $P_7$ that there are four fault-free $B_{n-2}$'s chosen as in Case 7 is $2\sum_{i=1}^{5}P_{7.i}=(2n^3-5n^2+3n)p^{4(n-2)!}+6(n^4-5n^3+8n^2-4n)p^{4(n-2)!-(n-3)!}+(4n^5-30n^4+83n^3-99n^2+42n)p^{4(n-2)!-2(n-3)!}$.

{\it Case 8.} For $\{a_1, a_2\}\cap\{d_1, d_2\}\neq\emptyset$ and arbitrary $\{i, j\}=\{b, c\}$, $i_1\neq j_1=a_1$, $b_2=c_2\neq d_2$ or $i_2\neq j_2=d_2$, $b_1=c_1\neq a_1$.

There are 4 possible cases, w.l.o.g, assume that $\{a_1, a_2\}\cap\{d_1, d_2\}\neq\emptyset$, $b_1\neq c_1=a_1$ and $b_2=c_2\neq d_2$.

The following facts will be applied directly in this case.

($\romannumeral1$) Observation \ref{ob:4} implies that $a_1 a_2 X^{n-2}$ and $X^{n-2} d_1 d_2$ are disjoint. 

($\romannumeral2$) If $a_2\neq c_2$, Observations \ref{ob:2} and \ref{ob:3} imply that $a_1 a_2 X^{n-2}$ and $b_1 X^{n-2} b_2$ are disjoint, and $a_1 a_2 X^{n-2}$ and $c_1 X^{n-2} c_2$ are not disjoint and $V(a_1 a_2 X^{n-2})\cap V(c_1 X^{n-2} c_2)=V(a_1 a_2 X^{n-3} c_2)$, and for any $i\in\{b, c\}$, $X^{n-2} d_1 d_2$ and $i_1 X^{n-2} i_2$ are disjoint. Thus, $|V(a_1 a_2 X^{n-2})\cup V(b_1 X^{n-2} b_2)\cup V(c_1 X^{n-2} c_2)\cup V(X^{n-2} d_1 d_2)|=|V(a_1 a_2 X^{n-2})|+|V(b_1 X^{n-2} b_2)|+|V(c_1 X^{n-2} c_2)|+|V(X^{n-2} d_1 d_2)|-|V(a_1 a_2 X^{n-2})\cap V(c_1 X^{n-2} c_2)|=4(n-2)!-(n-3)!$ in this case. 

($\romannumeral3$) If $a_2=c_2$, Observations \ref{ob:2} and \ref{ob:3} imply that for any $i\in\{b, c\}$, $a_1 a_2 X^{n-2}$ and $i_1 X^{n-2} i_2$ are disjoint, and $X^{n-2} d_1 d_2$ and $i_1 X^{n-2} i_2$ are disjoint. Thus, $|V(a_1 a_2 X^{n-2})\cup V(b_1 X^{n-2} b_2)\cup V(c_1 X^{n-2} c_2)\cup V(X^{n-2} d_1 d_2)|=|V(a_1 a_2 X^{n-2})|+|V(b_1 X^{n-2} b_2)|+|V(c_1 X^{n-2} c_2)|+|V(X^{n-2} d_1 d_2)|=4(n-2)!$ in this case.

{\it Case 8.1.} $a_1=d_1$ and $a_2=d_2$.

There are $n$ ways to choose $c_1(=a_1=d_1)$ from $N_n$, $n-1$ ways to choose $a_2(=d_2)$ from $N_n\setminus\{a_1\}$, $n-2$ ways to choose $b_2(=c_2)$ from $N_n\setminus\{c_1, d_2\}$, $n-3$ ways to choose $b_1$ from $N_n\setminus\{a_1, b_2\}$.
Therefore, there are $n(n-1)(n-2)^2=n^4-5n^3+8n^2-4n$ ways to choose four $B_{n-2}$'s as required. 
Clearly, in this case $a_2\neq c_2$. By ($\romannumeral1$) and ($\romannumeral2$), the probability $P_{8.1}$ that there are four fault-free $B_{n-2}$'s chosen as in Case 8.1 is $\frac{1}{2}(n^4-5n^3+8n^2-4n)p^{4(n-2)!-(n-3)!}$.

{\it Case 8.2.} $a_1=d_2$ and $a_2=d_1$.

{\it Case 8.2.1.} $a_2\neq c_2$.

There are $n$ ways to choose $a_1(=d_2=c_1)$ from $N_n$, $n-1$ ways to choose $a_2(=d_1)$ from $N_n\setminus\{a_1\}$, $n-2$ ways to choose $b_2(=c_2)$ from $N_n\setminus\{c_1, a_2\}$, $n-2$ ways to choose $b_1$ from $N_n\setminus\{a_1, b_2\}$.
Therefore, there are $n(n-1)(n-2)^2=n^4-5n^3+8n^2-4n$ ways to choose four $B_{n-2}$'s as required. 
By ($\romannumeral1$) and ($\romannumeral2$), the probability $P_{8.2.1}$ that there are four fault-free $B_{n-2}$'s chosen as in Case 8.2.1 is $\frac{1}{2}(n^4-5n^3+8n^2-4n)p^{4(n-2)!-(n-3)!}$.

{\it Case 8.2.2.} $a_2=c_2$.

There are $n$ ways to choose $a_1(=d_2=c_1)$ from $N_n$, $n-1$ ways to choose $a_2(=d_1=b_2=c_2)$ from $N_n\setminus\{a_1\}$, $n-2$ ways to choose $b_1$ from $N_n\setminus\{a_1, b_2\}$.
Therefore, there are $n(n-1)(n-2)=n^3-3n^2+2n$ ways to choose four $B_{n-2}$'s as required. 
By ($\romannumeral1$) and ($\romannumeral3$), the probability $P_{8.2.2}$ that there are four fault-free $B_{n-2}$'s chosen as in Case 8.2.2 is $\frac{1}{2}(n^3-3n^2+2n)p^{4(n-2)!}$.

Thus, the probability $P_{8.2}$ that there are four fault-free $B_{n-2}$'s chosen as in Case 8.2 is $P_{8.2.1}+P_{8.2.2}=\frac{1}{2}(n^3-3n^2+2n)p^{4(n-2)!}+\frac{1}{2}(n^4-5n^3+8n^2-4n)p^{4(n-2)!-(n-3)!}$.

{\it Case 8.3.} $a_1=d_1$, $a_2\notin\{d_1, d_2\}$ or $a_2=d_1$, $a_1\notin\{d_1, d_2\}$.

W.l.o.g, assume that the former applies.

{\it Case 8.3.1.} $a_2\neq c_2$.

There are $n$ ways to choose $a_1(=d_1=c_1)$ from $N_n$, $n-1$ ways to choose $d_2$ from $N_n\setminus\{d_1\}$, $n-2$ ways to choose $a_2$ from $N_n\setminus\{d_1, d_2\}$, $n-3$ ways to choose $b_2(=c_2)$ from $N_n\setminus\{d_2, c_1, a_2\}$, $n-2$ ways to choose $b_1$ from $N_n\setminus\{a_1, b_2\}$.
Therefore, there are $n(n-1)(n-2)^2(n-3)=n^5-8n^4+23n^3-28n^2+12n$ ways to choose four $B_{n-2}$'s as required. 
By ($\romannumeral1$) and ($\romannumeral2$), the probability $P_{8.3.1}$ that there are four fault-free $B_{n-2}$'s chosen as in Case 8.3.1 is $\frac{1}{2}(n^5-8n^4+23n^3-28n^2+12n)p^{4(n-2)!-(n-3)!}$.

{\it Case 8.3.2.} $a_2=c_2$.

There are $n$ ways to choose $a_1(=d_1=c_1)$ from $N_n$, $n-1$ ways to choose $d_2$ from $N_n\setminus\{d_1\}$, $n-2$ ways to choose $a_2(=b_2=c_2)$ from $N_n\setminus\{d_1, d_2\}$, $n-2$ ways to choose $b_1$ from $N_n\setminus\{a_1, b_2\}$.
Therefore, there are $n(n-1)(n-2)^2=n^4-5n^3+8n^2-4n$ ways to choose four $B_{n-2}$'s as required. 
By ($\romannumeral1$) and ($\romannumeral3$), the probability $P_{8.3.2}$ that there are four fault-free $B_{n-2}$'s chosen as in Case 8.3.2 is $\frac{1}{2}(n^4-5n^3+8n^2-4n)p^{4(n-2)!}$.

Thus, the probability $P_{8.3}$ that there are four fault-free $B_{n-2}$'s chosen as in Case 8.3 is $2(P_{8.3.1}+P_{8.3.2})=(n^4-5n^3+8n^2-4n)p^{4(n-2)!}+(n^5-8n^4+23n^3-28n^2+12n)p^{4(n-2)!-(n-3)!}$.

{\it Case 8.4.} $a_1=d_2$, $a_2\notin\{d_1, d_2\}$.

{\it Case 8.4.1.} $a_2\neq c_2$.

There are $n$ ways to choose $a_1(=d_2=c_1)$ from $N_n$, $n-1$ ways to choose $d_1$ from $N_n\setminus\{d_2\}$, $n-2$ ways to choose $a_2$ from $N_n\setminus\{d_1, d_2\}$, $n-2$ ways to choose $b_2(=c_2)$ from $N_n\setminus\{d_2, a_2\}$, $n-2$ ways to choose $b_1$ from $N_n\setminus\{a_1, b_2\}$.
Therefore, there are $n(n-1)(n-2)^3=n^5-7n^4+18n^3-20n^2+8n$ ways to choose four $B_{n-2}$'s as required. 
By ($\romannumeral1$) and ($\romannumeral2$), the probability $P_{8.4.1}$ that there are four fault-free $B_{n-2}$'s chosen as in Case 8.4.1 is $\frac{1}{2}(n^5-7n^4+18n^3-20n^2+8n)p^{4(n-2)!-(n-3)!}$.

{\it Case 8.4.2.} $a_2=c_2$.

There are $n$ ways to choose $a_1(=d_2=c_1)$ from $N_n$, $n-1$ ways to choose $d_1$ from $N_n\setminus\{d_2\}$, $n-2$ ways to choose $a_2(=b_2=c_2)$ from $N_n\setminus\{d_1, d_2\}$, $n-2$ ways to choose $b_1$ from $N_n\setminus\{a_1, b_2\}$.
Therefore, there are $n(n-1)(n-2)^2=n^4-5n^3+8n^2-4n$ ways to choose four $B_{n-2}$'s as required. 
By ($\romannumeral1$) and ($\romannumeral3$), the probability $P_{8.4.2}$ that there are four fault-free $B_{n-2}$'s chosen as in Case 8.4.2 is $\frac{1}{2}(n^4-5n^3+8n^2-4n)p^{4(n-2)!}$.

Thus, the probability $P_{8.4}$ that there are four fault-free $B_{n-2}$'s chosen as in Case 8.4 is $P_{8.4.1}+P_{8.4.2}=\frac{1}{2}(n^4-5n^3+8n^2-4n)p^{4(n-2)!}+\frac{1}{2}(n^5-7n^4+18n^3-20n^2+8n)p^{4(n-2)!-(n-3)!}$.

{\it Case 8.5.} $a_2=d_2$, $a_1\notin\{d_1, d_2\}$.

There are $n$ ways to choose $a_2(=d_2)$ from $N_n$, $n-1$ ways to choose $d_1$ from $N_n\setminus\{d_2\}$, $n-2$ ways to choose $a_1(=c_1)$ from $N_n\setminus\{d_1, d_2\}$, $n-2$ ways to choose $b_2(=c_2)$ from $N_n\setminus\{d_2, c_1\}$, $n-2$ ways to choose $b_1$ from $N_n\setminus\{a_1, b_2\}$.
Therefore, there are $n(n-1)(n-2)^3=n^5-7n^4+18n^3-20n^2+8n$ ways to choose four $B_{n-2}$'s as required. 
Clearly, in this case $a_2\neq c_2$. By ($\romannumeral1$) and ($\romannumeral2$), the probability $P_{8.5}$ that there are four fault-free $B_{n-2}$'s chosen as in Case 8.5 is $\frac{1}{2}(n^5-7n^4+18n^3-20n^2+8n)p^{4(n-2)!-(n-3)!}$.

In summary, the probability $P_8$ that there are four fault-free $B_{n-2}$'s chosen as in Case 8 is $4\sum_{i=1}^{5}P_{8.i}=2(3n^4-14n^3+21n^2-10n)p^{4(n-2)!}+8(n^5-7n^4+18n^3-20n^2+8n)p^{4(n-2)!-(n-3)!}$.

{\it Case 9.} For $\{a_1, a_2\}\cap\{d_1, d_2\}\neq\emptyset$ and arbitrary $\{i, j\}=\{b, c\}$, $i_1\neq j_1=a_1$, $b_2=c_2=d_2$ or $b_1=c_1=a_1$, $i_2\neq j_2=d_2$.

There are 4 possible cases, w.l.o.g, assume that $\{a_1, a_2\}\cap\{d_1, d_2\}\neq\emptyset$, $b_1\neq c_1=a_1$ and $b_2=c_2=d_2$.

The following facts will be applied directly in this case.

($\romannumeral1$) Observation \ref{ob:4} implies that $a_1 a_2 X^{n-2}$ and $X^{n-2} d_1 d_2$ are disjoint. 

($\romannumeral2$) If $a_2\neq c_2$ and $b_1\neq d_1$ and $c_1\neq d_1$, Observations \ref{ob:2} and \ref{ob:3} imply that $a_1 a_2 X^{n-2}$ and $b_1 X^{n-2} b_2$ are disjoint, and $a_1 a_2 X^{n-2}$ and $c_1 X^{n-2} c_2$ are not disjoint and $V(a_1 a_2 X^{n-2})\cap V(c_1 X^{n-2} c_2)=V(a_1 a_2 X^{n-3} c_2)$, and for any $i\in\{b, c\}$, $X^{n-2} d_1 d_2$ and $i_1 X^{n-2} i_2$ are not disjoint and $V(X^{n-2} d_1 d_2)\cap V(i_1 X^{n-2} i_2)=V(i_1 X^{n-3} d_1 d_2)$. Thus, $|V(a_1 a_2 X^{n-2})\cup V(b_1 X^{n-2} b_2)\cup V(c_1 X^{n-2} c_2)\cup V(X^{n-2} d_1 d_2)|=|V(a_1 a_2 X^{n-2})|+|V(b_1 X^{n-2} b_2)|+|V(c_1 X^{n-2} c_2)|+|V(X^{n-2} d_1 d_2)|-|V(a_1 a_2 X^{n-2})\cap V(c_1 X^{n-2} c_2)|-|V(X^{n-2} d_1 d_2)\cap V(b_1 X^{n-2} b_2)|-|V(X^{n-2} d_1$ $d_2)\cap V(c_1 X^{n-2} c_2)|=4(n-2)!-3(n-3)!$.

($\romannumeral3$) If $a_2\neq c_2$, $b_1\neq d_1$, $c_1=d_1$ or $a_2\neq c_2$, $b_1=d_1$, $c_1\neq d_1$ or $a_2=c_2$, $b_1\neq d_1$, $c_1\neq d_1$, w.l.o.g, assume that the first case applies. Observations \ref{ob:2} and \ref{ob:3} imply that $a_1 a_2 X^{n-2}$ and $b_1 X^{n-2} b_2$ are disjoint, and $a_1 a_2 X^{n-2}$ and $c_1 X^{n-2} c_2$ are not disjoint and $V(a_1 a_2 X^{n-2})\cap V(c_1 X^{n-2} c_2)=V(a_1 a_2 X^{n-3} c_2)$, and $X^{n-2} d_1 d_2$ and $b_1 X^{n-2} b_2$ are not disjoint and $V(X^{n-2} d_1 d_2)\cap V(b_1 X^{n-2} b_2)=V(b_1 X^{n-3} d_1 d_2)$, and $X^{n-2} d_1 d_2$ and $c_1 X^{n-2} c_2$ are disjoint. Thus, $|V(a_1 a_2 X^{n-2})\cup V(b_1 X^{n-2} b_2)\cup V(c_1 X^{n-2} c_2)\cup V(X^{n-2} d_1 d_2)|=|V(a_1 a_2 X^{n-2})|+|V(b_1 X^{n-2} b_2)|+|V(c_1 X^{n-2} c_2)|+|V(X^{n-2} d_1 d_2)|-|V(a_1 a_2 X^{n-2})\cap V(c_1 X^{n-2} c_2)|-|V(X^{n-2} d_1 d_2)\cap V(b_1 X^{n-2}$ $b_2)|=4(n-2)!-2(n-3)!$.

($\romannumeral4$) If $a_2=c_2$, $b_1\neq d_1$, $c_1=d_1$ or $a_2=c_2$, $b_1=d_1$, $c_1\neq d_1$, w.l.o.g, assume that the former applies. Observations \ref{ob:2} and \ref{ob:3} imply that for any $i\in\{b, c\}$, $a_1 a_2 X^{n-2}$ and $i_1 X^{n-2} i_2$ are disjoint, and $X^{n-2} d_1 d_2$ and $b_1 X^{n-2} b_2$ are disjoint, and $X^{n-2} d_1 d_2$ and $c_1 X^{n-2} c_2$ are not disjoint and $V(X^{n-2} d_1 d_2)\cap V(c_1 X^{n-2} c_2)=V(c_1 X^{n-3} d_1 d_2)$. Thus, $|V(a_1 a_2 X^{n-2})\cup V(b_1 X^{n-2} b_2)\cup V(c_1 X^{n-2} c_2)\cup V(X^{n-2} d_1 d_2)|=|V(a_1 a_2 X^{n-2})|+|V(b_1 X^{n-2} b_2)|+|V(c_1 X^{n-2} c_2)|+|V(X^{n-2} d_1 d_2)|-|V(X^{n-2} d_1 d_2)\cap V(c_1 X^{n-2} c_2)|=4(n-2)!-(n-3)!$.

{\it Case 9.1.} $a_1=d_1$ and $a_2=d_2$.

There are $n$ ways to choose $c_1(=a_1=d_1)$ from $N_n$, $n-1$ ways to choose $a_2(=d_2=b_2=c_2)$ from $N_n\setminus\{a_1\}$, $n-2$ ways to choose $b_1$ from $N_n\setminus\{a_1, b_2\}$.
Therefore, there are $n(n-1)(n-2)=n^3-3n^2+2n$ ways to choose four $B_{n-2}$'s as required.
Clearly, in this case $a_2=c_2$, $b_1\neq d_1$ and $c_1=d_1$. By ($\romannumeral1$) and ($\romannumeral4$), the probability $P_{9.1}$ that there are four fault-free $B_{n-2}$'s chosen as in Case 9.1 is $\frac{1}{2}(n^3-3n^2+2n)p^{4(n-2)!-(n-3)!}$.

{\it Case 9.2.} $a_1=d_1$ and $a_2\notin\{d_1, d_2\}$.

There are $n$ ways to choose $c_1(=a_1=d_1)$ from $N_n$, $n-1$ ways to choose $b_2(=c_2=d_2)$ from $N_n\setminus\{d_1\}$, $n-2$ ways to choose $a_2$ from $N_n\setminus\{d_1, d_2\}$, $n-2$ ways to choose $b_1$ from $N_n\setminus\{a_1, b_2\}$.
Therefore, there are $n(n-1)(n-2)^2=n^4-5n^3+8n^2-4n$ ways to choose four $B_{n-2}$'s as required.
Clearly, in this case $a_2\neq c_2$, $b_1\neq d_1$ and $c_1=d_1$. By ($\romannumeral1$) and ($\romannumeral3$), the probability $P_{9.2}$ that there are four fault-free $B_{n-2}$'s chosen as in Case 9.2 is $\frac{1}{2}(n^4-5n^3+8n^2-4n)p^{4(n-2)!-2(n-3)!}$.

{\it Case 9.3.} $a_2=d_1$ and $a_1\notin\{d_1, d_2\}$.

Clearly, in this case $a_2\neq c_2$ and $c_1\neq d_1$.

{\it Case 9.3.1.} $b_1\neq d_1$.

There are $n$ ways to choose $a_2(=d_1)$ from $N_n$, $n-1$ ways to choose $b_2(=c_2=d_2)$ from $N_n\setminus\{d_1\}$, $n-2$ ways to choose $a_1(=c_1)$ from $N_n\setminus\{d_1, d_2\}$, $n-3$ ways to choose $b_1$ from $N_n\setminus\{a_1, b_2, d_1\}$.
Therefore, there are $n(n-1)(n-2)(n-3)=n^4-6n^3+11n^2-6n$ ways to choose four $B_{n-2}$'s as required.
By ($\romannumeral1$) and ($\romannumeral2$), the probability $P_{9.3.1}$ that there are four fault-free $B_{n-2}$'s chosen as in Case 9.3.1 is $\frac{1}{2}(n^4-6n^3+11n^2-6n)p^{4(n-2)!-3(n-3)!}$.

{\it Case 9.3.2.} $b_1=d_1$.

There are $n$ ways to choose $a_2(=d_1=b_1)$ from $N_n$, $n-1$ ways to choose $b_2(=c_2=d_2)$ from $N_n\setminus\{d_1\}$, $n-2$ ways to choose $a_1(=c_1)$ from $N_n\setminus\{d_1, d_2\}$.
Therefore, there are $n(n-1)(n-2)=n^3-3n^2+2n$ ways to choose four $B_{n-2}$'s as required.
By ($\romannumeral1$) and ($\romannumeral2$), the probability $P_{9.3.2}$ that there are four fault-free $B_{n-2}$'s chosen as in Case 9.3.2 is $\frac{1}{2}(n^3-3n^2+2n)p^{4(n-2)!-2(n-3)!}$.

{\it Case 9.4.} $a_2=d_2$ and $a_1\notin\{d_1, d_2\}$.

Clearly, in this case $a_2=c_2$ and $c_1\neq d_1$.

{\it Case 9.4.1.} $b_1\neq d_1$.

There are $n$ ways to choose $a_2(=d_2=b_2=c_2)$ from $N_n$, $n-1$ ways to choose $d_1$ from $N_n\setminus\{d_2\}$, $n-2$ ways to choose $a_1(=c_1)$ from $N_n\setminus\{d_1, d_2\}$, $n-3$ ways to choose $b_1$ from $N_n\setminus\{a_1, b_2, d_1\}$.
Therefore, there are $n(n-1)(n-2)(n-3)=n^4-6n^3+11n^2-6n$ ways to choose four $B_{n-2}$'s as required.
By ($\romannumeral1$) and ($\romannumeral3$), the probability $P_{9.4.1}$ that there are four fault-free $B_{n-2}$'s chosen as in Case 9.4.1 is $\frac{1}{2}(n^4-6n^3+11n^2-6n)p^{4(n-2)!-2(n-3)!}$.

{\it Case 9.4.2.} $b_1=d_1$.

There are $n$ ways to choose $a_2(=d_2=b_2=c_2)$ from $N_n$, $n-1$ ways to choose $d_1(=b_1)$ from $N_n\setminus\{d_2\}$, $n-2$ ways to choose $a_1(=c_1)$ from $N_n\setminus\{d_1, d_2\}$.
Therefore, there are $n(n-1)(n-2)=n^3-3n^2+2n$ ways to choose four $B_{n-2}$'s as required.
By ($\romannumeral1$) and ($\romannumeral4$), the probability $P_{9.4.2}$ that there are four fault-free $B_{n-2}$'s chosen as in Case 9.4.2 is $\frac{1}{2}(n^3-3n^2+2n)p^{4(n-2)!-(n-3)!}$.

In summary, the probability $P_9$ that there are four fault-free $B_{n-2}$'s chosen as in Case 9 is $4(P_{9.1}+P_{9.2}+P_{9.3.1}+P_{9.3.2}+P_{9.4.1}+P_{9.4.2})=4(n^3-3n^2+2n)p^{4(n-2)!-(n-3)!}+4(n^4-5n^3+8n^2-4n)p^{4(n-2)!-2(n-3)!}+2(n^4-6n^3+11n^2-6n)p^{4(n-2)!-3(n-3)!}$.

By the above computations, $P(1,2,1)=\sum_{i=1}^{9}P_i=(2n^7-17n^6+65n^5-139n^4+173n^3-118n^2+34n)p^{4(n-2)!}+(8n^6-61n^5+199n^4-347n^3+315n^2-114n)p^{4(n-2)!-(n-3)!}+(n^6-32n^4+112n^3-143n^2+62n)p^{4(n-2)!-2(n-3)!}+2(n^5-8n^4+23n^3-28n^2+12n)p^{4(n-2)!-3(n-3)!}+\frac{1}{2}(n^8-12n^7+65n^6-207n^5+416n^4-525n^3+382n^2-120n)p^{4(n-2)!-(n-4)!}+2(n^7-12n^6+61n^5-171n^4+280n^3-249n^2+90n)p^{4(n-2)!-(n-3)!-(n-4)!}+(2n^6-25n^5+121n^4-281n^3+309n^2-126n)p^{4(n-2)!-2(n-3)!-(n-4)!}$.
\end{proof}

\bibliographystyle{model1-num-names}

\end{document}